\documentclass[11pt,reqno]{amsart}

\usepackage[latin1]{inputenc}
\usepackage{amsmath,amsthm,amssymb,amscd,mathrsfs}
\usepackage{mathtools,color,esint,bm}
\usepackage{booktabs}
\usepackage[shortlabels]{enumitem}
\usepackage[perpage]{footmisc}
\usepackage{subfig,caption}

\definecolor{darkgreen}{rgb}{0.0, 0.6, 0.13}

\usepackage{hyperref}

\newtheorem{thm}{Theorem}[section]
 \newtheorem{cor}[thm]{Corollary}
 \newtheorem{lem}[thm]{Lemma}
 \newtheorem{prop}[thm]{Proposition}

 \theoremstyle{definition}
 \newtheorem{df}[thm]{Definition}
 \theoremstyle{remark}
 \newtheorem{rem}[thm]{Remark}
 
 \numberwithin{equation}{section}

\newcommand{\Ab}{\mathbb A}
\newcommand{\Bb}{\mathbb B}
\newcommand{\Cb}{\mathbb C}
\newcommand{\Db}{\mathbb D}
\newcommand{\Eb}{\mathbb E}

\newcommand{\Mb}{\mathbb M}
\newcommand{\Nb}{\mathbb N}

\newcommand{\Pb}{\mathbb P}
\newcommand{\Qb}{\mathbb Q}
\newcommand{\Rb}{\mathbb R}

\newcommand{\Tb}{\mathbb T}

\newcommand{\Zb}{\mathbb Z}
\newcommand{\Ac}{\mathcal A}
\newcommand{\Bc}{\mathcal B}
\newcommand{\Cc}{\mathcal C}
\newcommand{\Dc}{\mathcal D}
\newcommand{\Ec}{\mathcal E}
\newcommand{\Fc}{\mathcal F}
\newcommand{\Gc}{\mathcal{G}}
\newcommand{\Hc}{\mathcal H}
\newcommand{\Ic}{\mathcal I}
\newcommand {\Jc}{\mathcal{J}}
\newcommand{\Kc}{\mathcal K}
\newcommand{\Lc}{\mathcal L}
\renewcommand{\Mc}{\mathcal M}
\newcommand{\Nc}{\mathcal N}

\newcommand{\Pc}{\mathcal P}
\newcommand{\Qc}{\mathcal Q}
\newcommand{\Rc}{\mathcal R}
\newcommand{\Sc}{\mathcal S}
\newcommand{\Tc}{\mathcal T}
\newcommand{\Uc}{\mathcal U}
\newcommand{\Vc}{\mathcal V}
\newcommand{\Wc}{\mathcal W}
\newcommand{\Xc}{\mathcal X}
\newcommand{\Yc}{\mathcal Y}
\newcommand{\Zc}{\mathcal Z}
\newcommand{\As}{\mathscr A}
\newcommand{\Bs}{\mathscr B}
\newcommand{\Cs}{\mathscr C}
\newcommand{\Ds}{\mathscr D}
\newcommand{\Es}{\mathscr E}
\newcommand{\Fs}{\mathscr F}

\newcommand{\Is}{\mathscr I}

\newcommand{\Ks}{\mathscr K}
\newcommand{\Ls}{\mathscr L}
\newcommand{\Ms}{\mathscr M}

\newcommand{\Ps}{\mathscr P}
\newcommand{\Qs}{\mathscr Q}
\newcommand{\Rs}{\mathscr R}

\newcommand{\Xs}{\mathscr X}
\newcommand{\Ys}{\mathscr Y}
\newcommand{\Zs}{\mathscr Z}

\newcommand{\Df}{\mathfrak D}

\newcommand{\Lf}{\mathfrak L}

\newcommand{\Nf}{\mathfrak N}

\newcommand{\Rf}{\mathfrak R}
\newcommand{\Sf}{\mathfrak S}

\newcommand{\Zf}{\mathfrak Z}

\newcommand{\bff}{\mathfrak b}
\newcommand{\cf}{\mathfrak c}

\newcommand{\ef}{\mathfrak e}
\newcommand{\ff}{\mathfrak f}

\newcommand{\hf}{\mathfrak h}

\newcommand{\lf}{\mathfrak l}
\newcommand{\mf}{\mathfrak m}
\newcommand{\nf}{\mathfrak n}

\newcommand{\pf}{\mathfrak p}
\newcommand{\qf}{\mathfrak q}
\newcommand{\rf}{\mathfrak r}

\newcommand{\vx}{\boldsymbol{x}}
\newcommand{\vy}{\boldsymbol{y}}

\newcommand{\vg}{\boldsymbol{g}}
\newcommand{\vh}{\boldsymbol{h}}
\newcommand{\vs}{\boldsymbol{s}}
\newcommand{\vxi}{\boldsymbol{\xi}}
\newcommand{\dirac}{\boldsymbol{\delta}}
\newcommand{\veta}{\boldsymbol{\eta}}

\newcommand{\vOmega}{\boldsymbol{\Omega}}
\newcommand{\valpha}{\boldsymbol{\alpha}}
\newcommand{\vbeta}{\boldsymbol{\beta}}
\newcommand{\vsigma}{\boldsymbol{\sigma}}
\newcommand{\vlambda}{\boldsymbol{\lambda}}

\begin{document}
\title{Full derivation of the wave kinetic equation}
\author{Yu Deng}
\address{\textsc{Department of Mathematics, University of Southern California, Los Angeles, CA, USA}}
\email{\texttt{yudeng@usc.edu}}
\author{Zaher Hani}
\address{\textsc{Department of Mathematics, University of Michigan, Ann Arbor, MI, USA}}
\email{\texttt{zhani@umich.edu}}
\begin{abstract} We provide the rigorous derivation of the wave kinetic equation from the cubic nonlinear Schr\"odinger (NLS) equation \emph{at the kinetic timescale}, under a particular \emph{scaling law} that describes the limiting process. This solves a main conjecture in the theory of \emph{wave turbulence}, i.e. the kinetic theory of nonlinear wave systems. Our result is the wave analog of Lanford's theorem on the derivation of the Boltzmann kinetic equation from particle systems, where in both cases one takes the thermodynamic limit as the size of the system diverges to infinity, and as the interaction strength of waves/radius of particles vanishes to $0$, according to a particular scaling law (Boltzmann-Grad in the particle case).

\smallskip
More precisely, in dimensions $d\geq 3$, we consider the (NLS) equation in a large box of size $L$ with a weak nonlinearity of strength $\alpha$. In the limit $L\to\infty$ and $\alpha\to 0$, under the scaling law $\alpha\sim L^{-1}$, we show that the long-time behavior of (NLS) is statistically described by the wave kinetic equation, with well justified approximation, up to times that are $O(1)$ (i.e independent of $L$ and $\alpha$) multiples of the kinetic timescale $T_{\text{kin}}\sim \alpha^{-2}$. This is the first result of its kind for any nonlinear dispersive system.
\end{abstract}

\maketitle
\tableofcontents

\section{Introduction}\label{intro} 
The kinetic theory of nonlinear wave systems is the formal basis of the non-equilibrium statistical physics of such systems. It is an extension of the kinetic framework, first laid out by Boltzmann in the context of particle systems, to nonlinear dispersive systems. The wave kinetic theory can be traced back to the work of Peierls in 1928 on anharmonic crystals \cite{Peierls}, which exhibited the very first wave kinetic equation (the phonon Boltzmann equation).
Soon after, the kinetic framework for waves was widely adopted in plasma theory \cite{Vedenov, Zaslavskii, Davidson,GS79}, water waves \cite{Hass1, Hass2, BennSaff, BennNewell}, and later formalized into a systematic approach to understand the effective long-time behavior of large systems of interacting waves undergoing weak nonlinear interactions \cite{ZLFBook,SpohnBE, Nazarenko}. This kinetic theory for waves came to be known as \emph{wave turbulence theory}, due to its surprising and profound implications on the spectral energy dynamics and cascades for nonlinear wave systems, similar to those made in Kolmogorov's theory of hydrodynamic turbulence.

\medskip 

The central object in wave turbulence theory is the \emph{wave kinetic equation} (WKE), which plays the analogous role of Boltzmann's kinetic equation for particles. The (WKE) was derived, at a heuristic level, in the physics literature to describe the effective behavior of the normal frequency amplitudes of solutions in some statistically averaged sense.
The analogy to Boltzmann's theory also comes from the \emph{thermodynamic limit} involved in both theories: The number of particles $N\to\infty$ in Boltzmann's theory is paralleled by the size $L\to\infty$ of the dispersive system in the wave kinetic theory, and the particle radius $r\to 0$ is paralleled by the strength of nonlinear wave interactions, which we shall denote by $\alpha\to 0$. A \emph{scaling law} is a rule that dictates how these two limits are taken; for example the well-known Boltzmann-Grad limit corresponds to the scaling law $N r^{d-1}\sim 1$ as $N\to \infty$ and $r\to 0$. 

\medskip

From the mathematical viewpoint, the fundamental problem is to give a rigorous justification or derivation of the wave kinetic equation starting from the nonlinear dispersive equation that governs the wave system as a first principle. This is Hilbert's Sixth Problem for the statistical theory of wave systems. It should be said, though, that this question is far from being a mere mathematical curiosity. In fact, it is a question that was posed by physicists as a means to better understand the exact regimes and limitations of the wave kinetic theory \cite{Nazarenko}. The particle analog of this problem is the rigorous derivation of the Boltzmann equation starting from the Newtonian dynamics of particles as a first principle. This was given by Lanford's celebrated theorem \cite{Lanford, CIP,GSRT}, which justifies the derivation in the above-mentioned Boltzmann-Grad scaling law where the particle number $N\to \infty$ and the particle size $r\to 0$ in such a way that $Nr^{d-1}\sim 1$. 

Despite being open for quite some time, progress on this problem for wave systems only started in the past twenty years. In part, this is due to the fact that it relied on techniques that didn't mature until then, like progress in the analysis of probabilistic nonlinear PDE, combinatorics of Feynman diagrams, and in some cases analytic number theory, all of which are components that address various facets of the problem. We shall survey the previous results leading up to this work in Section \ref{background}. In another part, as we shall see and explain below (see Section \ref{CriticalityComment}), the full resolution of this problem is a \emph{probabilistically-critical problem}, and prior to this work, no such result existed even in the parabolic setting.

We consider the nonlinear Schr\"odinger (NLS) equation as a {fundamental and prototypical system in nonlinear wave theory. This is partly due its unique \emph{universality property} in this class, in the sense that any Hamiltonian dispersive system gives (NLS) in an appropriate scaling limit (see \cite{SuSu})}. Our main result is a full rigorous derivation of the wave kinetic equation (WKE) up to $O(1)$ timescales. This means timescales that are independent of the asymptotic parameters involved in the thermodynamic limit, namely the size $L$ of the domain and the strength $\alpha$ of the nonlinearity. For the sake of definiteness, this will be done under the scaling law $\alpha L\sim 1$, which is of particular mathematical interest as we shall explain later. However, our approach is fairly general and allows treating some other scaling laws with minor modifications (cf. Section \ref{SLComment}).
\subsection{Statement of the main result} \subsubsection{(NLS) as the microscopic system}In dimension $d\geq 3$, consider the cubic nonlinear Schr\"{o}dinger equation
$$
(i\partial_t-\Delta)w+|w|^2w=0
$$
on a generic irrational torus of size $L\gg 1$. For convenience, we will adjust by dilations and work equivalently on the square torus $\Tb_L^d=[0,L]^d$ of size $L$, but with the twisted Laplacian
\begin{equation}\label{notations}\Delta_\beta=(2\pi)^{-1}(\beta^1 \partial_1^2+\cdots+\beta^d\partial_d^2).
\end{equation} Here $(2\pi)^{-1}$ is a normalizing constant, and $\beta=(\beta^1,\cdots,\beta^d)\in(\Rb^+)^d$ represents the aspect ratios of the torus. We assume $\beta$ is \emph{generic}, i.e. belongs to the complement of some Lebesgue null set $\Zf$, which is fixed by a set of explicit Diophantine conditions, stated precisely in Lemma \ref{genericity}. We will comment in Section \ref{SLComment} below in more detail on the necessity of this genericity condition, but roughly speaking, it is necessary for some scaling laws, including the one we impose in this paper, due to some number theoretic considerations. Other scaling laws, some of which can also be covered by our proof, do not require this genericity condition as we shall discuss later.

As mentioned above, the strength of the nonlinearity is the other asymptotic parameter in the wave kinetic theory. Of course, this strength is intimately tied to the size of solutions (say in terms of $L^2$ norm). To emphasize this size, we adopt the ansatz $w=\lambda u$ where $\lambda$ can be thought of as the the conserved $L^2$ norm of $w$. This leads us to study the equation
\begin{equation}\label{nls}\tag{NLS}
\left\{
\begin{split}&(i\partial_t-\Delta_\beta)u+\lambda^2|u|^2u=0,\quad x\in \Tb_L^d=[0,L]^d,\\
&u(0,x)=u_{\mathrm{in}}(x).
\end{split}
\right.
\end{equation} 
The defocusing sign of the nonlinearity adopted here is merely for concreteness purposes. The same results hold for the focusing case; this is due to the weak nonlinearity setting inherent in the wave kinetic theory we study here.

\medskip

The kinetic theory seeks to give the effective dynamics of frequency amplitudes $\mathbb E |\widehat u(t, k)|^2$ where\footnote{Here we note that one has freedom to choose a different normalization of the Fourier transform. We caution that, while this has no effect on the theory, it does change the expression for the strength of the nonlinearity $\alpha$ below, and hence the kinetic timescale $T_{\text{kin}}=1/2\alpha^2$, in terms of $\lambda$ and $L$. {For example, another common normalization is the one that puts $L^{-d/2}$ in front of the Fourier integral; there $\alpha$ would be $\lambda^2$ and $T_{\text{kin}}=1/2\alpha^{2}=1/2\lambda^4$.}} 
\begin{equation}\label{fourier}
\widehat u(t, k) =\int_{\Tb^d_L} u(t, x) e^{-2\pi i k\cdot x} \, dx, \qquad u(t,x) =\frac{1}{L^d}\sum_{k\in\Zb_L^d}\widehat{u}(t,k)e^{2\pi ik\cdot x},
\end{equation}
and the averaging happens over a random distribution of the initial data. Such random distribution is chosen in a way that allows for the kinetic description; we call such data \emph{well-prepared}. More precisely, we consider random homogeneous initial data given by
\begin{equation}
\label{data}\tag{DAT}u_{\mathrm{in}}(x)=\frac{1}{L^d}\sum_{k\in\Zb_L^d}\widehat{u_{\mathrm{in}}}(k)e^{2\pi ik\cdot x},\quad \widehat{u_{\mathrm{in}}}(k)=\sqrt{n_{\mathrm{in}}(k)}\eta_k(\omega),
\end{equation}
where $\Zb_L^d:=(L^{-1}\Zb)^d$, and $n_{\mathrm{in}}:\Rb^d\to[0,\infty)$ is a given Schwartz function, $\{\eta_k(\omega)\}$ is a collection of i.i.d. random variables. We assume that each $\eta_k$ is either a \emph{centered normalized complex Gaussian}, or \emph{uniformly distributed on the unit circle of $\Cb$}. This is sometimes called the \emph{random phase assumption} in the literature \cite{Nazarenko}. For simplicity, in the proof below, we will only consider the Gaussian case; the unimodular case can be treated with minor modifications (see for example Lemma 3.1 of \cite{DH}).

\smallskip

Given such random solutions, we define the \emph{strength of the nonlinearity parameter} to be $\alpha:=\lambda^2L^{-d}$. This nomenclature can be justified, heuristically at this point, by noting that if $u$ is a randomly chosen $L^2(\Tb^d_L)$ function with norm $O(1)$, then with high probability one has that $\|u\|_{L^\infty(\Tb^d_L)}\lesssim L^{-d/2}$, which makes the nonlinearity $\lambda^2 |u|^2u$ of size $\sim \lambda^2 L^{-d}=\alpha$ in $L^2(\Tb^d)$. This heuristic can be directly verified for the well-prepared initial data $u_{\mathrm{in}}$ using Gaussian hypercontractivity estimates, but it will follow from our proof that it is also true for the solution $u(t)$ itself at later timescales of interest to us. 

Finally, we define the \emph{kinetic timescale}
\[T_{\mathrm{kin}}:=\frac{1}{2\alpha^2}=\frac{1}{2}\cdot\frac{L^{2d}}{\lambda^4},\]
which will be the timescale at which the kinetic behavior will start exhibiting itself for (NLS). 

\subsubsection{The wave kinetic equation for NLS} Under the homogeneity assumption on the initial data in \eqref{data} (i.e. the independence of $\widehat u_{\mathrm{in}}(k)$ for different $k$), the relevant wave kinetic equation is also homogeneous (i.e. has no transport term) and is given by:

\begin{equation}\label{wke}\tag{WKE}
\left\{
\begin{split}&\partial_t n(t,k)=\Kc(n(t),n(t),n(t))(k),\\
&n(0,k)=n_{\mathrm{in}}(k),
\end{split}
\right.
\end{equation} where the nonlinearity
\begin{multline}\label{wke2}\tag{KIN}
\Kc(\phi_1,\phi_2,\phi_3)(k)=\int_{(\Rb^d)^3}\big\{\phi_1(k_1)\phi_2(k_2)\phi_3(k_3)-\phi_1(k)\phi_2(k_2)\phi_3(k_3)+\phi_1(k_1)\phi_2(k)\phi_3(k_3)\\-\phi_1(k_1)\phi_2(k_2)\phi_3(k)\big\}\times\dirac(k_1-k_2+k_3-k)\cdot\dirac(|k_1|_\beta^2-|k_2|_\beta^2+|k_3|_\beta^2-|k|_\beta^2)\,\mathrm{d}k_1\mathrm{d}k_2\mathrm{d}k_3.
\end{multline} Here and below $\dirac$ denotes the Dirac delta, and we define
\[|k|_\beta^2:=\langle k,k\rangle_\beta,\quad \langle k,\ell\rangle_\beta:=\beta^1 k^1 \ell^1+\cdots +\beta^dk^d\ell^d,\] where $k=(k^1,\cdots,k^d)$ and $\ell=(\ell^1,\cdots,\ell^d)$ are $\Zb_L^d$ or $\Rb^d$ vectors.

Note that the initial data of \eqref{wke} matches that for (NLS) in \eqref{data} in the sense that $\mathbb E|\widehat{u_{\mathrm{in}}}(k)|^2=n_{\mathrm{in}}(k)$, hence the description \emph{well-prepared} for \eqref{data}. We shall show as part of our proof (Proposition \ref{wkelwp}; see also an optimal local well-posedness result in \cite{GIT}) that given such initial data $n_{\mathrm{in}}(k)$, there exists $\delta>0$ small enough depending on $n_{\mathrm{in}}$, such that there exists a unique local solution $n=n(t,k)\,(k \in \Rb^d)$ of \eqref{wke} on the interval $[0,\delta]$.
\subsubsection{The main result} The main result of this manuscript is the rigorous and quantitative justification of \eqref{wke} over all the existence interval $[0,\delta]$, as the limit of the averaged (NLS) dynamics under the scaling law $\alpha L=1$.
\begin{thm}\label{main} Let $d\geq 3$, and consider the Lebesgue null set $\Zf\subset(\Rb^+)^d$ defined in Lemma \ref{genericity}. The followings hold for any fixed $\beta\in(\Rb^+)^d\backslash\Zf$.

Fix $A\geq 40d$, a Schwartz function $n_{\mathrm{in}}\geq 0$, and fix $\delta\ll 1$ depending on $(A,\beta,n_{\mathrm{in}})$. Consider the equation (\ref{nls}) with random initial data (\ref{data}), and assume $\lambda=L^{(d-1)/2}$ so that $\alpha=L^{-1}$ and $T_{\mathrm{kin}}=L^2/2$. Then, for sufficiently large $L$ (depending on $\delta$), the equation has a smooth solution up to time \[T=\frac{\delta L^2}{2}=\delta\cdot T_{\mathrm{kin}},\] with probability $\geq 1- L^{-A}$. Moreover we have (here $\widehat{u}$ is as in \eqref{fourier})
\begin{equation}\label{limit}\lim_{L\to\infty}\sup_{\tau\in[0,\delta]}\sup_{k\in\Zb_L^d}\left|\Eb\,|\widehat{u}(\tau\cdot T_{\mathrm{kin}},k)|^2-n(\tau,k)\right|=0,
\end{equation} where $n(\tau,k)$ is the solution to (\ref{wke}). 
\end{thm}

\medskip

A few remarks about this result are in order. First, we understand that the expected value $\Eb$ in (\ref{limit}) is taken only when (\ref{nls}) has a smooth solution on $[0,\delta\cdot T_{\mathrm{kin}}]$, {and the quantity that we take expectation of is defined to be $0$ otherwise}. As stated in Theorem \ref{main}, this is a set of probability $\geq  1- L^{-A}$, and hence its complement has no effect on \eqref{limit} (using the mass conservation of $u$). Second, the convergence as $L \to \infty$ is actually quantitative in the sense that there exists $\nu=\nu(d)>0$ and a constant $C$ independent of $L$ such that 
$$
\sup_{\tau\in[0,\delta]}\sup_{k\in\Zb_L^d}\left|\Eb\,|\widehat{u}(\tau\cdot T_{\mathrm{kin}},k)|^2-n(\tau,k)\right|\leq CL^{-\nu}.
$$ 
We also point out that the requirement that $n_{\mathrm{in}}$ be Schwartz is an overkill, and the proof only requires control on finitely many Schwartz semi-norms of $n_{\mathrm{in}}$.

Finally, we remark that Theorem \ref{main} extends, with essentially the same proof, to scaling laws of the form $\alpha=L^{-\kappa}$ for $\kappa$ smaller than and sufficiently close to 1. For such scaling laws, we do not need the genericity assumption for $\beta$, and \eqref{limit} holds independent of the shape of the torus. We shall discuss this in some more detail in Section \ref{SLComment} below.

\subsection{Comments on Theorem \ref{main}}
\subsubsection{Background Work}\label{background} Starting from the middle of the past century, wave turbulence has become a significant component in the study of nonlinear wave theory, and a vibrant field of scientific study in plasma theory \cite{Davidson}, oceanography \cite{Janssen,WMO}, crystal thermodynamics \cite{SpohnPBE} to mention only a few. We refer to \cite{ZLFBook, Nazarenko} for textbook treatments. Mathematically speaking, problems related to wave turbulence theory have attracted considerable attention in the last couple of decades. The focus was initially on constructing solutions to nonlinear dispersive equations that exhibited some form of energy cascade\footnote{Upper bounds on this cascades, measured in terms of the growth of high Sobolev norms was also investigated in \cite{Staffilani97,CollianderDKS01,Bourgain04,CatoireW10,Sohinger11,CollianderKO12,PTV17, Deng}.} \cite{Kuksin96, Kuksin97, Bourgain00b, CKSTT, GuardiaK12,Hani12,Guardia14,HaniPTV15, HausProcesi, GuardiaHP16, GerardG10, GHHMP}. This is one of the important conclusions of the wave kinetic theory, which predicates the presence of stationary power-like solutions to \eqref{wke}, called the \emph{forward and backward cascade spectra}. These are the wave-analogues of Kolmogorov spectra in hydrodynamic turbulence \cite{Zakh65, ZLFBook, Nazarenko}. The rigorous study of such solutions of the \eqref{wke} has been initiated in \cite{EV1, EV2}. The wide range of applicability of this kinetic theory, combined with its profound turbulence implications, emphasized the importance of setting it on rigorous mathematical foundations.

\medskip

In terms of justifying the kinetic formalism, several works addressed certain aspects of the problem \cite{LukSpohn,Faou, DK1, DK2} {(see also \cite{EY, ESY} for related results on the linear Schr\"{o}dinger equation with random potential)}. The full question of deriving the (WKE) starting from the unperturbed dispersive system was first treated in \cite{BGHS2}. There, the authors justify the derivation of \eqref{wke} for (NLS) up to timescales that are vanishingly small relative to the kinetic timescale, namely up to $L^{-\gamma}T_{\mathrm{kin}}$ for some $\gamma>0$. The later works in \cite{DH, CG1} were able to substantially improve such timescales of approximation all the way to $L^{-\varepsilon}T_{\mathrm{kin}}$ for arbitrarily small $\varepsilon$ and for some particular scaling laws. We shall elaborate a bit more on these works given their relevance to this manuscript, and the fact that they were the first to showcase the importance of the {scaling law} to this problem. 

The result in \cite{DH} suggested that the rigorous derivation of the wave kinetic equation depends on the scaling law at which $L$ diverges to $\infty$ and $\alpha$ vanishes to $0$. More precisely, it is shown that for \emph{two favorable scaling laws}, including the one studied in this manuscript, one can justify the approximation as in \eqref{limit} but up to times scales of the form $L^{-\varepsilon}$ for arbitrarily small $\varepsilon$. The main difficulty in such a result is in proving the existence of solutions to the (NLS) equation \emph{as a Feynman diagram expansion} up to times $T\sim L^{-\varepsilon}T_{\mathrm{kin}}$. This time $T$ plays the role of the radius of convergence of this power series expansion. When it comes to absolute convergence, the result in \cite{DH} gives optimal, up to $L^\varepsilon$ loss, estimates on this radius of convergence $T$, and proves that \eqref{limit} holds for such timescales. This is done for all admissible scaling laws (cf. Section \ref{SLComment}), and outside the two favorable scaling laws mentioned above, the time $T$ is {much shorter} than the conjectured kinetic timescale. In fact, we show that the expansion diverges absolutely in a certain sense for times longer than $T$, which raised the question whether one can justify the kinetic equation at the kinetic timescales outside the two scaling laws identified in \cite{DH}. This issue was also investigated in \cite{CG2} which further analyzed this divergence.

Of course, the central question, for any scaling law, is whether one can justify the approximation \eqref{limit} up to times that are $O(1)$ multiples of the kinetic timescale. Such a result, regardless of the scaling law, would allow transferring the rich set of behaviors exhibited by the wave kinetic equation (such as energy cascade or formation of condensate \cite{EV1, EV2}) on the interval of approximation into long-time behaviors of the cubic NLS equation. This includes NLS set on the unit torus by rescaling. Our main theorem provides such quantitative approximation, for the scaling law $\alpha\sim L^{-1}$. Moreover, as we shall discuss in Section \ref{SLComment} below, the proof extends with minor modifications to some close-by scaling laws.

Finally, we mention a recent deep work \cite{ST} of Staffilani-Tran, which was submitted to arXiv shortly after the completion of this manuscript. It concerns a higher dimensional KdV-type equation under a time-dependent Stratonovich stochastic forcing, which effectively randomizes the phases without injecting energy into the system. The authors derive the corresponding wave kinetic equation up to the kinetic timescale, for the specific scaling law $\alpha\sim L^{-0}$ (i.e. first taking $L\to\infty$ and then taking $\alpha\to 0$).
\subsubsection{Criticality of the problem}\label{CriticalityComment}
Criticality is one of the fundamental concepts in the study of nonlinear PDE. While {the classical} scaling criticality plays a central role in the study of deterministic equations, a different type of scaling takes the spotlight for probabilistic problems as showcased in several recent works both in the parabolic and dispersive setting \cite{Hairer, DNY2, DNY3}. To explain the difference, it is worth recalling the following robust definition of criticality: A problem is \emph{subcritical} if subsequent (Picard) iterates of the solution get better and better compared to previous ones; it is \emph{critical} if the iterates neither exhibit an improved nor worse behavior compared to previous ones, and \emph{supercritical} if the iterates successively deteriorate. For instance, the classical {(deterministic)} scaling criticality for the {cubic} (NLS) equation $i\partial_t v+\Delta v=\pm |v|^{2}v$ can be defined as the minimum regularity $s$ for which the first iterate of an $H^s$-normalized rescaled bump function of the form $u_{\mathrm{in}}:=N^{-s+\frac{d}{2}}\varphi(Nx)$ is better behaved than the zeroth iterate. This can be easily seen by comparing $|u_{\mathrm{in}}|^{2} u_{\mathrm{in}}$ and $\Delta u_{\mathrm{in}}$ to obtain that the problem is critical if $s=s_c:=\frac{d}{2}-1$ and subcrticial (resp. supercritical) if $s>s_c$ (resp. $s<s_c$).

The more relevant notion of criticality for us is that of \emph{probabilistic scaling criticality}. This can be formulated in terms of the $H^s$ regularity of the initial data for (NLS) on the unit torus as above, (see \cite{DNY2, DNY3}), but for our problem \eqref{nls} it translates (or rescales) into the trichotomy of whether the time interval $[0, T]$ on which we study the solutions satisfies $T\ll T_{\mathrm{kin}}$ (subcritical regime), $T\sim T_{\mathrm{kin}}$ (critical regime), or $T\gg T_{\mathrm{kin}}$ (supercritical regime). To see this, we note that the first iterate of \eqref{nls} is given in Fourier space by
\begin{equation}\label{firstiterate}
\widehat u^{(1)}(t,k):=i\frac{\lambda^2}{L^{2d}}\sum_{S(k)} \widehat{u_{\mathrm{in}}}({k_1}) \overline{\widehat{u_{\mathrm{in}}}(k_2)} \widehat{u_{\mathrm{in}}}({k_3}) \frac{e^{\pi i\Omega t}-1}{\pi i \Omega}, \quad \Omega=|k_1|_\beta^2-|k_2|_\beta^2+|k_3|_\beta^2-|k|_\beta^2,
\end{equation}
where $S(k)=\{(k_1, k_2, k_3) \in \Zb^d_L: k_1-k_2+k_3=k\}$. A deterministic analysis using the fact that $\widehat{u_{\mathrm{in}}}(k)$ decays like a Schwartz function (think of it as compactly supported in $B(0,1)$) shows that this term is bounded (up to logarithmic losses) by $\frac{\lambda^2}{L^{2d}} \sup_{m} |S_{T, m}|$ where $S_{T, m}=\{(k_1, k_2, k_3) \in S(k): |\Omega-m|\leq T^{-1}\}$. It's not too hard to see that $\sup_{m} |S_{T, m}| \sim L^{2d}T^{-1}$ (at least when $T\ll L^d$, see Lemma \ref{lem:counting} or Lemma {3.2} in \cite{DH}). However, {with random data $u_{\mathrm{in}}$} and using Gaussian hypercontractivity estimates, a major cancellation happens in the sum over $S(k)$ above, and with overwhelming probability, one has the much improved central-limit-theorem-type bound
\begin{equation}\label{firstiterate2}
|\widehat u^{(1)}(t,k)|\sim \frac{\lambda^2}{L^{2d}}\left(\sup_{m} |S_{T, m}|\right)^{1/2}\sim \frac{\lambda^2T^{1/2}}{L^{d}}.
\end{equation}
From this it is clear that the iterate $\widehat u^{(1)}(t,k)$ is much better behaved compared to the zeroth iterate $\widehat{u_{\mathrm{in}}}(k)$ on timescales $T\ll \frac{L^{2d}}{\lambda^4}\sim T_{\mathrm{kin}}$, and does not feature any improvement for times $T\sim T_{\mathrm{kin}}$. For this reason, all previous works \cite{BGHS2, DH, CG1, CG2} on this subject deal with the probabilistically subcritical setting, albeit the results in \cite{DH, CG1} cover the full subcritical regime $T<L^{-\varepsilon} T_{\mathrm{kin}}$ for the scaling law $\alpha L=1$, which we also adopt in this paper. Consequently, obtaining the rigorous derivation of the wave kinetic equation at the kinetic timescale $T_{\mathrm{kin}}$ as in Theorem \ref{main} is a quintessential \emph{probabilistically critical problem}. In fact, Theorem \ref{main} seems to be the first solution of a probabilistically critical problem, both in the dispersive and parabolic setting (here we should note that recent developments in the parabolic setting allow covering the full subcritical range \cite{Hairer, GIP}).

\subsubsection{On scaling laws and the torus genericity condition}\label{SLComment} Theorem \ref{main} justifies the kinetic approximation under the scaling law $\alpha L=1$, i.e. $\alpha$ goes to zero like $L^{-1}$. This is one of the two favorable scaling laws identified in \cite{DH}, and is also the one treated in \cite{CG1}. Moreover, it also holds a particular mathematical importance. In fact, Theorem \ref{main} scales back, in this scaling law, to time $\sim 1$ results (i.e. local well-posedness with precise description of statistical properties) for the cubic (NLS) equation on the unit torus, in the probabilistically critical space $H^{-1/2}$, which is linked to a main open problem raised in \cite{DNY3}. A particularly interesting case happens when $d=3$. There, for an appropriate choice of $n_{\mathrm{in}}$ (namely $\varphi(\xi)|\xi|^{-1}$ for some $\varphi \in \mathcal S(\Rb^3)$ vanishing near 0 and infinity), Theorem \ref{main} rescales into a local existence result for the Littlewood-Paley projection of data in (essentially) the support of the \emph{Gibbs measure} for the (NLS) equation on $\Tb^3$. Such local existence results for Gibbs measure initial data would be a central part of a potential proof of the invariance of the Gibbs measure. As is well-known, the Gibbs measure invariance problem for (NLS) on $\Tb^3$ is another outstanding probabilistically critical problem. In fact, after the work \cite{DNY2} which solves the two-dimensional case, it is the only remaining Gibbs measure invariance problem for (NLS), given that the question of existence (or lack thereof) of such measures is now well understood in constructive quantum field theory \cite{AC,BG,Fro,GJ,Simon}.

\medskip

As explained in \cite{DH}, not all scaling laws are admissible for the kinetic theory, and the admissibility of the scaling law depends on the whether the torus is generic or not. In fact, suppose one adopts the scaling law $\alpha=L^{-\kappa}$ for $\kappa\geq 0$. Here $\kappa=0$ means that one takes the $L\to\infty$ limit followed by the $\alpha\to 0$ limit, which incidentally was the other favorable scaling law identified in \cite{DH}. Since the kinetic timescale is given by $T_{\mathrm{kin}}\sim \alpha^{-2}=L^{2\kappa}$, restrictions on the admissible $\kappa$ come from any restriction posed by the kinetic theory on the time interval of approximation. The relevant restriction here is that the exact resonances, for which $\Omega=0$, in a sum like \eqref{firstiterate} should not overwhelm the quasi-resonances for which $0<|\Omega|\lesssim T^{-1}$. The latter interactions are the ones responsible for the emergence of the kinetic equation in the large box limit. For an arbitrary torus (including the rational or square torus), the exact resonances can have a contribution of $(L^{2d-2})^{1/2}$ to the sum in \eqref{firstiterate} (taking into account the Gaussian $\ell^2$ cancellation), which should be compared to the $(L^{2d}/T)^{1/2}$ estimate used above. This means that if the torus is rational, then the limitation of the kinetic theory is given by $T_{\mathrm{kin}}\ll L^2$. This translates into the requirement that $\kappa <1$ on a rational torus. On the other hand, on a generic torus, the contribution of exact resonances is much less, namely $(L^{d})^{1/2}$, which when compared to $(L^{2d}/T)^{1/2}$ yields the requirement that $T_{\mathrm{kin}}\ll L^d$, and hence $\kappa$ has to $< d/2$ on a generic torus. Note that our scaling law $\alpha=L^{-1}$ lies just outside the range of admissible scaling laws for a rational torus, but well within the range for a generic torus. This explains why Theorem \ref{main} is stated for a generic torus.

\medskip

\begin{center}
    \begin{tabular}{| l | l | l | l |}
    \hline
    Scaling Law $\alpha=L^{-\kappa}$ & $T_{\mathrm{kin}}=1/2\alpha^2$ & Torus type \\ \hline
    $0\leq \kappa <1$ & $ T_{\mathrm{kin}}\sim L^{2\kappa}\ll L^2$ &  Any torus\\ \hline
    $1\leq \kappa < L^{d/2}$ & $L^{2} \lesssim T_{\mathrm{kin}}\ll L^d$ & generic torus \\ \hline
    \end{tabular}
\end{center}

\medskip

That being said, our proof extends with minor modifications to scaling laws $\alpha=L^{-\kappa}$ for $\kappa$ smaller than but sufficiently close to $1$. This is within the admissible range of scaling laws {on an arbitrary torus}, and as such our result can be extended to such scaling laws {which require no restrictions on the shape of the torus}. Given the complexity of the proof, we chose to focus the discussion here to the single scaling law $\alpha=L^{-1}$. {We will address the remaining scaling laws $\kappa<1$ (on the arbitrary torus) in a separate forthcoming note}. Note that some challenges are apparent in the case $\kappa>1$, and new ideas seem to be needed there.
\begin{rem} {After the submission of this paper, the authors have completed the subsequent works \cite{DH2,DH3,DH4}. In particular \cite{DH4} addresses the full range of scaling laws $0<\kappa<1$ without genericity assumption; see also discussions in \cite{DH4} regarding the endpoint case $\gamma=0$, which is in fact not compatible with the continuum setting (the difficulty comes from the remainder terms $R_N$ in (\ref{duhamelintro}) below). Moreover, \cite{DH2} establishes important results including propagation of chaos and non-Gaussian density evolution, which are again true for the full range of scaling laws \cite{DH4}.}
\end{rem}
\subsection{A high-level sketch of the proof}\label{ProofComment}
A proper overview of the proof requires introducing quite a bit of notation and setup. We shall do this in Section \ref{proofoverview} after we set up the problem in Section \ref{setup}. Here, we shall be content with a zoomed-out overview of the proof. As in our previous work in \cite{DH} on the subcritical timescales, the idea is to expand the (NLS) solution as a power series (Feynman diagram expansion) of its iterates 
\begin{equation}\label{duhamelintro}
u=u^{(0)}+u^{(1)}+\ldots+u^{(N)} +R_N, 
\end{equation}
for sufficiently large $N$. Here, the $j$-th iterates $u^{j}$ can be written as a sum over ternary trees of scale $j$ (cf. Section \ref{setup}) and $R_N$ is the remainder. In the subcritical problem in which $T\leq L^{-\varepsilon} T_{\mathrm{kin}}$, it is sufficient to do a finite (but $O(\varepsilon^{-1})$ long) expansion to prove an approximation result like \eqref{limit}. Roughly speaking, the reason for that is that each iterate exhibits at least a $L^{-\varepsilon}$ improvement over the previous one. In particular, one does not need to keep track of any factorial dependences on $N$ when estimating the iterates and the remainder $\Rc_N$. Such factorial growth appear when one computes the correlations, like $\mathbb E \big(u^{(k)} \overline{u^{(\ell)}}\big)$, and hence in the estimates on the iterates and the remainder.

This becomes one of the major difficulties in the critical problem. In fact, in our critical setting here where $T=\delta T_{\mathrm{kin}}$, the only improvement in the successive iterates is $\sim \sqrt \delta$ (cf. \eqref{firstiterate2}), and as such the best estimates one can dream of for $R_N$ is to control it by $(\sqrt{\delta})^{N}$. For the contribution of $R_N$ in \eqref{limit} to vanish in the limit $L\to \infty$, one has to allow $N$ to diverge as $L\to \infty$. This means that one has to track carefully the factorial divergences in $N$ in the correlations $\mathbb E \big(u^{(k)} \overline{u^{(\ell)}}\big)$. In fact, such correlations can be represented as sums over pairs of ternary trees whose leaves are paired to each other. We call those such objects \emph{couples}, and the number of those couples is factorial in $n:=k+\ell$, which is called the scale of the couple. This brings us to the central idea in the proof: \emph{can one classify the couples into groups, such that those saturating or almost saturating the worst-case-scenario estimates are relatively few and do not lead to factorial losses in $n=k+\ell$, while the remaining (factorially many) couples satisfy much better estimates than the worst-case scenario, i.e. feature a gain of powers of $L$, which is sufficient to offset the factorial loss?}

The positive answer to this question constitutes the bulk of the proof. However, the answer is not as straightforward as one might first hope. In fact, one would hope that the couples with almost saturated estimates would be small perturbations of the ``leading" ones that converge to the iterates of the wave kinetic equation. Unfortunately, these are not the only ones. In our proof we will actually identify three families of couples with almost saturated estimates. The first family, which we call \emph{regular couples}, are essentially the leading ones that converge to the iterates of the wave kinetic equation, plus some similar couples whose contribution cancel out in the limit. The second family, which we call \emph{irregular chains}, can also lead to almost saturated estimates and is dealt with in Section \ref{irchaincancel}. The last family, which we call {Type II (molecular) chains, satisfy an $L^1$ bound that makes its contribution acceptable.} This is dealt with in Section \ref{l1coef}.

The good news is that there are only $O(C^n)$ couples that lead to almost saturated estimates, whereas the remaining (factorial in $n$) number of couples all feature a gain in $L$. This is the content of our main rigidity theorem in Section \ref{improvecount}. In fact, we show that if one performs a type of surgery on an arbitrary couple to remove all its regular sub-couples, all its irregular chains, and all its Type II molecular chains (which are exactly the structures that lead to almost saturated estimates), then we are left with a reduced {structure} whose estimate features a gain $L^{-r}$ where $r$ is comparable to the size of this structure! This is enough to offset the factorial divergence $r!$ that comes from the possibilities of these size $r$ structures, provided that $r$ is small enough relative to $L$. Since $r\leq N$, this is more than guaranteed if we pick $N\sim \log L$.  

We should mention that the analysis of each of the couple families mentioned above requires a different genre of argument, ranging from sophisticated combinatorial constructions in Sections \ref{domasymp}, \ref{irchaincancel} and \ref{improvecount}, to analytic ones in Section \ref{regasymp} and \ref{l1coef}, and number theoretic ones in Section \ref{numbertheory}. Once this picture is made clear, the estimate on the remainder term is relatively easier and can be derived from the analysis above. There are some subtleties involved, which will be treated in Section \ref{operatornorm}. For a more detailed discussion of the proof, See Section \ref{proofoverview}.
\subsubsection*{Acknowledgements} The first author is supported in part by NSF grant DMS-1900251 and a Sloan Fellowship. The second author is supported in part by NSF grant DMS-1654692 and a Simons Collaboration Grant on Wave Turbulence. The authors would like to thank Ruixiang Zhang, Fan Zheng and Herbert Spohn for helpful discussions, and Gigliola Staffilani and Minh-Binh Tran for explaining the work \cite{ST}. {The authors would like to thank the referee for carefully going through the manuscript and providing the detailed report and suggestions.}

\section{Basic setup}\label{setup}
\subsection{Preliminary reductions}Consider the equation (\ref{nls}). Let $M=\fint|u|^2$ be the conserved mass of $u$ (where $\fint$ takes the average on $\Tb_L^d$), and define $v:=e^{-2i\lambda^2Mt}\cdot u$, then $v$ satisfies the Wick ordered equation
\begin{equation}(i\partial_t-\Delta_\beta)v+\lambda^2\bigg(|v|^2v-2\fint|v|^2\cdot v\bigg)=0.
\end{equation} By switching to Fourier space, rescaling in time and taking back the linear flow, we can define
\begin{equation}a_k(t)=e^{-\delta\pi iL^2|k|_\beta^2t}\cdot\widehat{v}( \delta T_{\mathrm{kin}}\cdot t,k),
\end{equation} with $\widehat{v}$ as in (\ref{fourier}). {By the same calculations as in Section 2.1 of \cite{DH},} we obtain that $a:=a_k(t)$ satisfies the equation
 \begin{equation}\label{akeqn}
 \left\{
\begin{aligned}
\partial_ta_k &= \Cc_+(a,\overline{a},a)_k(t),\\
a_k(0) &=(a_k)_{\mathrm{in}}=\sqrt{n_{\mathrm{in}}(k)}\eta_k(\omega),
\end{aligned}
\right.
\end{equation} with the nonlinearity
\begin{equation}\label{akeqn2} \Cc_\zeta(f,g,h)_k(t):=\frac{\delta}{2L^{d-1}}\cdot(i\zeta)\sum_{k_1-k_2+k_3=k}\epsilon_{k_1k_2k_3}e^{\zeta\delta\pi iL^2\Omega(k_1,k_2,k_3,k)t}f_{k_1}(t)g_{k_2}(t)h_{k_3}(t).
\end{equation}for $\zeta\in\{\pm\}$. Here in (\ref{akeqn2}) and below, the summation is taken over $(k_1,k_2,k_3)\in(\Zb_L^d)^3$, and \begin{equation}\label{defcoef0}\epsilon_{k_1k_2k_3}=
\left\{
\begin{aligned}&1,&&\mathrm{if\ }k_2\not\in\{k_1,k_3\};\\
-&1,&&\mathrm{if\ }k_1=k_2=k_3;\\
&0,&&\mathrm{otherwise},
\end{aligned}
\right.\end{equation} and the resonance factor
\begin{equation}\label{res}
\Omega=\Omega(k_1,k_2,k_3,k):=|k_1|_\beta^2-|k_2|_\beta^2+|k_3|_\beta^2-|k|_\beta^2=2\langle k_1-k,k-k_3\rangle_\beta;\end{equation} {the last equality in (\ref{res}) requires $k_1-k_2+k_3=k$.} Note that $\epsilon_{k_1k_2k_3}$ is always supported in the non-resonant set \begin{equation}\label{defset}\Sf:=\big\{(k_1,k_2,k_3):\mathrm{\ either\ }k_2\not\in\{k_1,k_3\},\mathrm{\ or\ }k_1=k_2=k_3\big\}.\end{equation} The rest of this paper is focused on the system (\ref{akeqn})--(\ref{akeqn2}), with the relevant terms defined in (\ref{defcoef0})--(\ref{defset}), in the time interval $t\in[0,1]$.
\subsection{Trees and couples} Throughout the proof we will make extensive use of ternary trees and pairs of ternary trees, to characterize the expressions appearing in the formal expansion of solutions to (\ref{nls}). These are alternative formulations of the classical \emph{Feynman diagrams}.
\begin{df}\label{deftree} A \emph{ternary tree} $\Tc$ (see Figure \ref{fig:tree}, we will simply say a \emph{tree} below) is a rooted tree where each non-leaf (or \emph{branching}) node has exactly three children nodes, which we shall distinguish as the \emph{left}, \emph{mid} and \emph{right} ones. We say $\Tc$ is \emph{trivial} (and write $\Tc=\bullet$) if it consists only of the root, in which case this root is also viewed as a leaf.

We denote generic nodes by $\nf$, generic leaves by $\lf$, the root by $\rf$, the set of leaves by $\Lc$ and the set of branching nodes by $\Nc$. The \emph{scale} of a tree $\Tc$ is defined by $n(\Tc)=|\Nc|$, so if $n(\Tc)=n$ then $|\Lc|=2n+1$ and $|\Tc|=3n+1$.

A tree $\Tc$ may have sign $+$ or $-$. If its sign is fixed then we decide the signs of its nodes as follows: the root $\rf$ has the same sign as $\Tc$, and for any branching node $\nf\in\Nc$, the signs of the three children nodes of $\nf$ from left to right are $(\zeta,-\zeta,\zeta)$ if $\nf$ has sign $\zeta\in\{\pm\}$. Once the sign of $\Tc$ is fixed, we will denote the sign of $\nf\in\Tc$ by $\zeta_\nf$. Define the conjugate $\overline{\Tc}$ of a tree $\Tc$ to be the same tree but with opposite sign.
  \begin{figure}[h!]
  \includegraphics[scale=0.55]{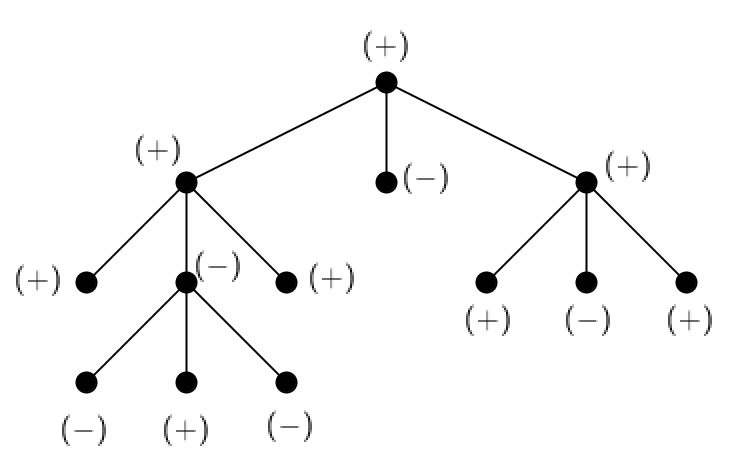}
  \caption{An example of a tree with $+$ sign (Definition \ref{deftree}).}
  \label{fig:tree}
\end{figure} 
\end{df}
\begin{df}\label{defcouple} A \emph{couple} $\Qc$ (see Figure \ref{fig:cpl}) is an unordered pair $(\Tc^+,\Tc^-)$ of two trees $\Tc^\pm$ with signs $+$ and $-$ respectively, together with a partition $\Ps$ of the set $\Lc^+\cup\Lc^-$ into $(n+1)$ pairwise disjoint two-element subsets, where $\Lc^\pm$ is the set of leaves for $\Tc^\pm$, and $n=n^++n^-$ where $n^\pm$ is the scale of $\Tc^\pm$. This $n$ is also called the \emph{scale} of $\Qc$, denoted by $n(\Qc)$. The subsets $\{\lf,\lf'\}\in\Ps$ are referred to as \emph{pairs}, and we require that $\zeta_{\lf'}=-\zeta_\lf$, i.e. the signs of paired leaves must be opposite. If both $\Tc^\pm$ are trivial, we call $\Qc$ the \emph{trivial couple} (and write $\Qc=\times$).

For a couple $\Qc=(\Tc^+,\Tc^-,\Ps)$ we denote the set of branching nodes by $\Nc^*=\Nc^+\cup\Nc^-$, and the set of leave by $\Lc^*=\Lc^+\cup\Lc^-$; for simplicity we will abuse notation and write $\Qc=\Tc^+\cup\Tc^-$. We also define a \emph{paired tree} to be a tree where \emph{some} leaves are paired to each other, according to the same pairing rule for couples. We say a paired tree is \emph{saturated} if there is only one unpaired leaf (called the \emph{lone leaf}). In this case the tree forms a couple with the trivial tree $\bullet$.
  \begin{figure}[h!]
  \includegraphics[scale=.5]{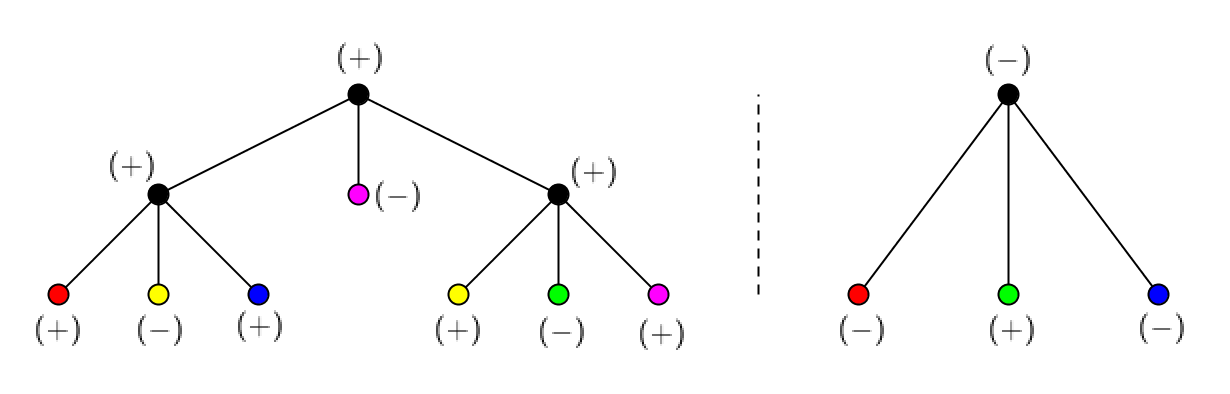}
  \caption{An example of a couple (Definition \ref{defcouple}). Here and below two nodes of same color (other than black) represent a pair of leaves.}
  \label{fig:cpl}
\end{figure} 
\end{df}
\begin{rem} Our notions about trees and couples will be fixed throughout, for example $\Lc^\pm$ will always mean the set of leaves for the tree $\Tc^\pm$, and $\Nc_j^*$ will always mean the $\Nc^*$ set for a couple $\Qc_j$, etc.
\end{rem}
\begin{df}\label{defdec} A \emph{decoration} $\Ds$ of a tree $\Tc$ (see Figure \ref{fig:dectree}) is a set of vectors $(k_\nf)_{\nf\in\Tc}$, such that $k_\nf\in\Zb_L^d$ for each node $\nf$, and that \[k_\nf=k_{\nf_1}-k_{\nf_2}+k_{\nf_3},\quad \mathrm{or\ equivalently}\quad \zeta_\nf k_\nf=\zeta_{\nf_1}k_{\nf_1}+\zeta_{\nf_2}k_{\nf_2}+\zeta_{\nf_3}k_{\nf_3},\] for each branching node $\nf\in\Nc$, where $\zeta_\nf$ is the sign of $\nf$ as in Definition \ref{deftree}, and $\nf_1,\nf_2,\nf_3$ are the three children nodes of $\nf$ from left to right. Clearly a decoration $\Ds$ is uniquely determined by the values of $(k_\lf)_{\lf\in\Lc}$. For $k\in\Zb_L^d$, we say $\Ds$ is a $k$-decoration if $k_\rf=k$ for the root $\rf$.\footnote{{Note that our notion of decoration is different from some earlier literature, in which vectors $k$ are assigned not to the nodes of the couple but to the edges connecting nodes to its children. These differences are of course non-essential.}}

Given a decoration $\Ds$, we define the coefficient
\begin{equation}\label{defcoef}\epsilon_\Ds:=\prod_{\nf\in\Nc}\epsilon_{k_{\nf_1}k_{\nf_2}k_{\nf_3}}\end{equation} where $\epsilon_{k_1k_2k_3}$ is as in (\ref{defcoef0}). Note that in the support of $\epsilon_\Ds$ we have that $(k_{\nf_1},k_{\nf_2},k_{\nf_3})\in\Sf$ for each $\nf\in\Nc$. We also define the resonance factor $\Omega_\nf$ for each $\nf\in\Nc$ by
\begin{equation}\label{defres}\Omega_\nf=\Omega(k_{\nf_1},k_{\nf_2},k_{\nf_3},k_\nf)=|k_{\nf_1}|_\beta^2-|k_{\nf_2}|_\beta^2+|k_{\nf_3}|_\beta^2-|k_\nf|_\beta^2.\end{equation}

A decoration $\Es$ of a couple $\Qc=(\Tc^+,\Tc^-,\Ps)$, see Figure \ref{fig:deccpl}, is a set of vectors $(k_\nf)_{\nf\in\Qc}$, such that $\Ds^\pm:=(k_\nf)_{\nf\in\Tc^\pm}$ is a decoration of $\Tc^\pm$, and moreover $k_\lf=k_{\lf'}$ for each pair $\{\lf,\lf'\}\in\Ps$. We define $\epsilon_\Es:=\epsilon_{\Ds^+}\epsilon_{\Ds^-}$, and define the resonance factors $\Omega_\nf$ for $\nf\in\Nc^*$ as in (\ref{defres}). Note that we must have $k_{\rf^+}=k_{\rf^-}$ where $\rf^\pm$ is the root of $\Tc^\pm$; again we say $\Es$ is a $k$-decoration if $k_{\rf^+}=k_{\rf^-}=k$. Finally, we can define decorations $\Ds$ of paired trees, as well as $\epsilon_\Ds$ and $\Omega_\nf$ etc., similar to the above.
  \begin{figure}[h!]
  \includegraphics[scale=.5]{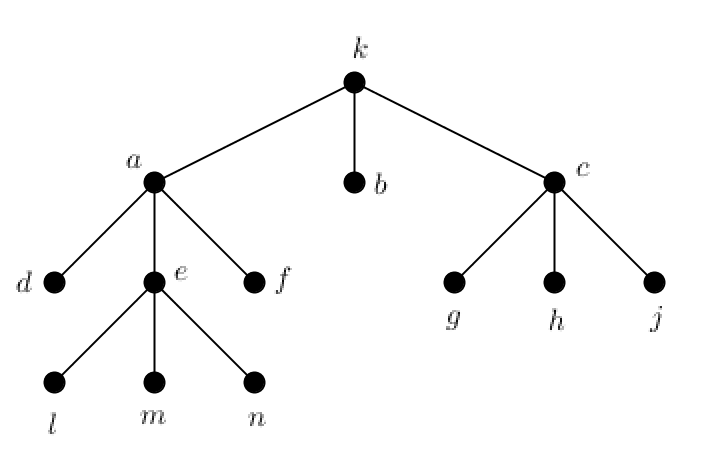}
  \caption{An example of a decorated tree (Definition \ref{defdec}). It satisfies $k=a-b+c$ and $a=d-e+f$ etc.}
  \label{fig:dectree}
\end{figure} 
  \begin{figure}[h!]
  \includegraphics[scale=.5]{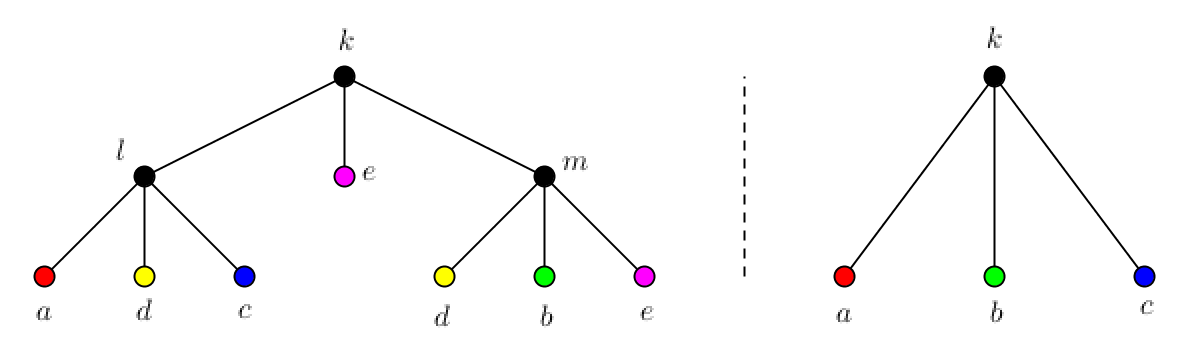}
  \caption{An example of a decorated couple (Definition \ref{defdec}). It satisfies $k=l-e+m$ and $l=a-d+c$ etc.}
  \label{fig:deccpl}
\end{figure} 
\end{df}
\subsubsection{Multilinear Gaussians associated with trees}\label{treeterm}
For any tree $\Tc$, define the expression $\Jc_\Tc$ inductively by
\begin{equation}\label{defjt}(\Jc_\Tc)_k(t)=\left\{
\begin{aligned}
&(a_k)_{\mathrm{in}}^\zeta,&&\textrm{if $\Tc=\bullet$,}\\
&\Ic\Cc_\zeta(\Jc_{\Tc_1},\Jc_{\Tc_2},\Jc_{\Tc_3})_k(t),&&\textrm{otherwise,}
\end{aligned}
\right.
\end{equation} where $\Tc_1$, $\Tc_2$ and $\Tc_3$ are subtrees of $\Tc$ from left to right, $\zeta$ is the sign of $\Tc$, $\Cc_\zeta$ is defined as in (\ref{akeqn2}), and linear Duhamel operator $\Ic$ is given by
\begin{equation}\label{duhamel}\Ic F(t)=\int_0^t F(s){\,\mathrm{d}s}.
\end{equation} Denote $z^+=z$ and $z^-=\overline{z}$ for complex numbers $z$ (note that similar expressions like $m^\pm$ or $\alpha_j^\pm$ that occur later may also have different meanings; this will depend on the context). By induction, we can show that if $\Tc$ has scale $n$, then $\Jc_\Tc$ has the expression
\begin{equation}\label{formulajt}(\Jc_\Tc)_k(t)=\bigg(\frac{\delta}{2L^{d-1}}\bigg)^n\prod_{\nf\in\Nc}(i\zeta_\nf)\sum_\Ds\epsilon_\Ds\cdot\Ac_\Tc(t,\delta L^2\cdot\Omega[\Nc])\prod_{\lf\in\Lc}\sqrt{n_{\mathrm{in}}(k_\lf)}\eta_{k_\lf}^{\zeta_\lf}(\omega),\end{equation} where $\Omega[\Nc]$ represents $(\Omega_\nf)_{\nf\in\Nc}$, and the sum is taken over all $k$-decorations $\Ds$ of $\Tc$ (or equivalently, all choices of $(k_\lf)_{\lf\in\Lc}$). In (\ref{formulajt}) the coefficient $\Ac_\Tc=\Ac_\Tc(t,\alpha[\Nc])$ is defined inductively by 
\begin{equation}\label{defcoefa} \Ac_\bullet(t,\alpha[\Nc])=1;\quad\Ac_\Tc(t,\alpha[\Nc])=\int_0^t e^{\zeta \pi i\alpha_\rf t'}\prod_{j=1}^3\Ac_{\Tc_j}(t',\alpha[\Nc_j])\,\mathrm{d}t',
\end{equation} where $\Tc_1$, $\Tc_2$ and $\Tc_3$ are subtrees of $\Tc$ from left to right so that $\Nc=\Nc_1\cup\Nc_2\cup\Nc_3\cup\{\rf\}$, and $\zeta$ is the sign of $\Tc$.
Finally, for $n\geq 0$ we define
\begin{equation}\Jc_n=\sum_{n(\Tc^+)=n}\Jc_{\Tc^+},
\end{equation} where the sum is taken over all trees $\Tc^+$ of scale $n$ that have $+$ sign.
\subsubsection{The ansatz and the remainder term}\label{ansatzrem} Define $N:=\lfloor \log L\rfloor$. With the definition of $\Jc_\Tc$ and $\Jc_n$ in Section \ref{treeterm} we may introduce the ansatz
\begin{equation}
\label{ansatz}a_k(t)=\sum_{n=0}^{N}(\Jc_n)_k(t)+b_k(t)=\sum_{n(\Tc^+)\leq N}(\Jc_{\Tc^+})_k(t)+b_k(t),
\end{equation} where again the sum is taken over all trees $\Tc^+$ of scale at most $N$ that have $+$ sign. The remainder term $b:=b_k(t)$ then satisfies the equation
\begin{equation}\label{eqnbk}b=\Rc+\Ls b+\Bs(b,b)+\Cs(b,b,b),
\end{equation} {see for example Section 2.2 of \cite{DH},} where the terms on the right hand side are defined by
\begin{equation}\label{eqnbk1.5}\Rc=\sum_{(0)}\Ic\Cc_+(u,\overline{v},w),\,\,\,\Ls b=\sum_{(1)}\Ic\Cc_+(u,\overline{v},w),\,\,\, \Bs(b,b)=\sum_{(2)}\Ic\Cc_+(u,\overline{v},w),\,\,\, \Cs(b,b,b)=\Ic\Cc_+(b,\overline{b},b).\end{equation} The above summations are taken over $(u,v,w)$, each of which being either $b$ or $\Jc_n$ for some $0\leq n\leq N$; moreover in the summation $\sum_{(j)}$ for $0\leq j\leq 2$, exactly $j$ inputs in $(u,v,w)$ equals $b$, and in the summation $\sum_{(0)}$ we require that the sum of the three $n$'s in the $\Jc_n$'s is at least $N$. Note that $\Ls$, $\Bs$ and $\Cs$ are $\Rb$-linear, $\Rb$-bilinear and $\Rb$-trilinear operators respectively, and (\ref{eqnbk}) is equivalent to
\begin{equation}\label{eqnbk2}b=(1-\Ls)^{-1}(\Rc+\Bs(b,b)+\Cs(b,b,b)),
\end{equation} provided $1-\Ls$ is invertible in a suitable space.
\subsubsection{Correlations associated with couples}\label{correlcal} Given $t,s\in[0,1]$ and $k\in\Zb_L^d$, we want to calculate the correlation $\Eb(a_k(t)\overline{a_k(s)})$. Neglecting the remainder $b$ for the moment, we obtain the main contribution
\begin{equation}\label{defe}\Ec:=\sum_{(\Tc^+,\Tc^-)}\Eb[(\Jc_{\Tc^+})_k(t)(\Jc_{\Tc^-})_k(s)],\end{equation} where the sum is taken over all pairs of trees $(\Tc^+,\Tc^-)$ where $\Tc^\pm$ has sign $\pm$ and scale at most $N$. For fixed $(\Tc^+,\Tc^-)$, by (\ref{formulajt}) we get that
\begin{multline}\Eb[(\Jc_{\Tc^+})_k(t)(\Jc_{\Tc^-})_k(s)]=\bigg(\frac{\delta}{2L^{d-1}}\bigg)^n\prod_{\nf\in\Nc^*}(i\zeta_\nf)\sum_{(\Ds^+,\Ds^-)}\epsilon_{\Ds^+}\epsilon_{\Ds^-}\cdot\Ac_{\Tc^+}(t,\delta L^2\cdot \Omega[\Nc^+])\\\times\Ac_{\Tc^-}(s,\delta L^2\cdot\Omega[\Nc^-])\cdot\Eb\bigg[\prod_{\lf\in\Lc^*}\sqrt{n_{\mathrm{in}}(k_\lf)}\eta_{k_\lf}^{\zeta_\lf}(\omega)\bigg],\end{multline}where $n$ equals the sum of scales of $\Tc^+$ and $\Tc^-$. By using the complex version of a specific case of the Isserlis' theorem, which is proved in Lemma \ref{isserlis0}, we conclude that
\begin{equation}\label{formulae}\Eb[(\Jc_{\Tc^+})_k(t)(\Jc_{\Tc^-})_k(s)]=\bigg(\frac{\delta}{2L^{d-1}}\bigg)^n\zeta^*(\Qc)\sum_\Ps\sum_\Es\epsilon_\Es\cdot\Bc_\Qc(t,s,\delta L^2\cdot\Omega[\Nc^*])\cdot{\prod_{\lf\in\Lc^*}^{(+)}n_{\mathrm{in}}(k_\lf)},\end{equation} where $\zeta^*(\Qc)$ and $\Bc_\Qc$ are defined by
\begin{equation}\label{defzetaq}\zeta^*(\Qc)=\prod_{\nf\in\Nc^*}(i\zeta_\nf),\end{equation}
\begin{equation}\label{defcoefb}\Bc_\Qc(t,s,\alpha[\Nc^*])=\Ac_{\Tc^+}(t,\alpha[\Nc^+])\cdot\Ac_{\Tc^-}(s,\alpha[\Nc^-]).\end{equation}Here in (\ref{formulae}) the first summation is taken over all possible sets of pairings $\Ps$ that make a couple $\Qc:=(\Tc^+,\Tc^-,\Ps)$, and the second summation is taken over all $k$-decorations $\Es$ of the couple $\Qc$. {The product $\prod_{\lf\in\Lc^*}^{(+)}$ is taken over $\lf\in\Lc^*$ that have sign $+$.}

By summing over all $(\Tc^+,\Tc^-)$, we conclude that $\Ec=\sum_\Qc\Kc_\Qc(t,s,k)$, where the summation is taken over all couples $\Qc=(\Tc^+,\Tc^-,\Ps)$ with both $\Tc_\pm$ having scale at most $N$, and $\Kc_\Qc$ is defined by
\begin{equation}\label{defkq}\Kc_\Qc(t,s,k):=\bigg(\frac{\delta}{2L^{d-1}}\bigg)^n\zeta^*(\Qc)\sum_\Es\epsilon_\Es\cdot\Bc_\Qc(t,s,\delta L^2\cdot\Omega[\Nc^*])\cdot{\prod_{\lf\in\Lc^*}^{(+)}n_{\mathrm{in}}(k_\lf)}.\end{equation} Here $n$ is the scale of $\Qc$.
\subsection{Notations and estimates} Here we state the main notations, norms and estimates.
\subsubsection{Parameters and norms}\label{notations} From now on we fix $\beta\in(\Rb^+)^d\backslash \Zf$ with $\Zf$ defined in Lemma \ref{genericity}. Let $C$ denote any large constant that depends only on the dimension $d$, and $C^+$ denote any large constant that depends only on $(d,\beta,n_{\mathrm{in}})$. These constants may vary from line to line. The notations $X\lesssim Y$ and $X=O(Y)$ will mean $X\leq C^+ Y$, unless otherwise stated.

Recall that $A\geq 40d$ is fixed in Theorem \ref{main}. We will fix $\nu=(100d)^{-1}\ll 1$, { and fix an even integer $p$ that is sufficiently large depending on $A$ and $\nu$, abbreviated as $p\gg_{A,\nu}1$ (same below). Then fix the value of $\delta$ in Theorem \ref{main}, such that $\delta\ll_{p,C^+}1$ (so $\delta$ is sufficiently small depending on $p$ and $C^+$).} Finally assume $L\gg_{\delta}1$ and fix $N=\lfloor \log L\rfloor$.

Let $\chi_0=\chi_0(z)$ be a smooth {even} function for $z\in\Rb$ that equals $1$ for $|z|\leq 1/2$ and equals $0$ for $|z|\geq 1$; define also $\chi_0(z^1,\cdots,z^d)=\chi_0(z^1)\cdots \chi_0(z^d)$ and $\chi_\infty=1-\chi_0$. By abusing notation, sometimes we may also use $\chi_0$ to denote other cutoff functions with slightly different supports. These functions, as well as the other cutoff functions, will be in Gevrey class $2$ (i.e. the $k$-th order derivatives are bounded by $(2k)!$). For a multi-index $\rho=(\rho_1,\cdots,\rho_m)$, we adopt the usual notations $|\rho|=\rho_1+\cdots+\rho_m$ and $\rho!=(\rho_1)!\cdots(\rho_m)!$, etc. For an index set $A$, we use the vector notation $\alpha[A]=(\alpha_j)_{j\in A}$ and $\mathrm{d}\alpha[A]=\prod_{j\in A}\mathrm{d}\alpha_j$, etc.

Define the time Fourier transform (the meaning of $\widehat{\cdot}$ later may depend on the context)
\[\widehat{u}(\lambda)=\int_\Rb u(t) e^{-2\pi i\lambda t}\,\mathrm{d}t,\quad u(t)=\int_\Rb \widehat{u}(\lambda)e^{2\pi i\lambda t}\,\mathrm{d}\lambda.\] Define the $X^\kappa$ norm for functions $F=F(t,k)$ or $G=G(t,s,k)$ by \[\|F\|_{X^\kappa}=\int_\Rb\langle\lambda\rangle^{\frac{1}{9}}\sup_k\langle k\rangle^{\kappa}|\widehat{F}(\lambda,k)|\,\mathrm{d}\lambda,\quad \|G\|_{X^\kappa}=\int_{\Rb^2}(\langle\lambda\rangle+\langle\mu\rangle)^{\frac{1}{9}}\sup_k\langle k\rangle^{\kappa}|\widehat{F}(\lambda,\mu,k)|\,\mathrm{d}\lambda\mathrm{d}\mu,\] where $\widehat{\cdot}$ denotes the Fourier transform in $t$ or $(t,s)$. If $F=F(t)$ or $G=G(t,s)$ does not depend on $k$, the norms are modified accordingly; they do not depend on $\kappa$ so we simply call it $X$. Define the localized version $X_{\mathrm{loc}}^\kappa$ (and similarly $X_{\mathrm{loc}}$) as 
\[\|F\|_{X_{\mathrm{loc}}^\kappa}=\inf\big\{\|\widetilde{F}\|_{X^\kappa}:\widetilde{F}=F\mathrm{\ for\ }0\leq t\leq 1\big\};\quad \|G\|_{X_{\mathrm{loc}}^\kappa}=\inf\big\{\|\widetilde{G}\|_{X^\kappa}:\widetilde{G}=G\mathrm{\ for\ }0\leq t,s\leq 1\big\}.\] If we will only use the value of $G$ in some set (for example $\{t>s\}$ in Proposition \ref{varregtree}), then in the above definition we may only require $\widetilde{G}=G$ in this set. Define the $Z$ norm for function $a=a_k(t)$,
\begin{equation}\label{defznorm}\|a\|_Z^2=\sup_{0\leq t\leq 1}L^{-d}\sum_{k\in\Zb_L^d}\langle k\rangle^{10d}|a_k(t)|^2.
\end{equation}
\subsubsection{Key estimates} In this section we state the key estimates of the paper. The rest of the paper until Section \ref{operatornorm} is devoted to the proof of these estimates, and in Section \ref{endgame} we use them to prove Theorem \ref{main}.
\begin{prop}\label{mainprop1} Let $\Jc_\Tc$ and $\Jc_n$ be defined as in Section \ref{treeterm}. Then, for each $0\leq n\leq N^3$, $k\in\Zb_L^d$ and $t\in[0,1]$ we have
\begin{equation}\label{mainest1}\Eb|(\Jc_n)_k(t)|^2\lesssim\langle k\rangle^{-20d}(C^+\sqrt{\delta})^n.
\end{equation}
\end{prop}
\begin{prop}\label{mainprop2} Let $\Ls$ be defined as in (\ref{eqnbk1.5}), note that $\Ls^n$ is an $\Rb$-linear operator for $n\geq 0$. Define its kernels $(\Ls^n)_{k\ell}^\zeta(t,s)$ for $\zeta\in\{\pm\}$ by
\[(\Ls^nb)_k(t)=\sum_{\zeta\in\{\pm\}}\sum_\ell{\int_0^t} (\Ls^n)_{k\ell}^\zeta(t,s) b_\ell(s)^\zeta\,\mathrm{d}s.\]Then for each $1\leq n\leq N$ and $\zeta\in\{\pm\}$, we can decompose \begin{equation}\label{decomposem}(\Ls^n)_{k\ell}^\zeta=\sum_{n\leq m\leq N^3}(\Ls^n)_{k\ell}^{m,\zeta},\end{equation} such that for any $n\leq m\leq N^3$ and $k,\ell\in\Zb^d_L$ and $t,s\in[0,1]$ with $t>s$ we have
\begin{equation}\label{mainest2}\Eb|(\Ls^n)_{k\ell}^{m,\zeta}(t,s)|^2\lesssim \langle {k-\zeta\ell}\rangle^{-20d}(C^+\sqrt{\delta})^mL^{40d}.\end{equation}
\end{prop}
\begin{prop}\label{mainprop4} Recall the nonlinearity $\Kc(\phi_1,\phi_2,\phi_3)$ defined in (\ref{wke2}). Define
\begin{equation}\label{iterate}\Mc_0(t,k)=n_{\mathrm{in}}(k);\quad \Mc_n(t,k)=\delta\sum_{n_1+n_2+n_3=n-1}\int_0^t\Kc(\Mc_{n_1}(t'),\Mc_{n_2}(t'),\Mc_{n_3}(t'))(k)\,\mathrm{d}t',\end{equation} which form the Taylor expansion of the solution to (\ref{wke}), see Proposition \ref{wkelwp}, then for each $0\leq n\leq N^3$, $k\in\Zb_L^d$ and $t\in[0,1]$, we have that
\[\bigg|\sum_{n(\Qc)=2n}\Kc_\Qc(t,t,k)-\Mc_n(t,k)\bigg|\lesssim\langle k\rangle^{-20d} (C^+\sqrt{\delta})^n L^{-\nu},\] where the summation is taken over all couples $\Qc$ of scale $n$, and $\Kc_\Qc$ is defined in (\ref{defkq}). If $2n$ is replaced by $2n+1$, then the same result holds without the $\Mc_n(t,k)$ term.
\end{prop}
\section{Overview of the proof}\label{proofoverview}
\subsection{The main challenge} We will focus on the analysis of the correlations $\Kc_\Qc$, since they also control the sizes of $\Jc_\Tc$ and $\Jc_n$ in the ansatz (\ref{ansatz}) in view of (\ref{formulae}). Recall that we have divided the proof of Theorem \ref{main} into three sub-tasks: Proposition \ref{mainprop1}---to obtain upper bounds for $\Kc_\Qc$, Proposition \ref{mainprop2}---to control the $\Rb$-linear operator $\Ls$ appearing in (\ref{eqnbk}), and Proposition \ref{mainprop4}---to evaluate the leading contributions of $\Kc_\Qc$ and match them with the Taylor expansion of $n(\tau,k)$. To demonstrate the main challenge of the problem, let us compare the current situation with the \emph{subcritical} situation which was solved in \cite{DH}, i.e. when one restricts $t\leq L^{-\varepsilon}$ in these propositions.

In the subcritical situation, it can be shown that each term in the expansion in (\ref{ansatz}) gains a power $L^{-\varepsilon}$ compared to the previous one, with high probability, in particular we have
\begin{equation}\label{subcrit}|\Kc_\Qc(t,s,k)|\lesssim \langle k\rangle^{-20d}C_nL^{-n\varepsilon},\end{equation} for any couple $\Qc$ of scale $2n$. Here for simplicity we only consider couples of even scale; the case of odd scale is treated in the same way. Note that (\ref{subcrit}) becomes negligible when $n$ is sufficiently large depending on $\varepsilon$, so the expansion (\ref{ansatz}) can be done to a finite order $N$ \emph{independent of} $L$, and any constant factors one may lose that depends on $N$ will be negligible compared to $L$. Of course it is still highly nontrivial to analyze $\Kc_\Qc$ for $\Qc$ with large scale, but this can be done using the combinatorial structure of $\Qc$, see \cite{DH}.

In the current critical situation, however, the best estimate one can hope for is that 
\begin{equation}\label{naive}|\Kc_\Qc(t,s,k)|\lesssim \langle k\rangle^{-20d}(C^+\delta)^n\end{equation} for couples $\Qc$ of scale $2n$ (in reality we will have $C^+\sqrt{\delta}$ instead of $C^+\delta$ due to a technical reason, see Proposition \ref{section8main}, but this is not important here). This means that, in order for the remainder $b$ in (\ref{ansatz}) to behave significantly better than the main terms, the expansion has to be done at least to order $N\geq \frac{\log L}{\log (1/\delta)}$; in fact as in Section \ref{notations} we have set $N=\lfloor\log L\rfloor$. Therefore the order of expansion \emph{grows} with $L$, which brings the fundamental difficulty of the problem.

One consequence of the largeness of $N$ is that, in many parts of the proof, one is not allowed to lose $\log L$ type factors; on the contrary, for (\ref{subcrit}), any logarithmic factors are negligible. This means that one needs to make every single estimate as sharp as possible, which is a main source of technical difficulties appearing in the proofs below.

A much more significant challenge, which also suggests our main proof strategy, is as follows. For fixed $n$, it is well known that the number of ternary trees is at most $C^N$; however the number of \emph{couples} $\Qc$ of scale $2n$ will grow like $n!$, due to the possibilities of pairings between leaves. This factorial loss, though negligible in the subcritical case (\ref{subcrit}), easily overwhelms the $\delta^n$ gain coming from (\ref{naive}) and seems to completely destroy the convergence.

However, there is one crucial observation that allows us to avoid this fate. Namely, although the total number of couples of scale $2n$ grows factorially on $n$, \emph{the number of couples that actually saturate (\ref{naive}) is in fact bounded by $C^n$.} In other words, even though the whole problem is critical, the \emph{vast majority} of couples $\Qc$ are actually of ``sub-critical" nature and satisfy much better estimates than (\ref{naive}). This fact seems to be unique for the dispersive equation (\ref{nls}), and we have not found a counterpart for stochastic heat equations.

With this observation, it is now clear what we should do with the couples $\Qc$. We shall divide them into different classes, depending on whether they saturate the estimate (\ref{naive}), or nearly saturate it, or neither. This will be controlled by an \emph{index} $r=r(\Qc)$, which plays the central role in the proof. We explain this in more detail in Section \ref{classifyintro} below.
\subsection{Classification of couples} \label{classifyintro} The fundamental objects in our classification of couples are what we call \emph{regular couples}, see Definition \ref{defreg}. These couples have relatively simple structure, and can be constructed by repeating two basic steps (which we call steps $\Ab$ and $\Bb$, see Figures \ref{fig:regcplcst1} and \ref{fig:regcplcst2}) starting from the trivial couple $\times$. As a result, the number of regular couples of a fixed scale $2N$ is bounded by $C^N$ for some absolute constant $C$ (Corollary \ref{countcouple1}). Moreover, these couples are exactly the ones that (formally) saturate (\ref{naive}); in fact a subset of these couples, called the \emph{dominant couples} (Definition \ref{defstd}), constitute exactly the leading contributions in Proposition \ref{mainprop4}.

With the notion of regular couples, it is natural to define the index $r$ as the ``distance" of a given couple to the set of regular couples, roughly as follows:
\begin{multline}\label{indexintro}r(\Qc) = \textrm{ the remaining size of $\Qc$, after repeatedly \emph{reverting}} \\\textrm{the steps $\Ab$ and $\Bb$, until no longer possible.}\end{multline} Note that the actual definition of $r$, see (\ref{indexintro2}), is slightly different from (\ref{indexintro}), due to the presence of particular structures called the \emph{irregular chains} and \emph{Type II) molecular chains}, which will be discussed in Sections \ref{irrechanintro} and \ref{moleintro} below. Here we will temporarily ignore the difference, and note that a couple of index $2r$ is essentially a regular couple \emph{up to ``perturbations" of size $\leq 2r$}.

It is now intuitively clear that the number of couples of scale $2n$ and index $2r$ is at most $C^nr!$ instead of $n!$. This is because a couple of size $2r$ has at most $r!C^r$ possibilities, while reverting steps $\Ab$ and $\Bb$ at most $n$ times leads to at most $C^n$ possible choices (Corollary \ref{corcpl}). Therefore, it remains to show that a couple of scale $2n$ and index $2r$ satisfies the following improvement to (\ref{naive}), namely
\begin{equation}\label{naive2}|\Kc_\Qc(t,s,k)|\lesssim \langle k\rangle^{-20d}(C^+\delta)^nL^{-\nu r}
\end{equation} for some absolute constant $\nu>0$. In the rest of this section we will briefly explain why (\ref{naive2}) is intuitively plausible, and how we shall prove it.

First, recall the definition (\ref{defkq}) of $\Kc_\Qc$. It is easy to show that the function $\Bc_\Qc(t,s,\alpha[\Nc^*])$ is bounded by a product of factors of form $\langle\rho\rangle^{-1}$ where each $\rho$ is a suitable linear combination of the $\alpha_j$ variables for $j\in \Nc^*$; see for example \cite{DH}, Proposition 2.3. As such, for each fixed $(t,s)$, the function $\Bc_\Qc(t,s,\alpha[\Nc^*])$, as a function of $\alpha[\Nc^*]$, is almost $L^1$ integrable. Note that we do need to carefully distinguish between \emph{genuine} and \emph{almost} integrability (see Section \ref{timeintintro} below), but here we will temporarily ignore this and simply assume $\Bc_\Qc\in L^1$. Assuming also $n_{\mathrm{in}}$ is supported in the unit ball, then (\ref{defkq}) is controlled by the upper bound for the following \emph{counting problem}
\begin{equation}\label{countingintro}\big\{\Es=(k_\nf)_{\nf\in\Qc}:|k_\lf|\leq 1\,(\forall \lf\in \Lc^*),\ |\Omega_\nf-\alpha_\nf|\leq (\delta L^2)^{-1}\,(\forall\nf\in\Nc^*)\big\}
\end{equation} for $k$-decorations $\Es$, where $k$ is fixed, and $\alpha_\nf\in\Rb$ are fixed real numbers.

Accurately estimating the number of solutions to (\ref{countingintro}) is a major component of this work (see Section \ref{moleintro}); for demonstration we will use a naive dimension counting argument here (which may not be precise but usually provides the correct heuristics). For example, the decorated couple in Figure \ref{fig:deccpl} corresponds to the counting problem for $(a,b,c,d,e,l,m)\in\Zb_L^{7d}$ such that
\begin{equation}\label{countingintro1}
\left\{
\begin{aligned}l-e+m&=k,&|l|_\beta^2-|e|_\beta^2+|m|_\beta^2-|k|_\beta^2&=\alpha_1+O(\delta^{-1}L^{-2});\\
a-d+c&=l,&|a|_\beta^2-|d|_\beta^2+|c|_\beta^2-|l|_\beta^2&=\alpha_2+O(\delta^{-1}L^{-2});\\
d-b+e&=m,&|d|_\beta^2-|b|_\beta^2+|e|_\beta^2-|m|_\beta^2&=\alpha_3+O(\delta^{-1}L^{-2});\\
a-b+c&=k,&|a|_\beta^2-|b|_\beta^2+|c|_\beta^2-|k|_\beta^2&=\alpha_4+O(\delta^{-1}L^{-2}).
\end{aligned}
\right.
\end{equation} Thus dimension counting yields a possible upper bound, which is $L^{4d}(\delta^{-1}L^{-2})^3=\delta^{-3}L^{4d-6}$ (note that the last line of equations in (\ref{countingintro1}) follows from the first three).

Now, a key feature of \emph{regular} couples is that, all its branching nodes can be paired such that for any decoration $\Es$ and any two paired branching nodes $\nf$ and $\nf'$, one must have $\Omega_{\nf'}=\pm\Omega_\nf$ (see Proposition \ref{branchpair}), i.e. each variable $\Omega_\nf$ occurs \emph{twice} in the $\Bc_\Qc$ function, and in the counting problem. For example, the following decorated couple (Figure \ref{fig:regcplex}) which is regular, corresponds to the counting problem for $(a,b,c,d,e,m)\in\Zb_L^{6d}$ such that
\begin{equation}\label{countingintro2}
\left\{
\begin{aligned}m-e+d&=k,&|m|_\beta^2-|e|_\beta^2+|d|_\beta^2-|k|_\beta^2&=\alpha_1+O(\delta^{-1}L^{-2});\\
a-b+c&=k,&|a|_\beta^2-|b|_\beta^2+|c|_\beta^2-|k|_\beta^2&=\alpha_4+O(\delta^{-1}L^{-2}).
\end{aligned}
\right.
\end{equation} Thus dimension counting yields a possible upper bound, which is $L^{4d}(\delta^{-1}L^{-2})^2=\delta^{-2}L^{4d-4}$.
  \begin{figure}[h!]
  \includegraphics[scale=.5]{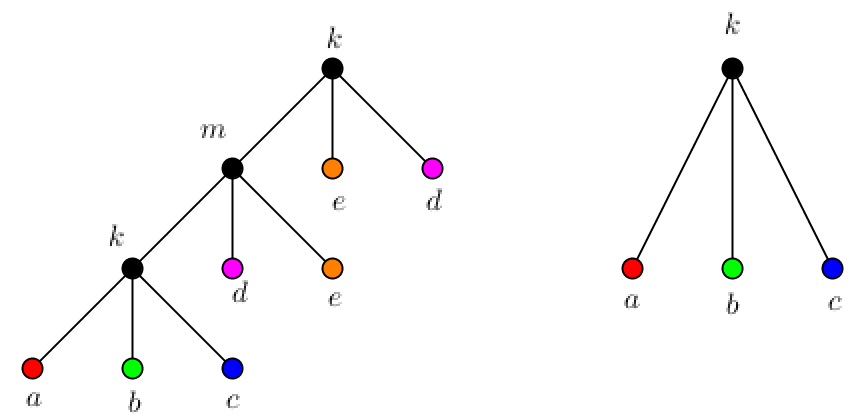}
  \caption{An example of a decorated regular couple. It satisfies $k=m-e+d$ and $k=a-b+c$ etc.}
  \label{fig:regcplex}
\end{figure}

It is clear that in both systems, the dimensions of the submanifolds determined by the \emph{linear} parts are the same (which is $4d$ here and $2nd$ if $\Qc$ has scale $2n$) The reason why the non-regular couple $\Qc$ in Figure \ref{fig:deccpl} enjoys better estimates than the regular couple in Figure \ref{fig:regcplex}, is that the corresponding system contains \emph{one more independent quadratic equation}, due to the fact that each $\Omega_\nf$ occurs twice for regular couples, but not for non-regular couples.

As such, it is natural to believe that $\Kc_\Qc$ for regular couples $\Qc$, which involve the least number of quadratic equations in the counting problem, will be the worst in terms of upper bounds and will saturate (\ref{naive}), while $\Kc_\Qc$ for non-regular couples $\Qc$ will enjoy better estimates. Moreover, if a couple $\Qc$ has ``distance" at least $2r$ to regular couples, i.e. it is obtained by making a size $2r$ perturbation to a regular couple, then it will contain at least $r$ extra quadratic equations in the corresponding counting problem, and thus satisfies the improved bound (\ref{naive2}).

A major part of this paper is to make the above heuristics rigorous. In addition, one has to calculate the asymptotics of $\Kc_\Qc$ for regular couples $\Qc$, and deal with the $\Rb$-linear operator $\Ls$. In the next section we start by considering regular couples.
\subsection{Regular couples}\label{regintro} Let $\Qc$ be a regular couple. Our goal is to calculate the asymptotics of $\Kc_\Qc$, then combine and match them with $\Mc_n(t,k)$ in Proposition \ref{mainprop4}; in this process we also obtain uniform bounds for $\Kc_\Qc$, as in (\ref{naive}), that lead to Proposition \ref{mainprop1}.

Note that in all previous works \cite{BGHS2,CG1,CG2,DH}, only the correlations $\Kc_\Qc$ for couples $\Qc$ \emph{up to scale $2$} are calculated, and they are matched with only the \emph{first order term} $\Mc_1(t,k)$ in the expansion of $n(\tau,k)$. In subcritical situations this is enough, as each term gains at least $L^{-\varepsilon}$ compared to the previous one; in the current work, however, it is necessary to calculate the correlations $\Kc_\Qc$ for couples $\Qc$ of \emph{any} scale. These correlations have much richer structure than $\Mc_n(t,k)$ which are obtained by simply iterating the nonlinearity (\ref{wke2}), so the fact that they still match the higher order iterations $\Mc_n(t,k)$ is quite remarkable.
\subsubsection{Approximation using circle method}\label{circleintro} The formal calculation of the asymptotics of $\Kc_\Qc$ is not difficult. In fact in the limit $L\to\infty$ the sum in (\ref{defkq}) can be viewed as a Riemann sum, which is then approximated by an integral, and we also have \begin{equation}\label{approxdirac}\Bc_\Qc(t,s,\delta L^2\Omega[\Nc^*])\approx (\delta L^2)^{-n}\int\Bc_\Qc\cdot \prod_{\nf} \dirac(\Omega_\nf),\end{equation} where the product is taken over all \emph{different} variables $\Omega_\nf$, and there are in total $n$ of them (half of the scale of $\Qc$). Thus heuristically we have (see Proposition \ref{asymptotics1} for the actual version)
\begin{equation}\label{approxintro}\Kc_\Qc(t,s,k)\approx 2^{-2n}\delta^n\zeta^*(\Qc)\int\Bc_\Qc\cdot\int\prod_\nf\dirac(\Omega_\nf)\prod_{\lf\in\Lc^*}^{(+)}n_{\mathrm{in}}(k_\lf)\,\mathrm{d}\sigma,\end{equation} where $\mathrm{d}\sigma$ is the surface measure for a suitable linear submanifold of $(k_\lf)$. Here note that the vectors involved in different variables $\Omega_\nf$ can be separated, for example for Figure \ref{fig:regcplex} and (\ref{countingintro2}), the two different $\Omega_\nf$ variables are
\[|m|_\beta^2-|e|_\beta^2+|d|_\beta^2-|k|_\beta^2=2\langle m-k,k-d\rangle_\beta\quad \textrm{and}\quad|a|_\beta^2-|b|_\beta^2+|c|_\beta^2-|k|_\beta^2=2\langle a-k,k-c\rangle_\beta,\] and the vectors they involve are $(m-k,k-d,a-k,k-c)$, which are independent variables. This is crucial for (\ref{approxdirac}) to be valid, as products like $\dirac(x\cdot y)\dirac(x\cdot z)$ etc. may not be well-defined in general.

\smallskip In order to make the approximation (\ref{approxintro}) rigorous, one first needs to perform a change of variables, so that the different $\Omega_\nf$ become $\langle x_j,y_j\rangle_\beta\,(1\leq j\leq n)$ for the new independent variables $(x_j,y_j)$. In the simple case $n=1$, we essentially need to prove
\[\sum_{x,y\in\Zb_L^d}\psi(x,y)\cdot\Bc(\delta L^2\langle x,y\rangle_\beta)\approx L^{2d-2}\delta^{-1}\int\Bc\cdot\int_{\Rb^{2d}}\psi(x,y)\cdot\dirac(\langle x,y\rangle_\beta)\mathrm{d}x\mathrm{d}y\] for a Schwartz function $\psi$ and an $L^1$ function $\Bc$, which was achieved in earlier works \cite{BGHS2,DH} etc. by applying the circle method and exploiting the genericity of $\beta$. The case of general $n$, which can be as large as $N=\lfloor\log L\rfloor$, follows from applying the circle method for the integration in each of the variables $(x_j,y_j)$, see Proposition \ref{approxnt}.

There is one main new ingredient, though, compared to previous works. Assuming $n_{\mathrm{in}}$ is supported in the unit ball, we know that each of the variables $(x_j,y_j)$ belongs to a ball of size at most $n$. If $n$ is independent of $L$, as in previous works, then any loss in terms of $n$ is negligible; however for $n\sim\log L$ this bound is not good enough, as a polynomial loss in $n$ for each variable $(x_j,y_j)$ will lead to a factorial net loss, which is not acceptable. The idea here is to make this restriction more precise, namely that each $(x_j,y_j)$ belongs to a ball of size $\lambda_j$ \emph{centered at some point determined by the previous $(x_\ell,y_\ell)$}, after fixing some strict partial ordering in $j$. Moreover individual $\lambda_j$ can be as large as $n$, but the product of all these $\lambda_j$ is bounded by $C^n$, which is then acceptable, see Lemma \ref{auxlem}. In addition, since each $(x_j,y_j)$ is supported in a ball not centered at the origin, one needs to apply a translation-invariant version of the circle method. This is mostly straightforward, but requires a new argument when dealing with major arcs, see Lemma \ref{NTSP}.
\subsubsection{Analysis of $\Bc_\Qc$}\label{timeintintro} In order to apply Proposition \ref{approxnt}, one needs to obtain $L^1$ bounds for the function $\Bc_\Qc=\Bc_\Qc(t,s,\alpha[\Nc^*])$ defined in (\ref{defcoefb}). Here the rough bound in \cite{DH}, Proposition 2.3 is not enough, as $\langle x\rangle^{-1}$ is not in $L^1$ and one cannot afford to lose $\log L$ type factors in the $L^1$ norm. Fortunately, since each variable $\Omega_\nf$ occurs twice in the function $\Bc_\Qc$, it in principle should also occur twice in the denominators, which allows one to recover $L^1$ boundedness, in view of the elementary inequality
\begin{equation}\label{elementary}\int_\Rb{\langle x-a\rangle}^{-1}{\langle x-b\rangle}^{-1}\,\mathrm{d}x\lesssim 1\end{equation} uniformly in $a$ and $b$.

To make the above heuristics precise, we will perform an inductive argument exploiting the structure of regular couples. First note that by induction, $\Bc_\Qc$ can be written as a multi-dimensional integral in the time variables $t_\nf$ in a domain $\Ec=\Ec_\Qc$ defined by $\Qc$, see (\ref{timedom}). Next, we apply the \emph{structure theorem} for regular couples, which is proved in Proposition \ref{structure1}, to construct $\Qc$ from a specific couple $\Qc_0$, by replacing each of its leaf pairs with a \emph{smaller} regular couple $\Qc_j$. We will assume this $\Qc_0$ is a so-called \emph{regular double chain}, see Definition \ref{defregchain}. Then, by considering the time domains $\Ec$ associated with $\Qc$, $\Qc_0$ and each $\Qc_j$, we can essentially express $\Bc_\Qc$ in terms of $\Bc_{\Qc_0}$ and $\Bc_{\Qc_j}$. Applying the induction hypothesis for each $\Bc_{\Qc_j}$, we can reduce the study of $\Bc_\Qc$ to that of $\Bc_{\Qc_0}$; since $\Qc_0$ has an explicit form, the corresponding function $\Bc_{\Qc_0}$ is also explicit and in fact equals the product of two functions of the form
\begin{equation}\label{regchainintro}K(t,\alpha_1,\cdots,\alpha_m)=\int_{t>t_1>\cdots >t_{2m}>0}e^{\pi i(\beta_1t_1+\cdots +\beta_{2m}t_{2m})}\,\mathrm{d}t_1\cdots\mathrm{d}t_{2m},
\end{equation} see (\ref{regchaink}). Here $t$ is replaced by $s$ in the other function, and $\{\beta_1,\cdots,\beta_{2m}\}$ is a permutation of $\{\pm\alpha_1,\cdots,\pm\alpha_{m}\}$ corresponding to a \emph{legal partition} (Definition \ref{deflegal}).

The analysis of the function $K$ is done in Section \ref{regchainest}, where we show that it is essentially $L^1$ in $(\alpha_1,\cdots,\alpha_m)$ for any $t$, except it may contain a few factors $1/\pi i\alpha_j\,(j\in Z)$ where $Z$ is a subset of $\{1,\cdots,m\}$, but then it will be $L^1$ in the remaining variables, see Lemma \ref{regchainlem8}. Using this result, we can proceed with the inductive step and finally prove Proposition \ref{maincoef}, which states that for each fixed $(t,s)$, the function $\Bc_\Qc(t,s,\alpha[\Nc^*])$ is the product of $\prod_{\nf\in Z}1/(\pi i\alpha_\nf)$ for some subset $Z$ of branching nodes, with an $L^1$ function of the remaining $\alpha_\nf$ variables. This then allows us to apply Proposition \ref{approxnt} and calculate the asymptotics of $\Kc_\Qc$ as in Section \ref{circleintro}. Note that the factors $\prod_{\nf\in Z}1/(\pi i\alpha_\nf)$ are not in $L^1$ but have the correct parity so that the circle method still applies, provided one treats the singularities using the Cauchy principal value.

Finally we need to calculate the integrals of (the integrable parts of) $\Bc_\Qc$, see (\ref{defintcal}). These values can again be calculated inductively; in fact we can identify a special class of regular couples, called \emph{dominant couples} (Definition \ref{defstd}), such that this integral vanishes for any non-dominant regular couple (Proposition \ref{maincancel}). For dominant couples, the above induction process yields a recurrence relation for the integrals $\Jc\Bc_\Qc$ of $\Bc_\Qc$. Such a recurrence relation then uniquely determines these integrals, which happen to be independent of $Z$. See Proposition \ref{regcplcal}.
\subsubsection{Combinatorics of leading terms} As in Sections \ref{circleintro} and \ref{timeintintro}, we are able to calculate the leading term of $\Kc_\Qc$ for each regular couple $\Qc$, and it just remains to put them altogether. Note that each of these leading terms has the form
\[(\Kc_\Qc)_{\mathrm{app}}(t,s,k)\sim\delta^n\sum_Z\prod_{\nf\in Z}\zeta_\nf\cdot\Jc\Bc_\Qc(t,s)\cdot \Mc_{\Qc,Z}^*(k),\] see (\ref{asymptotics1}). Here $Z$ is a subset of branching nodes, $\Jc\Bc_\Qc(t,s)$ is a function of $(t,s)$ only that is also independent of $Z$, and $\Mc_{\Qc,Z}^*(k)$ is an explicit multilinear integral expression of the initial data $n_{\mathrm{in}}$ depending on $\Qc$ and $Z$, see (\ref{defintegral}). Since $\Jc\Bc_\Qc$ vanishes for non-dominant couples we just need to consider dominant $\Qc$.

The natural idea is then to classify all these terms according to the form of the expression $\Mc_{\Qc,Z}^*$, and combine the coefficients $\Jc\Bc_\Qc(t,s)$. This leads to the definition of \emph{enhanced dominant couples} which depends on $Z$, and the notation of \emph{equivalence} between enhanced dominant couples which asserts that the forms of $\Mc_{\Qc,Z}^*$ are the same. See Definition \ref{equivcpl} and Proposition \ref{multiexp}. As such, we need to calculate the combinations of coefficients
\[\sum_{\Qs\in\Xs}\prod_{\nf\in Z}\zeta_\nf\cdot\Jc\Bc_\Qc(t,t)\] where the sum is taken over all enhanced dominant couples $(\Qc,Z)$ in a fixed equivalence class $\Xs$, and we restrict to $t=s$ as this is the case of interest in Proposition \ref{mainprop4}. It turns out, see Proposition \ref{nonemptyZ}, that for equivalence classes in which $Z\neq\varnothing$, the above combinations again vanish due to delicate cancellations involving the signs $\zeta_\nf$.

Finally, Proposition \ref{emptyZ} establishes that the combinations of coefficients corresponding to $Z=\varnothing$ exactly coincide with the coefficients occurring in $\Mc_n(t,k)$. As the corresponding multilinear expressions $\Mc_{\Qc,\varnothing}^*(k)$ also match precisely, see Propositions \ref{multiexp} and \ref{wkelwp}, this then completes the regular couple part of the proofs of Propositions \ref{mainprop1} and \ref{mainprop4}.
\subsection{Non-regular couples} We now turn to the non-regular couples. Compared to the regular couple case, here we only need to obtain upper bounds instead of asymptotics, but the structures of couples are much more complicated.

First, we reduce a general couple $\Qc$ by reverting the steps $\Ab$ and $\Bb$ as in (\ref{indexintro}) whenever possible. The result, say $\Qc_{sk}$, of these operations is called the \emph{skeleton} of $\Qc$ (Proposition \ref{reduceprocess}), and is \emph{prime} in the sense that it is not obtained from any other couple by performing $\Ab$ and $\Bb$. Now by Proposition \ref{structure2}, $\Qc$ can be obtained from $\Qc_{sk}$ by attaching regular sub-couples (as well as \emph{regular trees}, see Remark \ref{regtree}, which behave similarly). This allows us to express $\Kc_\Qc$ in terms of $\Kc_{\Qc_j}$ for these regular couples $\Qc_j$, and an expression similar to $\Kc_{\Qc_{sk}}$, see (\ref{bigformula2}).

Thanks to Section \ref{regintro} we have enough information about $\Kc_{\Qc_j}$; in particular they can be divided into a remainder term which gains an extra $L^{-\nu}$ power, and a leading term which satisfies (\ref{naive}) as well as differentiability in $k$ as in (\ref{holderbound}). For simplicity we only consider the leading terms below, which can be viewed effectively as $n_{\mathrm{in}}$ multiplied by a power of $C^+\delta$.
\subsubsection{Irregular chains}\label{irrechanintro} Now we have effectively reduced $\Kc_\Qc$ to $\Kc_{\Qc_{sk}}$. Since $\Qc_{sk}$ is a prime couple which does not have any regular sub-couple, it is tempting to guess that
\begin{equation}\label{easyguess}|\Kc_{\Qc_{sk}}(t,s,k)|\lesssim\langle k\rangle^{-20}(C^+\delta)^nL^{-\nu n}\end{equation} for constant $\nu>0$, where $2n$ is the scale of $\Qc_{sk}$. Clearly (\ref{easyguess}) would imply the desired (\ref{naive2}), in view of the definition (\ref{indexintro}), but unfortunately it is not true.

A main obstacle that prevents (\ref{easyguess}) is the so-called \emph{irregular chains} (we denote them by $\Hc$). These are chains of branching nodes, such that each one is the parent of the next one, and each one has a child leaf paired to a child of the next node, and a child leaf paired to a child of the previous node (see Figure \ref{fig:irrechanintro} below).
  \begin{figure}[h!]
  \includegraphics[scale=.5]{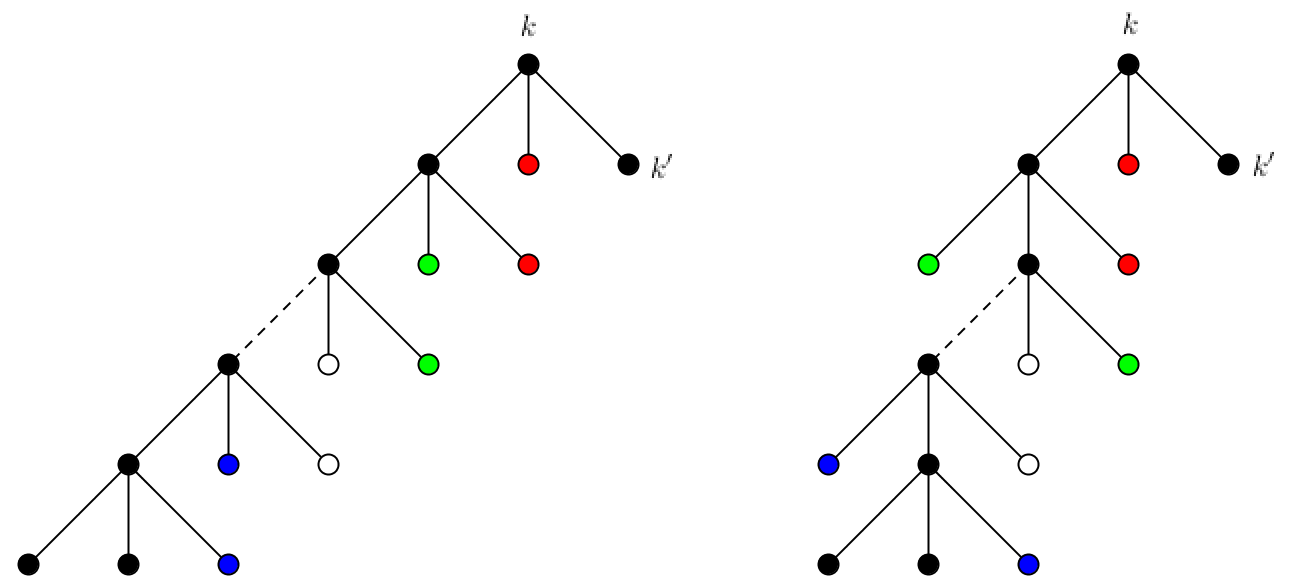}
  \caption{An example of an irregular chain, and another irregular chain congruent to it; see Section \ref{irchaincancel}. We also include the vectors $k$ and $k'$ in a decoration. Here a white leaf may be paired with a leaf in the omitted part.}
  \label{fig:irrechanintro}
\end{figure}

The irregular chains were already discussed in the earlier works \cite{CG2,DH}. In \cite{DH} it was noted that these chains create terms that diverge absolutely, which is a main challenge in reaching the sharp time scale for scaling laws $\alpha \sim L^{-\kappa}$ when $0<\kappa<1$. Here we are in the $\kappa=1$ case, and it can still be shown that if $\Qc_{sk}$ contains long irregular chains then $\Kc_{\Qc_{sk}}$ violates (\ref{easyguess}). More seriously, if the decoration in Figure \ref{fig:irrechanintro} satisfies $|k-k'|\sim L^{-1}$ (i.e. the \emph{small gap case} in Section \ref{sgcase}), then even the $\delta^n$ gain in (\ref{easyguess}) will be absent, and one can only hope for
\begin{equation}\label{badbound}|\Kc_{\Qc_{sk}}(t,s,k)|\lesssim\langle k\rangle^{-20}L^{-\nu}\end{equation} with a constant $\nu<1$ independent of $n$, which is clearly not sufficient.

Note, however, that such bad behavior is only for a \emph{single} irregular chain. In the small gap case, one can in fact group together \emph{different} irregular chains, such that the quantities $\Kc_\Qc$ for the corresponding couples $\Qc$ exhibit exact \emph{cancellations}. This leads to the definition of \emph{congruence} between different irregular chains and, by straightforward extensions, congruence between prime couples $\Qc_{sk}$ and general couples $\Qc$, see Definitions \ref{equivirrechain} and \ref{conggen}.

For two congruent irregular chains (or couples), there is a one-to-one correspondence between their sets of decorations, such that for any two decorations in correspondence, the values of $\zeta_\nf\Omega_\nf$ are exactly the same, see Proposition \ref{congdec}. The input functions in $\Kc_{\Qc_{sk}}$ for the two chains, which are either $n_{\mathrm{in}}$ or functions of similar form that come from regular sub-couples, differ only by a translation of length $|k-k'|\lesssim L^{-1}$, and the different signs
\[\zeta^*(\Qc)=\prod_{\nf}(i\zeta_\nf)\] for different chains in the same congruence class then leads to the desired cancellation, see (\ref{irrechainexp}) and (\ref{irrechainexp2}). {This} effectively improves the power $L^{-\nu}$ in (\ref{badbound}) to $L^{-q\nu}$ where $q$ is the length of the chain (with also the gain from other chains), which is then more than acceptable.

The above cancellation works only for the small gap case. For the complementary large gap case such cancellation is not available, but a direct calculation allows one to retain the $\delta^n$ gain in (\ref{easyguess}). It is still not possible, though, to achieve the negative powers of $L$ in (\ref{easyguess}), which means we need to modify the definition of $r$ in (\ref{indexintro}), see Section \ref{moleintro}. In either case the calculation involving irregular chains are done  similar to \cite{DH}, Section 3.4, using Poisson summation. See Sections \ref{sgcase} and \ref{lgcase}.

With the above analysis and by exploiting the cancellation in the small gap case, we can then reduce $\Kc_{\Qc_{sk}}$ to some expression similar to $\Kc_{\Qc_{sk}^\#}$, where $\Qc_{sk}^\#$ is the couple formed by deleting all irregular chains in $\Qc_{sk}$, see (\ref{section6fin2}). For simplicity we will denote it by $\Qc'$ below.
\subsubsection{Molecules}\label{moleintro} Now we proceed to analyze $\Qc'$, which is a prime couple and does not contain any irregular chains. At this point we will be able to accommodate logarithmic losses, so we may exploit the \emph{almost} integrability of $\Bc_{\Qc'}(t,s,\alpha[(\Nc')^*])$ in the $\alpha_\nf$ variables and reduce to the counting problem (\ref{countingintro}) described in Section \ref{classifyintro}.

In order to bound the number of solutions to (\ref{countingintro}), we notice that any such system, such as (\ref{countingintro1}) and (\ref{countingintro2}), consists of a number of quadruple equations of the form
\[a-b+c-d=0,\quad |a|_\beta^2-|b|_\beta^2+|c|_\beta^2-|d|_\beta^2=\alpha+O(\delta^{-1}L^{-2})\] which involves four vectors $(a,b,c,d)$.

The natural idea is to gradually reduce the size of the system by solving for the quadruples $(a,b,c,d)$ one at a time. Note that some quadruples will have nonempty intersection with others, hence by solving for one quadruple one may also decide some components of later quadruples. Therefore the \emph{order} in which we choose the quadruples is crucial, and we need to design a specific \emph{algorithm} depending on the structure of the couple $\Qc'$.

Before describing this algorithm, however, we need to make one shift in the point of view. Note that after solving for a quadruple and fixing some unknown vectors, we reduce to a smaller counting problem, but the new counting problem may not be coming from another couple (unless in special cases). Thus to validate the induction process, we need to shift to a structure more flexible than couples.

Note that each quadruple corresponds to a branching node and its three children in the couple $\Qc'$, and the only properties we need from $\Qc'$ are the pairwise intersections of these $4$-element subsets. We then define these $4$-element subsets as \emph{atoms} and their intersections as \emph{bonds}, to form a (non-simple) graph with maximum degree $4$, which we refer to as a \emph{molecule} (Definitions \ref{defmole0}, \ref{defmole}). Our counting problem for a couple then reduces to the counting problem for a molecule, where each unknown vector corresponds to a bond and each quadruple system corresponds to an atom; for example the system (\ref{countingintro1}) is represented by Figure \ref{fig:moleintro}. As such, solving for a quadruple corresponds to deleting an atom from the molecule, which simply results in a smaller molecule.
  \begin{figure}[h!]
  \includegraphics[scale=.5]{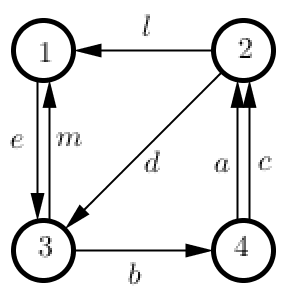}
  \caption{The molecule associated with (\ref{countingintro1}) and the couple in Figure \ref{fig:deccpl}. The arrows represent the signs of the corresponding vectors in the system, see Definition \ref{countingproblem} and Remark \ref{moledec}.}
  \label{fig:moleintro}
\end{figure}

We then design a particular \emph{molecule reduction algorithm}, by applying some specifically defined operations called \emph{steps} (Section \ref{oper}), following some particular rule (Section \ref{alg}). In each step we remove (or in some cases add) a finite number of bonds associated with some (at most $4$) atoms to reduce to a smaller molecule, and solve the local counting problem involving the corresponding quadruples. The upper bounds for such counting problems are provided by Lemma \ref{lem:counting}.

Note that the couple $\Qc'$ is prime, consequently the corresponding molecule $\Mb$ does not contain triple bonds. For such molecules, the application of the algorithm allows us to bound the number of solutions to (\ref{countingintro}) by (essentially)
\begin{equation}\label{gainintro}\Df\lesssim (\delta^{-1}L^{2d-2})^{n} L^{-\nu r_1},\end{equation}  where $2n$ is the scale of $\Qc;$, and $2r_1$ is the number of remaining atoms after removing (all copies of) the two specific structures---which we call \emph{type I and II (molecular) chains} (see Definition \ref{molechain})---from the molecule. This is proved in a \emph{rigidity theorem}, Proposition \ref{gain}, which is perhaps the single most important estimate in this paper.

Since the counting bound $\Df\lesssim(\delta^{-1}L^{2d-2})^{n}$ corresponds to the bound (\ref{naive}) for $\Kc_{\Qc'}$, we see from (\ref{gainintro}) that \emph{this $r_1$ should be defined as the index $r$, replacing the naive definition (\ref{indexintro}), to make (\ref{naive2}) valid.} More precisely, we redefine
\begin{multline}\label{indexintro2}r(\Qc)=\textrm{ the remaining size of $\Qc$, after reverting all steps $\Ab$ and $\Bb$, removing all irregular} \\\textrm{chains, and removing all type I and II chains in the resulting molecule.}
\end{multline} Although this $r$ is smaller than (\ref{indexintro}), we still have the upper bound $C^n(Cr)!$ for the number of couples with index $r$, because type I and II molecular chains and irregular chains are all explicit objects and inserting copies of them only leads to $C^n$ possibilities. Note also that a couple can be reconstructed from the corresponding molecule, again with at most $C^n$ possibilities (Proposition \ref{recover}).

The last piece of the puzzle is to guarantee the \emph{genuine} $L^1$ integrability of the $\Bc_{\Qc'}$ function in the variables \emph{associated with the type I and II chains}, as we can only afford losses of $(\log L)^{Cr}$ with the new definition (\ref{indexintro2}). As it turns out, type I chains in the molecule only come from irregular chains in the couple, which are already treated in Section \ref{irrechanintro}. As for type II chains, we can verify that each variable $\alpha_\nf$ associated with such chains again occurs twice in $\Bc_{\Qc'}$ (same as regular couples in Section \ref{timeintintro}), thus integrability can be proved in a similar manner. See Proposition \ref{section8main}.
\begin{rem} {Some concepts introduced in this work have also been discussed in earlier mathematical and physical literature such as \cite{ESY,LukSpohn}, under different names. For clarity we list some of the correspondences below:
\begin{itemize}
\item The trees, couples and molecules are different but equivalent ways to represent the standard Feynman diagrams in the literature (though the couples and molecules focus on different aspects of the structure, which is important for this paper);
\item The dominant couples, non-dominant regular couples and irregular chains are closely related to the leading diagrams, nested diagrams and necklace diagrams in earlier literature;
\item The $(1,1)$ mini-couples and mini-trees correspond to the gain and loss terms described in earlier literature;
\item The atomic counting bounds in Lemma \ref{lem:counting} is conceptually related to the crossing bounds in earlier literature; in particular the rigidity theorem, Proposition \ref{gain}, achieves the same ``gain per crossing" effect as in \cite{ESY}, but now in the \emph{nonlinear} setting.
\end{itemize}}
\end{rem}
\subsection{Operator $\Ls$, and the endgame}\label{remainder} We now discuss the $\Rb$-linear operator $\Ls$, which appears in the equation (\ref{eqnbk}) satisfied by the remainder $b$. Since $b$ will be assumed to have tiny norm in a high regularity space (Proposition \ref{finprop3}), the quadratic and cubic (in $b$) terms in (\ref{eqnbk}) are not a problem, and the only difficulty is the linear term $\Ls$.

Usually, to invert $1-\Ls$ one would like to construct a function space $\Xc$ and prove that $\Ls$ is a contraction mapping from $\Xc$ to itself. However in the current situation this seems to be problematic due to the critical nature of the problem. Indeed, in \cite{DH} the standard $X^{s,b}$ norm for $b>1/2$ is used, which certainly cannot be applied in the critical setting. One may try to use the critical $U^p$ and $V^p$ norms as in \cite{HTT}, but even they seem to be not precise enough; moreover they are $L^p$ based norms, while the classical $TT^*$ argument (see \cite{DH}, Section 3.3), which is the main tool in establishing norm bounds for random matrices or operators, works best in $L^2$.

In this paper we have found an interesting alternative to the above approach, which might be of independent interest. Namely, in order to invert $1-\Ls$ we do not really need $\Ls$ to \emph{have small norm from some space to itself}, all we need is that $\Ls$ \emph{has small spectral radius}\footnote{{This is well-known in the context of matrix analysis (see \cite{Davies}, Example 4.1.5), however we have not seen any prior example where it is applied to PDEs.}}. Note that the spectral radius of $\Ls$ is basically
\[\rho(\Ls)=\lim_{n\to\infty}\|\Ls^n\|^{1/n},\] where the norm can be chosen as the operator norm between any two reasonable spaces, and $\rho(\Ls)$ does \emph{not} really depend on any specific choice of norms. Therefore, the idea is to consider the powers $\Ls^n$, instead of $(\Ls\Ls^*)^n$ which depends on the specific choice of the Hilbert norm. This provides the motivation for Proposition \ref{mainprop2}.

Now, by (\ref{eqnbk1.5}), we can write $(\Ls b)_k(t)$ as an expression that is $\Rb$-linear in $b$ and $\Rb$-multilinear in the Gaussians; moreover this expression involves a summation over decorations of specific trees, which are obtained by attaching two sub-trees $\Tc_1$ and $\Tc_2$ to a single node. Repeating this $n$ times, we see that the kernels $(\Ls^n)_{k\ell}^\zeta(t,s)$ of $\Ls^n$, and the corresponding homogeneous components $(\Ls^n)_{k\ell}^{m,\zeta}(t,s)$, are given by expressions associated with specific trees (or more precisely a modified version of trees called \emph{flower trees}, see Definition \ref{flowers}), which have similar form as $\Jc_\Tc$ with only minor and manageable modifications, see (\ref{defkerLn}). In the same way, the correlations $\Eb|(\Ls^n)_{k\ell}^{m,\zeta}(t,s)|^2$ will have similar form as $\Kc_\Qc$ with minor modifications, see (\ref{defaltkq}). Therefore, the estimate for $\Ls^n$, as in Proposition \ref{mainprop2}, can be done without paying too much extra effort, by adapting the above proof for the estimates of $\Kc_\Qc$ and making only small changes. See Section \ref{operatornorm}.

Finally, to pass from Propositions \ref{mainprop1}--\ref{mainprop4} to Theorem \ref{main} we simply apply Lemma \ref{largedev}, exploiting the multilinear Gaussian form for $\Jc_\Tc$ to control the $L^p$ moments by $L^2$ moments for free. In controlling the operators $\Ls^n$ (Proposition \ref{finprop2}) one encounters a problem of reducing to finitely many values of $k$, which is more subtle than the similar problem occurring in \cite{DH}, but it still can be resolved by applying a refined version of Claim 3.7 in \cite{DH}. See Lemma \ref{finitelem}.
\subsection{The rest of this paper} In Section \ref{cplstructure} we examine the structure of trees and couples and prove some basic results that will be important in later proofs.

Then, Sections \ref{regasymp}--\ref{domasymp} are devoted to the analysis of regular couples. In Section \ref{regasymp} we study the integrability properties of the coefficients $\Bc_\Qc$, in Section \ref{numbertheory} we prove the number theoretic approximation lemma (Lemma \ref{approxnt}) and apply it to $\Kc_\Qc$, and in Section \ref{domasymp} we collect the asymptotics obtained in Section \ref{numbertheory} and match them with $\Mc_n(t,k)$.

Sections \ref{irchaincancel}--\ref{l1coef} are devoted to non-regular couples. In Section \ref{irchaincancel} we introduce the notion of irregular chains and exhibit the cancellation structure, in Section \ref{improvecount} we analyze the structure of the molecule obtained from the given couple $\Qc$ and use it to solve the counting problem associated with $\Kc_\Qc$, and in Section \ref{l1coef} we recover the $L^1$ integrability of $\Bc_\Qc$ in the type I and II chain variables, which finally allows us to prove Propositions \ref{mainprop1} and \ref{mainprop4}.

Finally, in Section \ref{operatornorm} we apply similar arguments as above to control the kernels of $\Ls^n$ and prove Proposition \ref{mainprop2}, and in Section \ref{endgame} we put everything together to prove Theorem \ref{main}.
\section{Structure of couples}\label{cplstructure} The central part in the proofs of Propositions \ref{mainprop1}--\ref{mainprop4} is the analysis of the correlations $\Kc_\Qc(t,s,k)$ for different couples $\Qc$, and superpositions thereof. Therefore the structure of couples will play a key role in the arguments. This will be analyzed in the current section.
\subsection{Regular couples}We start with the notion of \emph{regular couples}.
\begin{df}\label{defmini} A \emph{$(1,1)$-mini couple} is a couple formed by two ternary trees of scale $1$ with no siblings paired. It has two possibilities, shown in Figure \ref{fig:minicpl}. We assign the two-digit \emph{code} $00$ to the top one, and code $01$ to the bottom one in Figure \ref{fig:minicpl}.
  \begin{figure}[h!]
  \includegraphics[scale=.5]{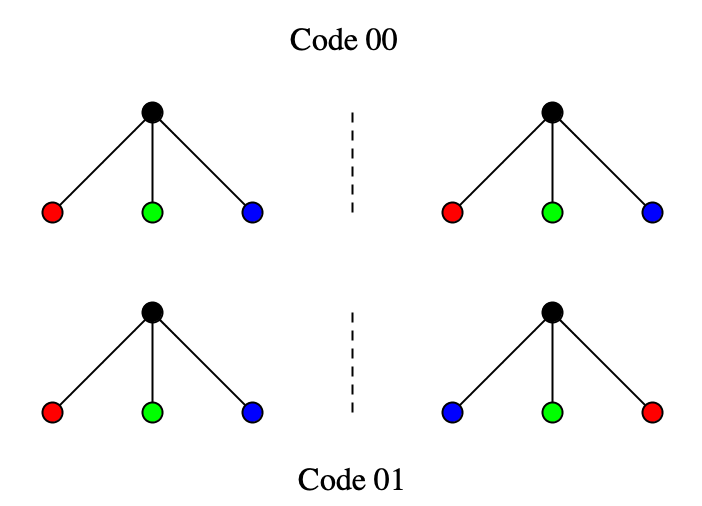}
  \caption{Two possibilities of $(1,1)$-mini couples (Definition \ref{defmini}).}
  \label{fig:minicpl}
\end{figure}

A \emph{mini tree} is a saturated paired tree of scale $2$ with no siblings paired. It has six possibilities, shown in Figure \ref{fig:minitree}; as in the figure we also assign them the two-digit codes in $\{10,\cdots,31\}$. We define a \emph{$(2,0)$-mini couple} to be the couple formed by a mini tree and a single node $\bullet$.
  \begin{figure}[h!]
  \includegraphics[scale=.5]{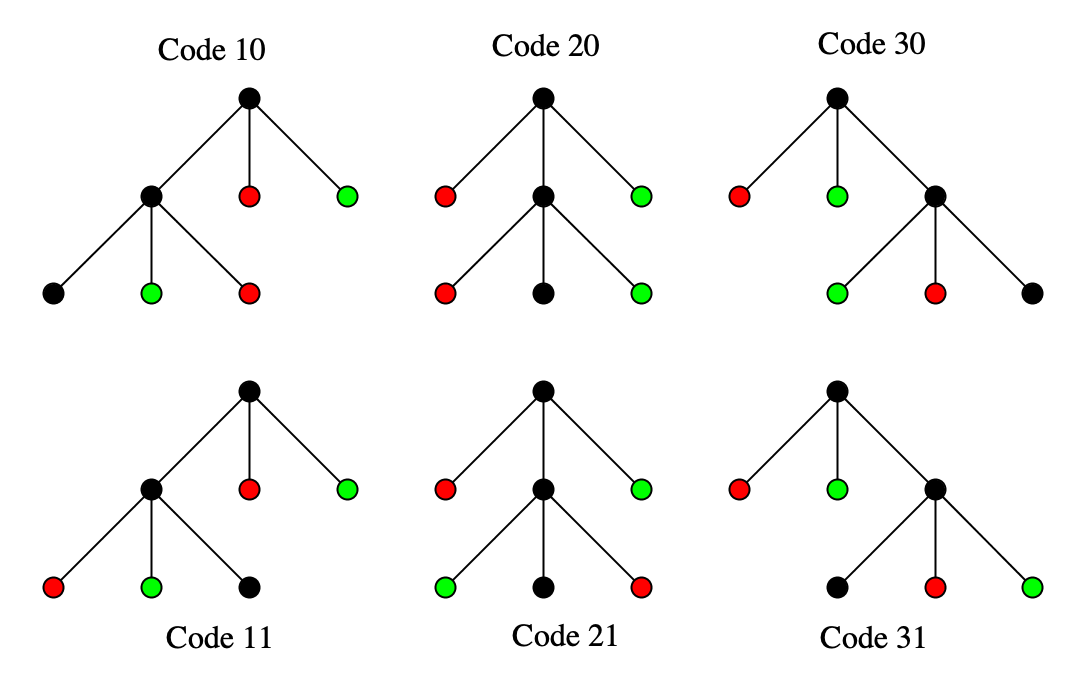}
  \caption{Six possibilities of mini trees (Definition \ref{defmini}).}
  \label{fig:minitree}
\end{figure}
\end{df}
\begin{df}\label{defreg} We define the \emph{regular couples} as follows. First the trivial couple $\times$ is regular. Suppose $\Qc$ is regular, then
\begin{enumerate}
\item The couple $\Qc_+$, formed by replacing a pair of leaves in $\Qc$ (which may or may not be in the same tree) with a $(1,1)$-mini couple, is regular (see Figure \ref{fig:regcplcst1}).
\item The couple $\Qc_+$, formed by replacing a node in $\Qc$ with a mini tree, is regular (see Figure \ref{fig:regcplcst2}).
\item All regular couples are of form (1) or (2).
\end{enumerate}
Note that the scale of a regular couple must be even. The operations described in (1) and (2) will be referred to as \emph{step $\Ab$} (acting at a pair of leaves) and \emph{step $\Bb$} (acting at a node) below.
  \begin{figure}[h!]
  \includegraphics[scale=.5]{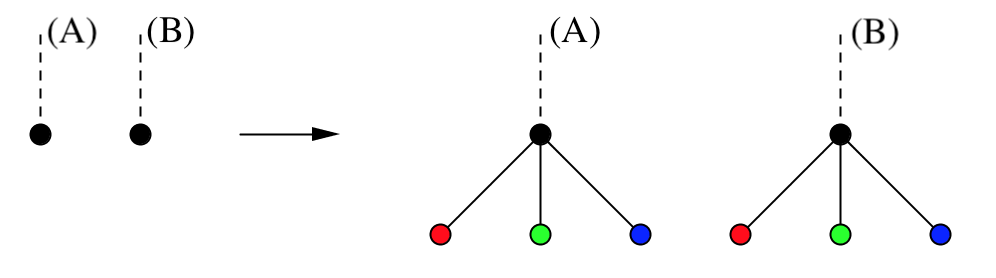}
  \caption{Step $\Ab$ of building regular couples (Definition \ref{defreg}). There are two possibilities depending on the mini couple. {The ends $A$ and $B$ represent the rest of the couple, which is unaffected by the step.}}
  \label{fig:regcplcst1}
\end{figure}
  \begin{figure}[h!]
  \includegraphics[scale=.5]{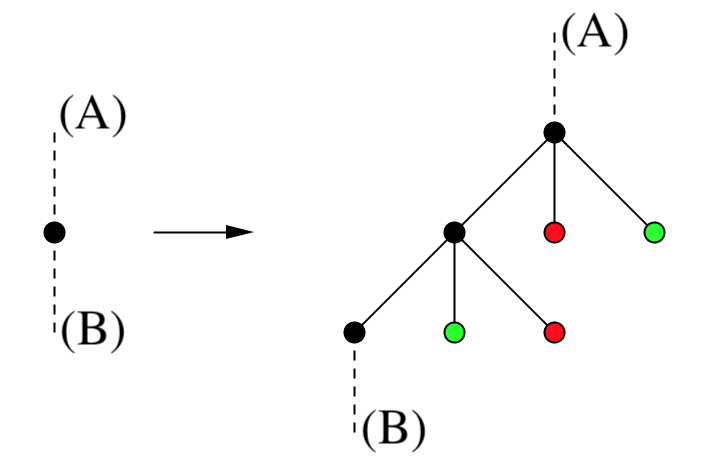}
  \caption{Step $\Bb$ of building regular couples (Definition \ref{defreg}). There are six possibilities depending on the mini tree. {The ends $A$ and $B$ represent the rest of the couple, which is unaffected by the step.}}
  \label{fig:regcplcst2}
\end{figure}
\end{df}
\begin{prop}\label{branchpair} Given any regular couple $\Qc$, there is a unique way to pair \emph{branching} nodes $\nf\in\Nc^*$ to each other, such that for any pair $\{\nf,\nf'\}$ and any decoration $\Es$ of $\Qc$ we have $\zeta_{\nf'}\Omega_{\nf'}=-\zeta_\nf\Omega_\nf$.
\end{prop}
\begin{proof} This is easily proved by induction. When $\Qc=\times$ there is nothing to prove. Suppose the result holds for $\Qc$, then let $\Qc_+$ be formed from $\Qc$ by step $\Ab$ or $\Bb$. In either case, we simply make the two new branching nodes into a pair (for step $\Ab$, these are the two roots of the $(1,1)$-mini couple which are also two leaves in $\Qc$; for step $\Bb$, these are the two branching nodes of the mini tree). See Figures \ref{fig:branpair} for a description of the corresponding decoration. {It is easy to verify that the pairings obtained this way does not depend on the order of applications of $\Ab$ and $\Bb$, hence the uniqueness.}
  \begin{figure}[h!]
  \includegraphics[scale=.5]{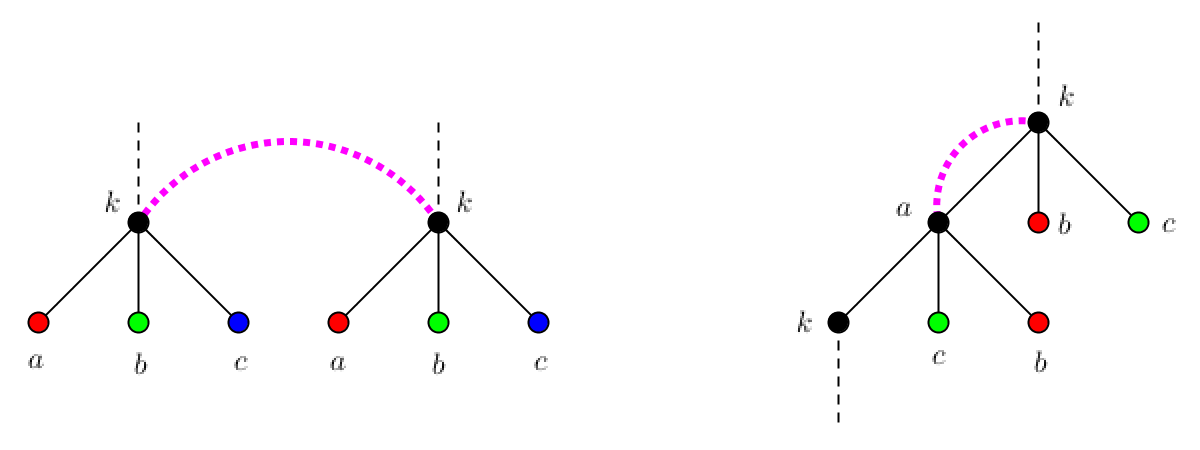}
  \caption{A new pair of branching nodes (connected by a pink dotted curve) formed by step $\Ab$ or $\Bb$; see Proposition \ref{branchpair}.}
  \label{fig:branpair}
\end{figure}
\end{proof}
\subsection{Structure of regular couples} We next analyze the structure of general regular couples.
\begin{df}\label{deflegal} Given $m\geq 0$, consider a partition $\Pc$ of $\{1,\cdots,2m\}$ into $m$ pairwise disjoint two-element subsets (or pairs). We say $\Pc$ is \emph{legal} if there do not exist $a<b<c<d$ such that $\{a,c\}\in\Pc$ and $\{b,d\}\in\Pc$. For example, when $m=3$, then $\Pc=\{\{1,2\},\{3,6\},\{4,5\}\}$ is legal, while $\Pc=\{\{1,6\},\{2,4\},\{3,5\}\}$ is not. We say $\Pc$ is \emph{dominant} if $\Pc=\{\{1,2\},\cdots,\{2m-1,2m\}\}$.
\end{df}
\begin{prop}\label{legalpair} (1) A legal partition can be obtained by inserting a pair of adjacent elements into a smaller legal partition. More precisely, $\Pc$ is legal if and only if either (i) $m=0$ and $\Pc=\varnothing$, or (ii) $m\geq 1$ and
\begin{multline*}\Pc=\big\{\{a,b\}:a<b<j,\,\{a,b\}\in\Pc_1\big\}\cup\big\{\{a,b+2\}:a<j\leq b,\,\{a,b\}\in\Pc_1\big\}\\\cup\big\{\{a+2,b+2\}:j\leq a<b,\,\{a,b\}\in\Pc_1\big\}\cup\big\{\{j,j+1\}\big\}\end{multline*} for some $1\leq j\leq 2m-1$ and some $\Pc_1$ associated with $m-1$ which is legal.

(2) Alternatively, a legal partition can be obtained by concatenating two smaller legal partitions, or enclosing a smaller legal partition in a new pair. More precisely, $\Pc$ is legal if either (i) $m=0$ and $\Pc=\varnothing$, or (ii) $m\geq 2$ and 
\[\Pc=\big\{\{a,b\}:a<b\leq 2k,\,\{a,b\}\in\Pc_1\big\}\cup\big\{\{a+2k,b+2k\}:a<b\leq 2(m-k),\,\{a,b\}\in\Pc_2\big\}\] for some $1\leq k\leq m-1$ and some $\Pc_1$ associated with $k$ and some $\Pc_2$ associated with $m-k$ which are legal, or (iii) $m\geq 1$ and 
\[\Pc=\big\{\{a+1,b+1\}:a<b\leq 2m-2,\,\{a,b\}\in\Pc_1\big\}\cup\big\{\{1,2m\}\big\}\] for some $\Pc_1$ associated with $m-1$ which is legal.
\end{prop}
\begin{proof} This is easily proved by induction.
\end{proof}
\begin{df}\label{defregchain} A \emph{regular chain} is a saturated paired tree, obtained by repeatedly applying step $\Bb$ at either a branching node or the lone leaf, as described in Definition \ref{defreg}, starting from the trivial tree $\bullet$. A \emph{regular double chain} is a couple consisting of two regular chains (where, of course, the lone leaves of the two regular chains are paired). It can also be obtained by repeatedly applying step $\Bb$ at either a branching node or a lone leaf, starting from the trivial couple $\times$.
\end{df}
\begin{prop}\label{prop3.4} The scale of a regular chain $\Tc$ is always an even number $2m$. The $2m$ branching nodes are naturally ordered by parent-child relation; denote them by $\nf_j\,(1\leq j\leq 2m)$ from top to bottom. Then, see Figure \ref{fig:regchain}, $\Tc$ is associated with a \emph{legal} partition $\Pc$ of $\{1,\cdots,2m\}$, and a code in $\{10,\cdots,31\}$ (as in Definition \ref{defmini}) for each pair, such that (i) the lone leaf is a child of $\nf_{2m}$, and (ii) for any pair $\{a,b\}\in\Pc\,(a<b)$, the two children leaves of $\nf_a$ are paired with the two children leaves of $\nf_b$ respectively, and the exact positions (relative to $\nf_a$ and $\nf_b$) and pairings of these nodes are just like in the mini tree (in which the root represents $\nf_a$ and the other branching node represents $\nf_b$) having the code of $\{a,b\}$. We also define $\Tc$ to be \emph{dominant} if the partition $\Pc$ is dominant in the sense of Definition \ref{deflegal}.
\end{prop}
\begin{proof} This is a direct consequence of Proposition \ref{legalpair} (1) and Definition \ref{defregchain}, because the trivial tree corresponds to $m=0$ and $\Pc=\varnothing$, and applying step $\Bb$ at a branching node or lone leaf, i.e. replacing it with a mini tree, just corresponds to inserting a pair of adjacent elements into $\Pc$.
\end{proof}
  \begin{figure}[h!]
  \includegraphics[scale=.5]{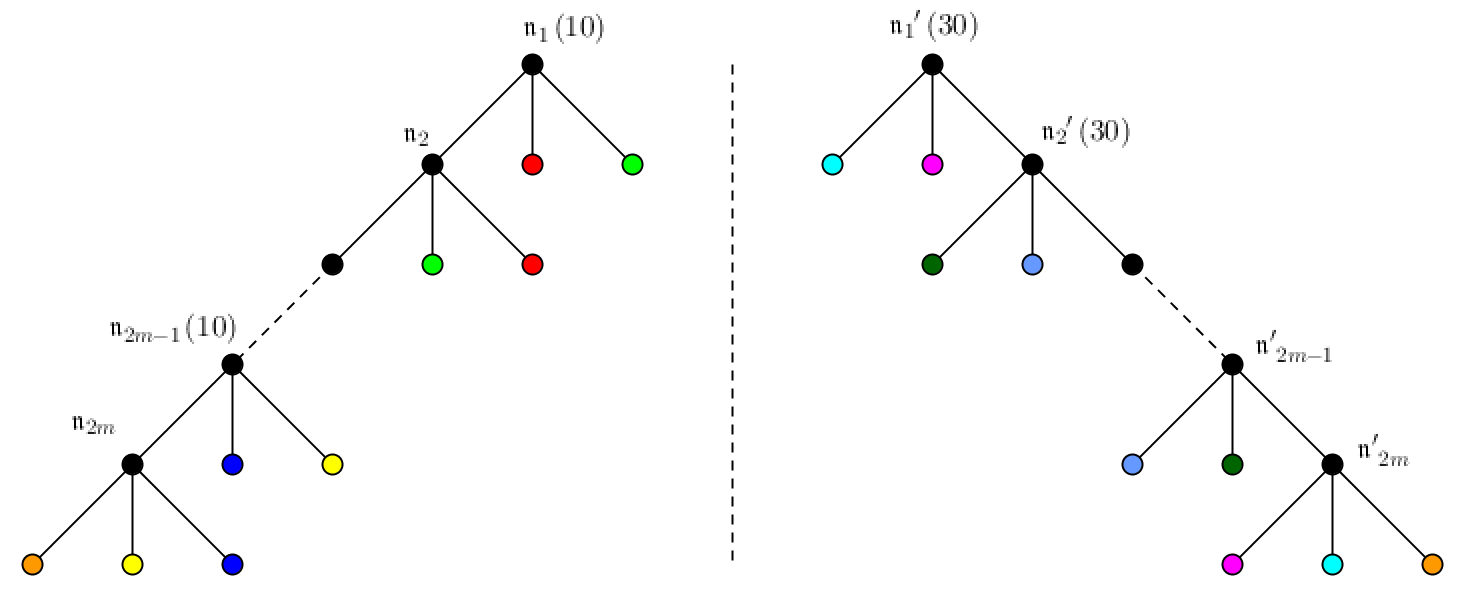}
  \caption{A regular double chain (as described in Proposition \ref{prop3.4}). The lone leaves are colored in orange. The left chain is dominant; the right chain is not, as the partition $\Pc'$ contains $\{1,2m\}$ and $\{2,2m-1\}$. The code of each mini tree is indicated beside the node $\nf_a$ and $\nf_a'$ as in Proposition \ref{prop3.4}.}
  \label{fig:regchain}
\end{figure}
 The following proposition describes (inductively) the structure of all regular couples.
\begin{prop}[Structure theorem for regular couples]\label{structure1} For any nontrivial regular couple $\Qc\neq \times$, there exists a regular couple $\Qc_0\neq\times$ which is either a $(1,1)$-mini couple or a regular double chain, such that $\Qc$ is formed by replacing each pair of leaves in $\Qc_0$ with a regular couple. Clearly each such couple has scale strictly smaller than that of $\Qc$, see Figure \ref{fig:regcpl}.
\end{prop}
\begin{proof} In the base case $n(\Qc)=2$, so $\Qc$ is either a $(1,1)$-mini couple or a $(2,0)$-mini couple (which is a regular double chain), so the result is true. Suppose the result is true for $\Qc$, with associated $\Qc_0$ and the leaf-pairs in $\Qc_0$ replaced by regular couples $\Qc_j\,(1\leq j\leq n)$. Let $\Qc_+$ be obtained from $\Qc$ by step $\Ab$ or $\Bb$ in Definition \ref{defreg}. Then:

(1) If $\Qc_0$ is any couple and one applies step $\Ab$, then this step $\Ab$ must be applied at a leaf-pair belonging to some regular couple $\Qc_i\,(i\geq 1)$. In this case the same $\Qc_0$ works for $\Qc_+$, the regular couples $\Qc_j\,(1\leq j\neq i)$ also remain the same, and the regular couple $\Qc_i$ is replaced by $\Ab\Qc_i$.

(2) If $\Qc_0$ is any couple and one applies step $\Bb$ at a node which belongs to some $\Qc_i\,(i\geq 1)$, then the same result holds as in (1) except that $\Qc_i$ is now replaced by $\Bb\Qc_i$.

(3) If we are not in case (1) or (2), and $\Qc_0$ is a $(1,1)$-mini couple, then the node where one applies step $\Bb$ must be one of the roots. In this case for $\Qc_+$ we may replace $\Qc_0$ by $\Qc_1$ which is a $(2,0)$-mini couple. Two leaf pairs in $\Qc_1$ remain leaf-pairs (note that a leaf pair can be viewed as the trivial couple), and the third leaf-pair in $\Qc_1$ is replaced by $\Qc$.

(4) If we are not in case (1) or (2), and $\Qc_0$ is a regular double chain, then the node where one applies step $\Bb$ must be a branching node of $\Qc_0$. In this case for $\Qc_+$ we may replace $\Qc_0$ by $\Bb\Qc_0$. The regular couples $\Qc_j\,(j\geq 1)$ remain the same, while the two new leaf-pairs in $\Bb\Qc_0$ (which do not belong to $\Qc_0$) remain leaf-pairs.

In any case we have verified the result for $\Qc_+$, which completes the inductive proof due to Definition \ref{defreg}.
\end{proof}
\begin{cor}\label{countcouple1} The number of regular couples of scale $n$ is at most $C^n$.
\end{cor}
\begin{proof} Let the number of regular couples of scale $n$ be $A_n$, then $A_0=1$. By Proposition \ref{structure1}, any regular couple $\Qc$ of scale $n\geq 1$ can be expressed in terms of a couple $\Qc_0$ (say of scale $1\leq m\leq n$) and regular couples $\Qc_j\,(1\leq j\leq m+1)$ of scale $n_j$ such that $n_1+\cdots +n_{m+1}=n-m$. Notice that $\Qc_0$ has at most $5^m$ choices, since $m=2m_1$ must be even, and the number of choices for the legal partition $\Pc$ in Proposition \ref{prop3.4} is the Catalan number $\binom{2m_1}{m_1}/(m_1+1)< 4^{m_1}$, and that $\Qc_0$ has $6^{m_1}$ choices, due to the codes in $\{10,\cdots,31\}$, once $\Pc$ is fixed (there are two possibilities of $(1,1)$-mini couples for $m_1=1$ but this does not affect the result), {leading to $24^{m_1}<5^m$.} This implies that
\[A_n\leq \sum_{m=1}^n5^m\sum_{n_1+\cdots +n_{m+1}=n-m}A_{n_1}\cdots A_{n_{m+1}}.\] Let $B_n$ be such that $B_0=1$ and equality holds in the above inequality for $B_n$, then $A_n\leq B_n$. Moreover the generating function $f(z)=\sum_{n\geq 0}B_nz^n$ satisfies that
\[
\begin{split}f(z)&=1+\sum_{n\geq 1}\sum_{1\leq m\leq n}(5z)^m\sum_{n_1+\cdots +n_{m+1}=n-m}B_{n_1}\cdots B_{n_{m+1}}z^{n-m}\\
&=1+\sum_{m=1}^\infty(5z)^m(f(z))^{m+1}=1+\frac{5z(f(z))^2}{1-5zf(z)}.
\end{split}\] Note that for $|z|\ll 1$ the equation \[f=1+\frac{5zf^2}{1-5zf}\Leftrightarrow f=1-5zf+10zf^2 \] has unique solution near $f=1$ which is an analytic function of $z$, we conclude that $B_n\leq C^n$ for some absolute constant $C$ (for example $C=100$), hence $A_n\leq C^n$.
\end{proof}
  \begin{figure}[h!]
  \includegraphics[scale=.5]{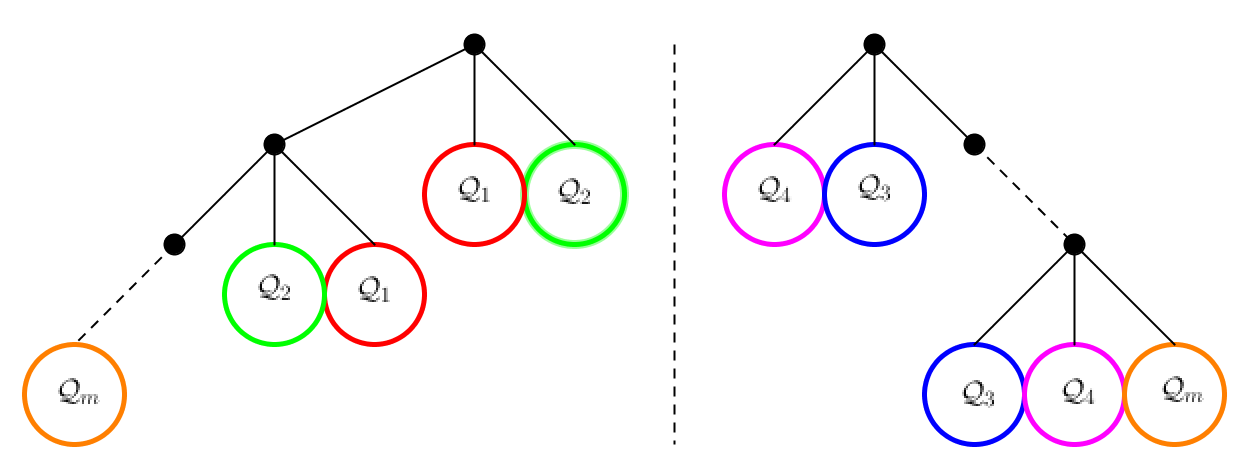}
  \caption{A regular couple with structure as described in Proposition \ref{structure1}. Here and below two circles of same color represent a regular couple $\Qc_j$. If we require $\Qc_m$ to have type 1, as in Proposition \ref{structure1.5}, then this representation is unique.}
  \label{fig:regcpl}
\end{figure}
Note that in Proposition \ref{structure1}, the choice of $\Qc_0$ may not be unique; however we can recover uniqueness under some extra assumptions.
\begin{prop}\label{structure1.5} For any regular couple $\Qc\neq\times$, we say it has \emph{type 1} if $\Qc_0$ is a $(1,1)$-mini couple in Proposition \ref{structure1}, and has \emph{type 2} if $\Qc_0$ is a regular double chain. Now, if $\Qc$ has type 2, then the choice of $\Qc_0$, as well as the whole representation, is unique, if we require that the regular couple replacing the pair of lone leaves in $\Qc_0$ is trivial or has type 1 (see Figure \ref{fig:regcpl}).
\end{prop}
\begin{proof} First, the type is well-defined, because if $\Qc_0$ is a $(1,1)$-mini couple, then for each child $\nf$ of the root of each tree, \emph{at least one} of its descendant leaves is paired with a leaf in the other tree (we shall call this property $\mathtt{L}$ in the proof below). However this is not true if $\Qc_0$ is a regular double chain.

Now suppose $\Qc$ has type 2. The roots of the chains of $\Qc_0$ are the roots of trees in $\Qc$. For each root, \emph{only one} of its three children nodes has property $\mathtt{L}$, and this must be the next branching node in $\Qc_0$. In the same way, all the subsequent branching nodes (and lone leaves) in $\Qc_0$ can be uniquely determined. The pairing structure of leaves in $\Qc_0$ is also uniquely determined by the pairing structure of $\Qc$. Moreover, the regular couple replacing the pair of lone leaves in $\Qc_0$ does \emph{not} have type 2 (i.e. it is either trivial of has type 1), if and only if neither of the chains in $\Qc_0$ can be further extended by the above process (i.e. by selecting the \emph{unique} child which has property $\mathtt{L}$). Thus the choice of $\Qc_0$ is unique. Once $\Qc_0$ is fixed, it is easy to see that the regular couples $\Qc_j$ in $\Qc$ replacing the leaf pairs in $\Qc_0$ are also uniquely determined. This completes the proof.
\end{proof}
\subsubsection{Relevant notations} For later use, let us introduce some notations related to regular couples with structure as in Proposition \ref{structure1}.
\begin{df}\label{defchoice} Given a regular couple $\Qc$, recall that the branching nodes in $\Nc^*$ are paired as in Proposition \ref{branchpair}. We shall fix a choice of $\Nc^{ch}\subset \Nc^*$ (here $ch$ means ``choice''), which contains exactly one branching node in each pair, as follows. First if $\Qc=\times$ then $\Nc^{ch}$ contains the single root of $+$ sign. If $\Qc\neq\times$, let $\Qc_0$ be uniquely fixed as in Propositions \ref{structure1} and \ref{structure1.5}.

\emph{Case 1}. If $\Qc$ has type 1, then $\Qc_0$ is a $(1,1)$-mini couple. Let $\Qc_j\,(1\leq j\leq 3)$ be the regular couples in $\Qc$ replacing the leaf-pairs of $\Qc_0$, counted from left to right in the tree whose root has $+$ sign (i.e. in the order \emph{red, green, blue} in Figure \ref{fig:minicpl}, assuming the left tree has $+$ root). Then we have $\Nc^*=\Nc_1^*\cup\Nc_2^*\cup\Nc_3^*\cup\{\rf,\rf'\}$, where $\rf$ and $\rf'$ are the two root nodes which are also paired; in particular define $\Nc^{ch}=\Nc_1^{ch}\cup\Nc_2^{ch}\cup\Nc_3^{ch}\cup\{\rf\}$, where $\rf$ is the root with $+$ sign.

\emph{Case 2}. If $\Qc$ has type 2, then $\Qc_0$ is a (nontrivial) regular double chain, which is formed by two regular chains $\Tc^+$ and $\Tc^-$ of scales $2m^+$ and $2m^-$ respectively. Let the branching nodes of $\Tc^\pm$ be $\nf_1^\pm,\cdots,\nf_{2m^\pm}^\pm$ from top to bottom, and let the legal partition of $\{1,\cdots,2m^\pm\}$ associated with $\Tc^\pm$ be $\Pc^\pm$ (see Proposition \ref{prop3.4}). Let the pair of lone leaves in $\Qc_0$ (which is a pair between a child leaf of $\nf_{2m^+}^+$ and a child leaf of $\nf_{2m^-}^-$) be replaced by a regular couple $\Qc_{lp}$ (which is trivial or has type 1; here $lp$ means ``lone pair''). If we list the pairs $\{a,b\}\in\Pc^+\,(a<b)$ in the increasing order of $a$, then the $j$-th pair $\{a,b\}$, where $1\leq j\leq m^+$, corresponds to a branching node pair $\{\nf_a^+,\nf_b^+\}$ in the sense of Proposition \ref{branchpair}. This also corresponds to a mini tree in Figure \ref{fig:minitree} (in which the root represents $\nf_a^+$ and the other branching node represents $\nf_b^+$) and two leaf-pairs in $\Qc_0$, see Proposition \ref{prop3.4}. We define the regular couple in $\Qc$ replacing the pair of \emph{red} leaves in Figure \ref{fig:minitree} by $\Qc_{j,+,1}$, and define the regular couple in $\Qc$ replacing the pair of \emph{green} leaves in Figure \ref{fig:minitree} by $\Qc_{j,+,2}$ (see Figure \ref{fig:notation} for an example). The same is done for the other regular chain $\Tc^-$. Then we have
\begin{equation}\label{newn*}\Nc^*=\bigg(\bigcup_{j,\epsilon,\iota}\Nc_{j,\epsilon,\iota}^*\bigg)\cup\Nc_{lp}^*\cup \big\{\nf_1^+,\cdots,\nf_{2m^+}^+\big\}\cup\big\{\nf_1^-,\cdots,\nf_{2m^-}^-\big\}
\end{equation} and then define
\begin{equation}\label{newnch}\Nc^{ch}=\bigg(\bigcup_{j,\epsilon,\iota}\Nc_{j,\epsilon,\iota}^{ch}\bigg)\cup\Nc_{lp}^{ch}\cup\big\{\nf_a^+:a<b\big\}\cup\big\{\nf_a^-:a<b\big\}.
\end{equation} Here in (\ref{newn*}) and (\ref{newnch}), the couples $\Qc_{j,\epsilon,\iota}$, where $\epsilon\in\{\pm\}$ and $\iota\in\{1,2\}$, are the ones described above, and $\Nc_{j,\epsilon,\iota}^*$ (and $\Nc_{j,\epsilon,\iota}^{ch}$) are defined correspondingly; similarly for $\Qc_{lp}$, $\Nc_{lp}^*$ and $\Nc_{lp}^{ch}$. Moreover in (\ref{newnch}) the $\nf_a^\pm$ are the nodes chosen above, such that $a<b$ for the pair $\{a,b\}\in\Pc^\pm$.
  \begin{figure}[h!]
  \includegraphics[scale=.5]{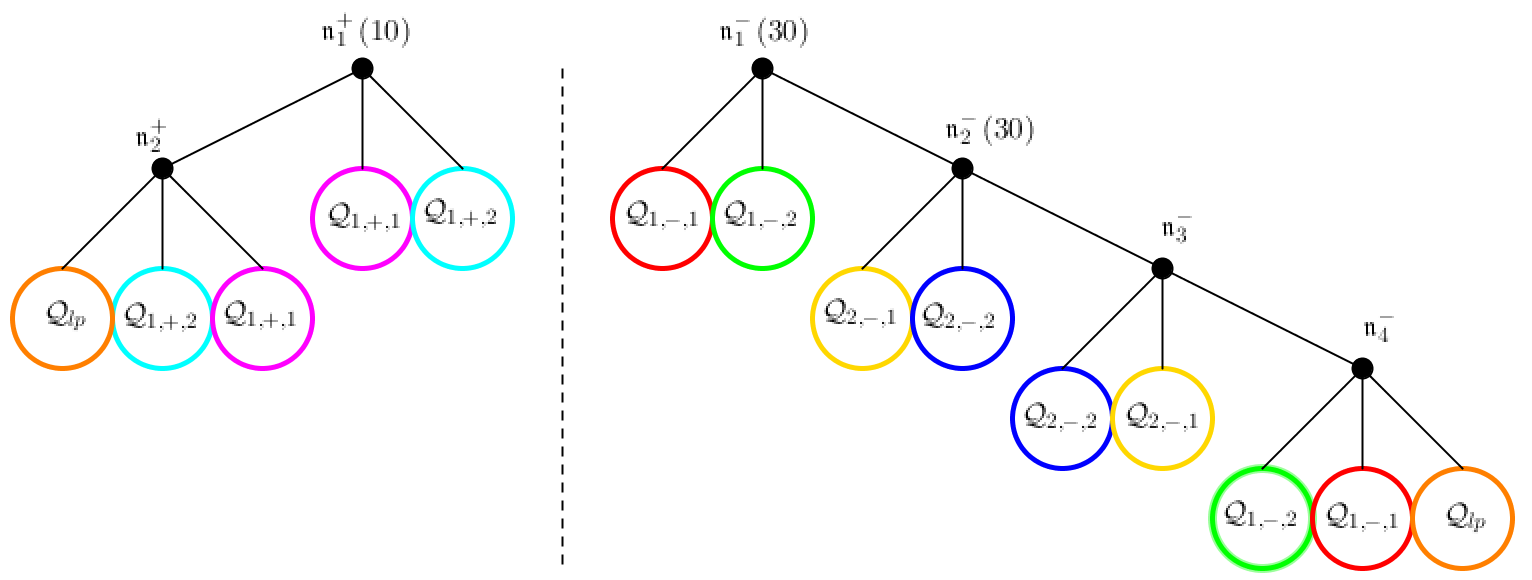}
  \caption{An example of the notations in Definition \ref{defchoice}. Here $m^+=1$ and $m^-=2$, $\Pc^+=\{\{1,2\}\}$ and $\Pc^-=\{\{1,4\},\{2,3\}\}$, and $\Qc_{lp}$ has type 1. The code of each mini tree is indicated beside the node $\nf_a^\pm$ as in Definition \ref{defchoice}.}
  \label{fig:notation}
\end{figure}
\end{df}
\subsection{Structure of general couples} We now turn to the structure of arbitrary couples.
\begin{df}\label{defsub} A \emph{prime} couple is a couple that cannot be formed from any other couple by applying steps $\Ab$ or $\Bb$ as in Definition \ref{defreg}. For example the couple in Figure \ref{fig:cpl} is prime.
\end{df}
\begin{prop}\label{reduceprocess} For any couple $\Qc$, there exists a unique prime couple $\Qc_{sk}$ such that $\Qc$ is obtained from $\Qc_{sk}$ by performing the operations in Definition \ref{defreg}; moreover $\Qc$ is regular if and only if $\Qc_{sk}=\times$. We call this $\Qc_{sk}$ the \emph{skeleton} of $\Qc$.
\end{prop}
\begin{proof} Recall the steps $\Ab$ and $\Bb$ defined in Definition \ref{defreg}, and denote the corresponding inverse operations by $\overline{\Ab}$ (where a $(1,1)$-mini sub-couple collapses to a leaf-pair) and $\overline{\Bb}$ (where a mini tree collapses to a single node). Note that it is possible that a $(1,1)$-mini sub-couple or a mini tree appears only after an operation $\overline{\Ab}$ or $\overline{\Bb}$, allowing for further operations that are not possible before this operation.

Now, starting from a couple $\Qc$, we may repeatedly apply $\overline{\Ab}$ and $\overline{\Bb}$ whenever possible until obtaining a couple $\Qc_{sk}$ where no more operation can be done. This $\Qc_{sk}$ will then be prime and satisfies the requirement. Now we need to prove the uniqueness of $\Qc_{sk}$. We first make a simple observation: if $\mathtt{D}_1$ and $\mathtt{D}_2$ each represents a $(1,1)$-mini sub-couple or a mini tree in $\Qc$, and let $\overline{\Db}_1$ and $\overline{\Db}_2$ be the corresponding inverse operations ($\overline{\Ab}$ or $\overline{\Bb}$) performed at $\mathtt{D}_1$ and $\mathtt{D}_2$ respectively, then the operations $\overline{\Db}_1$ and $\overline{\Db}_2$ commute. This can be easily verified using the definition of these inverse operations, as they are easily seen not to affect each other.

Now we can prove the uniqueness of $\Qc_{sk}$. In fact, the base case $\Qc=\times$ is obvious; suppose $\Qc_{sk}$ is unique for all $\Qc$ with smaller scale, then starting with any $\Qc$, we look for $(1,1)$-mini sub-couples and mini trees in $\Qc$. If there is none then $\Qc$ is already prime and we are done. Suppose there is at least one of them, then for each one, say $\mathtt{D}$, if the first inverse operation (say $\overline{\Db}$) is performed at $\mathtt{D}$, then the resulting $(\overline{\Db}\Qc)_{sk}$ is uniquely fixed (but may depend on $\mathtt{D}$), by applying the induction hypothesis for the smaller couple $\overline{\Db}\mathcal{Q}$. Now, let $\mathtt{D}_1$ and $\mathtt{D}_2$ be arbitrary, and let $\overline{\Db}_1$ and $\overline{\Db}_2$ be corresponding inverse operations, and let $\Qc_1=\overline{\Db}_1\overline{\Db}_2\Qc=\overline{\Db}_2\overline{\Db}_1\Qc$, then we must have $(\overline{\Db}_1\Qc)_{sk}=(\overline{\Db}_2\Qc)_{sk}=(\Qc_1)_{sk}$. This proves the uniqueness of $\Qc_{sk}$. Clearly by definition, $\Qc$ is regular if and only if $\Qc_{sk}=\times$.
\end{proof}
\begin{prop}[Structure theorem for general couples]\label{structure2} Let $\Qc$ be any couple with skeleton $\Qc_{sk}$. Then, see Figures \ref{fig:gencpl} and \ref{fig:blacksq}, $\Qc$ can be obtained from $\Qc_{sk}$ by (i) first replacing each branching node with a regular chain, and then (ii) replacing each pair of leaves in the resulting couple with a regular couple. This representation (i.e. the chain (i) and the couple in (ii) at each position) is also unique.
\end{prop}
\begin{proof} By Proposition \ref{reduceprocess}, $\Qc$ can be obtained from $\Qc_{sk}$ by applying steps $\Ab$ and $\Bb$. We induct on the scale of $\Qc$. The base case $\Qc=\Qc_{sk}$ is obvious by definition. Suppose the result is true for $\Qc$, and let $\Qc_+$ be obtained from $\Qc$ by applying $\Ab$ or $\Bb$. We know that $\Qc$ is obtained from $\Qc_{sk}$ by (i) first replacing each branching node with a regular chain, say $\Tc_j^\circ\,(1\leq j\leq n)$, resulting in a couple $\Qc_{int}$, and then (ii) replacing each leaf-pair in $\Qc_{int}$ by a regular couple, say $\Qc_j\,(1\leq j\leq m)$. Then:

(1) If one applies step $\Ab$, then this step $\Ab$ must be applied at a leaf-pair belonging to some regular couple $\Qc_i\,(1\leq i\leq m)$. In this case the $\Tc_j^\circ\,(1\leq j\leq n)$ remain the same for $\Qc_+$, the regular couples $\Qc_j\,(j\neq i)$ also remain the same, and the regular couple $\Qc_i$ is replaced by $\Ab\Qc_i$.

(2) If one applies step $\Bb$ at a node which belongs to some $\Qc_i\,(1\leq i\leq m)$, then the same result holds as in (1) except that $\Qc_i$ is now replaced by $\Bb\Qc_i$.

(3) If we are not in case (1) or (2), then the node where one applies step $\Bb$ must be a branching node of $\Qc_{int}$, hence it must be a branching node or the lone leaf of some regular chain $\Tc_i^\circ\,(1\leq i\leq n)$. In this case the regular chains $\Tc_j^\circ\,(j\neq i)$ remain the same for $\Qc_+$, and the regular couples $\Qc_j\,(1\leq j\leq m)$ also remain the same. The regular chain $\Tc_i^\circ$ is replaced by $\Bb\Tc_i^\circ$, while the two new leaf-pairs in $\Bb\Tc_i^\circ$ (which do not belong to $\Tc_i^\circ$) remain leaf-pairs. In any case we have verified the result for $\Qc_+$, which proves existence.

Now to prove uniqueness of the representation, let $\Qc_{int}$ be the couple formed after performing step (i). Given $\Qc_{sk}$, clearly $\Qc_{int}$ uniquely determines the regular chains in step (i) replacing the branching nodes in $\Qc_{sk}$, so it suffices to show that $\Qc$ uniquely determines $\Qc_{int}$ (once $\Qc_{int}$ is given, it is also clear that $\Qc$ uniquely determines the regular couples in step (ii) replacing the leaf pairs in $\Qc_{int}$). However we can show, via a case-by-case argument, that $\Qc_{int}$ contains no nontrivial regular sub-couple (i.e. no two subtrees rooted at two nodes in $\Qc_{int}$ form a nontrivial regular couple). Since $\Qc$ is formed from $\Qc_{int}$ by replacing each leaf pair with a regular couple, we see that $\Qc_{int}$ can be reconstructed by collapsing each \emph{maximal regular sub-couple} (under inclusion) in $\Qc$ to a leaf pair (because any regular sub-couple of $\Qc$ must be a sub-couple of one of the regular couples in $\Qc$ replacing a leaf pair in $\Qc_{int}$). Clearly, this collapsing process is commutative as explained in Proposition \ref{reduceprocess}, hence the resulting couple $\Qc_{int}$ is unique. This completes the proof.
\end{proof}
\begin{rem}\label{regtree} We will call a saturated paired tree, which is a regular chain with each leaf pair replaced by a regular couple, a ``regular tree". Thus in Proposition \ref{structure2}, $\Qc$ can also be formed from $\Qc_{sk}$ by replacing each branching node with a regular tree and each leaf pair with a regular couple; see Figures \ref{fig:gencpl} and \ref{fig:blacksq}. This representation is also unique.
\end{rem}
  \begin{figure}[h!]
  \includegraphics[scale=.5]{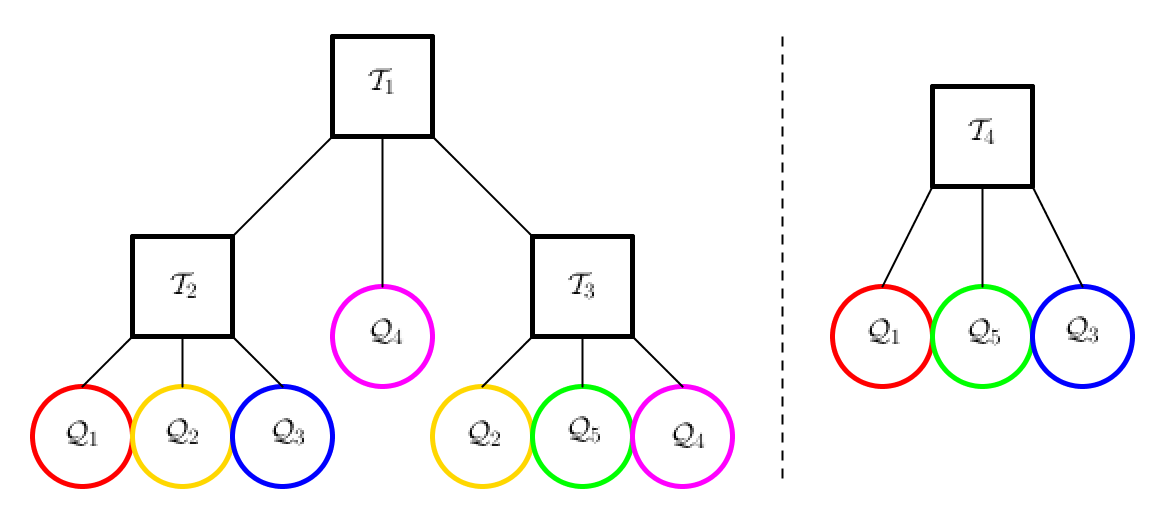}
  \caption{An example of a couple, with structure as described in Proposition \ref{structure2}, whose skeleton is the couple in Figure \ref{fig:cpl}. Here a black square represents a regular tree.}
  \label{fig:gencpl}
\end{figure}
  \begin{figure}[h!]
  \includegraphics[scale=.5]{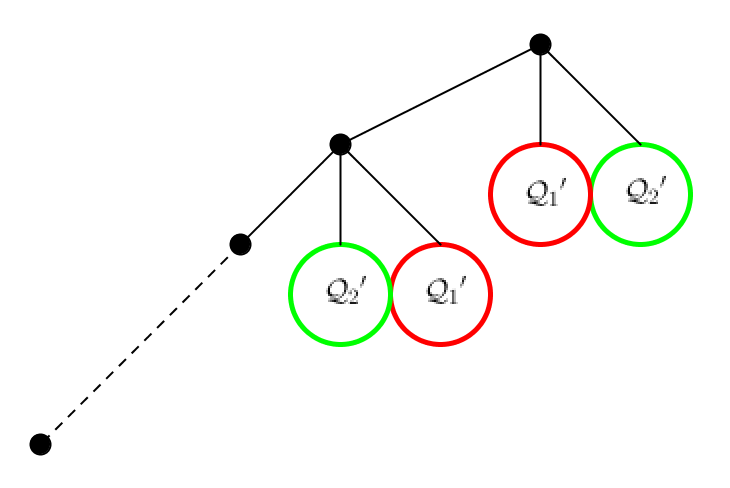}
  \caption{An example of a regular tree in Figure \ref{fig:gencpl}, as defined in Remark \ref{regtree}. The bottom black node is the lone leaf.}
  \label{fig:blacksq}
\end{figure}
\begin{cor}\label{corcpl} Fix any $\mathcal{Q}_{sk}$, the number of couples $\Qc$ with skeleton $\Qc_{sk}$ such that $n(\Qc)\leq n$ is at most $C^n$.
\end{cor}
\begin{proof} Let the couple formed after performing step (i) in the statement of Proposition \ref{structure2} be $\Qc_{int}$. If $n(\Qc_{sk})=m$, then $\Qc_{int}$ is determined by $m$ regular chains of total scale at most $n$, so the number of choices for $\Qc_{int}$ is at most
\[\sum_{n_1+\cdots +n_m\leq n}C_0^{n_1}\cdots C_0^{n_m}\leq (2C_0)^n\] for some constant $C_0$. For each fixed $\Qc_{int}$, let $n(\Qc_{int})=r$, then $\Qc$ is formed from $\Qc_{int}$ by performing step (ii) in the statement of Proposition \ref{structure2}, so it is determined by $r+1$ regular couples of total scale at most $n$, so the number of choices for $\Qc$ is at most \[\sum_{n_1+\cdots +n_{r+1}\leq n}C_1^{n_1}\cdots C_1^{n_{r+1}}\leq (2C_1)^n\] for some other constant $C_1$. Therefore the total umber of choices for $\Qc$ is at most $(4C_0C_1)^n$.
\end{proof}
\subsection{Dominant couples} We will identify a subclass of regular couples, namely the \emph{dominant couples}, which give rise to the nonzero leading terms.
\begin{df}\label{defstd} We define a regular couple $\Qc$ to be \emph{dominant} inductively as follows. First the trivial couple $\times$ is dominant. Suppose $\Qc\neq\times$, let $\Qc_0$ be uniquely determined by Propositions \ref{structure1} and \ref{structure1.5}, and let $\Qc_j\,(j\geq 1)$ be the regular couples in $\Qc$ replacing leaf pairs in $\Qc_0$. Then we define $\Qc$ to be dominant, if (i) $\Qc_0$ is either a $(1,1)$-mini couple or a regular double chain formed by two \emph{dominant} regular chains, and (ii) each regular couple $\Qc_j$ is dominant.
\end{df}
  \begin{figure}[h!]
  \includegraphics[scale=.45]{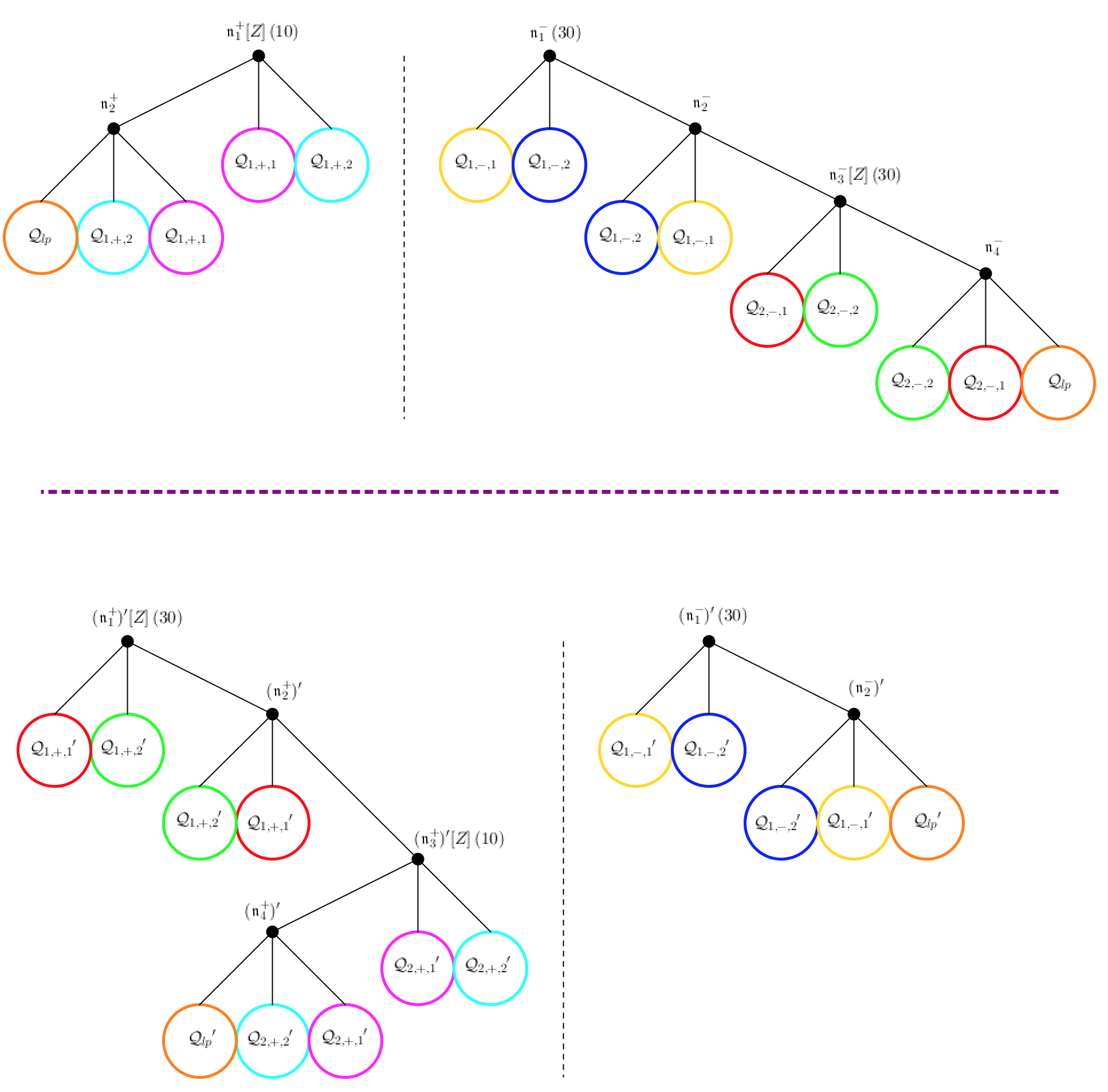}
  \caption{An example of two equivalent dominant couples with $(m^+,m^-)=(1,2)$ and $((m^+)',(m^-)')=(2,1)$. Here we assume that (i) couples represented by the same color are equivalent (with the corresponding $Z$ sets, which are omitted), and (ii) the symbol $[Z]$ means the value of $j$ corresponding to this branching node belongs to the suitable $Z^\pm$ or $(Z^\pm)'$ set. The code of each mini tree is indicated beside the node $\nf_a^\pm$ and $(\nf_a^\pm)'$, and $\Qc_{lp}$ and $\Qc_{lp}'$ have type 1.}
  \label{fig:equivdomcpl}
\end{figure}
\subsubsection{An equivalence relation} Given a dominant couple $\Qc$, recall that $\Nc^*$ is the set of branching nodes, and $\Nc^{ch}\subset \Nc^*$ is defined in Definition \ref{defchoice}. Let $Z$ be a \emph{special} subset of $\Nc^{ch}$, which will be defined inductively in Definition \ref{equivcpl} below; we call $\Qs:=(\Qc,Z)$ an \emph{enhanced dominant couple}, and when $Z=\varnothing$, we will also denote $\Qs=(\Qc,\varnothing)$ just by $\Qc$ for convenience.
\begin{df}\label{equivcpl} We inductively define special subsets $Z\subset\Nc^{ch}$, and an equivalence relation $\sim$ between enhanced dominant couples $\Qs:=(\Qc,Z)$, as follows. First $\varnothing$ is a special subset and the enhanced trivial couple $(\times,\varnothing)$ is only equivalent to itself, moreover two enhanced dominant couples where the $\Qc$ have different types are never equivalent.

Next, if $\Qs=(\Qc,Z)$ and $\Qs'=(\Qc',Z')$, where $\Qc$ and $\Qc'$ have type $1$ (recall the definition of type in Proposition \ref{structure1.5}), then we have $\Nc^{ch}=\Nc_1^{ch}\cup\Nc_2^{ch}\cup\Nc_3^{ch}\cup\{\rf\}$ where $\rf$ is the root with $+$ sign, see Definition \ref{defchoice}. Then $Z$ is special if and only if $Z=Z_1\cup Z_2\cup Z_3$ (i.e. $\rf$ is \emph{not} in $Z$) where $Z_j\subset \Nc_j^{ch}$ is special, and similarly for $\Qc'$. Let $\Qs_j=(\Qc_j,Z_j)$, we define $\Qs\sim\Qs'$ if and only if $\Qs_j\sim\Qs_j'$ for $1\leq j\leq 3$.

Now let $\Qs$ and $\Qs'$ be as before, but suppose $\Qc$ and $\Qc'$ have type $2$. Let $\Qc_0$ be associated with $\Qc$ as in Proposition \ref{structure1.5}, and similarly for $\Qc'$ (same for the other objects appearing below). Suppose the two regular chains of $\Qc_0$ have scale $2m^+$ and $2m^-$ respectively, and let the branching nodes in $\Qc_0$ be $\nf_a^\pm(1\leq a\leq 2m^\pm)$, where $\nf_{2j-1}^\pm$ is paired with $\nf_{2j}^\pm$ for $1\leq j\leq m^\pm$, see Figure \ref{fig:equivdomcpl}. We will use the notations in Definition \ref{defchoice}, and note that $\Qc$ is dominant and $\Qc_{lp}$ is trivial or has type 1. Recall that 
\begin{equation}\label{unionch}\Nc^{ch}=\bigg(\bigcup_{j,\epsilon,\iota}\Nc_{j,\epsilon,\iota}^{ch}\bigg)\cup\Nc_{lp}^{ch}\cup\big\{\nf_{2j-1}^+:1\leq j\leq m^+\big\}\cup\big\{\nf_{2j-1}^-:1\leq j\leq m^-\big\}
\end{equation} as in (\ref{newnch}); then $Z$ is special if and only if
\begin{equation}\label{unionz}Z=\bigg(\bigcup_{j,\epsilon,\iota}Z_{j,\epsilon,\iota}\bigg)\cup Z_{lp}\cup\big\{\nf_{2j-1}^+:j\in Z^+\big\}\cup\big\{\nf_{2j-1}^-:j\in Z^-\big\}
\end{equation} for some special subsets $Z_{j,\epsilon,\iota}\subset \Nc_{j,\epsilon,\iota}^{ch}$ and $Z_{lp}\subset\Nc_{lp}^{ch}$, and some subsets $Z^\pm\subset \{1,\cdots,m^\pm\}$. Similar representations are defined for $\Qs'$. For $\epsilon\in\{\pm\}$ and each $1\leq j\leq m^\epsilon$, consider the tuple $(\mathtt{I}_{j,\epsilon},\mathtt{c}_{j,\epsilon},\Xs_{j,\epsilon,1},\Xs_{j,\epsilon,2})$. Here $\mathtt{I}_{j,\epsilon}=1$ if $j\in Z^\epsilon$ and $\mathtt{I}_{j,\epsilon}=0$ otherwise, $\mathtt{c}_{j,\epsilon}\in\{1,2,3\}$ is the \emph{first digit of} the code of the mini tree associated with the pair $\{2j-1,2j\}\in \Pc^\epsilon$ (see Definition \ref{defchoice}; this code is also the code assigned for the pair $\{2j-1,2j\}\in \Pc^\epsilon$ as in Proposition \ref{prop3.4}). Moreover $\Xs_{j,\epsilon,\iota}$ is the equivalence class of the enhanced dominant couple $\Qs_{j,\epsilon,\iota}=(\Qc_{j,\epsilon,\iota},Z_{j,\epsilon,\iota})$ for $\iota\in\{1,2\}$, and let $\Ys$ be the equivalence class of the enhanced dominant couple $\Qs_{lp}=(\Qc_{lp},Z_{lp})$.

We now define $\Qs\sim\Qs'$, if and only if (i) $m^++m^-=(m^+)'+(m^-)'$, and (ii) the tuples coming from $\Qc_0$ (there are total $m^++m^-$ of them) form \emph{a permutation of} the corresponding tuples coming from $\Qc_0'$ (there are total $(m^+)'+(m^-)'$ of them), and (iii) $\Ys=\Ys'$. Finally, note that if $\Qs=(\Qc,Z)$ and $\Qs'=(\Qc',Z')$ are equivalent then $n(\Qc)=n(\Qc')$ and $|Z|=|Z'|$. When $\Qs\sim\Qs'$ with $Z=Z'=\varnothing$, we also say that $\Qc\sim\Qc'$.
\end{df}
\subsection{Encoded trees}\label{encodedtree} Let $\Tc$ be a tree, we will assign to each of its \emph{branching} nodes $\nf\in\Nc$ a \emph{code} $\mathtt{c}=\mathtt{c}_\nf\in\{0,1,2,3\}$ to form an \emph{encoded tree}.

Given a encoded tree $\Tc$, define the \emph{canonical path} to be the unique path $\gamma$ starting from the root $\rf$ and ending at either a leaf or a branching node with code $0$, such that any non-terminal node $\nf\in\gamma$ is a branching node with code $\mathtt{c}_\nf\in\{1,2,3\}$, and the next node $\nf'$ in $\gamma$ is the $\mathtt{c}_\nf$-th child of $\nf$ counting from left to right.
\begin{df}\label{codechain} An \emph{encoded chain} is an encoded tree whose canonical path $\gamma$ ends at a leaf, and for any non-terminal node $\nf\in\gamma$, the two children of $\nf$ other than $\nf'$ are both leaves, where $\nf'$ is the next node in {$\gamma$}. We call the endpoint of $\gamma$, which is a leaf, the \emph{tail leaf} of $\Tc$.
\end{df}
\begin{prop}\label{structure+} Given an encoded tree $\Tc\neq \bullet$, we say $\Tc$ has \emph{type 1} if its root $\rf$ has code $\mathtt{c}_\rf=0$, otherwise we say it has \emph{type 2}. Now, for any type 2 encoded tree $\Tc$, there is a unique encoded chain $\Tc_0\neq\bullet$, such that $\Tc$ is obtained from $\Tc_0$ by replacing each leaf with an encoded tree, and that the tail leaf is replaced by either $\bullet$ (i.e. remains a leaf) or an encoded tree of type 1.
\end{prop}
\begin{proof} This is straightforward from the definition. In fact $\Tc$ has type 1 if and only if $\mathtt{c}_\rf=0$ for the root $\rf$; suppose $\mathtt{c}_\rf\in\{1,2,3\}$, then $\Tc_0$ must have the same canonical path $\gamma$ as $\Tc$. Thus $\Tc_0$ must be the encoded tree formed by selecting each node in $\gamma$ and collapsing the subtree rooted at this node to a leaf. Such $\Tc_0$ is clearly unique, and is nontrivial when $\mathtt{c}_\rf\in\{1,2,3\}$.
\end{proof}
\begin{df}\label{equcodetree} We define the equivalence relation between encoded trees as follows. First the trivial tree $\bullet$ is only equivalent to itself, and encoded trees of different type are not equivalent. Now suppose $\Tc$ and $\Tc'$ are two encoded trees of type 1, then define $\Tc\sim \Tc'$ if and only if $\Tc_j\sim \Tc_j'$ for $1\leq j\leq 3$, where $\Tc_j$ and $\Tc_j'$ are the subtrees of $\Tc$ and $\Tc'$ respectively, from left to right.

Now suppose $\Tc$ and $\Tc'$ are two encoded trees of type 2, then by Proposition \ref{structure+} there exists a unique encoded chain $\Tc_0$ such that $\Tc$ is formed by replacing each leaf of $\Tc_0$ with an encoded tree, and the same holds for $\Tc'$. Let the branching nodes of $\Tc_0$ from top to bottom be $\nf_j\,(1\leq j\leq m)$. For each $1\leq j\leq m$, let $\Tc_{j,1}$ and $\Tc_{j,2}$ be the two encoded trees that replace the two children of $\nf_j$ other than $\nf_{j+1}$ (or the tail leaf), counted from left to right; moreover let $\Tc_{ta}$ be the encoded tree replacing the tail leaf, which is either trivial or has type 1. Consider the triples $(\mathtt{c}_j,\Zs_{j,1},\Zs_{j,2})$ for each $j$, where $\mathtt{c}_j$ is {the code} of $\nf_j$, and $\Zs_{j,\iota}$ is the equivalence class of $\Tc_{j,\iota}$ for $\iota\in\{1,2\}$. Then the encoded trees $\Tc$ and $\Tc'$ are equivalent, if and only if (i) $m=m'$ where $m'$ is defined similarly for $\Tc'$, (ii) the triples $(\mathtt{c}_j,\Zs_{j,1},\Zs_{j,2})$ form a \emph{permutation} of the corresponding triples coming from $\Tc_0'$, and (iii) $\Tc_{ta}$ is equivalent to $\Tc_{ta}'$.
\end{df}
\subsubsection{Dominant couples and encoded trees} Given any dominant couple $\Qc$, we can inductively define a unique encoded tree $\Tc$ associated to $\Qc$, as follows.
\begin{df}\label{cpl-tree} Let $\Qc$ be a dominant couple, we define the encoded tree $\Tc$ associated with $\Qc$ as follows. First if $\Qc=\times$ then define $\Tc=\bullet$. Suppose $\Qc$ has type 1, then let $\Qc_j\,(1\leq j\leq 3)$ be defined as in Definition \ref{defchoice}, then define $\Tc$ to be the encoded tree such that the root has code $0$, and the three subtrees are $\Tc_j\,(1\leq j\leq 3)$ which are associated with $\Qc_j$, from left to right.

Now suppose $\Qc$ is a dominant couple of type 2. Let the relevant notations like $\Qc_{j,\epsilon,\iota}$ and $\Qc_{lp}$ be as in Definition \ref{defchoice}. Let $m=m^++m^-$, consider the triples $(\mathtt{c}_{j,\epsilon},\Qc_{j,\epsilon,1},\Qc_{j,\epsilon,2})$ for $\epsilon\in\{\pm\}$ and $1\leq j\leq m^\epsilon$, where $\mathtt{c}_{j,\epsilon}$ is the first digit of the code of the mini-tree associated with the pair $\{2j-1,2j\}\in\Pc^\epsilon$ as in Definition \ref{equivcpl}; we rearrange them putting the $\epsilon=+$ triples before the $\epsilon=-$ ones, and in increasing order of $j$ for fixed sign. Let the rearranged tuples be $(\mathtt{c}_i,\Qc_{i,1},\Qc_{i,2})$ for $1\leq i\leq m$, then $\Tc$ is defined as follows. First let $\Tc_0$ be the encoded chain which has $m$ branching nodes $\nf_i\,(1\leq i\leq m)$ from top to bottom with code $\mathtt{c}_i$, then for each $i$, replace the two children leaves of $\nf_i$ other than $\nf_{i+1}$ (or the tail leaf) with $\Tc_{i,1}$ and $\Tc_{i,2}$ which are the encoded trees associated to $\Qc_{i,1}$ and $\Qc_{i,2}$ by induction hypothesis. Finally the tail leaf is replaced by $\Tc_{ta}$ which is the encoded tree associated with $\Qc_{lp}$.
\end{df}
\begin{prop}\label{bijection} The mapping from dominant couples $\Qc$ to encoded trees $\Tc$, as defined in Definition \ref{cpl-tree}, is surjective. Moreover two dominant couples $\Qc$ and $\Qc'$ are equivalent in the sense of Definition \ref{equivcpl}, if and only if the associated encoded trees $\Tc$ and $\Tc'$ are equivalent in the sense of Definition \ref{equcodetree}. In particular this mapping induces a bijection between the equivalence classes of dominant couples and equivalence classes of encoded trees.
\end{prop}
\begin{proof} The mapping is surjective because for any $\Tc$ one can always construct $\Qc$ by reverting the construction in Definition \ref{cpl-tree}, following the same induction process using Proposition \ref{structure+}. Now recall that equivalence between dominant couples $\Qc$ is defined as a special case in Definition \ref{equivcpl} with $Z=\varnothing$, thus in Definition \ref{equivcpl} for type 2 (type 1 is similar), the first component $\mathtt{I}_{j,\pm}$ of the tuple $(\mathtt{I}_{j,\pm},\mathtt{c}_{j,\pm},\Xs_{j,\pm,1},\Xs_{j,\pm,2})$ is always $0$. Therefore, part (ii) of the equivalence relation between $\Qc$ and $\Qc'$ can be described as the triples $(\mathtt{c}_{j,\pm},\Xs_{j,\pm,1},\Xs_{j,\pm,2})$ coming from $\Qc$ being a permutation of the triples coming from $\Qc'$ (as well as other similar conditions). By induction hypothesis, this is equivalent to the triples $(\mathtt{c}_i,\Zs_{i,1},\Zs_{i,2})$ coming from $\Qc$ being a permutation of the triples coming from $\Qc'$, where $\Zs_{i,\iota}$ is the equivalence class of the encoded tree associated to $\Qc_{i,\iota}$, and the triples $(\mathtt{c}_i,\Qc_{i,1},\Qc_{i,2})$ are rearranged from the triples $(\mathtt{c}_{j,\epsilon},\Qc_{j,\epsilon,1},\Qc_{j,\epsilon,2})$ as in Definition \ref{cpl-tree}. Then, using Definition \ref{equcodetree}, we see that this is equivalent to part (ii) of the equivalence relation between $\Tc$ and $\Tc'$. Similarly the other parts also match, therefore $\Qc$ being equivalent to $\Qc'$ is equivalent to $\Tc$ being equivalent to $\Tc'$.
\end{proof}
\subsubsection{A summary}\label{summary} We will be using the equivalence relations between enhanced dominant couples $\Qs=(\Qc,Z)$ (say $\sim_1$), between dominant couples $\Qc$ (say $\sim_2$), and between encoded trees (say $\sim_3$). Clearly $\sim_2$ is a special case of $\sim_1$, and Proposition \ref{bijection} establishes a bijection between equivalence classes under $\sim_2$ and equivalence classes under $\sim_3$. By abusing notation, below we will use the notation $\Xs$ (and similarly $\Ys$ etc.) to denote an equivalence class in each of these cases; the precise meaning will be clear from the context. For later use, we list a few easily verified facts about these equivalence classes below.

(1) We know that equivalent (enhanced) dominant couples and encoded trees must have the same scale and $|Z|$, and the same type. The bijection in Proposition \ref{bijection} also preserves the type; moreover, if $\Qc$ has scale $2n$, then the associated $\Tc$ has scale $n$. If $\Xs$ denotes the equivalence class for both objects, we will define the \emph{half-scale} of $\Xs$ to be $n$.

(2) If the net sign $\zeta^*(\Qc)$ of a dominant couple $\Qc$ is defined by (\ref{defzetaq}), and we define the net sign $\zeta^*(\Tc)$ of an encoded tree $\Tc$ by
\begin{equation}\label{defzetat}\zeta^*(\Tc)=\prod_{\nf\in\Nc}(-1)^{\mathtt{c}_\nf}\end{equation}where $\mathtt{c}_\nf$ is the code of $\nf$, then these signs are preserved under equivalence (including $\sim_1$), and also under the bijection in Proposition \ref{bijection} (so $\zeta^*(\Qc)=\zeta^*(\Tc)$ if $\Tc$ is associated to $\Qc$).

(3) Let $\Xs$ be an equivalence class of enhanced dominant couples. Then, if $\Xs$ has type 1, it can be uniquely determined (bijectively) by an \emph{ordered} triple $(\Xs_1,\Xs_2,\Xs_3)$ of equivalence classes of enhanced dominant couples. If $\Xs$ has type 2, it can be uniquely determined (bijectively) by the following objects:
\begin{itemize}
\item A positive integer $m\geq 1$;
\item An \emph{unordered} collection of (ordered) tuples $(\mathtt{I}_j,\mathtt{c}_j,\Xs_{j,1},\Xs_{j,2})$ for $1\leq j\leq m$, where each $\mathtt{I}_j\in\{0,1\}$, each $\mathtt{c}_j\in\{1,2,3\}$ and each $\Xs_{j,1}$ and $\Xs_{j,2}$ is an equivalence class of enhanced dominant couples;
\item An equivalence class $\Ys$ of enhanced dominant couples that is trivial or has \emph{type 1}.
\end{itemize}

(4) Let $\Xs$ be an equivalence class of dominant couples (with $Z=\varnothing$) or encoded trees. Then the same description in (3) is valid, except that the tuple $(\mathtt{I}_j,\mathtt{c}_j,\Xs_{j,1},\Xs_{j,2})$ should be replaced by the triple $(\mathtt{c}_j,\Xs_{j,1},\Xs_{j,2})$.

\section{Regular couples I: the $\Ac$ and $\Bc$ coefficients}\label{regasymp} We start with the analysis of $\Kc_\Qc$ associated to the regular couples $\Qc$, which will occupy up to Section \ref{domasymp}. The first step is to obtain suitable estimates for the coefficients $\Bc_\Qc$ occurring in (\ref{defkq}), which is based on $\Ac_\Tc$ occurring in (\ref{formulajt}).
\subsection{Properties of the coefficients $\Bc_\Qc$}\label{propertiesB}
Recall the coefficients $\Ac_\Tc=\Ac_\Tc(t,\alpha[\Nc])$ and $\Bc_\Qc=\Bc_\Qc(t,s,\alpha[\Nc^*])$ defined in (\ref{defcoefa}) and (\ref{defcoefb}). By induction, we can also write
\begin{equation}\Ac_\Tc(t,\alpha[\Nc])=\int_\Dc\prod_{\nf\in\Nc}e^{\zeta_\nf \pi i\alpha_\nf t_\nf}\,\mathrm{d}t_\nf,
\label{defcoefa2}
\end{equation} where the domain
\begin{equation}\label{timedom0}\Dc=\big\{t[\Nc]:0<t_{\nf'}<t_\nf<t,\mathrm{\ whenever\ }\nf'\mathrm{\ is\ a\ child\ node\ of\ }\nf\big\},\end{equation} and similarly
\begin{equation}\label{defcoefb2}\Bc_\Qc(t,s,\alpha[\Nc^*])=\int_\Ec\prod_{\nf\in\Nc^*}e^{\zeta_\nf \pi i\alpha_\nf t_\nf}\,\mathrm{d}t_\nf,
\end{equation} where the domain
\begin{multline}\label{timedom}\Ec=\big\{t[\Nc^*]:0<t_{\nf'}<t_\nf,\mathrm{\ whenever\ }\nf'\mathrm{\ is\ a\ child\ node\ of\ }\nf;\\t_\nf<t\mathrm{\ whenever\ }\nf\in\Nc^+\mathrm{\ and\ }t_\nf<s\mathrm{\ whenever\ }\nf\in\Nc^-\big\}.
\end{multline}Now suppose $\Qc$ is a \emph{regular} couple. {If we fix the pairing of branching nodes as in Proposition \ref{branchpair}, then for any decoration $\Es$ of $\Qc$, we must have $\zeta_{\nf'}\Omega_{\nf'}=-\zeta_\nf\Omega_\nf$ for any pair $\{\nf,\nf'\}$ of branching nodes.} Let $\Nc^{ch}$ be defined as in Definition \ref{defchoice}, then we may define $\widetilde{\Bc}_\Qc=\widetilde{\Bc}_\Qc(t,s,\alpha[\Nc^{ch}])$ by
\begin{equation}\label{deftildeb}\widetilde{\Bc}_\Qc(t,s,\alpha[\Nc^{ch}])=\Bc_\Qc(t,s,\alpha[\Nc^*]),\end{equation} assuming that $\alpha[\Nc^*\backslash\Nc^{ch}]$ is defined such that $\zeta_{\nf'}\alpha_{\nf'}=-\zeta_\nf\alpha_\nf$ for each pair $\{\nf,\nf'\}$.
\subsubsection{Structure of $\widetilde{\Bc}_\Qc$}\label{recursive} For a regular couple $\Qc\neq\times$, let $\Qc_0$ be uniquely determined by Propositions \ref{structure1} and \ref{structure1.5}, which is either a $(1,1)$-mini couple, or a nontrivial regular double chain, such that $\Qc$ is obtained from $\Qc_0$ by replacing each leaf-pair with a regular couple. We will use the notations of Definition \ref{defchoice}, including for example $\Pc^\pm$, $m^\pm$ and $\Qc_{j,\epsilon,\iota}$, $\Qc_{lp}$ and $\Nc_{j,\epsilon,\iota}^{ch}$, $\Nc_{lp}^{ch}$ etc.

\emph{Case 1}. If $\Qc$ has type 1, then by (\ref{defcoefb2}) and (\ref{deftildeb}), we deduce that
\begin{equation}\label{newexp1}\widetilde{\Bc}_\Qc(t,s,\alpha[\Nc^{ch}])=\int_0^t\int_0^s e^{\pi i\alpha_\rf(t_1-s_1)}\prod_{j=1}^3\widetilde{\Bc}_{\Qc_j}(t_1,s_1,\alpha[\Nc_j^{ch}])\,\mathrm{d}t_1\mathrm{d}s_1.
\end{equation} We remark that in (\ref{newexp1}), the variables $(t_1,s_1)$ appearing in $\widetilde{\Bc}_{\Qc_j}$ may be replaced by $(s_1,t_1)$ for some $j$, depending on the signs of the leaves of $\Qc_0$.

\emph{Case 2}. Suppose $\Qc$ has type 2. For each $1\leq j\leq m^\pm$, let $\{a,b\}$ be the $j$-th pair in $\Pc^\pm$ where $a<b$, then $\nf_a^+\in\Nc^{ch}$; we shall rename $\alpha_{\nf_a^+}:=\alpha_j^+$, and define $\beta_a^+:=\zeta_{\nf_a^+}\alpha_{\nf_a^+}=\epsilon_j^+\alpha_j^+$ where $\epsilon_j^+=\zeta_{\nf_a^+}\in\{\pm\}$ and $\beta_b^+:=\zeta_{\nf_b^+}\alpha_{\nf_b^+}=-\epsilon_j^+\alpha_j^+$. The same is done for the other regular chain $\Tc^-$. Then, by these definitions, and (\ref{defcoefb2}) and (\ref{deftildeb}), we deduce that
\begin{equation}\label{newexp2}
\begin{aligned}\widetilde{\Bc}_\Qc(t,s,\alpha[\Nc^{ch}])&=\int_{t>t_1>\cdots >t_{2m^+}>0}e^{\pi i(\beta_1^+t_1+\cdots +\beta_{2m^+}^+t_{2m^+})}\prod_{j=1}^{m^+}\prod_{\iota=1}^2\widetilde{\Bc}_{\Qc_{j,+,\iota}}\big(t_a,t_b,\alpha[\Nc_{j,+,\iota}^{ch}]\big)\\&\times\int_{s>s_1>\cdots >s_{2m^-}>0}e^{\pi i(\beta_1^-s_1+\cdots +\beta_{2m^-}^-s_{2m^-})}\prod_{j=1}^{m^-}\prod_{\iota=1}^2\widetilde{\Bc}_{\Qc_{j,-,\iota}}\big(s_a,s_b,\alpha[\Nc_{j,-,\iota}^{ch}]\big)\\&\times\widetilde{\Bc}_{\Qc_{lp}}\big(t_{2m^+},s_{2m^-},\alpha[\Qc_{lp}^{ch}]\big)\prod_{j=1}^{m^+}\mathrm{d}t_j\prod_{j=1}^{m^-}\mathrm{d}s_j.
\end{aligned}
\end{equation} As in \emph{Case 1}, we remark that in some factors $(t_a,t_b)$ may be replaced by $(t_b,t_a)$, and similarly for $(s_a,s_b)$ and $(t_{2m^+},s_{2m^-})$, depending on the signs of the relevant leaves. Note also that $\Qc_{lp}$ is trivial (in which case $\widetilde{\Bc}_{\Qc_{lp}}\equiv 1$) or has type 1; this is not needed here, but will be useful later.
\subsubsection{Estimates for $\widetilde{\Bc}_\Qc$} The goal of this section is to prove the following
\begin{prop}\label{maincoef} Let $\Qc$ be a regular couple of scale $2n$. Then, the function \[\widetilde{\Bc}_\Qc\big(t,s,\alpha[\Nc^{ch}]\big)\] is the sum of at most $2^{n}$ terms. For each term there exists a subset $Z\subset\Nc^{ch}$, such that this term has form
\begin{equation}\label{maincoef1}
\prod_{\nf\in Z}\frac{\chi_\infty(\alpha_\nf)}{\zeta_\nf\pi i\alpha_\nf}\cdot\int_{\Rb^2}\Cc\big(\lambda_1,\lambda_2,\alpha[\Nc^{ch}\backslash Z]\big)e^{\pi i(\lambda_1t+\lambda_2s)}\,\mathrm{d}\lambda_1\mathrm{d}\lambda_2
\end{equation} for $t,s\in[0,1]$, where $\chi_\infty$ is as in Section \ref{notations}. In (\ref{maincoef1}) the function $\Cc$ satisfies the estimate
\begin{equation}\label{maincoef2}\int \langle \lambda_1\rangle^{1/4}\langle \lambda_2\rangle^{1/4}\big|\partial_\alpha^\rho\Cc\big(\lambda_1,\lambda_2,\alpha[\Nc^{ch}\backslash Z]\big)\big|\,\mathrm{d}\alpha[\Nc^{ch}\backslash Z]\mathrm{d}\lambda_1\mathrm{d}\lambda_2\leq C^n(2|\rho|)!
\end{equation} for any multi-index $\rho$, and we also have the weighted estimate
\begin{equation}\label{maincoef2.5}\int \langle \lambda_1\rangle^{1/8}\langle \lambda_2\rangle^{1/8}\cdot\max_{\nf\in \Nc^{ch}\backslash Z}\langle \alpha_\nf\rangle^{1/8}\big|\Cc\big(\lambda_1,\lambda_2,\alpha[\Nc^{ch}\backslash Z]\big)\big|\,\mathrm{d}\alpha[\Nc^{ch}\backslash Z]\mathrm{d}\lambda_1\mathrm{d}\lambda_2\leq C^n.
\end{equation}We will denote the $(\lambda_1,\lambda_2)$ integral in (\ref{maincoef1}) by $\widetilde{\Bc}_{\Qc,Z}=\widetilde{\Bc}_{\Qc,Z}(t,s,\alpha[\Nc^{ch}\backslash Z])$, so we have \begin{equation}\label{maincoef3}\widetilde{\Bc}_\Qc(t,s,\alpha[\Nc^{ch}])=\sum_{Z\subset\Nc^{ch}}\prod_{\nf\in Z}\frac{\chi_\infty(\alpha_\nf)}{\zeta_\nf\pi i\alpha_\nf}\cdot \widetilde{\Bc}_{\Qc,Z}(t,s,\alpha[\Nc^{ch}\backslash Z]).\end{equation}
\end{prop} The proof of Proposition \ref{maincoef} is done by induction, using the recursive description in (\ref{newexp1}) and (\ref{newexp2}). Clearly the hardest case is \emph{Case 2}, where $\Qc_0$ is a regular double chain. Therefore, before proving Proposition \ref{maincoef} in Section \ref{proofmaincoef} below, we first need to analyze the expressions associated with regular chains. This will be done in Section \ref{regchainest}.
\subsection{Regular chain estimates}\label{regchainest} Let $\Pc$ be a legal partition of $\{1,\cdots,2m\}$. As in Section \ref{recursive}, we list the pairs $\{a,b\}\in\Pc\,(a<b)$ in the increasing order of $a$. If the $j$-th pair is $\{a,b\}$, we define $\beta_a=\epsilon_j\alpha_j$ and $\beta_b=-\epsilon_j\alpha_j$, where $1\leq j\leq m$ and $\epsilon_j\in\{\pm\}$. For this section, we also introduce the parameters $\lambda_a\,(1\leq a\leq 2m)$ and $\lambda_0$, and define $\mu_j=\lambda_a+\lambda_b$ if the $j$-th pair is $\{a,b\}$. Define now
\begin{equation}\label{regchaink}K(t,\alpha_1,\cdots,\alpha_m,\lambda_0,(\lambda_a)_{1\leq a\leq 2m}):=\int_{t>t_1>\cdots >t_{2m}>0}e^{\pi i[(\beta_1+\lambda_1)t_1+\cdots +(\beta_{2m}+\lambda_{2m})t_{2m}]+\pi i\lambda_0t_{2m}}\,\mathrm{d}t_1\cdots\mathrm{d}t_{2m}.\end{equation} If we define the operator
\[I_\beta f(t)=\int_0^t e^{\pi i\beta s}f(s)\,\mathrm{d}s,\] then we have
\begin{equation}\label{kusingI}K(t,\alpha_1,\cdots,\alpha_m,\lambda_0,(\lambda_a)_{1\leq a\leq 2m})=I_{\beta_1+\lambda_1}\cdots I_{\beta_{2m}+\lambda_{2m}}(e^{\pi i\lambda_0s})(t).\end{equation} By definition, if we replace $\alpha_j$ by $\widetilde{\alpha_j}=\alpha_j+\epsilon_j\lambda_a$, where $\{a,b\}\,(a<b)$ is the $j$-the pair in $\Pc$, and replace $\lambda_a$ by $\widetilde{\lambda_a}=0$ and replace $\lambda_b$ by $\widetilde{\lambda_b}=\lambda_a+\lambda_b=\mu_j$, it is easily seen that
\begin{equation}\label{translationk}K(t,\alpha_1,\cdots,\alpha_m,\lambda_0,\lambda_1,\cdots,\lambda_{2m})=K(t,\widetilde{\alpha_1},\cdots,\widetilde{\alpha_m},\lambda_0,\widetilde{\lambda_1},\cdots,\widetilde{\lambda_{2m}}).\end{equation} Therefore, in this section we will assume $\lambda_a=0$ and $\lambda_b=\mu_j$ for the $j$-th pair $\{a,b\}\,(a<b)$.

For the purpose of Section \ref{classjr} below, we also define the operators
\[J_{\alpha;\gamma_1,\gamma_2}f(t)=\int_0^t e^{\pi i\alpha(t-s)}e^{\pi i(\gamma_1t+\gamma_2s)}f(s)\,\mathrm{d}s\] and 
\[R_{\alpha,\beta;\gamma_1,\gamma_2,\gamma_3}f(t)=\int_0^t \frac{\chi_\infty(\alpha+\gamma_3)}{\alpha+\gamma_3}e^{\pi i\beta(t-s)}e^{\pi i(\gamma_1t+\gamma_2s)}f(s)\,\mathrm{d}s.\] Given variables $(\alpha_a,\cdots,\alpha_b)$, we define a \emph{{bundle}} to be any linear combination $y_a\alpha_a+\cdots+y_b\alpha_b$ where $y_j\in\{-1,0,1\}$. Moreover, below we always view the operators as mapping functions on $[0,1]$ to functions on $[0,1]$.
\subsubsection{Class $J$ and $R$ operators}\label{classjr}
\begin{df}\label{defclassjr}Let $E$ be a finite set of positive integers, and $A\subset E$. We define an operator $\Jc=\Jc_{\alpha[A],\mu[E]}$, which depends on the variables $\alpha[A]$ and $\mu[E]$,  to have class $J$ (and norm $\|\Jc\|=1$), if we have
\begin{equation}\label{typeJ0}\Jc_{\alpha[A],\mu[E]}=\int m(\alpha[A],\mu[E],\gamma_1,\gamma_2)J_{\ell;\gamma_1,\gamma_2}\,\mathrm{d}\gamma_1\,\mathrm{d}\gamma_2,\end{equation} where $\ell$ is a {bundle} of $\alpha[A]$, and $m=m(\alpha[A],\mu[E],\gamma_1,\gamma_2)$ is a function such that
\begin{equation}\label{typeJ2}{\int \bigg(1+\sum_{j\in A}|\alpha_j|+|\gamma_1|\bigg)^{1/4}\,|\partial_\alpha^\rho m(\alpha[A],\mu[E],\gamma_1,\gamma_2)|\,\mathrm{d}\alpha[A]\mathrm{d}\gamma_1\mathrm{d}\gamma_2\leq(2|\rho|)!},\end{equation} for all $\mu[E]$ and multi-index $\rho$. Note that the weight on the left hand side of (\ref{typeJ2}) does not involve $|\gamma_2|$.

We also define an operator $\Rc=\Rc_{\alpha[A],\mu[E]}$, which again depends on the variables $\alpha[A]$ and $\mu[E]$, to have class $R$ (and norm $\|\Rc\|=1$), if $1\in A$ (called the \emph{special index}), and we have
\begin{equation}\label{typeR0}\Rc_{\alpha[A],\mu[E]}=\int m(\alpha[A\backslash\{1\}],\mu[E],\gamma_1,\gamma_2,\gamma_3)R_{\ell_1+\epsilon\alpha_1,\ell_2+\epsilon\alpha_1;\gamma_1,\gamma_2,\gamma_3}\,\mathrm{d}\gamma_1\mathrm{d}\gamma_2\mathrm{d}\gamma_3\end{equation} where $\epsilon\in\{\pm\}$, $\ell_1$ and $\ell_2$ are two {bundle}s of $\alpha[A\backslash\{1\}]$, and $m=m(\alpha[A\backslash\{1\}],\mu[E],\gamma_1,\gamma_2,\gamma_3)$ is a function such that
\begin{equation}\label{typeR2}{\int\bigg(1+\sum_{j\in A\backslash \{1\}}|\alpha_j|+|\gamma_1|+|\gamma_3|\bigg)^{1/4}| \partial_\alpha^\rho m(\alpha[A\backslash\{1\}],\mu[E],\gamma_1,\gamma_2,\gamma_3)|\,\mathrm{d}\alpha[A\backslash\{1\}]\mathrm{d}\gamma_1\mathrm{d}\gamma_2\mathrm{d}\gamma_3\leq(2|\rho|)!},\end{equation} for all $\mu[E]$ and multi-index $\rho$. Note that $m=m(\alpha[A\backslash\{1\}],\mu[E],\gamma_1,\gamma_2,\gamma_3)$ does not depend on $\alpha_1$, and the weight on the left hand side of (\ref{typeR2}) also does not involve $|\gamma_2|$.

More generally, we also define an operator {$\Jc$} to have class $J$ (or $R$) if it can be written as a linear combination {(say $\Jc=\sum_\ell \alpha_\ell\Jc_\ell$)} of operators {$\Jc_\ell$} satisfying (\ref{typeJ0})--(\ref{typeJ2}) (or (\ref{typeR0})--(\ref{typeR2})) for different choices of $\ell$ (or $(\ell_1,\ell_2)$); define the norm $\|\Jc\|$ and $\|\Rc\|$ to be the infimum of {$\sum_\ell|\alpha_\ell|$ over all such representations}. Below we will study compositions of class $J$ and $R$ operators, and compositions of them with other explicit operators; when doing so we always understand that the variables $\alpha_j$ and $\mu_j$ involved in different operators are different.
\end{df}
\subsubsection{Compositions of class $J$ and $R$ operators}
\begin{lem}\label{regchainlem1} The composition of two class $J$ operators is of class $J$, and the norms satisfy that $\|\Jc^{(1)}\Jc^{(2)}\|\leq C\|\Jc^{(1)}\|\cdot\|\Jc^{(2)}\|$ (the same will be true for subsequent lemmas).
\end{lem}
\begin{proof} Let $\Jc_{\alpha[A],\mu[E]}^{(1)}$ and $\Jc_{\alpha[B],\mu[F]}^{(2)}$ be of class $J$; we may assume that each satisfies (\ref{typeJ0})--(\ref{typeJ2}). Let $\Jc^{(3)}$ be their composition, which is an operator depending on the variables $(\alpha[A\cup B],\mu[E\cup F])$, of form
\[\Jc_{\alpha[A\cup B],\mu[E\cup F]}^{(3)}=\int m^{(1)}(\alpha[A],\mu[E],\gamma_1,\gamma_2)m^{(2)}(\alpha[B],\mu[F],\gamma_3,\gamma_4)J_{\ell_1;\gamma_1,\gamma_2}J_{\ell_2;\gamma_3,\gamma_4}\,\mathrm{d}\gamma_1\mathrm{d}\gamma_2\mathrm{d}\gamma_3\mathrm{d}\gamma_4,\] where $\ell_1$ is a {bundle} of $\alpha[A]$, and $\ell_2$ is a {bundle} of $\alpha[B]$.

Now look at the operator $J:=J_{\ell_1;\gamma_1,\gamma_2}J_{\ell_2;\gamma_3,\gamma_4}$, we have
\[\begin{aligned}Jf(t)&=\int_0^t e^{\pi i\ell_1(t-z)+\pi i\gamma_1t+\pi i\gamma_2z}\,\mathrm{d}z\int_0^z e^{\pi i\ell_2(z-s)+\pi i\gamma_3z+\pi i\gamma_4s}f(s)\,\mathrm{d}s\\
&=\int_0^tf(s)\,\mathrm{d}s\int_0^{t-s}e^{\pi i\ell_1(t-s-u)+\pi i\ell_2u+\pi i(\gamma_1t+\gamma_4s)+\pi i(\gamma_2+\gamma_3)(s+u)}\,\mathrm{d}u\\
&=\int_0^te^{\pi i\ell_1(t-s)+\pi i\gamma_1t+\pi i(\gamma_2+\gamma_3+\gamma_4)s}f(s)\,\mathrm{d}s\int_0^{t-s}e^{\pi i(\ell_2-\ell_1+\gamma_2+\gamma_3)u}\,\mathrm{d}u.
\end{aligned} 
\] We decompose $J=J'+J''$ where in $J'$ we multiply the kernel by $\chi_0(\ell_2-\ell_1+\gamma_2+\gamma_3)$ and in $J''$ we multiply by $\chi_\infty(\ell_2-\ell_1+\gamma_2+\gamma_3)$.

To deal with $J'$, notice that $\chi_0(\gamma)\int_0^v e^{\pi i\gamma u}\,\mathrm{d}u$ equals a compactly supported Gevrey $2$ function in $\gamma$ and $v$ for $v\in[0,1]$ {(which can be explicitly written down, say by multiplying by $\chi_0(v-1/2)$)}, so we may rewrite
\[
\begin{aligned}J'f(t)&=\int_0^te^{\pi i\ell_1(t-s)+\pi i\gamma_1t+\pi i(\gamma_2+\gamma_3+\gamma_4)s}f(s)\,\mathrm{d}s\int_\Rb M(\ell_2-\ell_1+\gamma_2+\gamma_3,\sigma)e^{\pi i\sigma(t-s)}\,\mathrm{d}\sigma
\\&=\int_\Rb M(\ell_2-\ell_1+\gamma_2+\gamma_3,\sigma)\,\mathrm{d}\sigma\int_0^te^{\pi i\ell_1(t-s)+\pi i(\gamma_1+\sigma)t+\pi i(\gamma_2+\gamma_3+\gamma_4-\sigma)s}f(s)\,\mathrm{d}s
\end{aligned}\] where $M$ is a fixed decaying Gevrey $2$ function in two real variables. {We refer to \cite{Rod93} for basic properties of Gevrey functions.} Therefore, the contribution of $J'$ to $\Jc^{(3)}$ equals
\begin{multline*}\int m^{(1)}(\alpha[A],\mu[E],\gamma_1,\gamma_2)m^{(2)}(\alpha[B],\mu[F],\gamma_3,\gamma_4)\\\times M(\ell_2-\ell_1+\gamma_2+\gamma_3,\sigma)J_{\ell_1;\gamma_1+\sigma,\gamma_2+\gamma_3+\gamma_4-\sigma}\,\mathrm{d}\gamma_1\mathrm{d}\gamma_2\mathrm{d}\gamma_3\mathrm{d}\gamma_4\mathrm{d}\sigma.\end{multline*}
 Note that $\ell_1$ is also a {bundle} of $\alpha[A\cup B]$, we can choose (note that the $\gamma_j$ associated with the composition $m^{(3)}$ are called $\gamma_j'$; this will be assumed for subsequent lemmas as well)
\begin{multline*}m^{(3)}(\alpha[A\cup B],\mu[E\cup F],\gamma_1',\gamma_2')=\int_{\gamma_1+\sigma={\gamma_1'}}\int_{\gamma_2+\gamma_3+\gamma_4=\sigma+{\gamma_2'}}m^{(1)}(\alpha[A],\mu[E],\gamma_1,\gamma_2)\\\times m^{(2)}(\alpha[B],\mu[F],\gamma_3,\gamma_4)M(\ell_2-\ell_1+\gamma_2+\gamma_3,\sigma)\,\mathrm{d}\gamma_2\mathrm{d}\gamma_3\mathrm{d}\gamma_1.\end{multline*} which takes care of the contribution of $J'$. Then, if we do not take derivatives and do not count the weight in (\ref{typeJ2}), the norm for $m^{(3)}$ is easily bounded using the corresponding norms for $m^{(1)}$ and $m^{(2)}$. If we do not take derivatives but include the weight
\[\bigg(1+\sum_{j\in A\cup B}|\alpha_j|+|{\gamma_1'}|\bigg)^{1/4}\] in (\ref{typeJ2}), we may decompose it into three parts $(1+\sum_{j\in A}|\alpha_j|)^{1/4}$, $(\sum_{j\in B}|\alpha_j|)^{1/4}$ and $|{\gamma_1'}|^{1/4}$. The first two parts can be estimated using the corresponding norms for $m^{(1)}$ or $m^{(2)}$, while for $|{\gamma_1'}|^{1/4}$ we may use $|{\gamma_1'}|\leq|\gamma_1|+|\sigma|$, together with the corresponding norms for $m^{(1)}$ and the decay in $\sigma$.

Next consider the higher order derivative estimates. The argument will be the same for the subsequent lemmas, so we will not repeat this later. Note that $m^{(3)}$ is a trilinear expression of the functions $m^{(1)}$, $m^{(2)}$ and $M$; in subsequent lemmas we may have higher degrees of multilinearity, but they will never exceed $9$. Now by Leibniz rule we have
\begin{multline*}\partial_\alpha^\rho m^{(3)}(\alpha[A\cup B],\mu[E\cup F],{\gamma_1'},{\gamma_2'})=\sum_{\rho^1+\rho^2+\rho^3=\rho}\frac{\rho!}{(\rho^1)!(\rho^2)!(\rho^3)!}\int_{\gamma_1+\sigma={\gamma_1'}}\int_{\gamma_2+\gamma_3+\gamma_4=\sigma+{\gamma_2'}}\\\partial_\alpha^{\rho^1}m^{(1)}(\alpha[A],\mu[E],\gamma_1,\gamma_2)\cdot \partial_\alpha^{\rho^2}m^{(2)}(\alpha[B],\mu[F],\gamma_3,\gamma_4)\partial_\alpha^{\rho^3}M(\ell_2-\ell_1+\gamma_2+\gamma_3,\sigma)\,\mathrm{d}\gamma_2\mathrm{d}\gamma_3\mathrm{d}\gamma_1,\end{multline*} so the norm of $\partial_\alpha^\rho m^{(3)}$ can be bounded in the same way as above, but using the norms of $\partial_\alpha^{\rho^1}m^{(1)}$, $\partial_\alpha^{\rho^2}m^{(2)}$ and $\partial_\alpha^{\rho^3}M$. Compared to the versions without the derivatives, we now have extra factors $(2|\rho^1|)!$, $(2|\rho^2|)!$ and $(2|\rho^3|)!$ in view of (\ref{typeJ2}) and the fact that $M$ is Gevrey $2$. Therefore it suffices to show that
\[\sum_{\rho^1+\rho^2+\rho^3=\rho}\frac{\rho!}{(\rho^1)!(\rho^2)!(\rho^3)!}(2|\rho^1|)!(2|\rho^2|)!(2|\rho^3|)!\leq C(2|\rho|)!,
\] which follows from Lemma \ref{combineq}.

Now for $J''$, we continue to calculate
\[
\begin{aligned}J''f(t)&=\int_0^te^{\pi i\ell_1(t-s)+\pi i\gamma_1t+\pi i(\gamma_2+\gamma_3+\gamma_4)s}f(s)\cdot\frac{\chi_\infty(\ell_2-\ell_1+\gamma_2+\gamma_3)}{\pi i(\ell_2-\ell_1+\gamma_2+\gamma_3)}(e^{\pi i(\ell_2-\ell_1+\gamma_2+\gamma_3)(t-s)}-1)\,\mathrm{d}s\\
&=\frac{\chi_\infty(\ell_2-\ell_1+\gamma_2+\gamma_3)}{{\pi}i(\ell_2-\ell_1+\gamma_2+\gamma_3)}\int_0^t(e^{\pi i\ell_2(t-s)+\pi i(\gamma_1+\gamma_2+\gamma_3)t+\pi i\gamma_4s}-e^{\pi i\ell_1(t-s)+\pi i\gamma_1t+\pi i(\gamma_2+\gamma_3+\gamma_4)s})f(s)\,\mathrm{d}s.
\end{aligned}\] Note that both $\ell_1$ and $\ell_2$ are {bundle}s of $(\alpha_1,\cdots,\alpha_B)$, just like the above, we can choose either 
\begin{multline}\label{formula1}m^{(3)}(\alpha[A\cup B],\mu[E\cup F],{\gamma_1'},{\gamma_2'})=\int_{\gamma_1+\gamma_2+\gamma_3={\gamma_1'}}m^{(1)}(\alpha[A],\mu[E],\gamma_1,\gamma_2)\\\times m^{(2)}(\alpha[B],\mu[F],\gamma_3,{\gamma_2'})\frac{\chi_\infty(\ell_2-\ell_1+\gamma_2+\gamma_3)}{\pi i(\ell_2-\ell_1+\gamma_2+\gamma_3)}\,\mathrm{d}\gamma_2\mathrm{d}\gamma_3,\quad \mathrm{or}\quad\end{multline}
\begin{multline}\label{formula2}m^{(3)}(\alpha[A\cup B],\mu[E\cup F],{\gamma_1'},{\gamma_2'})=\int_{\gamma_2+\gamma_3+\gamma_4={\gamma_2'}}m^{(1)}(\alpha[A],\mu[E],{\gamma_1'},\gamma_2)\\\times m^{(2)}(\alpha[B],\mu[F],\gamma_3,\gamma_4)\frac{\chi_\infty(\ell_2-\ell_1+\gamma_2+\gamma_3)}{\pi i(\ell_2-\ell_1+\gamma_2+\gamma_3)}\,\mathrm{d}\gamma_2\mathrm{d}\gamma_3,\end{multline} which settles the contribution of $J''$ if we do not take derivatives and do not count the weight. As for the weight, notice that for (\ref{formula1}) we need to use $|\gamma_1'|\leq|\gamma_1|+|\gamma_2|+|\gamma_3|$, which seemingly involves $\gamma_2$; however in view of the denominator $\ell_2-\ell_1+\gamma_2+\gamma_3$ in (\ref{formula1}), we may replace $|\gamma_2|^{1/4}$ by either $|\ell_2-\ell_1+\gamma_2+\gamma_3|^{1/4}$, which is estimated using this denominator, or $|\ell_2-\ell_1+\gamma_3|^{1/4}$, which is estimated using the corresponding norms for $m^{(1)}$ and $m^{(2)}$. Similarly, one can treat \eqref{formula2}. The higher order derivatives can be treated in the same way as $J'$ above.
\end{proof}
\begin{lem}\label{regchainlem2} The composition of a class $J$ operator and a class $R$ operator is of class $J$.
\end{lem}
\begin{proof} Let $\Jc_{\alpha[A],\mu[E]}^{(1)}$ be of class $J$ and $\Rc_{\alpha[B],\mu[F]}^{(2)}$ be of class $R$ with special index $1\in B$. We first consider $\Jc^{(3)}=\Jc^{(1)}\Rc^{(2)}$. Similar to Lemma \ref{regchainlem1}, we only need to look at $J:=J_{\ell_1;\gamma_1,\gamma_2}R_{\ell_2+\epsilon\alpha_{1},\ell_3+\epsilon\alpha_{1};\gamma_3,\gamma_4,\gamma_5}$, where $\ell_1$ is a {bundle} of $\alpha[A]$, $\ell_2$ and $\ell_3$ are two {bundle}s of $\alpha[B\backslash\{1\}]$. We have
\[
\begin{aligned}Jf(t)&=\int_0^te^{\pi i\ell_1(t-z)+\pi i\gamma_1t+\pi i\gamma_2z}\,\mathrm{d}z\int_0^z\frac{\chi_\infty(\ell_2+\epsilon\alpha_{1}+\gamma_5)}{\ell_2+\epsilon\alpha_{1}+\gamma_5}e^{\pi i(\ell_3+\epsilon\alpha_{1})(z-s)}e^{\pi i(\gamma_3z+\gamma_4s)}f(s)\,\mathrm{d}s\\
&=\frac{\chi_\infty(\ell_2+\epsilon\alpha_{1}+\gamma_5)}{\ell_2+\epsilon\alpha_{1}+\gamma_5}\int_0^tf(s)\,\mathrm{d}s\int_0^{t-s}e^{\pi i\ell_1(t-s-u)+\pi i(\ell_3+\epsilon\alpha_{1})u+\pi i\gamma_1t+\pi i(\gamma_2+\gamma_3+\gamma_4)s+\pi i(\gamma_2+\gamma_3)u}\,\mathrm{d}u\\
&=\frac{\chi_\infty(\ell_2+\epsilon\alpha_{1}+\gamma_5)}{\ell_2+\epsilon\alpha_{1}+\gamma_5}\int_0^t e^{\pi i\ell_1(t-s)+\pi i\gamma_1t+\pi i(\gamma_2+\gamma_3+\gamma_4)s}f(s)\,\mathrm{d}s\int_0^{t-s}e^{\pi i(\ell_3-\ell_1+\epsilon\alpha_{1}+\gamma_2+\gamma_3)u}\,\mathrm{d}u.
\end{aligned}\] Again decompose $J=J'+J''$ where for $J'$ and $J''$ we multiply by $\chi_0(\ell_3-\ell_1+\epsilon\alpha_{1}+\gamma_2+\gamma_3)$ and $\chi_\infty(\ell_3-\ell_1+\epsilon\alpha_{1}+\gamma_2+\gamma_3)$ respectively, then as in Lemma \ref{regchainlem1}, the contribution of the $J'$ term to $\Jc^{(3)}$ will be
\begin{multline*}\int m^{(1)}(\alpha[A],\mu[E],\gamma_1,\gamma_2)m^{(2)}(\alpha[B\backslash\{1\}],\mu[F],\gamma_3,\gamma_4,\gamma_5)\\\times M(\ell_3-\ell_1+\epsilon\alpha_{1}+\gamma_2+\gamma_3,\sigma)\frac{\chi_\infty(\ell_2+\epsilon\alpha_{1}+\gamma_5)}{\ell_2+\epsilon\alpha_{1}+\gamma_5}J_{\ell_1;\gamma_1+\sigma,\gamma_2+\gamma_3+\gamma_4-\sigma}\,\mathrm{d}\gamma_1\mathrm{d}\gamma_2\mathrm{d}\gamma_3\mathrm{d}\gamma_4\mathrm{d}\gamma_5\mathrm{d}\sigma,
\end{multline*} which can be rewritten in the form of $\Jc_{\alpha[A\cup B],\mu[E\cup F]}^{(3)}$ with
\begin{multline*}m^{(3)}(\alpha[A\cup B],\mu[E\cup F],{\gamma_1'},{\gamma_2'})=\int_{\gamma_1+\sigma={\gamma_1'}}\int_{\gamma_2+\gamma_3+\gamma_4=\sigma+{\gamma_2'}}m^{(1)}(\alpha[A],\mu[E],\gamma_1,\gamma_2)\\\times m^{(2)}(\alpha[B\backslash\{1\}],\mu[F],\gamma_3,\gamma_4,\gamma_5)M(\ell_3-\ell_1+\epsilon\alpha_{1}+\gamma_2+\gamma_3,\sigma)\frac{\chi_\infty(\ell_2+\epsilon\alpha_{1}+\gamma_5)}{\ell_2+\epsilon\alpha_{1}+\gamma_5}\,\mathrm{d}\gamma_2\mathrm{d}\gamma_3\mathrm{d}\gamma_1\mathrm{d}\gamma_5.\end{multline*} This $m^{(3)}$ can be controlled if we do not take derivatives and do not count the weight, using the fact that
\begin{equation}\label{uniformbound0}\int_\Rb|M(\zeta_1+\epsilon\alpha_{1},\sigma)|\cdot\bigg|\frac{\chi_\infty(\zeta_2+\epsilon\alpha_{1})}{\zeta_2+\epsilon\alpha_{1}}\bigg|\,\mathrm{d}\alpha_{1}\leq C\langle \sigma\rangle^{-10}\end{equation} uniformly in $(\sigma,\zeta_1,\zeta_2)$. As for the weight, we can again decompose it into different parts; compared to Lemma \ref{regchainlem1}, the new part that needs consideration is $|\alpha_{1}|^{1/4}$. But we may replace it by either $|\ell_2+\epsilon\alpha_{1}+\gamma_5|^{1/4}$, which does not affect (\ref{uniformbound0}), or $|\ell_2+\gamma_5|^{1/4}$, which can be estimated using the weighted norm for $m^{(1)}$ or $m^{(2)}$. The higher order derivatives are also treated in the same way as in Lemma \ref{regchainlem1}.

As for $J''$, similarly we have
\begin{multline}\label{formulaj''}
J''f(t)=\frac{\chi_\infty(\ell_2+\epsilon\alpha_{1}+\gamma_5)}{\ell_2+\epsilon\alpha_{1}+\gamma_5}\cdot\frac{\chi_\infty(\ell_3-\ell_1+\epsilon\alpha_{1}+\gamma_2+\gamma_3)}{\pi i(\ell_3-\ell_1+\epsilon\alpha_{1}+\gamma_2+\gamma_3)}\\\times\int_0^t (e^{\pi i({\ell_3}+\epsilon\alpha_{1})(t-s)+\pi i(\gamma_1+\gamma_2+\gamma_3)t+\pi i\gamma_4s}-e^{\pi i\ell_1(t-s)+\pi i\gamma_1t+\pi i(\gamma_2+\gamma_3+\gamma_4)s})f(s)\,\mathrm{d}s,
\end{multline} so similar arguments as above imply that the corresponding contribution is of class $J$, where we have used the fact that both ${\ell_3}+\epsilon\alpha_{1}$ and $\ell_1$ are {bundle}s of $\alpha[A\cup B]$, and that
\[\int_\Rb\bigg|\frac{\chi_\infty(\zeta_1+\epsilon\alpha_{1})}{\zeta_1+\epsilon\alpha_{1}}\bigg|\cdot\bigg|\frac{\chi_\infty(\zeta_2+\epsilon\alpha_{1})}{\zeta_2+\epsilon\alpha_{1}}\bigg|\,\mathrm{d}\alpha_{1}\leq C\] uniformly in $(\zeta_1,\zeta_2)$. The part $|\alpha_1|^{1/4}$ of the weight can be treated in the same way as above, while the part $|\gamma_2|^{1/4}$ of the weight (which is part of $|\gamma_1'|^{1/4}$ in one of the two terms in (\ref{formulaj''})) can be replaced by either $|\ell_3-\ell_1+\epsilon\alpha_{1}+\gamma_2+\gamma_3|^{1/4}$ or $|\ell_3-\ell_1+\epsilon\alpha_{1}+\gamma_3|^{1/4}$ and treated in the same way either as above or as $|\alpha_1|^{1/4}$.

Now we look at $\Rc^{(2)}\Jc^{(1)}$. The proof is similar, where we now have $J:=R_{{\ell_2}+\epsilon\alpha_{1},{\ell_3}+\epsilon\alpha_{1};\gamma_3,\gamma_4,\gamma_5}J_{\ell_1;\gamma_1,\gamma_2}$. Similar calculations yield that
\[
\begin{aligned}Jf(t)&=\int_0^t\frac{\chi_\infty({\ell_2}+\epsilon\alpha_{1}+\gamma_5)}{{\ell_2}+\epsilon\alpha_{1}+\gamma_5}e^{\pi i({\ell_3}+\epsilon\alpha_{1})(t-z)+\pi i\gamma_3t+\pi i\gamma_4z}\,\mathrm{d}z\int_0^ze^{\pi i\ell_1(z-s)}e^{\pi i(\gamma_1z+\gamma_2s)}f(s)\,\mathrm{d}s\\
&=\frac{\chi_\infty({\ell_2}+\epsilon\alpha_{1}+\gamma_5)}{{\ell_2}+\epsilon\alpha_{1}+\gamma_5}\int_0^tf(s)\,\mathrm{d}s\int_0^{t-s}e^{\pi i({\ell_3}+\epsilon\alpha_{1})(t-s-u)+\pi i\ell_1u+\pi i\gamma_3t+\pi i(\gamma_4+\gamma_1+\gamma_2)s+\pi i(\gamma_4+\gamma_1)u}\,\mathrm{d}u\\
&=\frac{\chi_\infty({\ell_2}+\epsilon\alpha_{1}+\gamma_5)}{{\ell_2}+\epsilon\alpha_{1}+\gamma_5}\int_0^t e^{\pi i({\ell_3}+\epsilon\alpha_{1})(t-s)+\pi i\gamma_3t+\pi i(\gamma_4+\gamma_1+\gamma_2)s}f(s)\,\mathrm{d}s\int_0^{t-s}e^{\pi i(\ell_1-\ell_3-\epsilon\alpha_{1}+\gamma_4+\gamma_1)u}\,\mathrm{d}u.
\end{aligned}\] We proceed in basically the same way as for $\Jc^{(1)}\Rc^{(2)}$, except that (i) the roles of $\ell_1$ and ${\ell_3}+\epsilon\alpha_{1}$ are switched, but both are still {bundle}s of $\alpha[A\cup B]$; (ii) we now need to deal with the weight $|\gamma_4|^{1/4}$, but also $\gamma_4$ will be a part of the denominator so this will not affect the proof.
\end{proof}
\begin{lem}\label{regchainlem3} The composition of two class $R$ operators is of class $J$.
\end{lem}
\begin{proof} Let $\Rc_{\alpha[A],\mu[E]}^{(1)}$ and $\Rc_{\alpha[B],\mu[F]}^{(2)}$ be of class $R$, with special indices $1\in A$ and $2\in B$. Again we first consider the operator $J:=R_{{\ell_1}+\epsilon_1\alpha_{1},{\ell_2}+\epsilon_1\alpha_{1};\gamma_1,\gamma_2,\gamma_3}R_{{\ell_3}+\epsilon_2\alpha_{2},{\ell_4}+\epsilon_2\alpha_{2};\gamma_4,\gamma_5,\gamma_6}$, where $\ell_1$ and $\ell_2$ are {bundle}s of $\alpha[A\backslash\{1\}]$, and $\ell_3$ and $\ell_4$ are {bundle}s of $\alpha[B\backslash\{2\}]$. We have
\[
\begin{aligned}Jf(t)&=\int_0^t\frac{\chi_\infty({\ell_1}+\epsilon_1\alpha_1+\gamma_3)}{{\ell_1}+\epsilon_1\alpha_1+\gamma_3}e^{\pi i({\ell_2}+\epsilon_1\alpha_1)(t-z)+\pi i\gamma_1t+\pi i\gamma_2z}\,\mathrm{d}z\\&\hspace{0.5cm}\times\int_0^z\frac{\chi_\infty({\ell_3}+\epsilon_2\alpha_{2}+\gamma_6)}{{\ell_3}+\epsilon_2\alpha_{2}+\gamma_6}e^{\pi i({\ell_4}+\epsilon_2\alpha_{2})(z-s)+\pi i\gamma_4z+\pi i\gamma_5s}f(s)\,\mathrm{d}s\\
&=\frac{\chi_\infty({\ell_1}+\epsilon_1\alpha_1+\gamma_3)}{{\ell_1}+\epsilon_1\alpha_1+\gamma_3}\frac{\chi_\infty({\ell_3}+\epsilon_2\alpha_{2}+\gamma_6)}{{\ell_3}+\epsilon_2\alpha_{2}+\gamma_6}\int_0^tf(s)\,\mathrm{d}s
\\&\hspace{0.5cm}\times\int_0^{t-s}e^{\pi i({\ell_2}+\epsilon_1\alpha_1)(t-s-u)+\pi i({\ell_4}+\epsilon_2\alpha_{2})u+\pi i\gamma_1t+\pi i(\gamma_2+\gamma_4+\gamma_5)s+\pi i(\gamma_2+\gamma_4)u}\,\mathrm{d}u\\
&=\frac{\chi_\infty({\ell_1}+\epsilon_1\alpha_1+\gamma_3)}{{\ell_1}+\epsilon_1\alpha_1+\gamma_3}\frac{\chi_\infty({\ell_3}+\epsilon_2\alpha_{2}+\gamma_6)}{{\ell_3}+\epsilon_2\alpha_{2}+\gamma_6}\int_0^t e^{\pi i({\ell_2}+\epsilon_1\alpha_1)(t-s)+\pi i\gamma_1t+\pi i(\gamma_2+\gamma_4+\gamma_5)s}f(s)\,\mathrm{d}s\\&\hspace{0.5cm}\times\int_0^{t-s}e^{\pi i(\ell_4-\ell_2+\epsilon_2\alpha_{2}-\epsilon_1{\alpha_1}+\gamma_2+\gamma_4)u}\,\mathrm{d}u.
\end{aligned}
\] Now, we make the decomposition again by multiplying \[\chi_0(\ell_4-\ell_2+\epsilon_2\alpha_{2}-\epsilon_1{\alpha_1}+\gamma_2+\gamma_4)\quad\textrm{or}\quad\chi_\infty(\ell_4-\ell_2+\epsilon_2\alpha_{2}-\epsilon_1{\alpha_1}+\gamma_2+\gamma_4),\] and denote the resulting terms by $J'$ and $J''$. Then repeating the same arguments before we can show that the contribution of both $J'$ and $J''$ are class $J$ operators. The key points here are that (i) both ${\ell_2}+\epsilon_1\alpha_1$ and ${\ell_4}+\epsilon_2\alpha_{2}$ are {bundle}s of $\alpha[A]$ and $\alpha[ B]$ respectively, and that (ii) the bound
\[\int_{\Rb^2}|M(\zeta_1+\epsilon_2\alpha_{2}-\epsilon_1\alpha_1,\sigma)|\cdot\bigg|\frac{\chi_\infty(\zeta_2+\epsilon_1\alpha_1)}{\zeta_2+\epsilon_1\alpha_1}\bigg|\cdot\bigg|\frac{\chi_\infty(\zeta_3+\epsilon_2\alpha_{2})}{\zeta_3+\epsilon_2\alpha_{2}}\bigg|\,\mathrm{d}\alpha_1\mathrm{d}\alpha_{2}\leq C\langle \sigma\rangle^{-10}\] holds uniformly in $(\sigma,\zeta_1,\zeta_2,\zeta_3)$, and similarly
\[\int_{\Rb^2}\bigg|\frac{\chi_\infty(\zeta_1+\epsilon_2\alpha_{2}-\epsilon_1\alpha_1)}{\zeta_1+\epsilon_2\alpha_{2}-\epsilon_1\alpha_1}\bigg|\cdot\bigg|\frac{\chi_\infty(\zeta_2+\epsilon_1\alpha_1)}{\zeta_2+\epsilon_1\alpha_1}\bigg|\cdot\bigg|\frac{\chi_\infty(\zeta_3+\epsilon_2\alpha_{2})}{\zeta_3+\epsilon_2\alpha_{2}}\bigg|\,\mathrm{d}\alpha_1\mathrm{d}\alpha_{2}\leq C\] holds uniformly in $(\zeta_1,\zeta_2,\zeta_3)$, which follow from elementary calculus. The parts of the weight that need consideration are (i) $|\alpha_1|^{1/4}$ and $|\alpha_2|^{1/4}$, which can be treated using the denominators $\ell_1+\epsilon_1\alpha_1+\gamma_3$ and $\ell_3+\epsilon_2\alpha_2+\gamma_6$ respectively, and {(ii)} $|\gamma_2|^{1/4}$ (which is part of $|\gamma_1'|^{1/4}$ in one of the terms), which can be treated using the denominator $\ell_4-\ell_2+\epsilon_2\alpha_{2}-\epsilon_1{\alpha_1}+\gamma_2+\gamma_4$.
\end{proof}
\begin{lem}\label{regchainlem4} Suppose $\Jc$ is an operator of class $J$, then the operator $I_{\epsilon\alpha_1}\Jc I_{\mu_1-\epsilon\alpha_1}$, where $\epsilon\in\{\pm\}$, can be decomposed into an operator of class $J$ and an operator of class $R$ (with special index $1$).
\end{lem}
\begin{proof} Let $\Jc=\Jc_{\alpha[A],\mu[E]}^{(1)}$ be of class $J$, where we assume $1\not\in E$. Again we first consider the operator $X:=I_{\epsilon\alpha_1}J_{\ell;\gamma_1,\gamma_2}I_{\mu_1-\epsilon\alpha_1}$, where $\ell$ is a {bundle} of $\alpha[A]$. Then
\[
\begin{aligned}Xf(t)&=\int_0^t e^{\epsilon\pi i\alpha_1 z}\,\mathrm{d}z\int_0^z e^{\pi i\ell(z-v)+\pi i\gamma_1z+\pi i\gamma_2v}\,\mathrm{d}v\int_0^v e^{-\epsilon\pi i\alpha_1 s+\pi i\mu_1 s}f(s)\,\mathrm{d}s\\
&=\int_0^te^{\pi i(\gamma_1+\gamma_2+\mu_1)s}f(s)\,\mathrm{d}s\int_0^{t-s}e^{\pi i(\gamma_1+\gamma_2+\epsilon\alpha_1)u}\,\mathrm{d}u\int_0^{t-s-u}e^{\pi i(\ell+\epsilon\alpha_1+\gamma_1)w}\,\mathrm{d}w.
\end{aligned}\] {Our estimates will be uniform in $\mu_1$ due to the only position it appears, and the fact that the left hand side of (\ref{typeJ2}) does not involve $\gamma_2$.} When $s$ is fixed, by making the $(\chi_0,\chi_\infty)$ decomposition twice, we can reduce the inner $(u,w)$ integral to $6$ different terms, namely:
\[
\begin{aligned}\mathrm{I}&:=\int_{\Rb^2}M(\gamma_1+\gamma_2+\epsilon\alpha_1-\sigma_1,\sigma_2)M(\ell+\epsilon\alpha_1+\gamma_1,\sigma_1)\cdot e^{\pi i(\sigma_1+\sigma_2)(t-s)}\,\mathrm{d}\sigma_1\mathrm{d}\sigma_2,\\
\mathrm{II}&:=\int_{\Rb}M(\ell+\epsilon\alpha_1+\gamma_1,\sigma)\cdot\frac{\chi_\infty(\gamma_1+\gamma_2+\epsilon\alpha_1-\sigma)}{\pi i(\gamma_1+\gamma_2+\epsilon\alpha_1-\sigma)}(e^{\pi i(\gamma_1+\gamma_2+\epsilon\alpha_1)(t-s)}-e^{\pi i\sigma(t-s)}),\\
\mathrm{III}&:={-}\int_\Rb M(\gamma_1+\gamma_2+\epsilon\alpha_1,\sigma)\frac{\chi_\infty(\ell+\epsilon\alpha_1+\gamma_1)}{\pi i(\ell+\epsilon\alpha_1+\gamma_1)}\cdot e^{\pi i\sigma(t-s)}\,\mathrm{d}\sigma,\\
\mathrm{IV}&:={-}\frac{\chi_\infty(\ell+\epsilon\alpha_1+\gamma_1)}{\pi i(\ell+\epsilon\alpha_1+\gamma_1)}\cdot\frac{\chi_\infty(\gamma_1+\gamma_2+\epsilon\alpha_1)}{\pi i(\gamma_1+\gamma_2+\epsilon\alpha_1)}(e^{\pi i(\gamma_1+\gamma_2+\epsilon\alpha_1)(t-s)}-1),\\
\mathrm{V}&:=\int_\Rb M(\gamma_2-\ell,\sigma)\frac{\chi_\infty(\ell+\epsilon\alpha_1+\gamma_1)}{\pi i(\ell+\epsilon\alpha_1+\gamma_1)} e^{\pi i(\ell+\epsilon\alpha_1+\gamma_1+\sigma)(t-s)}\,\mathrm{d}\sigma,\\
\mathrm{VI}&:=\frac{\chi_\infty(\ell+\epsilon\alpha_1+\gamma_1)}{\pi i(\ell+\epsilon\alpha_1+\gamma_1)}\cdot\frac{\chi_\infty(\gamma_2-\ell)}{\pi i(\gamma_2-\ell)}(e^{\pi i(\gamma_1+\gamma_2+\epsilon\alpha_1)(t-s)}-e^{\pi i(\ell+\epsilon\alpha_1+\gamma_1)(t-s)}).
\end{aligned}
\] Now, the terms $\mathrm{I}\sim\mathrm{IV}$ will give rise to class $J$ operators; for example, consider the term $\mathrm{IV}$ where we choose the term $e^{\pi i(\gamma_1+\gamma_2+\epsilon\alpha_1)(t-s)}$ from the last parenthesis, then we may express the contribution of this term as a class $J$ operator with coefficient
\begin{multline*}m^{(2)}(\alpha[A\cup\{1\}],\mu[E\cup\{1\}],{\gamma_1'},{\gamma_2'})=\dirac ({\gamma_2'}-\mu_1)\\\times\int_{\gamma_1+\gamma_2={\gamma_1'}}\frac{\chi_\infty(\ell+\epsilon\alpha_1+\gamma_1)}{\pi i(\ell+\epsilon\alpha_1+\gamma_1)}\cdot\frac{\chi_\infty(\gamma_1+\gamma_2+\epsilon\alpha_1)}{\pi i(\gamma_1+\gamma_2+\epsilon\alpha_1)}m^{(1)}(\alpha[A],\mu[E],\gamma_1,\gamma_2)\,\mathrm{d}\gamma_1.\end{multline*} Here $m^{(1)}$ is the coefficient associated with $\Jc^{(1)}$, and the Dirac $\dirac$ function can be removed by using the fact that
\begin{equation}\label{extraid}e^{\pi i\mu_1 s}=\int_\Rb M(\sigma)e^{\pi i(\mu_1+\sigma)s}\,\mathrm{d}\sigma\end{equation} for some {analytic} function $M$ and all $s\in[0,1]$, which allows to replace $\dirac({\gamma_2'}-\mu_1)$ by $M({\gamma_2'}-\mu_1)$; moreover the integral in $\alpha_1$ is uniformly bounded given the other variables.

Finally, terms $\mathrm{V}\sim\mathrm{VI}$ lead to class $R$ operators. For example, consider the term $\mathrm{VI}$ where we choose the term $e^{\pi i(\ell+\epsilon\alpha_1+\gamma_1)(t-s)}$ from the last parenthesis, then we may express the contribution of this term as a class $R$ operator with $\ell_1$ and $\ell_2$ replaced by $\ell$, and the coefficient
\[m^{(2)}(\alpha[A],\mu[E],\gamma_1',\gamma_2',\gamma_3')=\frac{1}{\pi i}\dirac(\gamma_1'-\gamma_3')\cdot\frac{\chi_\infty(\gamma_2'-\mu_1-\ell)}{\pi i(\gamma_2'-\mu_1-\ell)}m^{(1)}(\alpha[A],\mu[E],\gamma_1',\gamma_2'-\mu_1).\] The Dirac $\dirac$ function can again be removed using (\ref{extraid}), and in all cases (both for $\mathrm{I}\sim\mathrm{IV}$ and $\mathrm{V}\sim\mathrm{VI}$) the weight can be treated in the same way as before, using the denominator $\ell+\epsilon\alpha_1+\gamma_1$ to estimate $|\alpha_1|^{1/4}$, and using the denominator $\gamma_1+\gamma_2+\epsilon\alpha_1$ or $\gamma_2-\ell$ to estimate $|\gamma_2|^{1/4}$ (which may appear as part of $|\gamma_1'|^{1/4}$ for some terms).
\end{proof}
\begin{lem}\label{regchainlem5} Suppose $\Rc$ is an operator of class $R$, then the operator $I_{\epsilon\alpha_1}\Rc I_{\mu_1-\epsilon\alpha_1}$, where $\epsilon\in\{\pm\}$, can be decomposed into an operator of class $J$ and an operator of class $R$ (with special index $1$).
\end{lem}
\begin{proof} Let $\Rc=\Rc_{\alpha[A],\mu[E]}^{(1)}$ be of class $R$ with special index $2$, where $1\not\in E$. Consider $X=I_{\epsilon\alpha_1} R_{\ell_1+\epsilon_2\alpha_2,\ell_2+\epsilon_2\alpha_2;\gamma_1,\gamma_2,\gamma_3}I_{\mu_1-\epsilon\alpha_1}$, where $\ell_1$ and $\ell_2$ are {bundle}s of $\alpha[A\backslash\{2\}]$. {Our estimates will be uniform in $\mu_1$ as in Lemma \ref{regchainlem4}}. We proceed in basically the same way as in Lemma \ref{regchainlem4}, and obtain the same expressions $\mathrm{I}\sim\mathrm{VI}$, except that here $\ell$ is replaced by $\ell_2+\epsilon_2\alpha_2$, and that we have an extra factor $\frac{\chi_\infty(\ell_1+\epsilon_2\alpha_2+\gamma_3)}{\ell_1+\epsilon_2\alpha_2+\gamma_3}$. The terms $\mathrm{I}\sim\mathrm{IV}$ still give rise to class $J$ operators, since the factors in the coefficients that depend on $(\alpha_1,\alpha_2)$, which are bounded by
\[\frac{1}{\langle\gamma_1+\gamma_2+\epsilon\alpha_1\rangle}\cdot\frac{1}{\langle\ell_2+\epsilon\alpha_1+\epsilon_2\alpha_2+\gamma_1\rangle}\cdot\frac{1}{\langle\ell_1+\epsilon_2\alpha_2+\gamma_3\rangle},\] are integrable in $(\alpha_1,\alpha_2)$ uniformly in the other variables.

As for terms $\mathrm{V}\sim\mathrm{VI}$, they will give rise to class $R$ operators with special index 1. For this we only need the integrability in $\alpha_2$ (uniformly in the other variables) of the factors in the coefficients that depend on $\alpha_2$, which follows from the upper bound \[\frac{1}{\langle\ell_1+\epsilon_2\alpha_2+\gamma_3\rangle}\cdot\frac{1}{\langle\gamma_2-\epsilon_2\alpha_2-\ell_2\rangle}.\] Again, in all cases the weight can be treated in the same way as before, where for terms $\mathrm{I}\sim\mathrm{IV}$ we use the denominator $\ell_1+\epsilon_2\alpha_2+\gamma_3$ to estimate $|\alpha_2|^{1/4}$, use the denominator $\ell_2+\epsilon\alpha_1+\epsilon_2\alpha_2+\gamma_1$ to estimate $|\alpha_1|^{1/4}$, and use the denominator $\gamma_1+\gamma_2+\epsilon\alpha_1$ to estimate $|\gamma_2|^{1/4}$ which may appear as part of $|\gamma_1'|^{1/4}$ for some terms. For terms $\mathrm{V}\sim\mathrm{VI}$ we use the denominator $\ell_1+\epsilon_2\alpha_2+\gamma_3$ to estimate $|\alpha_2|^{1/4}$, and use the denominator $\gamma_2-\epsilon_2\alpha_2-\ell_2$ to estimate $|\gamma_2|^{1/4}$.
\end{proof}
\begin{lem}\label{regchainlem6}The operator $I_{\epsilon\alpha_1}I_{\mu_1-\epsilon\alpha_1}$, where $\epsilon\in\{\pm\}$, can be decomposed as 
\[I_{\epsilon\alpha_1}I_{\mu_1-\epsilon\alpha_1}=\frac{\chi_\infty(\alpha_1)}{\epsilon\pi i\alpha_1}\Jc_{\alpha[\ ],\mu_1}+\Jc_{\alpha_1,\mu_1}+\Rc_{\alpha_1,\mu_1},\] where $\Jc_{(\cdots)}$ and $\Rc_{(\cdots)}$ are of class $J$ and $R$ (with special index $1$) respectively, and $\alpha[\ ]$ indicates the corresponding set $A=\varnothing$ in (\ref{typeJ0}).
\end{lem}
\begin{proof} We directly calculate $X=I_{\epsilon\alpha_1}I_{\mu_1-\epsilon\alpha_1}$ such that
\[Xf(t)=\int_0^t e^{\epsilon\pi i\alpha_1 z}\,\mathrm{d}z\int_0^z e^{\pi i(\mu_1-\epsilon\alpha_1)s}f(s)\,\mathrm{d}s=\int_0^t e^{\pi i\mu_1 s}f(s)\,\mathrm{d}s\int_0^{t-s}e^{\epsilon\pi i\alpha_1 u}\,\mathrm{d}u,\] {note that the estimates are again uniform in $\mu_1$.} Now $X$ can be decomposed into three terms,
\[
\begin{aligned}\mathrm{I}&:=\int_\Rb M(\epsilon\alpha_1,\sigma)J_{0;\sigma,\mu_1-\sigma}\,\mathrm{d}\sigma,\\
\mathrm{II}&:={-}\frac{\chi_\infty(\epsilon\alpha_1)}{\epsilon\pi i\alpha_1}J_{0;0,\mu_1},\\
\mathrm{III}&:=R_{\epsilon\alpha_1,\epsilon\alpha_1;0,\mu_1,0}.
\end{aligned}
\] Clearly $\mathrm{I}=\Jc_{\alpha_1,\mu_1}$ has class $J$, as we can always convolve by a decaying analytic function as in \eqref{extraid}; likewise $\mathrm{II}=\frac{\chi_\infty(\alpha_1)}{\epsilon\pi i\alpha_1}\cdot\Jc_{\alpha[\ ],\mu_1}$ with the operator also having class $\Jc$. Now looking at $\mathrm{III}$, we can introduce the integration in $(\gamma_1,\gamma_2)$ by convolution; to see it is of class $R$, which involves $\gamma_3$ integration, we simply rewrite
\[\frac{\chi_\infty(\epsilon\alpha_1)}{\epsilon \pi i\alpha_1}=\int_\Rb M(\sigma)\frac{\chi_\infty(\epsilon\alpha_1)}{\epsilon \pi i\alpha_1}\,\mathrm{d}\sigma\] with some fixed compactly supported Gevrey $2$ function $M$; then we replace $\frac{\chi_\infty(\epsilon\alpha_1)}{\epsilon\pi i\alpha_1}$ by $\frac{\chi_\infty(\sigma+\epsilon\alpha_1)}{\pi i(\sigma+\epsilon\alpha_1)}$ to produce the $\gamma_3$ integral (which is just the $\sigma$ integral here), noticing that
\[\int_\Rb\bigg|\frac{\chi_\infty(\epsilon\alpha_1)}{\epsilon\pi i\alpha_1}-\frac{\chi_\infty(\sigma+\epsilon\alpha_1)}{\pi i(\sigma+\epsilon\alpha_1)}\bigg|\,\mathrm{d}\alpha_1\leq C\] for $|\sigma|\leq1 $, so the error term introduced in this way is of form $\Jc_{\alpha_1,\mu_1}$ which has class $J$.
\end{proof}
\subsubsection{Regular chain expressions}
\begin{lem}\label{regchainlem7} Let $\Pc$ be a legal partition of $\{1,\cdots,2m\}$, where $m\geq 1$, and consider the operator \[I:=X_0I_{\beta_1+\lambda_1}X_1I_{\beta_2+\lambda_2}\cdots X_{2m-1}I_{\beta_{2m}+\lambda_{2m}}X_{2m},\] again assume $\lambda_a=0$ and $\lambda_b=\mu_j$ for the $j$-th pair $\{a,b\}\,(a<b)$. Suppose that each $X_a\,(0\leq a\leq 2m)$ is either of class $J$ or $R$, or $X_a=\mathrm{Id}$, such that $X_a\neq\mathrm{Id}$ if $\{a,a+1\}\in\Pc$. Then $I$ can be decomposed into an operator of class $J$ and an operator of class $R$ (with special index 1). Moreover the norms of these operators are at most \[C^m\prod_{a=0}^{2m}\|X_a\|,\] with $\|X_a\|=1$ if $X_a=\mathrm{Id}$.
\end{lem}
\begin{proof} First note that we may always assume $X_0=X_{2m}=\mathrm{Id}$ in view of Lemmas \ref{regchainlem1}--\ref{regchainlem3}. Now we induct on $m$ (for convenience of induction we may replace the power $C^m$ by $C^{2m-1}$). When $m=1$ we may assume $I=I_{\epsilon\alpha_1}\widetilde{I} I_{\mu_1-\epsilon\alpha_1}$ where $\widetilde{I}$ has class $J$ or $R$, so the result follows from Lemmas \ref{regchainlem4}--\ref{regchainlem5}. Suppose the result is true for $m'<m$, and consider any legal partition $\Pc$ of $\{1,\cdots,2m\}$. We know that $\Pc$ is formed either by concatenating two smaller legal partitions $\Pc'$ and $\Pc''$, or by enclosing a legal partition $\Pc'$ into the pair $\{1,2m\}$. In the first case we have $I=I'I''$ where $I'$ and $I''$ are the operators corresponding to $\Pc'$ and $\Pc''$ respectively (with obvious choices of the $X_a$'s), so the result follows from Lemmas \ref{regchainlem1}--\ref{regchainlem3}. In the second case we have $I=I_{\epsilon\alpha_1} I'I_{\mu_1-\epsilon\alpha_1}$ where $I'$ is the operator corresponding to $\Pc'$ (with obvious choices of the $X_a$'s), so the result follows from Lemmas \ref{regchainlem4}--\ref{regchainlem5}.
\end{proof}
\begin{lem}\label{regchainlem8} Consider now the operator $I=I_{\beta_1+\lambda_1}\cdots I_{\beta_{2m}+\lambda_{2m}}$ with $\lambda_a$ satisfying the same conditions as in Lemma \ref{regchainlem7}. Then, $I$ is a sum of at most $2^{m}$ terms: for each term there exists a set $Z\subset\{1,\cdots,m\}$, such that this term has form
\[\prod_{j\in Z}\frac{\chi_\infty(\alpha_j)}{\epsilon_j\pi i\alpha_j}\cdot \widetilde{I},\] where $\widetilde{I}$ is another operator that depends only on the variables $(\mu_1,\cdots,\mu_m)$ and $\alpha[W]$ with $W=\{1,\cdots,m\}\backslash Z$, and has either class $J$ or class $R$, with norm $\|\widetilde{I}\|\leq C^m$.
\end{lem}
\begin{proof} We first consider all adjacent pairs $\{a,a+1\}\in\Pc$. For such pairs we have a factor in $I$ that is $I_{\epsilon_j\alpha_j} I_{\mu_j-\epsilon_j\alpha_j}$ for some $j$, so by Lemma \ref{regchainlem6} we can decompose it into $\frac{\chi_\infty(\alpha_j)}{\epsilon_j\pi i\alpha_j}\Jc_{\alpha[\ ],\mu_j}+\Jc_{\alpha_j,\mu_j}+\Rc_{\alpha_j,\mu_j}$. Now if we select the term $\frac{\chi_\infty(\alpha_j)}{\epsilon_j \pi i\alpha_j}\Jc_{\alpha[\ ],\mu_j}$ we already get one factor $\frac{\chi_\infty(\alpha_j)}{\epsilon_j\pi i\alpha_j}$, while the remaining part of the operator will no long depend on $\alpha_j$; if we select $\Jc_{\alpha_j,\mu_j}$ or $\Rc_{\alpha_j,\mu_j}$ we will leave it as is, and note that in any case our operator has class $J$ or class $R$.

Now, after removing all adjacent pairs $\{a,a+1\}$ we can reduce $\Pc$ to a smaller legal partition $\Pc'$. Then, apart from the possible $\frac{\chi_\infty(\alpha_j)}{\epsilon_\pi i\alpha_j}$ factors, the remaining part of the operator, denoted by $\widetilde{I}$, will be of form $I$ described in Lemma \ref{regchainlem7}. Note that if $\Pc'=\varnothing$, then $\widetilde{I}$ is already a composition of at most $m$ class $J$ or $R$ operators, so the result follows directly from Lemmas \ref{regchainlem1}--\ref{regchainlem3}. If $\Pc'\neq\varnothing$, applying Lemma \ref{regchainlem7} then yields that $\widetilde{I}$ has class $J$ or $R$, and that the norm
\[\|\widetilde{I}\|\leq C^{m'}\prod_{a=0}^{2m'}\|X_a\|.\] Here we assume that (after relabeling form smallest to largest) $\Pc'$ is a legal partition of $\{1,\cdots,2m'\}$, and each $X_a$ is in fact a composition of class $J$ and $R$ operators appearing in Lemma \ref{regchainlem6}, and the number $n_a$ of such operators is the number of adjacent pairs in $\Pc$ between the elements $a$ and $a+1$ (again after relabeling) of $\Pc'$ (note that $n_a\geq 1$ if $\{a,a+1\}\in\Pc'$). Therefore $\|X_a\|\leq C^{n_a}$ by iterating Lemmas \ref{regchainlem1}--\ref{regchainlem3}, and since $n_0+\cdots +n_{2m'}=m-m'$, we conclude that $\|\widetilde{I}\|\leq C^m$, which completes the proof.
\end{proof}

\subsection{Proof of Proposition \ref{maincoef}}\label{proofmaincoef} In this section we prove Proposition \ref{maincoef}. The proof is done by induction on the scale of $\Qc$, and Lemma \ref{regchainlem8} plays a key role in the inductive step.
\begin{proof}[Proof of Proposition \ref{maincoef}] We induct on $n$. The base case $n=0$ is trivial as $\widetilde{\Bc}_{\times}(t,s)\equiv 1$. Suppose the result is true for regular couples of smaller scales, and consider a regular couple $\Qc$ of scale $2n$. By the discussion in Section \ref{recursive}, we know that $\widetilde{\Bc}_\Qc$ can be expressed as in either (\ref{newexp1}) or (\ref{newexp2}), such that the regular couples appearing on the right hand sides all have scale strictly less than $2n$.

\emph{Case 1}. Suppose we have (\ref{newexp1}), then by induction hypothesis we have
\begin{multline}\label{newexp3}\widetilde{\Bc}_\Qc(t,s,\alpha[\Nc^{ch}])={\sum_{Z_1,Z_2,Z_3}}\prod_{j=1}^3\prod_{\nf\in Z_j}\frac{\chi_\infty(\alpha_\nf)}{\zeta_\nf\pi i\alpha_\nf}\int_{\Rb^6}\prod_{j=1}^3\Cc_j(\lambda_{2j-1},\lambda_{2j},\alpha[\Nc_j^{ch}\backslash Z_j])\prod_{j=1}^6\mathrm{d}\lambda_j\\\times\int_0^t\int_0^s e^{\pi i\alpha_\rf(t_1-s_1)}e^{\pi i(\lambda^* t_1+\lambda^{**}s_1)}\,\mathrm{d}t_1\mathrm{d}s_1.\end{multline} Here in (\ref{newexp3}), each $Z_j$ is a subset of $\Nc_j^{ch}$, each $\Cc_j=\Cc_j(\lambda_{2j-1},\lambda_{2j},\alpha[\Nc_j^{ch}\backslash Z_j])$ is a function satisfying (\ref{maincoef2}), and $(\lambda^*,\lambda^{**})=(\lambda_1+\lambda_3+\lambda_5,\lambda_2+\lambda_4+\lambda_6)$. {Note that $\Cc_j$ is either the function $\Cc$ associated with $\Qc_j$ and $Z_j$ as in (\ref{maincoef1}), or is the same function with the variables $(\lambda_{2j-1},\lambda_{2j})$ replaced by $(\lambda_{2j},\lambda_{2j-1})$. The same is true in \emph{Case 2} below.}

By integrability in $(\lambda_1,\cdots,\lambda_6)$, we may fix the choices of these parameters, and also exploit the weight $\langle \lambda_1\rangle^{1/4}\cdots\langle \lambda_6\rangle^{1/4}$ from (\ref{maincoef2}) {if needed}. We then explicitly calculate the expression
\begin{equation}\label{regcplexpl}\int_0^t\int_0^s e^{\pi i\alpha_\rf(t_1-s_1)}e^{\pi i(\lambda^* t_1+\lambda^{**}s_1)}\,\mathrm{d}t_1\mathrm{d}s_1=\frac{e^{\pi i(\alpha_\rf+\lambda^*)t}-1}{\pi i(\alpha_\rf+\lambda^*)}\cdot\frac{e^{\pi i(-\alpha_\rf+\lambda^{**})s}-1}{\pi i(-\alpha_\rf+\lambda^{**})}.\end{equation} By inserting the cutoffs $\chi_0(\alpha_\rf+\lambda^*)$ or $\chi_\infty(\alpha_\rf+\lambda^*)$, and $\chi_0(-\alpha_\rf+\lambda^{**})$ or $\chi_\infty(-\alpha_\rf+\lambda^{**})$ as in Section \ref{classjr}, we can easily show that the above expression, as a function of $(t,s,\alpha_\rf)$, can be written in the form
\begin{equation}\label{regcplex1}\int_{\Rb^2}\Cc'(\lambda_1',\lambda_2',\alpha_\rf)e^{\pi i(\lambda_1't+\lambda_2's)}\,\mathrm{d}\lambda_1'\mathrm{d}\lambda_2',\end{equation} where $\Cc'$ is such that
\begin{equation}\label{regcplex2}{\int \langle \lambda_1'\rangle^{1/4}\langle \lambda_2'\rangle^{1/4}|\partial_{\alpha_\rf}^\rho\Cc'(\lambda_1',\lambda_2',\alpha_\rf)|\,\mathrm{d}\lambda_1'\mathrm{d}\lambda_2'\mathrm{d}\alpha_\rf\leq C(2|\rho|)!}\end{equation} uniformly in the choices of $\lambda_j$. For example, if we insert the cutoffs $\chi_\infty(\alpha_\rf+\lambda^*)$ and $\chi_\infty(-\alpha_\rf+\lambda^{**})$, and choose the terms $e^{\pi i(\alpha_\rf+\lambda^*)t}$ and $e^{\pi i(-\alpha_\rf+\lambda^{**})s}$ from the numerators in (\ref{regcplexpl}), then by using (\ref{extraid}) we can write
\[\Cc'(\lambda_1',\lambda_2',\alpha_\rf)=M(\lambda_1'-\alpha_\rf-\lambda^*)M(\lambda_2'+\alpha_\rf-\lambda^{**})\cdot\frac{\chi_\infty(\alpha_\rf+\lambda^{*})}{\pi i(\alpha_\rf+\lambda^{*})}\cdot\frac{\chi_\infty(-\alpha_\rf+\lambda^{**})}{\pi i(-\alpha_\rf+\lambda^{**})},\] for which (\ref{regcplex2}) is easily verified, noticing also that $\chi_\infty$ is Gevrey $2$. The other terms can be treated similarly.

Now, for the regular couple $\Qc$ and the associated $\widetilde{\Bc}_\Qc$, we may choose $Z=Z_1\cup Z_2\cup Z_3$. Using (\ref{newexp3}), and rewriting (\ref{regcplexpl}) as the form (\ref{regcplex1}), we can then easily prove that it has the form (\ref{maincoef1}) and satisfies (\ref{maincoef2}); in fact, assume $\Qc_j$ has scale $n_j$, then the induction hypothesis together with (\ref{regcplex2}) bounds the {left hand side of (\ref{maincoef2})} without derivatives by
\[C^{n_1}\cdot C^{n_2}\cdot C^{n_3}\cdot C=C^{n_1+n_2+n_3+1}=C^n.\] As for the higher order derivatives estimate, notice that the whole expression (\ref{newexp3}) is a linear combination of terms that are \emph{factorized} as a product of functions of $\alpha[\Nc_j^{ch}]$ for $1\leq j\leq 3$ and a function of $\alpha_\rf$. For any multi-index $\rho$, suppose we want to control {the $\partial_\alpha^\rho$ derivative of the relevant quantity}, and let the multi-index of derivatives falling on each of the above four sets of variables be $\rho_j\,(1\leq j\leq 4)$, then by induction hypothesis and (\ref{regcplex2}), the {left hand side of (\ref{maincoef2})} is at most
\[{C^n(2|\rho_1|)!\cdots (2|\rho_4|)!\leq C^n(2|\rho|)!},\] which is what we need.

\emph{Case 2}. Suppose we have (\ref{newexp2}), again by induction hypothesis we have
\begin{equation}\label{newexp4}
\begin{aligned}
\widetilde{\Bc}_\Qc(t,s,\alpha[\Nc^{ch}])&={\sum_{(Z_{j,\epsilon,\iota})}}\prod_{\epsilon\in\{\pm\}}\prod_{j=1}^{m^\epsilon}\prod_{\iota=1}^2\prod_{\nf\in Z_{j,\epsilon,\iota}}\frac{\chi_\infty(\alpha_\nf)}{\zeta_\nf\pi i\alpha_\nf}\cdot\prod_{\nf\in Z_{lp}}\frac{\chi_\infty(\alpha_\nf)}{\zeta_\nf\pi i\alpha_\nf}\\&\times\int\prod_{\epsilon\in\{\pm\}}\prod_{j=1}^{m^\epsilon}\prod_{\iota=1}^2\Cc_{j,\epsilon,\iota}\big(\lambda_{a,\epsilon,\iota},\lambda_{b,\epsilon,\iota},\alpha[\Nc_{j,\epsilon,\iota}^{ch}\backslash Z_{j,\epsilon,\iota}]\big)\,\mathrm{d}\lambda_{a,\epsilon,\iota}\lambda_{b,\epsilon,\iota}\\&\times\int\Cc_{lp}(\lambda_{lp,+},\lambda_{lp,-},\alpha[\Nc_{lp}^{ch}\backslash Z_{lp}])\,\mathrm{d}\lambda_{lp,+}\mathrm{d}\lambda_{lp,-}\\
&\times\int_{t>t_1>\cdots >t_{2m^+}>0}e^{\pi i[(\beta_1^++\lambda_1^+)t_1+\cdots +(\beta_{2m^+}^++\lambda_{2m^+}^+)t_{2m^+}]+\pi i\lambda_0^+t_{2m^+}}\,\mathrm{d}t_1\cdots\mathrm{d}t_{2m^+}\\
&\times\int_{s>s_1>\cdots >s_{2m^-}>0}e^{\pi i[(\beta_1^-+\lambda_1^-)s_1+\cdots +(\beta_{2m^-}^-+\lambda_{2m^-}^-)s_{2m^-}]+\pi i\lambda_0^-s_{2m^-}}\,\mathrm{d}s_1\cdots\mathrm{d}s_{2m^-}
\end{aligned}
\end{equation} Here in (\ref{newexp4}) each $Z_{j,\epsilon,\iota}$ is a subset of $\Nc_{j,\epsilon,\iota}^{ch}$ and $Z_{lp}$ is a subset of $\Nc_{lp}^{ch}$, each $\Cc_{j,\epsilon,\iota}$ and $\Cc_{lp}$ is a function satisfying (\ref{maincoef2}), and $\lambda_a^\pm=\lambda_{a,\pm,1}+\lambda_{a,\pm,2}$ for $1\leq a\leq 2m^\pm$, $\lambda_0^\pm=\lambda_{lp,\pm}$.

By integrability in $(\lambda_{a,\epsilon,\iota})$ and $(\lambda_{lp,\pm})$, we may fix the choices of these parameters. Note that we can also exploit the weight
\begin{equation}\label{extraweight}\prod_{j,\epsilon,\iota}\langle \lambda_{a,\epsilon,\iota}\rangle^{1/4}\langle \lambda_{b,\epsilon,\iota}\rangle^{1/4}\cdot \langle \lambda_{lp,+}\rangle^{1/4}\langle\lambda_{lp,-}\rangle^{1/4}\end{equation} from (\ref{maincoef2}), whenever needed. Once these parameters are fixed, the relevant term in (\ref{newexp4}) is then reduced to the product of a function of $(t,\alpha_1^+,\cdots,\alpha_{m^+}^+)$, and a function of $s$ and $(s,\alpha_1^-,\cdots,\alpha_{m^-}^-)$. Let us look at the function depending on $t$, since the other can be treated in the same way.

As in (\ref{regchaink}), this function can be written as\begin{equation}\label{regchainfunck}K(t,\alpha_1^+,\cdots,\alpha_{m^+}^+,\lambda_0^+,\lambda_1^+,\cdots,\lambda_{2m^+}^+).\end{equation}For $1\leq j\leq m^+$, as in (\ref{translationk}), define $\widetilde{\alpha_j}=\alpha_j^++\epsilon_j^+\lambda_a^+$ and $\mu_j^+=\lambda_a^++\lambda_b^+$ where $(a,b)$ is the $j$-th pair, then using (\ref{kusingI}) and Lemma \ref{regchainlem8}, we can write
\begin{equation}\label{regcplform1}K(t,\alpha_1^+,\cdots,\alpha_{m^+}^+,\lambda_0^+,\lambda_1^+,\cdots,\lambda_{2m^+}^+)=\sum_{Z^+\subset \{1,\cdots,m^+\}}\prod_{j\in Z^+}\frac{\chi_\infty(\widetilde{\alpha_j})}{\epsilon_j^+\pi i\widetilde{\alpha_j}}\cdot \widetilde{I}(e^{\pi i\lambda_0^+s})(t),
\end{equation} where $Z^+$ is a subset of $\{1,\cdots,m^+\}$, $\widetilde{I}$ is an operator depending on the variables $(\mu_1^+,\cdots,\mu_{m^+}^+)$ and $\widetilde{\alpha}[W^+]:=(\widetilde{\alpha_j})_{j\in W^+}$ where $W^+:=\{1,\cdots,m^+\}\backslash Z^+$, that is the sum of a class $J$ operator and a class $R$ operator in the sense of Definition \ref{defclassjr}. There are then two sub cases.

\emph{Case 2.1}. Suppose $\widetilde{I}$ is of class $J$, then $\widetilde{I}$ has form (\ref{typeJ0}). The point is that, when the variables $(\lambda_0^+,\gamma_1,\gamma_2,\mu_1^+,\cdots,\mu_{m^+}^+)$ and $\widetilde{\alpha}[W^+]$ are fixed, and $\ell$ is a {bundle} of $\widetilde{\alpha}[W^+]$, then we can write
\begin{equation}\label{expandG}J_{\ell;\gamma_1,\gamma_2}(e^{\pi i\lambda_0^+s})(t)=\int_\Rb G(\lambda)e^{\pi it\lambda}\,\mathrm{d}\lambda\end{equation} for $t\in[0,1]$, where (viewing $\widetilde{\alpha}[W^+]$ as parameters)
\begin{equation}\label{propertyG}{\int_\Rb \langle \lambda\rangle^{1/4}|\partial_\alpha^\rho G(\lambda)|\,\mathrm{d}\lambda\leq C(2|\rho|)!\bigg(1+\sum_{j\in W^+}|\widetilde{\alpha_j}|+|\gamma_1|\bigg)^{1/4}}.\end{equation} This is in fact obvious by calculating
\[J_{\ell;\gamma_1,\gamma_2}(e^{\pi i\lambda_0^+s})(t)=e^{\pi i(\ell+\gamma_1)t}\int_0^t e^{\pi i(\gamma_2+\lambda_0^+-\ell)s}\,\mathrm{d}s\] and inserting the cutoffs $\chi_0(\gamma_2+\lambda_0^+-\ell)$ or $\chi_\infty(\gamma_2+\lambda_0^+-\ell)$ as in \emph{Case 1} above, noticing that {if part of the weight is $|\gamma_2+\lambda_0^+-\ell|^{1/4}$, it can be treated by exploiting the denominator which contains the same expression $\gamma_2+\lambda_0^+-\ell$.} Now using (\ref{typeJ0}) to expand $\widetilde{I}$ as a linear combination of $J_{\ell;\gamma_1,\gamma_2}$, and combining (\ref{propertyG}) with (\ref{typeJ2}) (using Leibniz rule and Lemma \ref{combineq} if necessary), we obtain that
\begin{equation}\label{expandH}\widetilde{I}(e^{\pi i\lambda_0^+s})(t)=\int_\Rb H(\lambda,\lambda_0^+,\mu_1^+,\cdots,\mu_{m^+}^+,\widetilde{\alpha}[W^+])e^{\pi it\lambda}\,\mathrm{d}\lambda\end{equation} for $t\in[0,1]$, where
\begin{equation}\label{propertyH}{\int \langle \lambda\rangle^{1/4}|\partial_\alpha^\rho H(\lambda,\lambda_0^+,\mu_1^+,\cdots,\mu_{m^+}^+,\widetilde{\alpha}[W^+])|\,\mathrm{d}\widetilde{\alpha}[W^+]\mathrm{d}\lambda\leq C^{m^+}(2|\rho|)!}.\end{equation}

\emph{Case 2.2}. Suppose $\widetilde{I}$ is of class $R$, then $\widetilde{I}$ has form (\ref{typeR0}), say with special index $1$. The arguments are similar to \emph{Case 2.1}, except that now we are considering the function
\[R_{\ell_1+\epsilon\widetilde{\alpha_1},\ell_2+\epsilon\widetilde{\alpha_1};\gamma_1,\gamma_2,\gamma_3}(e^{\pi i\lambda_0^+s})(t)=\frac{\chi(\ell_1+\epsilon\widetilde{\alpha_1}+\gamma_3)}{\ell_1+\epsilon\widetilde{\alpha_1}+\gamma_3}e^{\pi i(\ell_2+\epsilon\widetilde{\alpha_1}+\gamma_1)t}\int_0^t e^{\pi i(\gamma_2+\lambda_0^+-\ell_2-\epsilon\widetilde{\alpha_1})s}\,\mathrm{d}s.\] By inserting the cutoffs $\chi_0(\gamma_2+\lambda_0^+-\ell_2-\epsilon\widetilde{\alpha_1})$ or $\chi_\infty(\gamma_2+\lambda_0^+-\ell_2-\epsilon\widetilde{\alpha_1})$ as above, and using the fact that the resulting coefficient depending on $\widetilde{\alpha_1}$ is bounded by
\begin{equation}\label{integrableterm}\frac{1}{\langle\ell_1+\epsilon\widetilde{\alpha_1}+\gamma_3\rangle}\cdot\frac{1}{\langle\gamma_2+\lambda_0^+-\ell_2-\epsilon\widetilde{\alpha_1}\rangle}\end{equation} which is integrable in $\widetilde{\alpha_1}$ uniformly in other variables, we can conclude that, similar to (\ref{expandG}) and (\ref{propertyG}), we have
\begin{equation}\label{expandG1}R_{\ell_1+\epsilon\widetilde{\alpha_1},\ell_2+\epsilon\widetilde{\alpha_1};\gamma_1,\gamma_2,\gamma_3}(e^{\pi i\lambda_0^+s})(t)=\int_\Rb G_1(\lambda,\widetilde{\alpha_1})e^{\pi it\lambda}\,\mathrm{d}\lambda\end{equation} for $t\in[0,1]$, where ({again viewing $\widetilde{\alpha}[W^+\backslash\{1\}]$ as parameters})
\begin{equation}\label{propertyG1}{\int_{\Rb^2} \langle \lambda\rangle^{1/4}|\partial_\alpha^\rho G_1(\lambda,\widetilde{\alpha_1})|\,\mathrm{d}\lambda\mathrm{d}\widetilde{\alpha_1}\leq C(2|\rho|)!\bigg(1+\sum_{1\neq j\in W^+}|\widetilde{\alpha_j}|+|\gamma_1|+|\gamma_3|\bigg)^{1/4}}.\end{equation} Note that in deducing (\ref{propertyG1}) we have used the fact that the integrability of (\ref{integrableterm}) remains true uniformly in the other variables, even if one of the denominators is raised to the $3/4$-th power. Here {if part of the weight is $|\gamma_2+\lambda_0^+-\ell_2-\epsilon\widetilde{\alpha_1}|^{1/4}$, it can be estimated using the denominator which contains the same expression $\gamma_2+\lambda_0^+-\ell_2-\epsilon\widetilde{\alpha_1}$;} if part of the weight is $|\widetilde{\alpha_1}|^{1/4}$, it can be estimated using the denominator $\ell_1+\epsilon\widetilde{\alpha_1}+\gamma_3$. Now using (\ref{typeR0}) to expand $\widetilde{I}$ as a linear combination of $R_{\ell_1+\epsilon\widetilde{\alpha_1},\ell_2+\epsilon\widetilde{\alpha_1};\gamma_1,\gamma_2,\gamma_3}$, and combining (\ref{propertyG1}) with (\ref{typeR2}) as in \emph{Case 2.1} above, we obtain that (\ref{expandH}) and (\ref{propertyH}) remain true in this case.

Now, in either case, we have obtained the formula (\ref{expandH}) and the estimate (\ref{propertyH}), which are enough to treat the $\widetilde{I}(e^{\pi i\lambda_0^+s})$ part of $K$ in (\ref{regcplform1}), noticing that $\widetilde{\alpha_j}$ is a translation of $\alpha_j^+$ given the $(\lambda_{a,\epsilon,\iota})$ and $(\lambda_{lp,\pm})$ variables. However, for $j\in Z^+$, we still have $\widetilde{\alpha_j}=\alpha_j^++\epsilon_j^+\lambda_a^+$ instead of $\alpha_j^+$ in (\ref{regcplform1}). But this is easily resolved using the simple fact that\begin{equation}\label{differenceloss}\int_\Rb\bigg|\frac{\chi_\infty(\alpha_j^++\epsilon_j^+\lambda_a^+)}{\epsilon_j^+\pi i(\alpha_j^++\epsilon_j^+\lambda_a^+)}-\frac{\chi_\infty(\alpha_j^+)}{\epsilon_j^+\pi i\alpha_j^+}\bigg|\,\mathrm{d}\alpha_j^+\leq C\log(2+|\lambda_a^+|)\leq C\langle \lambda_a^+\rangle^{1/12}.\end{equation} Therefore, if we encounter the difference term $\frac{\chi_\infty(\alpha_j^++\epsilon_j^+\lambda_a^+)}{\epsilon_j^+\pi i(\alpha_j^++\epsilon_j^+\lambda_a^+)}-\frac{\chi_\infty(\alpha_j^+)}{\epsilon_j^+\pi i\alpha_j^+}$ for some $j$, we simply move this $j$ from $Z^+$ to $W^+$ and exploit the integrability of this difference in $\alpha_j^+$. The higher order derivatives can be treated in the same way as in \emph{Case 1}, and the total loss caused by the right hand sides of (\ref{differenceloss}) is bounded by
\begin{equation}\label{propertyH3}C^{m^+}\prod_{j\in Z^+}\prod_\iota\langle\lambda_{a,+,\iota}\rangle^{1/12}\langle\lambda_{b,+,\iota}\rangle^{1/12}\end{equation} (where $Z^+$ is the set before moving the elements $j$). {Clearly \ref{propertyH3} can be controlled by the part in (\ref{extraweight}) with sign $+$, up to a $C^{m^+}$ factor.}

Now we can return to the formula (\ref{newexp4}). Clearly the arguments for the functions depending on $t$ can be repeated for the function depending on $s$, obtaining a set $Z^-$. By combining these arguments, {as well as the arguments moving some of the $j\in Z^\pm$ to $W^\pm$}, we can write $\widetilde{\Bc}_\Qc$ in the form of (\ref{maincoef1}), where
\begin{equation*}Z=\bigg(\bigcup_{j,\epsilon,\iota} Z_{j,\epsilon,\iota}\bigg)\cup Z_{lp}\cup\big\{\nf_a^+:a<b,\,j\in Z^+\big\}\cup\big\{\nf_a^-:a<b,\,j\in Z^-\big\};\end{equation*} note also that $\epsilon_j^+=\zeta_{\nf_a^+}$ for $j\in Z^+$ (hence $\nf_a^+\in Z$) as in Section \ref{recursive}, and that the sets $Z^\pm$ may have been modified after moving the elements $j$ as described above. The bound (\ref{maincoef2}) without derivatives then follows from the estimates obtained above including (\ref{propertyH}) and (\ref{differenceloss}); in fact, if the regular couples $\Qc_{j,\epsilon,\iota}$ has scale $n_{j,\epsilon,\iota}$ etc., then the induction hypothesis together with the above estimates bounds the left hand side of (\ref{maincoef2}) without derivatives by
\[\prod_{j,\epsilon,\iota}C^{n_{j,\epsilon,\iota}}\cdot C^{n_{lp}}\cdot C^{m^+}\cdot C^{m^-}=C^n,\] noticing that
\[\sum_{j,\epsilon,\iota}n_{j,\epsilon,\iota}+n_{lp}+m^++m^-=n.\] The higher order derivative estimates in (\ref{maincoef2}) can be proved in the same way as in \emph{Case 1} using the factorized structure. This completes the proof of (\ref{maincoef1}) and (\ref{maincoef2}).

Finally we prove (\ref{maincoef2.5}) using a modification of the above inductive arguments. The proof scheme is the same, except that we induct (\ref{maincoef2.5}) in addition to (\ref{maincoef2}). In the inductive step, if $\nf_*:=\arg\!\max\langle\alpha_\nf\rangle$ belongs to one of $\Nc_{j,\epsilon,\iota}^{ch}\backslash Z_{j,\epsilon,\iota}$ or $\Nc_{lp}^{ch}\backslash Z_{lp}$, then we repeat the above arguments using (\ref{maincoef2.5}) for $\Cc_{j,\epsilon,\iota}$ or $\Cc_{lp}$, and noticing that changing the exponent $1/4$ to $1/8$ does not affect any part of the above proof.

 Now suppose $\nf_*=\nf_a^\pm$ with $a<b$ and $j\in W^\pm$ (in \emph{Case 2}, or $\nf_*=\rf$ in \emph{Case 1} which is similar), then for the functions $\Cc_{j,\epsilon,\iota}$ and $\Cc_{lp}$ we do not need to gain the $\alpha_\nf$ weight, so we can use the bound (\ref{maincoef2}) for them. We shall replace the weight $\langle\lambda\rangle^{1/4}$ in (\ref{propertyG}) and (\ref{propertyG1}) by $\langle \lambda\rangle^{1/8}\langle\alpha_{\nf_*}\rangle^{1/8}$. For example, if $\nf_*=\nf_a^+$ in \emph{Case 2.1} then instead of (\ref{propertyH}) we now have
\begin{equation*}\int \langle \lambda\rangle^{1/8}\langle\alpha_{\nf_*}\rangle^{1/8}|H(\lambda,\lambda_0^+,\mu_1^+,\cdots,\mu_{m^+}^+,\widetilde{\alpha}[W^+])|\,\mathrm{d}\widetilde{\alpha}[W^+]\mathrm{d}\lambda\leq C^{m^+}\langle\lambda_{a}^+\rangle^{1/8},\end{equation*} using also that $\widetilde{\alpha_j}=\alpha_{\nf_*}\pm\lambda_a^+$; the other cases are similar. Since $\lambda_a^+=\lambda_{a,+,1}+\lambda_{a,+,2}$, we know that the factor $\langle\lambda_{a}^+\rangle^{1/8}$ can be added to (\ref{propertyH3}) which is then still controlled by the part in (\ref{extraweight}) with sign $+$. This means that we can insert the power $\langle \alpha_{\nf_*}\rangle^{1/8}$ at the price of weakening $\langle \lambda\rangle^{1/4}$ to $\langle \lambda\rangle^{1/8}$. Finally, if $\nf_*$ occurs in the process of moving $j$ from $Z^+$ to $W^+$, then the same result is true because (\ref{differenceloss}) is still true, with the right hand side replaced by $C\langle \lambda_a^+\rangle^{1/7}$, if the integrand on the left hand side is multiplied by $\langle \alpha_j^+\rangle^{1/8}$. This proves (\ref{maincoef2.5}).
\end{proof}
\section{Regular couples II: approximation by integrals}\label{numbertheory} With the properties of $\Bc_\Qc$ obtained in Section \ref{regasymp}, we now calculate the asymptotics of $\Kc_\Qc$ as in (\ref{defkq}) for regular $\Qc$, using number theoretic methods.
\subsection{A general approximation result} We prove here a general approximation result, which we apply to $\Kc_\Qc$ in Section \ref{applyKQ}.
\begin{prop}\label{approxnt} Fix $\beta\in(\Rb^+)^d\backslash\Zf$. Consider the following expression
\begin{equation}\label{defi}I:=\sum_{(x_1,\cdots,x_n)}\sum_{(y_1,\cdots,y_n)}W(x_1,\cdots,x_n,y_1,\cdots,y_n)\cdot\Psi(L^2\delta\langle x_1,y_1\rangle_\beta,\cdots,L^2\delta\langle x_n,y_n\rangle_\beta).\end{equation} where $(x_1,\cdots,x_n,y_1,\cdots,y_n)\in(\Zb_L^d)^{2n}$ in (\ref{defi}). Assume there is a (strict) partial ordering $\prec$ on $\{1,\cdots,n\}$, and that the followings hold for the functions $W$ and $\Psi$:

(1) The function $W$ satisfies the bound (here $\widehat{W}$ denotes the Fourier transform in $(\Rb^d)^{2n}$)
\begin{equation}\label{propertyw1}\|\widehat{W}\|_{L^1}+\|\widehat{\partial W}\|_{L^1}\leq (C^+)^n.
\end{equation}

(2) This $W$ is supported in the set\begin{equation}\label{propertyw2}
E:=\big\{(x_1,\cdots,x_n,y_1,\cdots,y_n):|\widetilde{x_j}-a_j|,\,|\widetilde{y_j}-b_j|\leq \lambda_j,\,\forall 1\leq j\leq n\big\},
\end{equation} where $1\leq\lambda_j\leq (\log L)^4$ are constants, $a_j$ and $b_j$ are constant vectors. Each $\widetilde{x_j}$ is a linear function that equals either $x_j$, or $x_j\pm x_{j'}$ or $x_j\pm y_{j'}$ for some $j'\prec j$, similarly each $\widetilde{y_j}$ equals either $y_j$, or $y_j\pm x_{j''}$ or $y_j\pm y_{j''}$ for some $j''\prec j$.

(3) For some set $J\subset\{1,\cdots,n\}$, the function $\Psi$ has the expression
\begin{equation}\label{propertypsi1}\Psi(\Omega_1,\cdots,\Omega_n)=\prod_{j\in J}\frac{\chi_\infty(\Omega_j)}{\Omega_j}\cdot\Psi_1(\Omega[J^c]),
\end{equation} where $\chi_\infty$ is as in Section \ref{notations}, and for any $|\rho|\leq 10n$ we have
\begin{equation}\label{propertypsi2}
\|\partial^\rho\Psi_1\|_{L^1}\leq C^n(4|\rho|)!,\,\, \big\|\max_{j\in J^c}\langle \Omega_j\rangle^{1/8}\cdot\Psi_1\big\|_{L^1}\leq C^n.
\end{equation}

Assume $n\leq (\log L)^3$. Then we have
\begin{multline}\label{conclusion1}\bigg|I-L^{2dn}\int_{(\Rb^d)^{2n}}W(x_1,\cdots,x_n,y_1,\cdots,y_n)\cdot\Psi(L^2 \delta \langle x_1,y_1\rangle_\beta,\cdots,L^2 \delta\langle x_n,y_n\rangle_\beta)\,\mathrm{d}x_1\cdots\mathrm{d}x_n\mathrm{d}y_1\cdots\mathrm{d}y_n\bigg|\\\leq (\lambda_1\cdots\lambda_n)^C(C^+L^{2d-2}\delta^{-1})^{n} L^{-2\nu}.
\end{multline} {Here we recall that the choice of $\nu$, and the convention for $C$ and $C^+$, are fixed in Section \ref{notations}.} Moreover, defining
\begin{multline}\label{intappr}I_{\mathrm{app}}=(L^{2d-2}\delta^{-1})^{n}\int\Psi_1\mathrm{d}\Omega[J^c]\cdot\int_{(\Rb^d)^{2n}}W(x_1,\cdots,x_n,y_1,\cdots,y_n)\\\times\prod_{j\in J}\frac{1}{\langle x_j,y_j\rangle_\beta}\prod_{j\not\in J}\dirac(\langle x_j,y_j\rangle_\beta)\mathrm{d}x_1\cdots\mathrm{d}x_n\mathrm{d}y_1\cdots\mathrm{d}y_n,
\end{multline} where the singularities $1/\langle x_j,y_j\rangle_\beta$ are treated using the Cauchy principal value, we have
\begin{equation}\label{conclusion2}|I_{\mathrm{app}}|\leq (\lambda_1\cdots\lambda_n)^C(C^+L^{2d-2}\delta^{-1})^{n},\quad |I-I_{\mathrm{app}}|\leq(\lambda_1\cdots\lambda_n)^C(C^+L^{2d-2}\delta^{-1})^{n} L^{-2\nu}.
\end{equation}
\end{prop}
Before proving Proposition \ref{approxnt}, we first need to establish a few auxiliary results. In these results we will use the notation $e(z)=e^{2\pi iz}$, and fix $\lambda$ such that $1\leq \lambda\leq (\log L)^4$; moreover we will set $\upsilon=1/40$, so that $\nu\ll\upsilon$ by our choice.
\begin{lem}\label{NTSP}
	Suppose $\Phi:\Rb\times \Rb^d \times \Rb^d \to \Cb$ is a function satisfying the bounds
	\begin{equation}\label{NTphibounds}
	\sup_{s, x, y}|\partial_x^\alpha \partial_y^\beta \Phi(s,x, y)|\leq D
	\end{equation}
for all multi-indices $|\alpha|, |\beta|\leq 10d$. Then we have:

(1) The following bound
		\begin{equation}\label{NTSPest0}
		\left|\int_{\Rb}\int_{\Rb^{2d}}	\chi_0\big(\frac{x-a}{\lambda}\big)\chi_0\big(\frac{y-b}{\lambda}\big)\Phi(s,x, y)\cdot e(\xi \cdot x+\eta\cdot y+s\langle x, y\rangle_\beta) \, \mathrm{d}x \mathrm{d}y \mathrm{d}s\right| \lesssim D\lambda^{2d}
		\end{equation}
holds uniformly in $(\xi, \eta, a, b)\in \Rb^{4d}$.	

(2) Suppose, in addition, that $\Phi$ satisfies one of the following two requirements:
\begin{enumerate}[(a)]
\item $\Phi(s, x, y)=\chi_0(\frac{s}{K})\widehat \psi(\frac{s}{\delta L^2})\Phi'(x, y)$, where $\psi:=\frac{\chi_\infty(t)}{t}$ and $\Phi'$ satisfies \eqref{NTphibounds} without $s$, or 
\item $\Phi(s, x,y)=\chi_0(\frac{s}{K})\Phi'(\frac{s}{\delta L^2}, x, y)$, where $\Phi'$ satisfies \eqref{NTphibounds}, and $\mathcal F_s\Phi'(\cdot, x,y)$ is supported on an interval of length $O(1)$ in $\Rb$ which does not depend on $(x,y)\in \Rb^{2d}$.
\end{enumerate}
Here $K:= \lambda^{-1}L^{1-\upsilon}\gg 1$. Then, there holds
\begin{equation}\label{NTSPest}
\begin{split}
\bigg|\sum_{0\neq(g,h)\in \Zb^{2d}} \int_{\Rb} \int_{\Rb^{2d}} \chi_0\big(\frac{x-a}{\lambda}\big)\chi_0\big(\frac{y-b}{\lambda}\big)\Phi(s, x, y)\cdot e[(Lg+\xi) \cdot x +(Lh +\eta)\cdot y+s\langle x, y\rangle_\beta]\bigg|\\
\lesssim D\lambda^{2d} \min\left[(1+|\xi|+|\eta|)L^{-\frac{4}{10}\upsilon}, 1\right]
\end{split}
\end{equation}
uniformly in $(a, b)\in \Rb^{2d}$. In particular, we have
\begin{equation}
		\bigg|\sum_{(g,h) \in \Zb^{2d}} \int_{\Rb} \int_{\Rb^{2d}} \chi_0\big(\frac{x-a}{\lambda}\big)\chi_0\big(\frac{y-b}{\lambda}\big)\Phi(s, x, y)\cdot e[(Lg+\xi) \cdot x +(Lh +\eta)\cdot y+s\langle x, y\rangle_\beta]\bigg|
		\lesssim D\lambda^{2d}\label{NTSPest2}
\end{equation}
		uniformly in $(\xi, \eta,a,b)\in \Rb^{4d}$.
\end{lem}
\begin{proof} (1) By translating $x$ and $y$, it is enough to consider the case $a=b=0$. The result then follows by an application of the stationary phase lemma (or direct integration of the Gaussian phase) to bound the integral in $(x,y)$ by $\lambda^{2d}\langle s\rangle^{-d}$ which is integrable.

(2) Let us denote the left hand side of \eqref{NTSPest} without absolute value by $M(\xi, \eta)$ and also $\Omega=\langle x, y\rangle_\beta$. Here, we split the discussion into two cases depending on the size of $a$ and $b$:

\emph{Case 1: if $\max(|a|,|b|)\leq \lambda L^{\upsilon/2}$.} Here, we argue via a stationary phase analysis for the phase function $\varphi(x, y)=(Lg+\xi) \cdot x +(Lh +\eta)\cdot y+s\Omega$. Noting that 
\begin{equation*}
\nabla_{x}\varphi=L g+\xi+s(\beta^1 y^1, \ldots, \beta^d y^d), \quad \nabla_{y}\varphi=L h+\eta+s(\beta^1 x^1, \ldots, \beta^d x^d),
\end{equation*}
and using our assumption on the $s$ support of $\Phi$, we can bound the norms of $s(\beta^1 y^1, \ldots, \beta^d y^d)$ and $s(\beta^1 x^1, \ldots, \beta^d x^d)$ by $L/10$ if $L$ is large enough. If $|Lg+\xi|+|Lh+\eta|\geq L/5$ (which happens for all but one value of $(g,h)$), we integrate by parts at most $2d$ times in $x$ or $y$ in the $\mathrm{d}x\mathrm{d}y$ integral, and gain a denominator that is bounded below by $(L+ |Lg+\xi|+|Lh+\eta|)^{2d}$. For the only remaining value of $(g,h)$, we use a stationary phase estimate similar to part (1). In the end we get 
\begin{align*}
|M(\xi, \eta)|\lesssim & D\lambda^{2d} \sum_{(g,h) \neq 0}
\int_{|s|\leq \lambda^{-1}L^{1-\upsilon}}  \left[(L+ |Lg+\xi|+|Lh+\eta|)^{-2d}+ \langle s\rangle^{-d}\mathbf 1_{|Lg+\xi|+|Lh+\eta|< L/5}\right]\\
\lesssim & D\lambda^{2d}  \bigg[L^{-\upsilon}+\sum_{(g,h) \neq 0} \mathbf 1_{|Lg+\xi|+|Lh+\eta|< L/5}\bigg]\\
\lesssim &  D\lambda^{2d} \left(L^{-\upsilon}+\mathbf 1_{|\xi|+|\eta|\gtrsim L}\right)\lesssim D\lambda^{2d} \min\left(1,(1+|\xi|+|\eta|)L^{-\upsilon}\right),
\end{align*} as needed.

\emph{Case 2: if $\max(|a|,|b|)\geq \lambda L^{\upsilon/2}$.} Here the analysis is slightly more delicate, and we will obtain the estimate
\begin{equation}\label{NTcase2est}
|M(\xi, \eta)|\lesssim D\lambda^{2d}\min\left[1,(1+|\xi|+|\eta|)L^{-\frac{4}{10}\upsilon}\right]
\end{equation}
uniformly in $( a,b)$, which finishes the proof of the lemma.

We start by splitting the integral appearing in $M(\xi, \eta)$ into two parts: one with $|\nabla_x \varphi|+|\nabla_y \varphi| \leq L^{1-2\upsilon}$ giving a contribution $M_1(\xi, \eta)$ and the complementary region giving a contribution $M_1'(\xi, \eta)$. More precisely, define
\begin{align}\label{NTdefofM1}
M_1(\xi, \eta)=&\sum_{0\neq (g,h)\in\Zb^{2d}} \int_{\Rb} \int_{\Rb^{2d}} \chi_0\big(\frac{x-a}{\lambda}\big)\chi_0\big(\frac{y-b}{\lambda}\big)\Phi(s, x, y)\chi_0\left(\frac{\nabla_x \varphi}{L^{1-2\upsilon}}\right)\chi_0\left(\frac{\nabla_y \varphi}{L^{1-2\upsilon}}\right)\\\nonumber
&\qquad \qquad \qquad \times e[(Lg+\xi) \cdot x +(Lh +\eta)\cdot y+s \Omega]\,\mathrm{d}x\mathrm{d}y\mathrm{d}s,\\
M_1'(\xi, \eta)=&M(\xi, \eta)-M_1(\xi, \eta).\nonumber
\end{align}
The contribution of $M_1'(\xi, \eta)$ can be bounded easily by integrating by parts $10d$ times in either $x$ or $y$ (depending on whether $|\nabla_x \varphi|$ or $|\nabla_y \varphi|$ is $\gtrsim L^{1-2\upsilon}$) and estimated by $(\lambda^{-1}L^{1-\upsilon})\lambda^{2d}L^{-9d}$ which is more than acceptable. As such, we reduce to obtaining the bound \eqref{NTcase2est} for $M_1(\xi, \eta)$.

\medskip 

Next, we would like to localize in $\Omega$. For this let $J:={\max(|a|,|b|)}\geq \lambda L^{\upsilon/2}$. Here, the analysis is different depending on whether we make the assumption (a) or (b) on $\Phi(s, x, y)$. Under assumption (a), we note that $\widehat \psi(s)$ is odd and fastly decaying at infinity, and has a jump discontinuity at 0; all its derivatives are also uniformly bounded and decaying, except at 0 where they are not defined. We would like to integrate by parts in $s$ once, in the region when $|\Omega|\gtrsim JL^{-1+\frac{21}{10}\upsilon}$. Such integration by parts produces a new $s$-integrand which has the same form and is bounded as
$$
\lesssim |\Omega|^{-1}\left(\delta^{-1}L^{-2}+\frac{1}{K}+\frac{J}{L^{1-2\upsilon}}\right){\lesssim}|\Omega|^{-1}\left(\frac{\lambda}{ L^{1-\upsilon}}+\frac{JL^\upsilon}{L^{1-\upsilon}}\right)\lesssim \frac{JL^{2\upsilon}}{|\Omega|L}\lesssim L^{-\upsilon/10},
$$
and a boundary term (due to the discontinuity of $\widehat{\psi}(s)$ at 0) that gives a contribution of 
\begin{multline*}
	\sum_{0\neq(g,h)\in \Zb^{2d}} \chi_0\left(\frac{Lg+\xi}{L^{1-2\upsilon}}\right)\chi_0\left(\frac{Lh+\eta}{L^{1-2\upsilon}}\right)\int_{{|\Omega|\gtrsim JL^{-1+21\upsilon/10}}}\Omega^{-1}{\chi_\infty(\delta L^2\Omega)} \chi_0\big(\frac{x-a}{\lambda}\big)\chi_0\big(\frac{y-b}{\lambda}\big)\\\times\Phi'(x, y)
	\cdot e[(Lg+\xi) \cdot x +(Lh +\eta)\cdot y]\,\mathrm{d}x\mathrm{d}y
\end{multline*}
{(note $\chi_\infty(\delta L^2\Omega)=1$ with the given lower bound of $\Omega$)}, which can be estimated using Lemma \ref{NTOmegainv} below by 
$$
\lesssim D\lambda^{2d}\sum_{(g,h) \neq 0} \mathbf 1_{|Lg+\xi|+|Lh+\eta|< L/10}\lesssim D\lambda^{2d} \min\left(1, (1+|\xi|+|\eta|)L^{-\upsilon}\right),
$$
as needed. Repeating this integration by parts {$C\upsilon^{-1}$} times, we are reduced to obtaining the bound \eqref{NTcase2est} for
\begin{align*}
M_2(\xi, \eta)=&\sum_{0\neq(g,h)\in \Zb^{2d}} \int_{\Rb} \int_{\Rb^{2d}} \chi_0\big(\frac{x-a}{\lambda}\big)\chi_0\big(\frac{y-b}{\lambda}\big) \Phi(s, x, y)\chi_0\left(\frac{\nabla_x \varphi}{L^{1-2\upsilon}}\right)\chi_0\left(\frac{\nabla_y \varphi}{L^{1-2\upsilon}}\right)\chi_0\left(\frac{\Omega}{JL^{-1+\frac{21\upsilon}{10}}}\right)\\
&\qquad \qquad \qquad \times e[(Lg+\xi) \cdot x +(Lh +\eta)\cdot y+s\cdot \Omega]\,\mathrm{d}x\mathrm{d}y\mathrm{d}s.
\end{align*}

Now, under assumption (b) on $\Phi(s, x, y)$, we have, for some $p \in \Rb$, $\Phi'(s, x, y)=e^{-2\pi isp}\Phi''(s,x,y)$ where $\mathcal F_s \Phi''(\cdot, x,y)$ is supported in (say) $[-1,1]$ for any $(x,y)\in \Rb^{2d}$. As before, we split $M_1(\xi, \eta)$ in two parts depending on the size of $|\Omega-\frac{p}{\delta L^2}|$. Let $\widetilde M_2(\xi, \eta)$ denote the contribution of the region where $|\Omega-\frac{p}{\delta L^2}|\leq JL^{-1+\frac{21}{10}\upsilon}$ and $\widetilde M_2'(\xi, \eta)$ the contribution of the complementary region, namely
\begin{align*}
\widetilde M_2(\xi, \eta)=&\sum_{0\neq(g,h)\in \Zb^{2d}} \int_{\Rb} \int_{\Rb^{2d}} \chi_0\big(\frac{x-a}{\lambda}\big)\chi_0\big(\frac{y-b}{\lambda}\big) \chi_0\big(\frac{s}{K}\big)\Phi''\big(\frac{s}{\delta L^2}, x, y\big)\chi_0\bigg(\frac{\nabla_x \varphi}{L^{1-2\upsilon}}\bigg)\chi_0\bigg(\frac{\nabla_y \varphi}{L^{1-2\upsilon}}\bigg)\\
&\qquad \qquad \qquad \chi_0\bigg(\frac{\Omega-\frac{p}{\delta L^2}}{JL^{-1+\frac{21\upsilon}{10}}}\bigg)e\big[(Lg+\xi) \cdot x +(Lh +\eta)\cdot y+s\big(\Omega-\frac{p}{\delta L^2}\big)\big]\,\mathrm{d}x\mathrm{d}y\mathrm{d}s,\\
\widetilde M_2'(\xi, \eta)=&M_1(\xi, \eta)-\widetilde{M_2}(\xi, \eta).
\end{align*}

The contribution of $\widetilde M_2'$ can be bounded by integrating sufficiently many times in $s$: Each integration by parts produces a factor of $(\Omega-\frac{p}{\delta L^2})^{-1}$ at the expense of having an $s$ derivative fall on either $\chi_0(s/K)$ or $\Phi''(s/(\delta L^2), x,y)$ (giving a factor bounded by $D(\delta L^2)^{-1}$ given the Fourier support assumption on $\Phi''$) or the $\nabla \varphi$ factors in the other spatial cutoffs (which gives a factor bounded by $\frac{JL^\upsilon}{L^{1-\upsilon}}$). In effect, the net gain of this integration by part step is 
$$
\lesssim \big|\Omega-\frac{p}{\delta L^2}\big|^{-1}\left(\frac{1}{\delta L^2}+\frac{1}{K}+\frac{J}{L^{1-2\upsilon}}\right)
\lesssim \frac{JL^{2\upsilon}}{|\Omega-\frac{p}{\delta L^2}|L}\lesssim L^{-\upsilon/10},
$$
using the lower bound on $|\Omega-\frac{p}{\delta L^2}|$ in $\widetilde M_2'$. As such, we can integrate by parts {$C\upsilon^{-1}$} times in $s$ to obtain that the contribution of $\widetilde M_2'$ is acceptable as well.

\medskip

As such, in both cases (a) and (b), we are left with the contribution of $M_2(\xi, \eta)$ and $\widetilde M_2(\xi, \eta)$ respectively, which we will estimate in the same way (thanks to the bounds $ |\Phi'|,|\Phi''|\lesssim D$) by
\begin{align*}
D\int_{\Rb^{2d}} \chi_0\big(\frac{x-a}{\lambda}\big)\chi_0\big(\frac{y-b}{\lambda}\big)\chi_0\bigg(\frac{\Omega-\frac{p}{\delta L^2}}{JL^{-1+\frac{21\upsilon}{10}}}\bigg)\int_{\Rb}\chi_0\big(\frac{s}{K}\big)\sum_{0\neq (g,h)\in\Zb^{2d}} \chi_0\left(\frac{\nabla_x \varphi}{L^{1-2\upsilon}}\right)\chi_0\bigg(\frac{\nabla_y \varphi}{L^{1-2\upsilon}}\bigg)\mathrm{d}s\mathrm{d}x\mathrm{d}y.
\end{align*}

Now notice that the volume of the set of $(x,y)$ satisfying $|x-a|\lesssim\lambda,|y-b|\lesssim\lambda$ and {$|\Omega-\frac{p}{\delta L^2}|\lesssim JL^{-1+\frac{21\upsilon}{10}}$} is bounded by $\lambda^{2d-1}L^{-1+\frac{21\upsilon}{10}}$ since $|\nabla_{x,y} \Omega|\sim J$ (using the coarea formula or change of variables). Furthermore, for fixed $(\xi, \eta, x, y)$, when $(g,h)$ varies, we claim that the measure of the set $S(x, y)$ on which the $s$ integrand is supported is bounded by $L^{1-2\upsilon}\cdot \frac{(1+JL^{-\upsilon})}{J}$. To see this, suppose without loss of generality that $|y^1|\sim J$ for example, then $s\in S(x,y)$ implies that 
$$
\bigg|\frac{s\beta^1y^1}{L} +g^1 +\frac{\xi^1}{L}\bigg|\lesssim L^{-2\upsilon} \Rightarrow \bigg\{\frac{s\beta^1y^1}{L} +\frac{\xi^1}{L}\bigg\}\leq L^{-2\upsilon},
$$
where $\{\cdot\}$ denotes the distance to the nearest integer. Since also $|\frac{s\beta^1y^1}{L}|\lesssim JKL^{-1}{\lesssim}JL^{-\upsilon}$, we conclude that $\frac{s\beta^1y^1}{L}$ belongs to an interval of length $\lesssim JL^{-\upsilon}$, intersected by the $L^{-2\upsilon}$ neighborhood of the lattice $\frac{\xi_1}{L}+\Zb$. Thus $\frac{sy^1}{L}$ belongs to a set of measure at most $(1+JL^{-\upsilon})L^{-2\upsilon}$ and hence
$$
|S(x,y)|\lesssim \frac{L}{J}(1+JL^{-\upsilon})L^{-2\upsilon},
$$
as claimed. With this estimate in hand, we can bound both $M_2$ and $\widetilde M_2$ by
\begin{align*}
\lesssim &D\int_{\Rb^{2d}} \chi_0\big(\frac{x-a}{\lambda}\big)\chi_0\big(\frac{y-b}{\lambda}\big)\chi_0\bigg(\frac{\Omega-\frac{p}{\delta L^2}}{JL^{-1+\frac{21\upsilon}{10}}}\bigg)\int_{ S(x,y)} \sum_{0\neq(g,h)\in \Zb^{2d}} \chi_0\bigg(\frac{\nabla_x \varphi}{L^{1-2\upsilon}}\bigg)\chi_0\bigg(\frac{\nabla_y \varphi}{L^{1-2\upsilon}}\bigg)\mathrm{d}s\mathrm{d}x\mathrm{d}y\\
\lesssim&D\int_{\Rb^{2d}} \chi_0\big(\frac{x-a}{\lambda}\big)\chi_0\big(\frac{y-b}{\lambda}\big)\chi_0\bigg(\frac{\Omega-\frac{p}{\delta L^2}}{JL^{-1+\frac{21\upsilon}{10}}}\bigg) |S(x,y)| \,\mathrm{d}x\mathrm{d}y\\
\lesssim&D\lambda^{2d-1}L^{-1+\frac{21\upsilon}{10}} \times \frac{L}{J}{(1+JL^{-\upsilon})}L^{-2\upsilon}=D\lambda^{2d-1}L^{\upsilon/10}{(J^{-1}+L^{-\upsilon})}\lesssim D\lambda^{2d-1}L^{-\frac{4\upsilon}{10}}.
\end{align*}

This finishes the proof of \eqref{NTcase2est}, and hence that of \eqref{NTSPest}.
\end{proof}
\begin{lem}\label{NTOmegainv} Suppose that $\Theta=\Theta(x,y)$ satisfies \eqref{NTphibounds} without $s$.

(1) The following bound holds
$$
\left|\int_{\Rb^{2d}}\Omega^{-1}\chi_\infty(\mu\Omega) \chi_0\big(\frac{x-a}{\lambda}\big)\chi_0\big(\frac{y-b}{\lambda}\big)\Theta(x, y)
e(\xi \cdot x +\eta\cdot y)\,\mathrm{d}x\mathrm{d}y\right|\lesssim D\lambda^{2d},
$$	
uniformly in $\mu>1$ and $(a, b,\xi, \eta)\in \Rb^{4d}$. { In addition, the following limit of principal value type} 
$$
\lim_{\mu \to \infty}\int_{\Rb^{2d}}\Omega^{-1}\chi_\infty(\mu\Omega) \chi_0\big(\frac{x-a}{\lambda}\big)\chi_0\big(\frac{y-b}{\lambda}\big)\Theta(x, y)
e(\xi \cdot x +\eta\cdot y)\,\mathrm{d}x\mathrm{d}y
$$
exists and is $\lesssim D\lambda^{2d}$ uniformly in $(a, b,\xi, \eta)\in \Rb^{4d}$.

(2) The following estimate holds for the difference
$$
\left|\int_{\Rb^{2d}}\Omega^{-1}\chi_0(\mu\Omega) \chi_0\big(\frac{x-a}{\lambda}\big)\chi_0\big(\frac{y-b}{\lambda}\big)\Theta(x, y)
e(\xi \cdot x +\eta\cdot y)\,\mathrm{d}x\mathrm{d}y\right|\lesssim D\lambda^{2d}\mu^{-1}(1+|\xi|+|\eta|),
$$ uniformly in $(a, b)$.
\end{lem}
\begin{proof} (1) It is enough to consider the region $|\Omega|\lesssim 1$. Also, without loss of generality we only need to consider the region with $|x^1|\sim \max (|x|, |y|)$. Recall that $\chi_0$ is extended to $\Rb^d$ as in Section \ref{notations}; by abusing notation we may also use some $\chi_0$ of different support. Let $x=(x^1,x')$, and $\beta=(\beta^1,\beta')$ etc., it is enough to consider
	\begin{align*}
	N(\mu, \xi, \eta):=&\int_{\Rb^{2d}}\Omega^{-1}\chi_\infty(\mu\Omega)\chi_0(\Omega) \chi_0\big(\frac{x-a}{\lambda}\big)\chi_0\big(\frac{y-b}{\lambda}\big)\Theta(x, y)\chi_0\big(\frac{x'}{|x^1|}\big)\chi_0\big(\frac{y}{|x^1|}\big)
	e(\xi \cdot x +\eta\cdot y)\mathrm{d}x\mathrm{d}y\\
	=&\int_{\Rb^d}\chi_0\big(\frac{x-a}{\lambda}\big)\chi_0\big(\frac{x'}{|x^1|}\big)e(\xi\cdot x)\mathrm{d}x\int_{\Rb^d}\Omega^{-1}\chi_\infty(\mu\Omega)\chi_0(\Omega) \chi_0\big(\frac{y-b}{\lambda}\big)\Theta(x, y)\chi_0\big(\frac{y}{|x^1|}\big)
	e(\eta\cdot y)\mathrm{d}y.
	\end{align*} We change variables in the $y^1$ integral by setting $u=\Omega=\beta^1x^1 y^1+\langle x', y'\rangle_{\beta'}$, and write
	\begin{align}
	N(\mu, \xi, \eta)=&\int_{\Rb^d}|\beta^1x^1|^{-1}\chi_0\big(\frac{x-a}{\lambda}\big)\chi_0\big(\frac{x'}{|x^1|}\big)e(\xi\cdot x)\mathrm{d}x\int_{\Rb^{d-1}}\chi_0\big(\frac{y'-b'}{\lambda}\big)e\bigg[\eta'\cdot y' - \frac{\eta^1\langle x', y'\rangle_{\beta'}}{\beta^1 x^1}\bigg]\mathrm{d}y'\nonumber\\\label{NTdefNmuxi}
	&\qquad \times\int_{\Rb}\frac{\chi_\infty(\mu u)\chi_0(u)}{u}\chi_0\bigg(\frac{u-\langle x', y'\rangle_{\beta'}-\beta^1 x^1 b^1}{\beta^1 x^1\lambda}\bigg) \widetilde \Theta(x, u, y')
	e\big( \frac{\eta^1u}{\beta^1 x^1}\big)\mathrm{d}u.
	\end{align}
	where $\widetilde \Theta(x, u, y')=\Theta(x,y)\chi_0(\frac{y}{|x^1|})$ evaluated at $y^1=\frac{u-\langle x', y'\rangle_{\beta'}}{\beta^1 x^1}$.

	Now, we write the integral over $u$ as 
\[\int_{0}^\infty\frac{\chi_\infty(\mu u)\chi_0(u)}{u}\Delta(u, y',x)\,\mathrm{d}u,\] where
	\begin{align*}
	\Delta(u, y',x)&=\chi_0\bigg(\frac{u-\langle x', y'\rangle_{\beta'}-\beta^1 x^1 b^1}{\beta^1 x^1\lambda}\bigg) \widetilde \Theta(x, u,y')
	e\big(\frac{\eta^1u}{\beta^1x^1}\big)\\
	&\qquad\qquad\qquad\qquad\qquad\qquad-\chi_0\bigg(\frac{-u-\langle x', y'\rangle_{\beta'}-\beta^1 x^1 b^1}{\beta^1 x^1\lambda}\bigg) \widetilde\Theta(x, -u, y)e\big(\frac{-\eta^1u}{\beta^1 x^1}\big)\\
	&=2i\chi_0\bigg(\frac{u-\langle x', y'\rangle_{\beta'}-\beta^1 x^1 b^1}{\beta^1 x^1\lambda}\bigg) \widetilde \Theta(x, u,y')
	\sin\big(2\pi\frac{\eta^1u}{\beta^1 x^1}\big)+D\cdot O\bigg(\frac{|u|}{\beta^1 |x^1|}+\frac{|u|}{\beta^1 |x^1|^2}\bigg).
	\end{align*}
	
	Noting that the above integral is actually supported on the interval $[(2\mu)^{-1}, 2]$, the contribution of the terms $O(\frac{u}{\beta^1 |x^1|}+\frac{u}{\beta^1 |x^1|^2})$ is acceptable. Hence, we are left with bounding 	
	\begin{align*}
	&\int_{0}^\infty\chi_\infty(\mu u)\chi_0(u)\chi_0\bigg(\frac{u-\langle x', y'\rangle_{\beta'}-\beta^1 x^1 b^1}{\beta^1 x^1\lambda}\bigg) \widetilde \Theta(x, u,y')
	\sin\big( 2\pi\frac{\eta^1u}{\beta^1 x^1}\big)\frac{\mathrm{d}u}{u}\\
	&=\pm \int_{0}^\infty\chi_\infty\big(\frac{\mu \beta^1 |x^1|u}{|\eta^1|}\big)\chi_0\big(\frac{ \beta^1 |x^1|u}{|\eta^1|}\big)\chi_0\big(\frac{ u-Q}{\lambda |\eta^1|}\big) \widetilde \Theta\big(x, \frac{ \beta^1 x^1u}{\eta^1},y'\big)
	\sin( 2\pi u)\frac{\mathrm{d}u}{u}
	\end{align*}
	for some $Q$ depending on $x,y', \beta, \lambda$. The integral for $0\leq u\leq 5$ is bounded by $O(1)$, so we may use a smooth cutoff $\varphi(u)$ to restrict to $u\geq 5$. Then we integrate by parts in $u$ once. If the derivative falls on $1/u$ or $\varphi(u)$ the resulting contribution is bounded by $O(1)$. If the derivative falls on anything but the first factor, then the new integrand is bounded by $\frac{\beta^1|x^1|}{|\eta^1|}(1+\frac{1}{|x^1|}+\frac{1}{|x^1|^2})$, so the resulting contribution is bounded by $1+\frac{1}{|x^1|}+\frac{1}{|x^1|^2}$ even without using the $1/u$ factor, as $u$ belongs to an interval of length $\frac{|\eta^1|}{ \beta^1 |x^1|}$ due to the second factor. If the derivative falls on the first factor, then in the new integrand we will have $\frac{\mu\beta^1|x^1|}{|\eta^1|}$ instead of $\frac{\beta^1|x^1|}{|\eta^1|}$, but $u$ also belongs to a smaller interval of length $\frac{|\eta^1|}{ \mu\beta^1 |x^1|}$, so the conclusion will be the same.
	
	This leads to an acceptable contribution to $N(\mu, \xi, \eta)$ in \eqref{NTdefNmuxi}, and gives a bound that is uniform in $\mu, \xi, \eta$. The statement about the $\lim_{\mu \to \infty}$ follows directly from the above argument and dominated convergence. 
	
	(2) Arguing exactly as above, it is enough to bound 
	\begin{align*}
	\int_{0}^\infty\frac{\chi_0(\mu u)}{u}\Delta(u, y',x)du&\lesssim \int_{0}^\infty\frac{\chi_0(\mu u)}{u} u\left(\frac{1}{\beta^1 |x^1|}+\frac{1}{\beta^1 |x^1|^2 }+\frac{|\eta^1|}{\beta^1 |x^1|}\right) \, du\\
	&\lesssim  \mu^{-1}\left(\frac{1}{\beta^1 |x^1| }+\frac{1}{\beta^1 |x^1|^2 }+\frac{|\eta^1|}{\beta^1 |x^1|}\right)
	\end{align*}
	which gives the needed bound when substituted in \eqref{NTdefNmuxi}.
\end{proof}	
\begin{lem}\label{lem:minorarcs}
	Suppose that $\Phi(s,x,y): \Rb\times \Rb^d \times \Rb^d\to \Cb$ is a function satisfying (\ref{NTphibounds}).

(1) If $\Phi$ is supported on $|s|< L^2$, then the following bound holds uniformly in $(a, b, \xi, \eta)\in \Rb^{4d}$:
\begin{equation}\label{NTlem21}
\int_{\Rb}\bigg|\sum_{(x, y)\in \Zb^{2d}_L} \Phi(s, x, y) \chi_0\big(\frac{x-a}{\lambda}\big)\chi_0\big(\frac{y-b}{\lambda}\big)e(x\cdot \xi+y\cdot \eta+s \langle x, y\rangle_\beta)\bigg|\mathrm{d}s \lesssim D \lambda^{4d} L^{2d}.
\end{equation}

(2) If $\Phi(s,x,y)$ is supported on the set $|s|\gtrsim L^{1-\upsilon}$, then the following improved estimate holds uniformly in $(a, b, \xi,\eta)\in\Rb^{4d}$:
\begin{equation}\label{NTlem22}
\int_{\Rb}\big\langle\frac{s}{\delta L^2}\big\rangle^{-2}\bigg|\sum_{(x, y)\in \Zb^{2d}_L} \Phi(s, x, y) \chi_0\big(\frac{x-a}{\lambda}\big)\chi_0\big(\frac{y-b}{\lambda}\big)e(x\cdot \xi+y\cdot \eta+s \langle x, y\rangle_\beta)\bigg|\mathrm{d}s \lesssim D \lambda^{4d} L^{2d-\upsilon}.
\end{equation}	
\end{lem}	

\begin{proof} Recall that, $\Omega(x,y)=\sum_{j=1}^d\beta^j x^j y^j$ where $x^j, y^j\in \Zb_L$. We make the change of variables 
	$$
	L^{-1}p^j={x^j+y^j}, \qquad L^{-1}q^j={x^j-y^j}, \qquad p^j\equiv q^j \pmod 2.
	$$
The sum in $(x^j,y^j)\in \Zb_L^{2}$ then becomes the linear combination of four sums, which are taken over $(p^j, q^j)\in \Zb^{2}$, or $(p^j, q^j)\in 2\Zb\times\Zb$, or $(p^j, q^j)\in \Zb\times 2\Zb$, or $(p^j, q^j)\in (2\Zb)^2$. We will only consider the first sum, and it will be obvious from the proof that the other sums are estimated similarly. Define \[ \Upsilon (s,z, w)=\Phi\big(s, \frac{z+w}{2},\frac{z-w}{2}\big)\chi_0\big(\frac{z+w-2a}{2\lambda}\big)\chi_0\big(\frac{z-w-2b}{2\lambda}\big),\] which has all derivatives in $(z,w)$ up to order $10d$ uniformly bounded, and is supported in the set $\{g^j\leq Lz^j\leq g^j+2\lambda L,\,h^j\leq Lw^j\leq h^j+2\lambda L\}$, where $(g^j,h^j)\in\Zb^2$ are determined by $(a,b)$.

Now, by possibly redefining $(s,\xi,\eta)$, we need to show that the function 
	\begin{align*}
	 B(\xi, \eta)&=\int_{\Rb} \bigg|\sum_{(p, q)\in \Zb^{2d}} \Upsilon\big(s; pL^{-1}, qL^{-1}\big)e\big[sL^{-2}(|p|_\beta^2-|q|_\beta^2)+p\cdot  \xi+y\cdot \eta\big]\bigg| \, \mathrm{d}s\\
	&=\int_{\Rb} \bigg|\sum_{(p, q)\in \Zb^{2d}} \Upsilon\big(s,pL^{-1}, qL^{-1}\big)\prod_{j=1}^d e\big[sL^{-2}\beta^j (p^j)^2 +p^j \xi^j]\cdot e[-sL^{-2}\beta^j (q^j)^2 +q^j \eta^j\big]\bigg| \, \mathrm{d}s
	\end{align*}
	satisfies the bounds in \eqref{NTlem21} when $\Upsilon$ is supported on $|s|<L^2$, and that the corresponding integral with $\langle s/\delta L^2\rangle^{-2}$ (which we denote by $\widetilde{B}(\xi,\eta)$) satisfies \eqref{NTlem22} when $\Upsilon$ is supported on $|s|\gtrsim L^{1-\upsilon}$. Note that in the above sum we must have $p^j\in[g^j,g^j+20\lambda L]$ and $q^j\in[h^j,h^j+20\lambda L]$.

Recall the Gauss sums $G_h(s,r,n)$ and $G_h(s,r,x)$ defined in (\ref{huagausssum}). Notice that since $\partial_x G_h(s, r; x)=\sum_{p\in \Nb}e(s (h+p)^2+r(h+p)) \dirac(x-p)$, we can write
	\begin{align*}
	 B(\xi, \eta)&=\int_{\Rb} \bigg|\int_{(u, v)\in \Rb_+^{2d}} \Upsilon\big(s, (u+g)L^{-1}, (v+h)L^{-1}\big)\prod_{j=1}^d \partial_{u^j}G_{g^j}(\beta^j sL^{-2}, \xi^j , u^j) \\&\qquad\qquad\qquad\qquad\qquad\qquad\qquad\qquad\qquad\qquad\qquad\qquad\,\,\times\partial_{v^j}G_{h^j}(-\beta^j sL^{-2}, \eta^j, v^j) \, \mathrm{d}u \mathrm{d}v\bigg| \, \mathrm{d}s\\
	&\leq L^{-2d}\int_{\Rb}\int_{(u, v)\in \Rb_+^{2d}} \big| D^\alpha\Upsilon\big(s,(u+g)L^{-1}, (v+h)L^{-1}\big)\big|\\&\qquad\qquad\qquad\qquad\qquad\qquad\qquad\qquad\times\prod_{j=1}^d\big|G_{g^j}(\beta^j sL^{-2}, \xi^j , u^j) G_{h^j}(-\beta^j sL^{-2}, \eta^j, v^j)\big| \, \mathrm{d}u\mathrm{d}v\mathrm{d}s,
	\end{align*}
	where $g=(g^1,\cdots,g^d)$ etc., and $D^\alpha \Upsilon$ is obtained from $ \Upsilon$ by taking one derivative in each of the variables $u_j, v_j$ (and hence has the same support properties).
	
(1) We first note that if $\Phi$, and hence $\Upsilon$, is supported on the set $|s|<L^2$, then we have the bound (upon rescaling in $s$ by $L^2$)
	\begin{align*}
	|B(\xi, \eta)|\lesssim  DL^{-2d+2}\int_{(u, v)\in \Rb_+^{2d}}\chi_0\big(\frac{u}{100\lambda L}\big)\chi_0\big(\frac{v}{100\lambda L}\big)\int_{\Rb} \chi_0(s) \prod_{j=1}^d \left|G_{g^j}(\beta^j  s, \xi^j , u^j) G_{h^j}(-\beta^j  s, \eta^j, v^j)\right|\mathrm{d}u\mathrm{d}v\mathrm{d}s.
	\end{align*}
Now, we use Lemma \ref{hua} (and that $2d\geq 6$) to get the needed bound, namely
\begin{align*}
|B(\xi, \eta)|\lesssim& DL^{-2d+2}\int_{(u, v)\in \Rb_+^{2d}}\chi_0\big(\frac{u}{100\lambda L}\big)\chi_0(\frac{v}{100\lambda L}) \prod_{j=1}^d (u^j)^{1-1/d}(v^j)^{1-1/d} \, \mathrm{d}u\mathrm{d}v\\
\lesssim& D L^{-2d+2} (\lambda L)^{2d+2d-2}\lesssim D\lambda^{4d}L^{2d}.
\end{align*}

(2) To obtain the improved bound in \eqref{NTlem22} for $\widetilde{B}(\xi,\eta)$, we argue a bit differently. Without loss of generality, we can assume $\beta^1=1$ and $\beta^2\in[1,2]$. We start by writing the product of Gauss sums in $\widetilde B(\xi, \eta)$ as
\begin{align*}
\prod_{j=1}^d G_{g^j}(\beta^j sL^{-2}, \xi^j , u^j) G_{h^j}(-\beta^j sL^{-2}, \eta^j, v^j)&=\prod_{j=1}^2G_{g^j}(\beta^jsL^{-2}, \xi^j, u^j) \prod^* G_{g^j}(\cdots) G_{h^j}(\cdots)\\
&=\Mc(s, \xi, \eta, u) \prod^* G_{g^j}(\cdots) G_{h^j}(\cdots)
\end{align*}
where $\Mc(s, \xi, \eta, u)=G_{g^1}(sL^{-2}, \xi^1 , u^1)G_{g^2}(\beta^2sL^{-2}, \xi^2 , u^2)$ and $\prod^*$ is the product of the remaining $2d-2$ Gauss sums. We claim that, for $L^{1-\upsilon} \lesssim |s|\leq L^{2+5\upsilon}$, the following estimate holds for $\Mc$:
\begin{equation}\label{McestimateNT}
\sup_{s, \xi, \eta, u}|\Mc(s, \xi, \eta, u)|\lesssim (\lambda L)^{2-10\upsilon}.
\end{equation}
	
	Before proving this claim, let us see why it implies \eqref{NTlem22}. Using (\ref{McestimateNT}), we have
	\begin{align*}
	| \widetilde{B}(\xi,\eta)|
	\lesssim& D L^{-2d+2}\int_{(u, v)\in \Rb_+^{2d}} \chi_0\big(\frac{u}{100\lambda L}\big)\chi_0\big(\frac{v}{100\lambda L}\big)\bigg((\lambda L)^{2-10\upsilon}\int_{L^{-1-\upsilon} \lesssim |s|\leq L^{5\upsilon}} \prod^* |G_{g^j}(\beta^j s, \xi^j , u^j)|\\\times&| G_{h^j}(-\beta^j s, \eta^j, v^j)|\,\mathrm{d}s+\int_{|s|\geq L^{5\upsilon}} \langle \delta^{-1}s\rangle^{-2}\prod_{j=1}^d \left|G_{g^j}(\beta^j s, \xi^j , u^j) G_{h^j}(-\beta^j s, \eta^j, v^j)\right|\mathrm{d}s\bigg)\mathrm{d}u\mathrm{d}v.
	\end{align*}
Splitting the $s$ region into intervals of length 1, and using Lemma \ref{hua} again on each subinterval we obtain that 
\begin{align*}
	| \widetilde{B}(\xi,\eta)|& \lesssim DL^2 L^{-2d}\int_{(u, v)\in \Rb_+^{2d}} \chi_0\big(\frac{u}{100\lambda L}\big)\chi_0\big(\frac{v}{100\lambda L}\big)
	\\&\qquad\times\bigg[(\lambda L)^{2-10\upsilon}L^{6\upsilon}{(\lambda L)^{2d-4}}+\sum_{|k|\geq L^{5\upsilon}} \langle \delta^{-1}k\rangle^{-2}{(\lambda L)^{2d-2}}\bigg]\mathrm{d}u\mathrm{d}v\lesssim D \lambda^{4d} L^{2d-\upsilon}.
	\end{align*}
	
Thus, it remains to prove \eqref{McestimateNT}. We may assume without loss of generality that $L^{-1-\upsilon}\lesssim sL^{-2}\leq L^{5\upsilon}$ (since $G(-s, r;n)=\overline {G(s,-r; n )}$ and the estimates we shall use for $G$ are independent of $r$). To bound the Gauss sums, we write 
	$$
	sL^{-2}=n+\tau^1, \qquad \beta^2 sL^{-2}=m +\tau^2;\qquad n, m \in \Nb\cup \{0\}, \tau^1, \tau^2 \in [0,1),
	$$
	and use Dirichlet's approximation to find, for $j\in\{1, 2\}$, integers $0\leq a^j < q^j\leq u^j$ such that $(a^j, q^j)=1$ and 
	$$
	\bigg|\tau_j -\frac{a^j}{q^j}\bigg|<\frac{1}{q^j u^j}, \qquad j=1,2.
	$$
	By periodicity of sum $G(s, r,x)$ in $s$ and the Gauss lemma for such sums we have
	\begin{equation}\label{GaussLem}
	\left|G_{g^1}(sL^{-2}, \xi^1,u^1)G_{g^2}(\beta^2sL^{-2}, \xi^2,u^2)\right| \leq \frac{u^1 u^2{\log L}}{\sqrt{q^1q^2}(1+u^1|\tau^1-\frac{a^1}{q^1}|^{1/2})(1+u^2|\tau^2-\frac{a^2}{q^2}|^{1/2})}.
	\end{equation}
	
	We start by dealing with the case $n=0$ (i.e. when $ sL^{-2}=\tau^1<1$). Note that this implies $m \in \{0, 1\}$. Here we will use the fact that $L^{-1-\upsilon}\lesssim \tau^1<1 $. First note that if $a^1=0$, then we have
	$$\left|G_{g^1}(sL^{-2}, \xi^1 , u^1)G_{g^2}(\beta^2sL^{-2}, \xi^2 , u^2)\right| \leq \frac{u^1 u^2{\log L}}{1+u^1|\tau^1|^{1/2}}\lesssim \frac{u^2{\log L}}{|\tau^1|^{1/2}} \lesssim \lambda L L^{\frac{1+\upsilon}{2}}{\log L}\ll \lambda L^{7/4},$$
	which satisfies \eqref{McestimateNT}. Therefore, we now consider the case $a^1\neq 0$. Computing
	\begin{align*}
	q^1 q^2 \big(\beta^2\big|\tau^1 -\frac{a^1}{q^1}\big| +\big|\tau^2 -\frac{a^2}{q^2}\big|\big)=& q^1 q^2 \big(\big|\beta^2 sL^{-2} -\frac{\beta^2a^1}{q^1}\big| +\big|\beta^2 sL^{-2}-m-\frac{a^2}{q^2}\big|\big)\geq q^1 q^2 \big| \frac{\beta^2a^1}{q^1}-m-\frac{a^2}{q^2}\big|\\
	=&|\beta^2 a^1 q^2-(mq^1 q^2+a^2q^1)|\gtrsim \frac{\log^{-4}(2q^1q^2)}{a^1q^2+mq^1q^2+a^2q^1}\gtrsim \frac{1}{(q^1 q^2)^{1.01}},
	\end{align*}
	where we have used the diophantine condition (\ref{generic1}) and the fact that $0<a^j<q^j$ and $m \in \{0,1\}$. Therefore, we obtain
	$$\max\big(\big|\tau^1 -\frac{a^1}{q^1}\big|,\big|\tau^2 -\frac{a^2}{q^2}\big|\big)\gtrsim \frac{1}{(q^1q^2)^{2.01}},
	$$
	which when plugged into \eqref{GaussLem} gives the bound
	\begin{equation}\label{Mcestaux1}
	\lesssim \frac{u^1 u^2{\log L}}{\sqrt{q^1q^2}[1+\min(u^1, u^2)\frac{1}{(q^1 q^2)^{1.01}}]}\lesssim \min\bigg(\frac{u^1 u^2}{\sqrt{q^1q^2}}, \max(u^1, u^2) (q^1 q^2)^{0.51}\bigg){\log L}\lesssim (\lambda L)^{\frac{7}{4}},
	\end{equation}
	since $q^j \leq u^j \lesssim \lambda L$, which is better than \eqref{McestimateNT}. 
	
	It remains to consider the case when $1\leq n \leq L^{5\upsilon}$. Here, we argue similar to the above, to obtain
	\begin{align*}
	q^1 q^2 \big(\beta^2\big|\tau^1 -\frac{a^1}{q^1}\big| +|\tau^2 -\frac{a^2}{q^2}|\big)&= q^1 q^2 \big(\big|\beta^2 sL^{-2}-\beta^2 n -\beta^2\frac{a^1}{q^1}\big| +\big|\beta^2 sL^{-2}-m-\frac{a^2}{q^2}\big|\big)
	\\&\geq q^1 q^2 \big| \beta^2\frac{a^1}{q^1}+\beta^2 n-m-\frac{a^2}{q^2}\big|=|\beta_2 (nq^1 q^2+a^1 q^2)-(mq^1 q^2+a^2q^1)|\\&\gtrsim \frac{\log^{-4}(2nq^1q^2)}{nq^1q^2+a^1q^2+mq^1q^2+a^2q^1}\gtrsim \frac{1}{(nq^1 q^2)^{1.01}}.
	\end{align*}
	Since $n \leq L^{5\upsilon}$, we can repeat the same estimates as above and obtain the needed bound (notice the room in \eqref{Mcestaux1} compared to needed bound in \eqref{McestimateNT}).
\end{proof}
\begin{lem}\label{NTintlem}
Suppose that $\Phi(x,y)$ is a function satisfying \eqref{NTphibounds} without $s$. Let $\Omega(x,y)=\langle x, y\rangle_\beta$.

(1) Suppose {$\mu\geq 1$ and} $\psi$ is a function such that $\|\psi\|_{L^1(\Rb)}\leq D$, then
	\begin{equation}\label{NTintlem1}
	\left|{\mu}\int_{\Rb^{2d}}\psi(\mu \Omega) \chi_0\big(\frac{x-a}{\lambda}\big)\chi_0\big(\frac{y-b}{\lambda}\big)\Phi(x,y) e(x\cdot\xi+y\cdot \eta) \, \mathrm{d}x\mathrm{d}y\right|\lesssim D \lambda^{2d}
	\end{equation}
	uniformly in $(a, b, \xi, \eta)\in \Rb^{4d}$. The same holds if $\psi(\mu\Omega)\Phi(x,y)$ is replaced by $\Psi(\mu\Omega,x,y)$ where $\Psi=\Psi(u,x,y)$ satisfies $\big\|\sup_{x,y}|\partial_x^\alpha\partial_y^\beta\Psi|\big\|_{L_u^1}\leq D$ for all multi-indices $|\alpha|,|\beta|\leq 10d$.
	
(2) Suppose further that $\| \langle y\rangle^{\frac{1}{8}} \psi\|_{L^1(\Rb)}\leq D$, then 
	\begin{align}
	&\left|{\mu}\int_{\Rb^{2d}}\psi(\mu \Omega) \chi_0\big(\frac{x-a}{\lambda}\big)\chi_0\big(\frac{y-b}{\lambda}\big) \Phi(x,y)e(x\cdot\xi+y\cdot \eta) \, \mathrm{d}x\mathrm{d}y-\right.\label{NTintlem2}\\
	&\qquad \qquad \left.\left(\int \psi\right)\int_{\Rb^{2d}}\dirac(\Omega) \chi_0\big(\frac{x-a}{\lambda}\big)\chi_0\big(\frac{y-b}{\lambda}\big)\Phi(x,y) e(x\cdot\xi+y\cdot \eta) \,\mathrm{d}x\mathrm{d}y\right|\lesssim D \lambda^{2d}\mu^{-\frac{1}{9}}(1+|\xi|+|\eta|),\nonumber
	\end{align}
		uniformly in $(a, b)\in \Rb^{2d}$.
		\end{lem}
\begin{proof} Using a smooth partition of unity, it is enough to consider the region when $|x^1|\sim \max(|x|, |y|)$ (other regions are treated symmetrically). In this case, we do the same change of variables as in the proof of Lemma \ref{NTOmegainv}, replacing the variable $y^1$ by $u=\Omega$ to write the corresponding integral in \eqref{NTintlem1} as
 	\begin{align*}
 	&\int_{\Rb^d} |\beta^1x^1|^{-1}\chi_0\big(\frac{x-a}{\lambda}\big) \chi_0\big(\frac{x'}{|x^1|}\big)e(x\cdot \xi) \int_{\Rb^{d-1}}\chi_0\big(\frac{y'-b'}{\lambda}\big)\chi_0\big(\frac{y'}{|x^1|}\big)e(y'\cdot \eta')\,\mathrm{d}y'\mathrm{d}x\\
 	&\qquad\qquad\qquad\qquad \int_{\Rb} {\mu}\psi(\mu u) \chi_0\bigg(\frac{u-\langle x', y'\rangle_{\beta'} -\beta^1 x^1 b^1}{\beta^1 x^1 \lambda}\bigg)\widetilde\Phi(x, u, y')e\bigg(\frac{u\eta^1 -\langle x', y'\rangle_{\beta'}\eta^1}{\beta^1 x^1}\bigg)\mathrm{d}u,
 	\end{align*}
 	where $\widetilde\Phi(x, u, y')=\Phi(x, y^1, y')\chi_0(\frac{y}{|x^1|})$ evaluated at $y^1=\frac{u-\langle x', y'\rangle_{\beta'}}{\beta^1 x^1}$. 
 	From this one can directly obtain \eqref{NTintlem1}, as well as the extension with $\psi\cdot\Phi$ replaced by $\Psi$. To obtain \eqref{NTintlem2}, we look at the difference  
 		\begin{align*}
&\int_{\Rb} {\mu}\psi(\mu u) \chi_0\bigg(\frac{u-\langle x', y'\rangle_{\beta'} -\beta^1 x^1 b^1}{\beta^1 x^1 \lambda}\bigg)\widetilde\Phi(x, u, y')e\bigg(\frac{u\eta^1 -\langle x', y'\rangle_{\beta'}\eta^1}{\beta^1 x^1}\bigg)\mathrm{d}u\\
&\qquad\qquad\qquad\qquad\qquad- \int_{\Rb} {\mu}\psi(\mu u) \chi_0\bigg(\frac{-\langle x', y'\rangle_{\beta'} -\beta^1 x^1 b^1}{\beta^1 x^1 \lambda}\bigg)\widetilde\Phi(x, 0, y')e\big(\frac{-\langle x', y'\rangle_{\beta'}\eta^1}{\beta^1 x^1}\big)\mathrm{d}u,
\end{align*}
which can be written as $\int_{\Rb} {\mu}\psi(\mu u) \widetilde \Delta(u;x, y')\mathrm{d}u$, where
\begin{align*}
\widetilde \Delta(u;x,y')&=\chi_0\bigg(\frac{u-\langle x', y'\rangle_{\beta'} -\beta^1 x^1 b^1}{\beta^1 x^1 \lambda}\bigg)\widetilde\Phi(x, u, y')e\bigg(\frac{u\eta^1 -\langle x', y'\rangle_{\beta'}\eta^1}{\beta^1 x^1}\bigg)\\
 &\qquad \qquad\qquad\qquad-\chi_0\bigg(\frac{-\langle x', y'\rangle_{\beta'} -\beta^1 x^1 b^1}{\beta^1 x^1 \lambda}\bigg)\widetilde\Phi(x, 0, y')e\big(\frac{-\langle x', y'\rangle_{\beta'}\eta^1}{\beta^1 x^1}\big).
 	\end{align*}
It is easy to see that $|\widetilde \Delta|\lesssim D\cdot \min\big[1, u\cdot\big(\frac{|\eta^1|}{\beta^1 |x^1|}+\frac{1}{\beta^1|x^1|\min(1, |x^1|)}\big)\big]$. Using this, we split the $u$ integral into two regions. If $|u|\leq \mu^{-1/9}$, we can use the second of the two bounds on $\widetilde \Delta$ to obtain a contribution $D\mu^{-1/9}|\eta^1|\lambda^{2d}$ to \eqref{NTintlem2}; if $u\geq \mu^{-1/9}$, we can use the first bound and the weighted norm $\|\langle y\rangle^{1/8}\psi\|_{L^1}$ to obtain a contribution $\mu^{-1/9}D\lambda^{2d}$ to \eqref{NTintlem2}. This finishes the proof.
\end{proof}
With the help of Lemmas \ref{NTSP}--\ref{NTintlem}, we can now prove Proposition \ref{approxnt}.
\begin{proof}[Proof of Proposition \ref{approxnt}]
We start with some simplifying notation and reductions. Set $\vx:=(x_1, \ldots, x_n)\in\Rb^{dn}$, $\vy:=(y_1, \ldots, y_n)\in \Rb^{dn}$, $\vs:=(s_1, \ldots, s_n)\in \Rb^n$, $\vOmega=(\Omega_1, \ldots, \Omega_n)\in \Rb^n$ where $\Omega_j=\langle x_j, y_j \rangle_\beta$, and $\mu=L^2\delta$. We will also use the notations like $\vx_{\leq j}:=(x_1, \ldots, x_j)\in\Rb^{dj}$ and similarly for the other variables $\vy$ and $\vOmega$. 

Write $W(x,y)=W(\vx,\vy)\Upsilon(\vx, \vy)$ where $\Upsilon$ is a smooth function supported in the set described in \eqref{propertyw2}, namely
\begin{equation}\label{NTdefofUpsilon}
\Upsilon(\vx, \vy)=\prod_{j=1}^n \chi_0\big(\frac{\widetilde{x_j}-a_j}{\lambda_j}\big)\chi_0\big(\frac{\widetilde{y_j}-b_j}{\lambda_j}\big)
\end{equation}
where $\widetilde{x_j}$ and $\widetilde{y_j}$ are as in \eqref{propertyw2}, and we may use a different $\chi_0$ as said in Section \ref{notations}. Also note that we can assume (by rearranging the indices) that if $j'\prec j$ then $j'< j$. 

Set $K:=(10d\max{\lambda_j})^{-1}L^{1-\upsilon}$, and write (with $\widehat{\Psi}$ being the Fourier transform on $\Rb^n$)
\begin{align*}
S&=\sum_{(\vx, \vy)\in \Zb_L^{2dn}} W(\vx, \vy)\Psi(\mu \vOmega)=\int_{\Rb^n}\mu^{-n} \widehat\Psi(\mu^{-1}\vs)\bigg[\sum_{(\vx, \vy)\in \Zb_L^{2dn}} W(\vx, \vy)e(\vs\cdot \vOmega)\bigg]\, \mathrm{d}\vs\\
&=\int_{\Rb^n}\mu^{-n} \widehat\Psi(\mu^{-1}\vs)\prod_{j=1}^n\chi_0\big(\frac{s_j}{K}\big)\left[\ldots \right] \, d\vs +\int_{\Rb^n}\mu^{-n} \widehat\Psi(\mu^{-1}\vs)\bigg(1-\prod_{j=1}^n\chi_0\big(\frac{s_j}{K}\big)\bigg)\left[\ldots\right]  \, \mathrm{d}\vs\\
&=:I_{\textrm{major}}+I_{\textrm{minor}}.
\end{align*}

\medskip

\noindent{ $\bullet$ \bf Major arc contribution:} By Poisson summation, there holds
\begin{align*}
I_{\textrm{major}}&=\int_{\Rb^n}\mu^{-n} \widehat\Psi(\mu^{-1}\vs)\prod_{j=1}^n\chi_0\big(\frac{s_j}{K}\big)\bigg[\sum_{(\vg, \vh)\in \Zb^{2dn}}\int_{\Rb^{2dn}} W\big(\frac{\vx}{L}, \frac{\vy}{L}\big)e(\vg \cdot \vx +\vh \cdot \vy+L^{-2}\vs \cdot \vOmega)\,\mathrm{d}\vx \mathrm{d}\vy\bigg]\mathrm{d}\vs\\
&=L^{2dn}\mu^{-n}\int_{\Rb^n} \widehat\Psi(\mu^{-1}\vs)\prod_{j=1}^n\chi_0\big(\frac{s_j}{K}\big)\bigg[\sum_{(\vg, \vh)\in \Zb^{2dn}}\int_{\Rb^{2dn}} W({\vx}, {\vy})e(L\vg \cdot \vx +L\vh \cdot \vy+\vs\cdot \Omega)\,\mathrm{d}\vx \mathrm{d}\vy\bigg] \mathrm{d}\vs\\
&=L^{2dn}\mu^{-n}\left(\int_{\Rb^n} \widehat\Psi(\mu^{-1}\vs)\bigg[\int_{\Rb^{2dn}} W({\vx}, {\vy})e(\vs\cdot \Omega)\,\mathrm{d}\vx \mathrm{d}\vy\right]\mathrm{d}\vs\\
&\qquad -\int_{\Rb^n} \widehat\Psi(\mu^{-1}\vs)\bigg(1-\prod_{j=1}^n\chi_0\big(\frac{s_j}{K}\big)\bigg)\left[\int_{\Rb^{2dn}} W({\vx}, {\vy})e(\vs\cdot \Omega)\,\mathrm{d}\vx\mathrm{d}\vy\right]\mathrm{d}\vs\\
&\qquad+\int_{\Rb^n} \widehat\Psi(\mu^{-1}\vs)\prod_{j=1}^n\chi_0\big(\frac{s_j}{K}\big)\bigg[\sum_{0\neq(\vg, \vh)\in \Zb^{2dn}}\int_{\Rb^{2dn}} W({\vx}, {\vy})e(L\vg \cdot \vx +L\vh \cdot \vy+\vs\cdot \Omega)\,\mathrm{d}\vx\mathrm{d}\vy\bigg]\mathrm{d}\vs\bigg)\\
&=: I_{\textrm{major-A}}+I_{\textrm{major-B}}+I_{\textrm{major-C}}.
\end{align*}
Noticing that $I_{\textrm{major-A}}$ is nothing but the integral in \eqref{conclusion1}, it remains to show that $I_{\textrm{major-B}},I_{\textrm{major-C}}$ and $I_{\textrm{minor}}$ can all be bounded by the right hand side of \eqref{conclusion1}.

To bound $I_{\textrm{major-B}}$, we use the following bound:

\begin{equation}\label{stphasebd}
\bigg|\int_{\Rb^{2dn}} W(\vx, \vy) e(\vs\cdot \Omega) \,\mathrm{d}\vx \mathrm{d}\vy\bigg|\leq \frac{C^n(\lambda_1\ldots \lambda_n)^{2d}}{\langle s_1\rangle^d\ldots \langle s_n\rangle^d}\|\widehat W\|_{L^1(\Rb^{2dn})}.
\end{equation} This bound is obtained by writing 
$$
\int_{\Rb^{2dn}} W(\vx, \vy) e(\vs\cdot \Omega) \,\mathrm{d}\vx \mathrm{d}\vy=\int_{\Rb^{2dn}}\widehat W(\vxi, \veta) \left[\int_{\Rb^{2dn}} \Upsilon(\vx, \vy)e(\vxi\cdot \vx +\veta \cdot \vy+\vs\cdot \Omega) \,\mathrm{d}\vx \mathrm{d}\vy\right] \mathrm{d}\vxi\mathrm{d}\veta,
$$
and applying stationary phase (when $|s_j|\geq 1$) in the inner integral (or using the Fourier transform of the Gaussian since the phase is essentially the difference of two Gaussians). Using this bound, we can estimate 
\begin{align*}
I_{\textrm{major-B}}&\leq (C^+)^n (\lambda_1\ldots \lambda_n)^{2d}L^{2dn}\mu^{-n} \sum_{j=1}^n \int_{|s_j|\geq K} \langle s_j \rangle^{-d} \,\mathrm{d}s_j \prod_{k\neq j }\int_{\Rb} \langle s_k \rangle^{-d}\,\mathrm{d}s_k\\
&\leq (C^+)^n (\lambda_1\ldots \lambda_n)^{2d} L^{2dn}\mu^{-n} K^{-(d-1)}\lesssim (C^+)^n (\lambda_1\ldots \lambda_n)^d L^{2dn}\mu^{-n} L^{-\upsilon}.
\end{align*}

Moving to $I_{\textrm{major-C}}$, we write 
\begin{align*}
I_{\textrm{major-C}}=&L^{2dn}\mu^{-n} \int_{\Rb^{2dn}} \widehat W(\vxi, \veta) H(\vxi, \veta)  d\vxi d\veta,\\
H(\vxi, \veta):=&\sum_{0\neq(\vg, \vh)\in \Zb^{2dn}} \int_{\Rb^n} \widehat\Psi(\mu^{-1}\vs)\prod_{j=1}^n\chi_0\big(\frac{s_j}{K}\big)\int_{\Rb^{2dn}} \Upsilon({\vx}, {\vy})e[(L\vg+\vxi) \cdot \vx +(L\vh +\veta)\cdot \vy+\vs\cdot \vOmega] \,\mathrm{d}\vx \mathrm{d}\vy\mathrm{d}\vs.
\end{align*}

Recalling the form of $\Psi$ in \eqref{propertypsi1}, it will be enough to show that 
\begin{equation}\label{HvetaEst}
|H(\vxi, \veta)|\leq (C^+)^n(1+|\xi|+|\eta|)(\lambda_1\cdots\lambda_n)^C L^{-\frac{4}{10}\upsilon}\|\Psi_1\|_{L^1}.
\end{equation}

For each $j\notin J$, we use a partition of unity of $\Rb^{n-|J|}$ subordinate to cubes of size 1 in order to write:
$$
\Psi_1(\Omega[J^c])=\sum_{\kappa \in \Zb^{n-|J|}} \Psi_1^{(\kappa)}(\Omega[J^c]).
$$
where each $\Psi_1^{(\kappa)}$ is supported in a unit cube of $\Rb^{n-|J|}$. Since $\Psi_1\in L^1(\Rb^{n-|J|})$, it is enough to obtain the bound \eqref{HvetaEst} with $\Psi^1$ replaced by $\Psi_1^{(\kappa)}$. In what follows, we will omit the superscript $(\kappa)$ and just assume that $\Psi_1$ is supported on a unit cube of $\Rb^{n-|J|}$.

\medskip

Since the sum is over $(\vg, \vh)\neq 0$, let $1\leq \ell \leq n$ be the largest integer such that $(g_\ell, h_\ell)\neq 0$. It is enough to estimate the contribution for each fixed $1\leq \ell\leq n$ since polynomial losses in \eqref{HvetaEst} can be absorbed by modifying the $(C^+)^n$ factor.  As such, by abusing notation, we may assume in the definition of $H(\vxi, \veta)$ above that for some fixed $1\leq \ell \leq n$, the sum in $H(\vxi, \veta)$ is over $(\vg, \vh)_{<\ell} \in \Zb^{2d(\ell-1)}$, $(g_\ell, h_\ell)\in \Zb^{2d}\setminus \{0\}$, and $(\vg, \vh)_{>\ell}=0$. Hence,
\begin{align*}
H(\vxi, \veta):=&\sum_{(\vg, \vh)_{<\ell}\in (\Zb^{2d})^{\ell-1}} \int_{\Rb^{\ell-1}}\mathrm{d}\vs_{<\ell} \prod_{j=1}^{\ell-1}\chi_0\big(\frac{s_j}{K}\big)\int_{\Rb^{2d(\ell-1)}}\mathrm{d}\vx_{<\ell} \mathrm{d}\vy_{<\ell} \Upsilon_{<\ell}({\vx}_{<\ell}, {\vy}_{<\ell})\\
&\quad F_{<\ell}(\vs_{<\ell}, \vx_{<\ell}, \vy_{<\ell})\cdot e[(L\vg_{<\ell}+\vxi_{<\ell}) \cdot \vx_{<\ell} +(L\vh_{<\ell} +\veta_{<\ell})\cdot \vy_{<\ell}+\vs_{<\ell}\cdot \vOmega_{<\ell}],
\end{align*}
where 
\begin{equation}\label{NTlemUpsilonell}
\Upsilon_{<\ell}(\vx_{<\ell}, \vy_{<\ell})=\prod_{j=1}^{\ell-1} \chi_0\big(\frac{\widetilde{x_j}-a_j}{\lambda_j}\big)\chi_0\big(\frac{\widetilde{y_j}-b_j}{\lambda_j}\big),
\end{equation}
\begin{align*}
F_{<\ell}(\vs_{<\ell}, \vx_{<\ell}, \vy_{<\ell}):=&\sum_{0\neq(g_\ell, h_\ell)\in \Zb^{2d}} \int_{\Rb}\mathrm{d}s_\ell \, \int_{\Rb^{2d}}\mathrm{d}x_\ell \mathrm{d}y_\ell \cdot\chi_0\big(\frac{\widetilde{x_\ell}-a_\ell}{\lambda_\ell}\big)\chi_0\big(\frac{\widetilde{y_\ell}-b_\ell}{\lambda_\ell}\big)\chi_0\big(\frac{s_\ell}{K}\big)G_{\leq \ell}(\vs_{\leq \ell}, \vx_{\leq \ell}, \vy_{\leq \ell})\\
&\qquad e[(Lc_\ell+\xi_\ell) \cdot x_\ell +(Ld_\ell +\eta_\ell)\cdot y_\ell+s_\ell\cdot \Omega_\ell],
\end{align*}
and
\begin{align*}
G_{\leq \ell}(\vs_{\leq \ell}, \vx_{\leq \ell}, \vy_{\leq \ell}):=&\int_{\Rb^{n-\ell}}\mathrm{d}\vs_{>\ell} \prod_{j=\ell+1}^{n}\chi_0\big(\frac{s_j}{K}\big)\widehat\Psi(\mu^{-1}\vs)\int_{\Rb^{2d(n-\ell)}}\mathrm{d}\vx_{>\ell} \mathrm{d}\vy_{>\ell} \prod_{j=\ell+1}^{n} \chi_0\big(\frac{\widetilde{x_j}-a_j}{\lambda_j}\big)\chi_0\big(\frac{\widetilde{y_j}-b_j}{\lambda_j}\big)\\
&\quad e(\vxi_{>\ell}\cdot \vx_{>\ell} +\veta_{>\ell}\cdot \vy_{>\ell}+\vs_{>\ell}\cdot \vOmega_{>\ell}).
\end{align*}

Notice that $G_{\leq \ell}(\vs_{\leq \ell}, \vx_{\leq \ell}, \vy_{\leq \ell})$ only depends on the variables 
$(\vx_{\leq \ell}, \vy_{\leq \ell})$ through the possible occurrences of these variables in $\widetilde{x_j}$ and $\widetilde{y_j}$ in the $\chi_0$ factors when $j \geq \ell+1$.

We start by bounding $G_{\leq \ell}$ by applying Lemma \ref{NTSP} $(n-\ell)$ times starting with the last integration variables $(s_n, x_n , y_n)$. Indeed, by induction, one can show that after integrating in $(s_{k+1}, x_{k+1}, y_{k+1})$ for some $\ell\leq k \leq n-1$, we end up with an expression of the form given in \eqref{NTSPest0}, with $a\in a_k+\{0,\pm x_{k'},\pm y_{k'}\}$ and $b\in b_k+\{0,\pm x_{k''},\pm y_{k''}\}$ (cf. \eqref{propertyw2}) and $\Phi(s_k, x_k, y_k)$ given by 
\begin{align*}
\Phi&=\chi_0\big(\frac{s_k}{K}\big)\int_{\Rb^{n-k}}\mathrm{d}s_{>k}\int_{\Rb^{2d(n-k)}}\prod_{j=k+1}^{n}\chi_0\big(\frac{s_j}{K}\big)\widehat\Psi(\mu^{-1}\vs)\int_{\Rb^{2d(n-\ell)}}\mathrm{d}\vx_{>k} \mathrm{d}\vy_{>k} \prod_{j=k+1}^{n} \chi_0\big(\frac{\widetilde{x_j}-a_j}{\lambda_j}\big)\chi_0\big(\frac{\widetilde{y_j}-b_j}{\lambda_j}\big)\\
&\qquad \qquad \qquad \qquad e(\vxi_{>k}\cdot \vx_{>k} +\veta_{>k}\cdot \vy_{>k}+\vs_{>k}\cdot \vOmega_{>k})
\end{align*}
which satisfies the bound in \eqref{NTphibounds} with 
$$
D\leq (C^+)^{n-k}(\lambda_{k+1}\ldots \lambda_n)^{2d}\|\Psi_1\|_{L^1},
$$ uniformly in the parameters $(\vs_{<k},\vx_{<k},\vy_{<k})$ and together with all derivatives in the parameters $(\vx_{<k},\vy_{<k})$. Note that, if we differentiate $\Phi$ in $x_k$ and $y_k$ at most $10d$ times in (\ref{NTphibounds}), these derivatives may fall on some of the $\chi_0$ factors; however even if we do this at every step of induction, each single $\chi_0$ factor will be differentiated at most $20d$ times in total, because $\widetilde x_j$ (and similarly $\widetilde{y_j}$) depends only on $x_j$ and at most one other variable.

This gives that $\chi_0(\frac{s_\ell}{K})G_{\leq \ell}(\vs_{\leq \ell}, \vx_{\leq \ell}, \vy_{\leq \ell})$ satisfies the conditions of part (2) of Lemma \ref{NTSP} in the $(s_\ell, x_\ell, y_\ell)$ integration (with condition 2(a) holding if $\ell\in J$ and 2(b) if $\ell\notin J$). Thus, for any multi-indices $\valpha_{<\ell},\vbeta_{<\ell}$ satisfying $|\alpha_k|,|\beta_k|\leq 10d$ for each $1\leq k<\ell$, it holds that
$$
\sup_{\vs_{<\ell}} \left|\partial_{\vx_{<\ell}}^{\valpha_{<\ell}} \partial_{\vy_{<\ell}}^{\vbeta_{<\ell}} F_{<\ell}(\vs_{<\ell}, \vx_{<\ell}, \vy_{<\ell})\right|\leq (1+|\xi_\ell|+|\eta_\ell|)L^{-\frac{4\upsilon}{10}}(C^+)^{n-\ell+1}(\lambda_{\ell}\ldots \lambda_n)^{2d}\|\Psi_1\|_{L^1}
$$
which allows us to start applying estimate \eqref{NTSPest2} inductively starting with the $(s_{\ell-1},x_{\ell-1},y_{\ell-1})$ integral all the way to the integral over $(s_1, x_1, y_1)$ giving the desired bound in \eqref{HvetaEst}.

\medskip

\noindent{ $\bullet$ \bf Minor arc contribution:} Now we move to bound the contribution of the minor arc. This can be written as a sum of $2^n-1$ terms of the following form: For any set $F\subset \{1, \ldots, n\}$ such that $|F|=f\geq 1$, we consider
\begin{align*}
I_{\textrm{minor}}^F:=&\mu^{-n}\int_{\Rb^n}\prod_{j \in F}\chi_\infty\big(\frac{s_j}{K}\big)\prod_{j\notin F } \chi_0\big(\frac{s_j}{K}\big) \widehat\Psi(\mu^{-1}\vs)\bigg[\sum_{(\vx, \vy)\in \Zb_L^{2dn}} W(\vx, \vy)e(\vs\cdot \vOmega)\bigg] \mathrm{d}\vs\\
&=\mu^{-n} \int_{\Rb^{2dn}} \widehat{W}(\vxi, \veta) B(\vxi, \veta)\,\mathrm{d}\vxi \mathrm{d}\veta\\
B(\vxi, \veta):=& \int_{\Rb^n}\prod_{j \in F}\chi_\infty\big(\frac{s_j}{K}\big)\prod_{j\notin F } \chi_0\big(\frac{s_j}{K}\big)\widehat\Psi(\mu^{-1}\vs)\bigg[\sum_{(\vx, \vy)\in \Zb_L^{2dn}} \Upsilon(\vx, \vy)e(\vs\cdot \vOmega+\vx\cdot \vxi+\vy\cdot \veta)\bigg]\mathrm{d}\vs,
\end{align*}
where $\Upsilon$ is as defined in \eqref{NTdefofUpsilon}. We shall show that: 
\begin{equation}\label{NTminorest}
|B(\vxi, \veta)|\lesssim (C^+)^n L^{2dn-\upsilon f}(\lambda_1\ldots \lambda_n)^C \bigg\|\prod_{j\in F}\langle \sigma_j \rangle^2 \widehat \Psi(\vsigma)\bigg\|_{L^\infty},
\end{equation}
uniformly in $\vxi$, $\veta$, $a_j$, and $b_j$ ($1\leq j \leq n$). Once this estimate is established, we use the bound \[\bigg\|\prod_{j\in F}\langle \sigma_j \rangle^2 \widehat \Psi(\vsigma)\bigg\|_{L^\infty}\leq C^n (8f)!\] by \eqref{propertypsi1}--\eqref{propertypsi2} and that $1\leq f \leq n\leq (\log L)^3$, to conclude 
that $|B(\vxi, \veta)|$ can be bounded by the right hand side of \eqref{conclusion1} as needed. 

To prove \eqref{NTminorest}, we can bound \[\widehat \Psi(\mu^{-1}\vs)\leq \prod_{j\in F} \big\langle \frac{s_j}{\delta L^2}\big\rangle^{-2}\cdot\bigg\|\prod_{j\in F}\langle \sigma_j \rangle^2 \widehat \Psi(\vsigma)\bigg\|_{L^\infty}.\] Afterwards, we apply $n$ times Lemma \ref{lem:minorarcs}, going backwards in $n$, using part (1) for $j \notin F$ and part (2) for $j \in F$. Each application gives a factor of $\lambda_j^{4d}L^{2d}$ for $j \notin F$ and $\lambda_j^{4d}L^{2d-\upsilon}$ for $j\in F$, which gives (\ref{NTminorest}) and finishes the proof of \eqref{conclusion1}. 

\medskip

\noindent{ $\bullet$ \bf Deducing \eqref{conclusion2} from \eqref{conclusion1}:} We again start by writing 
\begin{align}
\mu^n\int_{\Rb^{2dn}} W(\vx, \vy)\Psi(\mu \vOmega)\mathrm{d}\vx\mathrm{d}\vy=\int_{\Rb^{2dn}} \widehat W(\vxi, \veta)\bigg[\mu^n\int_{\Rb^{2dn}}\Psi(\mu \vOmega)\Upsilon(\vx, \vy) e(\vx\cdot \vxi+\vy\cdot \veta)d\vx d\vy\bigg]\mathrm{d}\vxi\mathrm{d}\veta, \label{NTfinalint}
\end{align}
where $\Upsilon$ is defined in \eqref{NTdefofUpsilon} and recall that we have rearranged indices so if $j'\prec j$ then $j'< j$. Next, we start applying part (2) of either Lemma \ref{NTOmegainv} or Lemma \ref{NTintlem} for the $\mathrm{d}x_j\mathrm{d}y_j$ integral (depending on whether $j \in J$ or not) backwards in $n$ starting with the $\mathrm{d}x_n\mathrm{d}y_n$ integral. At the first application, we replace either $\frac{\chi_\infty(\mu \Omega_n)}{\Omega_n}$ by $\mathrm{p.v.}\frac{1}{\Omega_n}$ or $\mu \Psi(\mu\vOmega_{<n}, \mu\Omega_n)$ by $\left[\int_{\Rb}\Psi(\mu\vOmega_{<n},\omega_n) d\omega_n\right]\cdot\dirac(\langle x_n, y_n\rangle_\beta)$ plus an additive error term that can be bound by $\lambda^{2d}L^{-\frac{1}{6}}(1+|\xi_n|+|\eta_n|)$ uniformly in $(a_n, b_n, \vx_{<n}, \vy_{<n}, \vxi_{<n}, \veta_{<n})$. The contribution of this additive error term to \eqref{NTfinalint} can be bounded by repeatedly applying part (1) of either Lemma \ref{NTOmegainv} or Lemma \ref{NTintlem} (using the more general form of part (1) of Lemma \ref{NTintlem} if needed) and gives a total contribution 
\begin{equation}\label{NTadditiveerror}
\leq (C^+)^n(\lambda_1 \ldots \lambda_n)^{2d} L^{-\frac{1}{6}}\int_{\Rb^{2dn}}|\widehat W(\vxi, \veta)|(1+|\xi_n|+|\eta_n|) \mathrm{d}\vxi \mathrm{d}\veta\leq (C^+)^n(\lambda_1 \ldots \lambda_n)^{2d} L^{-\frac{1}{6}}\left(\|\widehat W\|_{L^1}+\|\widehat{\partial W}\|_{L^1}\right).
\end{equation}

This leaves us only with the main part contribution which corresponds to replacing \eqref{NTfinalint} by
\begin{align*}
&\int_{\Rb^{2dn}} \widehat W(\vxi, \veta)\bigg[\mu^{n-1}\int_{\Rb^{2d(n-1)}}\Psi_{<n}(\mu \vOmega_{<n})\Upsilon_{<n}(\vx, \vy)\Gamma_{n-1}(\vx_{<n}, \vy_{<n}, \xi_n, \eta_n) \\
&\qquad \qquad \qquad \qquad \qquad \qquad \qquad \qquad e(\vx_{<n}\cdot \vxi_{<n}+\vy_{<n}\cdot \veta_{<n})\,\mathrm{d}\vx_{<n}\mathrm{d}\vy_{<n}\bigg]\mathrm{d}\vxi\mathrm{d}\veta,
\end{align*} where $\Upsilon_{<n}$ is as in \eqref{NTlemUpsilonell}, 
\[
\Psi_{<n}(\Omega_1, \ldots, \Omega_{n-1})=
\left\{
\begin{aligned}
&\int_\Rb \Psi(\Omega_1, \ldots, \Omega_{n-1}, \omega_n)\,\mathrm{d}\omega_n, &\mathrm{if\ }n \notin J,\\
&\prod_{j\in J\setminus\{n\}}\frac{\chi_\infty(\Omega_j)}{\Omega_j}\cdot\Psi_1(\Omega[J^c]), &\mathrm{if\ }n \in J,
\end{aligned}
\right.
\] and
\[
\Gamma_{n-1}=
\left\{\begin{aligned}
&\int_{\Rb^d\times \Rb^d} \dirac(\langle x_n, y_n\rangle_\beta)\chi_0\big(\frac{\widetilde{x_n}-a_n}{\lambda_n}\big)\chi_0\big(\frac{\widetilde{y_n}-b_n}{\lambda_n}\big) e(x_n \cdot \xi_n+y_n \cdot \eta_n )\,\mathrm{d}x_n\mathrm{d}y_n, &\mathrm{if\ }n \notin J,\\
\mathrm{p.v.}&\int_{\Rb^d\times \Rb^d}\frac{1}{\langle x_n, y_n\rangle_\beta}\chi_0\big(\frac{\widetilde{x_n}-a_n}{\lambda_n}\big)\chi_0\big(\frac{\widetilde{y_n}-b_n}{\lambda_n}\big) e(x_n \cdot \xi_n+y_n \cdot \eta_n )\, \mathrm{d}x_n \mathrm{d}y_n, &\mathrm{if\ } n \in J.
\end{aligned}
\right.
\]This allows to repeat the above argument $n-1$ times, each time producing an additive error term bounded by \eqref{NTadditiveerror}, until finally \eqref{NTfinalint} is replaced by $I_{\mathrm{app}}$ in \eqref{intappr}. This allows us to bound $I-I_{\mathrm{app}}$ as in (\ref{conclusion2}), but the bound of $I_{\mathrm{app}}$ follows from the same arguments, so the proof of Proposition \ref{approxnt} is complete.
\end{proof}
\subsection{Asymptotics of $\Kc_{\Qc}$ for regular couples $\Qc$}\label{applyKQ} Using Proposition \ref{maincoef} and Proposition \ref{approxnt}, we can calculate the leading term in the asymptotic expression for the correlation $\Kc_\Qc(t,s,k)$ defined in (\ref{defkq}), as well as upper bounds for the error term.
\begin{lem}\label{auxlem} Let $\Tc$ be a tree of scale $n$. For any node $\nf\in\Tc$ define $\mu_\nf$ to be the number of leaves in the subtree rooted at $\nf$. Then, for any $\nf\in\Nc$, consider the values of $\mu_\mf$ where $\mf$ is a child of $\nf$, and let the \emph{second maximum} of these values be $\mu_\nf^\circ$. Then we have
\begin{equation}\label{auxineq}\prod_{\nf\in\Nc}\mu_\nf^\circ\leq \frac{3^n}{2n+1}.\end{equation}
\end{lem}
\begin{proof} We prove by induction. If $n=0$ the result is obvious. Suppose the result holds for smaller $n$, for any tree $\Tc$, let the subtrees be $\Tc_1$, $\Tc_2$ and $\Tc_3$ from left to right, with scale $n_1$, $n_2$ and $n_3$. If the root of $\Tc$ is $\rf$ and root of $\Tc_j$ is $\rf_j$, then by induction hypothesis we know that
\begin{equation}\prod_{\nf\in\Nc}\mu_\nf^\circ=\mu_\rf^\circ\cdot\prod_{j=1}^3\prod_{\nf\in\Nc_j}\mu_\nf^\circ\leq \frac{3^{n_1+n_2+n_3}}{(2n_1+1)(2n_2+1)(2n_3+1)}\mu_\rf^\circ\leq\frac{3^n}{2n+1}.\end{equation} In the last inequality we have used that $n=n_1+n_2+n_3+1$, which also implies $2n+1\leq 3\cdot\max(2n_1+1,2n_2+1,2n_3+1)$, and that $\mu_\rf^\circ$ equals the second maximum of $2n_j+1\,(1\leq j\leq 3)$. This completes the proof.
\end{proof}
\begin{prop}\label{asymptotics1} Let $\Qc$ be a regular couple of scale $2n$ where $n\leq N^3$, then we have $\Kc_\Qc(t,s,k)=\sum_Z\Kc_{\Qc,Z}(t,s,k)$, where $Z\subset\Nc^{ch}$ is the set that appears in Proposition \ref{maincoef}, and
\[\Kc_{\Qc,Z}(t,s,k)=2^{-2n}\delta^n\zeta^*(\Qc)\prod_{\nf\in Z}\frac{1}{\zeta_\nf \pi i}\cdot\int\widetilde{\Bc}_{\Qc,Z}\big(t,s,\alpha[\Nc^{ch}\backslash Z]\big)\,\mathrm{d}\alpha[\Nc^{ch}\backslash Z]\cdot\Mc_{\Qc,Z}^*(k)+\Rs,\] where the error term $\Rs$ satisfies $\|\Rs\|_{X_{\mathrm{loc}}^{40d}}\lesssim (C^+\delta)^n L^{-2\nu}$. The expression $\Mc_{\Qc,Z}^*(k)$ is defined by 
\begin{equation}\label{defintegral}\Mc_{\Qc,Z}^*(k)=\int_{\Sigma}{\prod_{\lf\in\Lc^*}^{(+)}n_{\mathrm{in}}(k_\lf)}\cdot\prod_{\nf\in\Nc^{ch}\backslash Z}\dirac(\Omega_{\nf})\prod_{\nf\in Z}\frac{1}{\Omega_{\nf}}\,\mathrm{d}\sigma.\end{equation} Here $k_\nf\in\Rb^d$ for each node $\nf$, and $\Sigma$ denotes the linear submanifold defined by the equations $k_{\rf^\pm}=k$ and $k_\nf=k_{\nf_1}-k_{\nf_2}+k_{\nf_3}$ for each branching node $\nf$ (where $\nf_1$, $\nf_2$ and $\nf_3$ are children nodes of $\nf$ from left to right), and $k_\lf=k_{\lf'}$ for each pair of leaves $\{\lf,\lf'\}$. {If we choose all the leaves of sign $+$ and list them as $\lf_1,\cdots,\lf_{2n+1}$}, then there is a linear bijection (up to a permutation of indices) from $\Sigma$ to some hyperplane $\{(k_{\lf_1},\cdots,k_{\lf_{2n+1}}):\pm k_{\lf_{2m+1}}\cdots\pm k_{\lf_{2n+1}}=k\}$ where $0\leq m\leq n$. The measure $\mathrm{d}\sigma$ is then defined by $\mathrm{d}\sigma=\mathrm{d}k_{\lf_1}\cdots\mathrm{d}k_{\lf_{2n}}$. {The product $\prod_{\lf\in\Lc^*}^{(+)}$ is taken over all $\lf\in\{\lf_1,\cdots,\lf_{2n+1}\}$,} and $\Omega_\nf=\Omega(k_{\nf_1},k_{\nf_2},k_{\nf_3},k_\nf)$. The singularities $1/\Omega_\nf$ are treated using the Cauchy principal value.
\end{prop}
\begin{proof} We start with the summation in (\ref{defkq}). Since for any decoration we must have $\zeta_{\nf'}\Omega_{\nf'}=-\zeta_\nf \Omega_\nf$ for any branching node pair $\{\nf,\nf'\}$ as in Proposition \ref{branchpair}, in (\ref{defkq}) we can replace the factor $\Bc_\Qc(t,s,\delta L^2\cdot\Omega[\Nc^*])$  by $\widetilde{\Bc}_\Qc(t,s,\delta L^2\cdot\Omega[\Nc^{ch}])$. Then, by Proposition \ref{maincoef}, we may write (\ref{defkq}) as a sum in $Z$ of terms
\begin{multline}\label{defkqz}\Kc_{\Qc,Z}(t,s,k)=\bigg(\frac{\delta}{2L^{d-1}}\bigg)^{2n}\zeta^*(\Qc)\prod_{\nf\in Z}\frac{1}{\zeta_\nf \pi i}\cdot\sum_{\Es}\epsilon_\Es\prod_{\nf\in Z}\frac{\chi_\infty(\delta L^2\Omega_\nf)}{\delta L^2\Omega_\nf}\\\times\widetilde{\Bc}_{\Qc,Z}(t,s,\delta L^2\Omega[\Nc^{ch}\backslash Z])\cdot{\prod_{\lf\in\Lc^*}^{(+)}n_{\mathrm{in}}(k_\lf)}.\end{multline} To analyze $\Kc_{\Qc,Z}$, we can use the formula (\ref{maincoef1}) to write $\widetilde{\Bc}_{\Qc,Z}$ as an integral in $(\lambda_1,\lambda_2)$, and apply Proposition \ref{approxnt} for fixed $(\lambda_1,\lambda_2)$. For simplicity of presentation we will not explicitly show this step below, but notice that this allows us to estimate the error term in $L_{\lambda_1,\lambda_2}^1$ type norms such as $X^\kappa$. We carefully note here that the bound (\ref{maincoef2}) involves different choices of $\rho$; however for Proposition \ref{approxnt}  we only need (\ref{maincoef2}) for $|\rho|\leq 10n$, so this leads to at most $C^n$ loss, since the number of such multi-indices $\rho$ is at most $C^n$.

Before applying Proposition \ref{approxnt}, we need a few preparation steps. First, for any $\nf\in\Nc^{ch}$ we define $x_\nf=k_{\nf_1}-k_\nf$ and $y_{\nf}=k_\nf-k_{\nf_3}$, so we have $\Omega_\nf=2\langle x_\nf,y_\nf\rangle$ by (\ref{res}). It is easy to check by induction that $(x_\nf,y_\nf)$, where $\nf\in\Nc^{ch}$ (there are $n$ such nodes $\nf$), are free variables and uniquely determine a point on $\Sigma$, and the linear mapping
\begin{equation}\label{substitution}(x_\nf,y_\nf)_{\nf\in\Nc^{ch}}\leftrightarrow (k_{\lf_1},\cdots, k_{\lf_{2n}})\end{equation} is volume preserving and preserves the lattice $(\Zb_L^d)^{2n}$. Therefore, we can rewrite the sum in (\ref{defkqz}) as 
\begin{equation}\label{newsum1}\sum_{(x_\nf,y_\nf):\nf\in\Nc^{ch}}\epsilon\cdot\prod_{\nf\in Z}\frac{\chi_\infty(2\delta L^2\langle x_\nf,y_\nf\rangle_\beta)}{2\delta L^2\langle x_\nf,y_\nf\rangle_\beta}\cdot\widetilde{\Bc}_{\Qc,Z}(t,s,2\delta L^2\langle x_\nf,y_\nf\rangle_\beta:\nf\in\Nc^{ch}\backslash Z)\cdot W(x[\Nc^{ch}],y[\Nc^{ch}]),\end{equation} where $\epsilon=\epsilon_\Es$ and
\begin{equation}\label{newsum1coef}W(x[\Nc^{ch}],y[\Nc^{ch}])=\prod_{j=1}^{2n}n_{\mathrm{in}}(k_{\lf_j})\cdot n_{\mathrm{in}}(\pm k\pm k_{\lf_{2m+1}}\cdots\pm k_{\lf_{2n}}).\end{equation}Next we will replace the $\epsilon$ in (\ref{newsum1}) by 1; the difference caused will be an error term that can be handled in the same way as the main term, and will be left to the end. Then, we decompose (\ref{newsum1coef}) into functions supported in $|k_{\lf_j}-a_j^*|\leq 1$, where $a_j^*\in\Zb_L^d$ for $1\leq j\leq 2n$, and $|\pm k\pm k_{\lf_{2m+1}}\cdots\pm k_{\lf_{2n}}-a_{2n+1}^*|\leq 1$, using a partition of unity. Since $n_{\mathrm{in}}$ is Schwartz, for such a term we can freely gain the decay factors $\prod_{j=1}^{2n+1}\langle a_j^*\rangle^{-80d}$; this easily allows us to sum in $(a_j^*)$, and also addresses the weight $\langle k\rangle^{40d}$ in the $X^{40d} $ norm, as $|k|\leq (2n+1)\max_j|a_j^*|$.

Now we can apply Proposition \ref{approxnt}. First (\ref{propertyw1}) is true, because it is true if $W$ is regarded as a function of $(k_{\lf_j})_{1\leq j\leq 2n}$. Moreover the change of variables (\ref{substitution}) is volume preserving, so it also preserves the Fourier $L^1$ norm, and similarly the Fourier $L^1$ norm with one derivative gets amplified by at most $O(n)$ under this change of variables. Second, the function $\Psi$ here clearly satisfies (\ref{propertypsi1})--(\ref{propertypsi2}) due to Proposition \ref{maincoef}, so we only need to verify the support condition (\ref{propertyw2}).

Since the condition (\ref{propertyw2}) allows for translation, we may assume $a_j^*=0$ in the previous reduction. Then we have $|k_\lf|\leq 1$ for any leaf $\lf$. For any $\nf\in\Nc^{ch}$, let $\nf'$ be the branching node paired with $\nf$, and let $d(\nf)$ be the maximum depth, counting from the root node(s), of $\nf$ and $\nf'$. Define the partial order $\prec$ such that $\nf_1\prec\nf_2$ if and only if $d(\nf_1)>d(\nf_2)$. Now for any $\nf\in\Nc^{ch}$, consider the variable $x_\nf$ (the other one $y_{\nf}$ is the same). We may assume $d(\nf)$ equals the depth of $\nf$, since otherwise we have $x_{\nf}\in\{\pm x_{\nf'},\pm y_{\nf'}\}$ and we can perform the same argument for $\nf'$. Let $\nf_j\,(1\leq j\leq 3)$ be the children nodes of $\nf$, then $x_\nf=k_{\nf_1}-k_\nf$. Using the notations in Lemma \ref{auxlem}, if $\mu_{\nf_1}=\max(\mu_{\nf_1},\mu_{\nf_2},\mu_{\nf_3})$, then
\[|x_\nf|=|k_{\nf_2}-k_{\nf_3}|\leq 2\max(\mu_{\nf_2},\mu_{\nf_3})=2\mu_{\nf}^\circ.\] Suppose now $\max(\mu_{\nf_1},\mu_{\nf_2},\mu_{\nf_3})$ is not $\mu_{\nf_1}$, say it is $\mu_{\nf_2}$ (the case of $\mu_{\nf_3}$ being similar), then $\nf_2$ is not a leaf. Let its children be $\nf_{21}$, $\nf_{22}$ and $\nf_{23}$ from left to right, then consider $\max(\mu_{\nf_{21}},\mu_{\nf_{22}},\mu_{\nf_{23}})$; we assume this maximum is not $\mu_{\nf_{21}}$ (otherwise it is not $\mu_{\nf_{23}}$ and we can argue similarly replacing $x_{\nf_2}$ by $-y_{\nf_2}$), then
\[|x_\nf+x_{\nf_2}|=|{k_{\nf_{21}}-k_{\nf_3}}|\leq\mu_{\nf_{21}}+\mu_{\nf_3}\leq \mu_{\nf_2}^\circ+\mu_\nf^\circ.\]Moreover, since $\nf_2$ is a child of $\nf$, by definition we know that either $\nf_2\prec\nf$ (if $\nf_2\in\Nc^{ch}$) or $\nf_2'\prec\nf$ (if $\nf_2$ is paired with some $\nf_2'\in\Nc^{ch}$, note also that $x_{\nf_2}\in\{\pm x_{\nf_2'},\pm y_{\nf_2'}\}$). Summarizing, in any case we get (\ref{propertyw2}) with $\lambda_\nf=2\max\{\mu_{\nf}^\circ,\mu_{\nf_j}^\circ\}$ where $\nf_j$ is a child of $\nf$ that is not a leaf. Note that by (\ref{auxineq}) we also have
\[\prod_{\nf\in\Nc^{ch}}\lambda_\nf\leq C^{n+1}.\]  By translation, the same bound is true for any $(a_j^*)$, with suitable choices of $(a_j)$ and $(b_j)$ in (\ref{propertyw2}).

With all the preparations, we can apply Proposition \ref{approxnt} to get
\begin{multline}\label{defint2}(\ref{newsum1})=(L^{2d-2}\delta^{-1})^n\int\widetilde{\Bc}_{\Qc,Z}\big(t,s,\alpha[\Nc^{ch}\backslash Z]\big)\,\mathrm{d}\alpha[\Nc^{ch}\backslash Z]\cdot\int W(x[\Nc^{ch}],y[\Nc^{ch}])\\\times\prod_{\nf\in Z}\frac{1}{2\langle x_\nf,y_\nf\rangle_\beta}\prod_{\nf\in\Nc^{ch}\backslash Z}\dirac(2\langle x_\nf,y_\nf\rangle_\beta)\,\mathrm{d}x[\Nc^{ch}]\mathrm{d}y[\Nc^{ch}]+\Rs_0\end{multline} with $\|\Rs_0\|_{X_{\mathrm{loc}}^{40d}}\lesssim (C^+L^{2d-2}\delta^{-1})^nL^{-2\nu}$. Note that in the $X^{40d}$ norm we are taking supremum in $k$ for fixed $(\lambda_1,\lambda_2)$, which is allowed because the bounds obtained by applying Proposition \ref{approxnt} are uniform in $k$. Then, reversing the change of variables (\ref{substitution}), we can rewrite the $\mathrm{d}x[\Nc^{ch}]\mathrm{d}y[\Nc^{ch}]$ integral in (\ref{defint2}) as $\mathrm{d}\sigma$ integral in (\ref{defintegral}), so this integral becomes $\Mc_{\Qc,Z}^*(k)$, noticing also that $2\langle x_\nf,y_\nf\rangle_\beta=\Omega_\nf$. This already proves the desired result, provided we replace the $\epsilon_\Es$ factor by 1.

Finally, consider the case when $\epsilon_\Es\neq 1$. By definition (\ref{defcoef}) we know that $\epsilon_\Es\neq 1$ only when some $x_\nf=0$ or $y_\nf=0$ (or both). If this happens, say $x_\nf=0$, then $\Omega_\nf=0$. Also $y_\nf\in\Zb_L^d$ satisfies $|y_{\nf}|\leq C n\leq C(\log L)^3$ up to translation, so it has at most $L^d(\log L)^{3d}$ choices. Then in the summation (\ref{newsum1}) we may first fix $(x_\nf,y_\nf)$ which has at most $L^d(\log L)^{3d}$ choices, then treat the remaining sum in the same way as above. We can easily verify (for example by using Sobolev embedding) that the assumptions of Proposition \ref{approxnt} are preserved upon fixing some of the variables $x_\nf$, $y_\nf$ or $\Omega_\nf$. Since the summation in $(x_\nf,y_\nf)$ only gives $L^d(\log L)^{3d}\leq L^{2d-2}\delta\cdot L^{-1/2}$, we can see that the bound satisfied by any such difference term will put it in the remainder term $\Rs$.
\end{proof}
\begin{rem}\label{holder} The main term $\Mc_{\Qc,Z}^*(k)$ defined by (\ref{defintegral}) satisfies the bound
\begin{equation}\label{holderbound}\sup_{|\rho|\leq 40d}|\partial^\rho\Mc_{\Qc,Z}^*(k)|\lesssim (C^+)^n\langle k\rangle^{-40d}.
\end{equation} In fact, if without derivatives, this bound follows the same argument as in the proof of Proposition \ref{approxnt} (the decay in $k$ can be included using that $n_{\mathrm{in}}$ is Schwartz as above). Suppose one takes a $\partial_k$ derivative in (\ref{defintegral}), then since the $\mathrm{d}\sigma$ integral can be rewritten as $\mathrm{d}x[\Nc^{ch}]\mathrm{d}y[\Nc^{ch}]$, the corresponding result will have the same form as (\ref{defintegral}), except that one of the input functions $n_{\mathrm{in}}$ is replaced by its partial derivative. Iterating this fact we can obtain control for $\partial^\rho\Mc_{\Qc,Z}^*$ for $|\rho|\leq 40d$.
\end{rem}
\begin{rem}\label{timebound} The integral
\begin{equation}\label{defintcal}\Jc\widetilde{\Bc}_{\Qc,Z}(t,s):=\int\widetilde{\Bc}_{\Qc,Z}(t,s,\alpha[\Nc^{ch}\backslash Z])\,\mathrm{d}\alpha[\Nc^{ch}\backslash Z]
\end{equation} will be studied in detail in Section \ref{domasymp}. For now we just note that it satisfies the simple bound $\|\Jc\widetilde{\Bc}_{\Qc,Z}\|_{X_{\mathrm{loc}}}\lesssim (C^+)^n$, which easily follows from (\ref{maincoef2}). This, together with Proposition \ref{asymptotics1} and (\ref{holderbound}), implies that $\|\Kc_\Qc(t,s,k)\|_{X_{\mathrm{loc}}^{40d}}\lesssim (C^+\delta)^n$ for each regular couples $\Qc$ of scale $2n$.
\end{rem}
We conclude this section with a similar asymptotic formula for regular trees.
\begin{prop}\label{varregtree} Let $\Tc$ be a regular tree of scale $2n$ with lone leaf $\lf_*$. Let $\Nc$ be the set of branching nodes, and $\Lc$ the set of leaves. Define the function (slightly different from (\ref{defcoefa2})--(\ref{defcoefb2}))
\begin{equation}\label{varat}\Ac_\Tc^*(t,s,\alpha[\Nc]):=\int_\Dc\prod_{\nf\in\Nc}e^{\zeta_\nf\pi i\alpha_\nf t_\nf}\,\mathrm{d}t_\nf,
\end{equation} where the domain
\[\Dc=\big\{t[\Nc]:t_{(\lf_*)^p}>s;\,\,0<t_{\nf'}<t_\nf<t,\mathrm{\ whenever\ }\nf'\mathrm{\ is\ a\ child\ node\ of\ }\nf\big\},\] with $(\lf_*)^p$ being the parent of $\lf_*$. For $t>s$, consider the expression
\begin{equation}\label{varkt}\Kc_\Tc^*(t,s,k)=\bigg(\frac{\delta}{2L^{d-1}}\bigg)^{2n}\widetilde{\zeta}(\Tc)\sum_{\Ds}\epsilon_\Ds\cdot\Ac_\Tc^*(t,s,\delta L^2\cdot\Omega[\Nc])\cdot{\prod_{\lf\in\Lc\backslash\{\lf_*\}}^{(+)}n_{\mathrm{in}}(k_\lf)}.
\end{equation} Here the sum is taken over all $k$-decorations $\Ds$ of the regular tree $\Tc$, $\widetilde{\zeta}(\Tc)$ is defined similar to (\ref{defzetaq}) but with $\Nc^*$ replaced by $\Nc$, {and the product is taken over $\lf\in\Lc\backslash\{\lf_*\}$ that has sign $+$.} Then, we can decompose $\Kc_\Tc^*=(\Kc_\Tc^*)_{\mathrm{app}}+\Rs^*$, where $(\Kc_\Tc^*)_{\mathrm{app}}(t,s,k)$ is the sum of at most $2^n$ terms each having form $\delta^n\cdot \Jc\Ac^*(t,s)\cdot \Mc^*(k)$, and we have the bounds
\begin{equation}\label{vardecomp}\|\Jc\Ac^*\|_{X_{\mathrm{loc}}}\lesssim (C^+)^n,\quad \sup_{|\rho|\leq 40d}|\partial^\rho\Mc^*(k)|\lesssim (C^+)^n,\quad \|\Rs^*\|_{X_{\mathrm{loc}}^0}\lesssim (C^+\delta)^nL^{-2\nu};
\end{equation}
\end{prop}
\begin{proof} Note that $\Qc=(\Tc,\bullet)$ is a regular couple of scale $2n$. A $k$-decoration $\Ds$ can be viewed as a $k$-decoration of $\Qc$, and we always have $k_{\lf_*}=k$. We can pair the branching nodes of $\Tc$ as in Proposition \ref{branchpair}, such that $\zeta_{\nf'}\Omega_{\nf'}=-\eta_\nf\Omega_\nf$, and define $\Nc^{ch}$ as in Definition \ref{defchoice}, so in particular $\Ac_\Tc^*(t,s,\delta L^2\Omega[\Nc])=\widetilde{\Ac}_\Tc(t,s,\delta L^2\Omega[\Nc^{ch}])$ is a function of $t,s$ and $\Omega[\Nc^{ch}]$ only.

Since $\Tc$ is formed from a regular chain by replacing each leaf pair with a regular couple, by using Proposition \ref{maincoef} for these regular couples, and analyzing the regular chain similar to Section \ref{regchainest}, we can show that $\widetilde{\Ac}_\Tc(t,s,\alpha[\Nc^{ch}])$ has form (\ref{maincoef1}) that satisfies (\ref{maincoef2})--(\ref{maincoef2.5}), for $t>s$, with some choice of $Z\subset \Nc^{ch}$. Here the weights $\langle \lambda_1\rangle^{\frac{1}{4}}\langle\lambda_2\rangle^{\frac{1}{4}}$ and $\langle \lambda_1\rangle^{\frac{1}{8}}\langle\lambda_2\rangle^{\frac{1}{8}}$ in (\ref{maincoef2}) and (\ref{maincoef2.5}) will be replaced by the weaker ones $(\langle \lambda_1\rangle+\langle\lambda_2)\rangle^{\frac{1}{4}}$ and $(\langle \lambda_1\rangle+\langle\lambda_2)\rangle^{\frac{1}{8}}$, but they still suffice to prove the desired $X_{\mathrm{loc}}$ and $X_{\mathrm{loc}}^0$ bounds. Moreover, the product $\prod_{\lf\in\Lc\backslash\{\lf_*\}}^{(+)}$ in (\ref{varkt}), compared to the product {$\prod_{\lf\in\Lc^*}^{(+)}$} in (\ref{defkqz}), only misses one factor $n_{\mathrm{in}}(k)$. Therefore, we can define the approximation $(\Kc_\Tc^*)_{\mathrm{app}}$ similar to Proposition \ref{asymptotics1} and prove (\ref{vardecomp}) using Proposition \ref{approxnt}, similar to the proof of Proposition \ref{asymptotics1}. Here, due to the absence of the $n_{\mathrm{in}}(k)$ factor, we can no longer control the weight $\langle k\rangle^{40d}$, so the second inequality in (\ref{vardecomp}) does not have the same weight as (\ref{holderbound}), and the third inequality only involves the $X_{\mathrm{loc}}^0$ norm instead of $X_{\mathrm{loc}}^{40d}$. Other than these, the proof is basically the same so we omit the details.
\end{proof}
\section{Regular couples III: full asymptotics}\label{domasymp} In this section we further analyze the asymptotics obtained in Proposition \ref{asymptotics1}. Clearly the main goal is to evaluate the integral (\ref{defintcal}). Like Proposition \ref{maincoef}, this will be done by inducting on the scale of $\Qc$, so the operators and $K$ functions associated with regular chains, which are studied in Section \ref{regchainest}, will also play a key role here. Once this is done, we will combine these terms in Section \ref{combine} to calculate the full asymptotics.
\subsection{Regular chain calculations} For any function $F=F(\alpha[W])$ we define 	
\begin{equation}\label{gausspv}[\mathrm{G}]\int F=\lim_{\theta\to 0}\int F(\alpha[W])\prod_{j\in W}e^{-\pi\theta^2\alpha_j^2}\,\mathrm{d}\alpha_j,
\end{equation} if the limit exists. This can be seen as a Gaussian version of principal value integral; clearly if $F\in L^1$ then $[\mathrm{G}]\int$ coincides with the usual Lebesgue integration.
\begin{lem}\label{regchaincancel1} Let \[I:=X_0I_{\beta_1+\lambda_1}X_1I_{\beta_2+\lambda_2}\cdots X_{2m-1}I_{\beta_{2m}+\lambda_{2m}}X_{2m}\] be as in Lemma \ref{regchainlem7}, for a legal partition $\Pc$ of $\{1,\cdots,2m\}$ with $m\geq 1$; in particular it depends on $(\alpha_1,\cdots,\alpha_m)$ and $(\mu_1,\cdots,\mu_m)$, and also on the $\alpha[A]$ and $\mu[E]$ variables appearing in the $X_a\,(0\leq a\leq 2m)$ operators. Denote the collection of all these $\alpha_j$ variables by $\alpha[W]$. For any $\lambda_*$, consider the expression $K=I(e^{\pi i\lambda_*s})(t)$. If we fix $(\lambda_*,t)$ and all the $\mu_j$ variables, and view $K$ as a function of $\alpha[W]$, then $K\in L^1$ and
\begin{equation}\int K(\alpha[W])\,\mathrm{d}\alpha[W]=0.
\end{equation}
\end{lem}
\begin{proof} By Lemma \ref{regchainlem7} we know that $I$ is a sum of an operator of class $J$ and an operator of class $R$. By repeating the proof of Proposition \ref{maincoef}, we know that $K\in L^1$. Let $W_1=W\backslash\{1\}$, we will fix $\alpha_j=\alpha_j^*$ for $j\in W_1$, and view $K=K(\alpha_1)$ as a function of $\alpha_1$. Clearly for a.e. $\alpha^*[W_1]:=(\alpha_j^*)_{j\in W_1}$ we have $K(\alpha_1)\in L^1$, so it suffices to prove that
\begin{equation}\label{intvanish}[\mathrm{G}]\int K(\alpha_1)\,\mathrm{d}\alpha_1=0\end{equation} holds for each $\alpha^*[W_1]$. Now once $\alpha^*[W_1]$ is fixed, we can simply write
\[K(\alpha_1)=Y_0I_{\epsilon\alpha_1}Y_1I_{\mu_1-\epsilon\alpha_1}G(t),\] where $Y_1$ have bounded kernel, i.e. $Y_1f(t)=\int_0^t Y_1(t,s)f(s)\mathrm{d}s$ with $Y_1\in L^\infty$, $Y_0$ is either $\mathrm{Id}$ or has bounded kernel, and $G$ is a bounded function. This gives that
\[K(\alpha_1)=\int_{t>u>v>w>s>0}Y_0(t,u)e^{\epsilon \pi i\alpha_1v}Y_1(v,w)e^{\pi i(\mu_1-\epsilon\alpha_1)s}G(s)\,\mathrm{d}u\mathrm{d}v\mathrm{d}w\mathrm{d}s;\] here $Y_0(t,s)$ may be replaced by $\dirac(t-s)$. We calculate
\[\int_\Rb e^{-\pi\theta^2\alpha_1^2} K(\alpha_1)\,\mathrm{d}\alpha_1=\int_{t>u>v>w>s>0}Y_0(t,u)Y_1(v,w)e^{\pi i\mu_1s}G(s)\,\mathrm{d}u\mathrm{d}v\mathrm{d}w\mathrm{d}s\int_\Rb e^{\epsilon\pi i(v-s)\alpha_1-\pi\theta^2\alpha_1^2}\mathrm{d}\alpha_1.\] The last integral in $\alpha_1$ can be calculated explicitly and equals $\theta^{-1}e^{-\pi(v-s)^2/4\theta^2}$, hence
\begin{multline*}\bigg|\int_\Rb e^{-\pi\theta^2\alpha_1^2} K(\alpha_1)\,\mathrm{d}\alpha_1\bigg|\lesssim\theta^{-1}\int_{t>v>w>s>0}e^{-\frac{\pi(v-s)^2}{4\theta^2}}\,\mathrm{d}v\mathrm{d}w\mathrm{d}s\\=\theta^{-1}\int_{t>v>s>0}e^{-\frac{\pi(v-s)^2}{4\theta^2}}(v-s)\,\mathrm{d}v\mathrm{d}s\to 0,\end{multline*} which proves (\ref{intvanish}).
\end{proof}
\begin{lem}\label{regchaincancel2} Let \[I:=I_{\beta_1+\lambda_1}I_{\beta_2+\lambda_2}\cdots I_{\beta_{2m}+\lambda_{2m}}\] be as in Lemma \ref{regchainlem8}, for a legal partition $\Pc$ of $\{1,\cdots,2m\}$ that is \emph{not} dominant, and let $K=I(e^{\pi i\lambda_0s})(t)$. We decompose $I$ into $\prod_{j\in Z}\frac{\chi_\infty(\alpha_j)}{\epsilon_j\pi i\alpha_j}\cdot\widetilde{I}$ as in Lemma \ref{regchainlem8}, where $Z\subset\{1,\cdots,m\}$ and $\widetilde{I}$ depends only on the variables $(\mu_1,\cdots,\mu_m)$ and $\alpha[W]$ with $W=\{1,\cdots,m\}\backslash Z$, and define $\widetilde{K}=\widetilde{I}(e^{\pi i\lambda_0s})(t)$. Then, for any choice of $Z$, if we fix $(t,\lambda_0,\mu_1,\cdots,\mu_m)$ and view $\widetilde{K}$ as a function of $\alpha[W]$, then $\widetilde{K}\in L^1$ and
\begin{equation}\label{regchaincancelref}\int \widetilde{K}(\alpha[W])\,\mathrm{d}\alpha[W]=0.
\end{equation}
\end{lem}
\begin{proof} This is a direct consequence of Lemma \ref{regchaincancel1}. Namely, if we carry out the construction process of $\widetilde{I}$ in the proof of Lemma \ref{regchainlem8}, then this $\widetilde{I}$ will have the form described in Lemma \ref{regchainlem7}; moreover as $\Pc$ is not dominant, there will be at least one pair left after removing all adjacent pairs (which corresponds to $m\geq 1$ in Lemma \ref{regchainlem7} and Lemma \ref{regchaincancel1}), so Lemma \ref{regchaincancel1} will be applicable.
\end{proof}
\begin{lem}\label{regchaincal} Let \[I:=I_{\beta_1+\lambda_1}I_{\beta_2+\lambda_2}\cdots I_{\beta_{2m}+\lambda_{2m}}\] be as in Lemma \ref{regchainlem8}, where $\Pc=\{\{1,2\},\cdots,\{2m-1,2m\}\}$ is the dominant partition in the sense of Definition \ref{deflegal}. Then we have $\beta_{2j-1}=\epsilon_j\alpha_j$ and $\beta_{2j}=-\epsilon_j\alpha_j$ where $\epsilon_j\in\{\pm\}$ for $1\leq j\leq m$. Given also $\lambda_0$, define $\widetilde{I}$ and $\widetilde{K}$ associated with $Z\subset\{1,\cdots,m\}$ as in Lemma \ref{regchaincancel2}. Then for any $Z$ we have $\widetilde{K}\in L^1$, and
\begin{equation}\label{regchaincalref}\int \widetilde{K}(\alpha[W])\,\mathrm{d}\alpha[W]=\int_{t>t_1>\cdots>t_m>0}e^{\pi i(\mu_1t_1+\cdots +\mu_mt_m)+\pi i\lambda_0t_m}\,\mathrm{d}t_1,\cdots\mathrm{d}t_m.
\end{equation}
\end{lem}
\begin{proof} For $1\leq j\leq m$, by Lemma \ref{regchainlem6} we decompose
\[I_{\epsilon_j\alpha_j}I_{\mu_j-\epsilon_j\alpha_j}=\frac{\chi_\infty(\alpha_j)}{\epsilon_j \pi i\alpha_j}I_{\mu_j}+\Jc_{\alpha_j,\mu_j}+\Rc_{\alpha_j,\mu_j},\] therefore $\widetilde{I}$ is the composition of $m$ operators, where the $j$-th operator is $I_{\mu_j}$ if $j\in Z$, and is $\Jc_{\alpha_j,\mu_j}+\Rc_{\alpha_j,\mu_j}$ if $j\in W$. Thus $\widetilde{I}$ is of class $J$ or $R$, so $\widetilde{K}\in L^1$.

Now, let $I^*$ be the operator where $\Jc_{\alpha_j,\mu_j}+\Rc_{\alpha_j,\mu_j}$ is replaced by $I_{\epsilon_j\alpha_j}I_{\mu_j-\epsilon_j\alpha_j}$ for each $j\in W$ in $\widetilde{I}$, and define $K^*$ accordingly. Then $I^*$ can be expanded into finitely many terms, one of them being $\widetilde{I}$; the other terms have form
\[\prod_{j\in Z_1}\frac{\chi_\infty(\alpha_j)}{\epsilon_j\pi i\alpha_j}\cdot I^{**}\] for some $\varnothing\neq Z_1\subset W$, where $I^{**}$ depends on $(\mu_1,\cdots,\mu_m)$ and $\alpha[W\backslash Z_1]$, and has class $J$ or $R$; define $K^{**}$ accordingly. By the factorized structure and symmetry, we trivially have
\[[\mathrm{G}]\int\prod_{j\in Z_1}\frac{\chi_\infty(\alpha_j)}{\epsilon_j\pi i\alpha_j}\cdot K^{**}(\alpha[W\backslash Z_1])\,\mathrm{d}\alpha[W]=0\] for any fixed $(t,\lambda_0,\mu_1,\cdots,\mu_m)$, noticing also $K^{**}\in L^1$. Therefore, to calculate $\int \widetilde{K}(\alpha[W])\mathrm{d}\alpha[W]$, it suffices to calculate $[\mathrm{G}]\int K^*(\alpha[W])$; by switching signs we may assume $\epsilon_j=1$. Now $K^{*}$ can be written in the following form (with $t_0=t$)
\[K^{*}=\int_{t>t_1>\cdots>t_m>0}e^{\pi i(\mu_1t_1+\cdots +\mu_mt_m)+\pi i\lambda_0t_m}\,\mathrm{d}t_1\cdots\mathrm{d}t_m\cdot\prod_{j\in W}\int_{t_j<s_j<t_{j-1}}e^{\pi i\alpha_j(s_j-t_j)}\,\mathrm{d}s_j;\] therefore, for $\theta>0$, we have
\begin{multline*}\int K^{*}(\alpha[W])\prod_{j\in W}e^{-\pi\theta^2\alpha_j^2}\,\mathrm{d}\alpha_j\\=\int_{t>t_1>\cdots>t_m>0}e^{\pi i(\mu_1t_1+\cdots +\mu_mt_m)+\pi i\lambda_0t_m}\,\mathrm{d}t_1\cdots\mathrm{d}t_m\cdot\prod_{j\in W}\int_{t_j<s_j<t_{j-1}}\theta^{-1}e^{-\frac{\pi(s_j-t_j)^2}{4\theta^2}}\,\mathrm{d}s_j.
\end{multline*} For each fixed $(t_1,\cdots,t_m)$, the integral in $s_j$ is uniformly bounded; moreover for any $\tau>0$ we have
\begin{equation}\label{approxid}\lim_{\theta\to 0}\int_{0<\eta<\tau}\theta^{-1}e^{-\frac{\pi\eta^2}{4\theta^2}}\,\mathrm{d}\eta=2\int_{\xi>0}e^{-\pi\xi^2}\,\mathrm{d}\xi=1,\end{equation} so (\ref{regchaincalref}) follows.
\end{proof}
\subsection{Non-dominant couples} For any regular but non-dominant couples, the leading term in the asymptotics obtained in Proposition \ref{asymptotics1} simply vanishes.
\begin{prop}\label{maincancel} Let $\Qc$ be a regular couple that is not dominant, then for any $Z\subset\Nc^{ch}$ that appears in Proposition \ref{maincoef} and any $(t,s)$ we have (recall (\ref{defintcal}) {for} definition)
\begin{equation}\label{maincancel1}\Jc\widetilde{\Bc}_{\Qc,Z}(t,s)=0.
\end{equation}
\end{prop}
\begin{proof} We induct on the scale of $\Qc$. Suppose (\ref{maincancel1}) is true for all regular couples with smaller scales (the base case will follow in the same way), and consider a regular couple $\Qc$ of scale $2n$. As in Section \ref{recursive}, let $\Qc$ be obtained from $\Qc_0$ and by replacing the leaf-pairs with regular couples $\Qc_j\,(j\geq 1)$ with $n(\Qc_j)<2n$.

\emph{Case 1}. If $\Qc_j$ is non-dominant for some $j\geq 1$, then by (\ref{newexp1}) and (\ref{newexp2}), we know that the only way in which $\widetilde{\Bc}_\Qc$ depends on the variables $\alpha[\Nc_j^{ch}]$ is via $\widetilde{\Bc}_{\Qc_j}$, thus we can write
\[\widetilde{\Bc}_\Qc(t,s,\alpha[\Nc^{ch}])=\int_{\Rb^2}\Wc(t,s,t',s',\alpha[\Nc^{ch}\backslash\Nc_j^{ch}])\cdot\widetilde{\Bc}_{\Qc_j}(t',s',\alpha[\Nc_j^{ch}])\,\mathrm{d}t'\mathrm{d}s'\] for some kernel $\Wc$. For any $Z\subset\Nc^{ch}$, let $Z_1=Z\cap\Nc_j^{ch}$ and $Z_2=Z\backslash\Nc_j^{ch}$, then the component
\[\prod_{\nf\in Z}\frac{\chi_\infty(\alpha_\nf)}{\zeta_\nf\pi i\alpha_\nf}\cdot\widetilde{\Bc}_{\Qc,Z}(t,s,\alpha[\Nc^{ch}\backslash Z])\] of $\widetilde{\Bc}_\Qc$ must come from the components
\[\prod_{\nf\in Z_1}\frac{\chi_\infty(\alpha_\nf)}{\zeta_\nf\pi i\alpha_\nf}\cdot\widetilde{\Bc}_{\Qc_j,Z_1}(t,s,\alpha[\Nc_j^{ch}\backslash Z_1])\quad\mathrm{and}\quad\prod_{\nf\in Z_2}\frac{\chi_\infty(\alpha_\nf)}{\zeta_\nf\pi i\alpha_\nf}\cdot\Wc_{Z_2}(t,s,t',s',\alpha[(\Nc^{ch}\backslash\Nc_j^{ch})\backslash Z_2])\] of $\widetilde{\Bc}_{\Qc_j}$ and $\Wc$ respectively (with at most a $\pm$ sign), where $\Wc_{Z_2}$ is a suitable kernel, and that we must have
\begin{equation*}\widetilde{\Bc}_{\Qc,Z}(t,s,\alpha[\Nc^{ch}\backslash Z])=\int_{\Rb^2}\Wc_{Z_2}(t,s,t',s',\alpha[(\Nc^{ch}\backslash\Nc_j^{ch})\backslash Z_2])\cdot\widetilde{\Bc}_{\Qc_j,Z_1}(t',s',\alpha[\Nc_j^{ch}\backslash Z_1])\,\mathrm{d}t'\mathrm{d}s'.\end{equation*} This, together with the induction hypothesis (\ref{maincancel1}) for $\widetilde{\Bc}_{\Qc_j,Z_1}$, clearly implies that (\ref{maincancel1}) also holds for $\widetilde{\Bc}_{\Qc,Z}$.

\emph{Case 2}. If $\Qc_j$ is dominant for each $j\geq 1$, since $\Qc$ is not dominant, by Definition \ref{defstd}, we know that $\Qc_0$ must be a regular double chain with at least one of the regular chains being non-dominant, say $\Tc^+$ is non-dominant. Following the proof of Proposition \ref{maincoef} in Section \ref{proofmaincoef}, we see that the only way in which $\widetilde{\Bc}_{\Qc}$ depends on the variables $(\alpha_1^+,\cdots,\alpha_{m^+}^+)$ is via the function $K=K(t,\alpha_1^+,\cdots,\alpha_{m^+}^+,\lambda_0^+,\lambda_1^+,\cdots,\lambda_{2m^+}^+)$ in (\ref{regchainfunck}); in the same way as in \emph{Case 1} above, we have
\begin{multline}\widetilde{\Bc}_{\Qc,Z}(t,s,\alpha[\Nc^{ch}\backslash Z])=\int \widetilde{K}_{Z^+}\big(t,\lambda_0^+,\mu_1^+,\cdots,\mu_{m^+}^+,\widetilde{\alpha}[W^+]\big)\\\times \Wc\big(s,\lambda_0^+,\lambda_1^+,\cdots,\lambda_{2m^+}^+,\alpha[W_1]\big)\,\mathrm{d}\lambda_0^+\prod_{j\not\in Z^+}\mathrm{d}\lambda_a^+\mathrm{d}\lambda_b^+.
\end{multline} Here we assume that $Z^+$ is the subset of $\{1,\cdots,m^+\}$ appearing in (\ref{regcplform1}), $W^+=\{1,\cdots,m^+\}\backslash Z^+$, and $Z\subset \Nc^{ch}$ is determined (among other things) by $Z^+$, $W_1=(\Nc^{ch}\backslash Z)\backslash \{\nf_a^+:a<b\}$ where the set $\{\nf_a^+:a<b\}$ is the one appearing on the right hand side of (\ref{newnch}), $\widetilde{K}_{Z^+}$ is the function $\widetilde{I}(e^{\pi i\lambda_0^+s})(t)$ appearing on the right hand side of (\ref{regcplform1}), and $\Wc$ is some function. Note that $\mu_j^+=\lambda_a^++\lambda_b^+$ and $\widetilde{\alpha_j}=\alpha_j^++\epsilon_j^+\lambda_a^+$ is a translation of $\alpha_j^+$, it will suffice to prove that for any fixed $(t,\lambda_0^+,\mu_1^+,\cdots,\mu_{m^+}^+)$, we must have
\begin{equation}\label{nondomcancel}\int \widetilde{K}_{Z^+}(\widetilde{\alpha}[W^+])\,\mathrm{d}\widetilde{\alpha}[W^+]=0.\end{equation} However, this $\widetilde{K}_{Z^+}$ is just the function $\widetilde{K}$ defined in Lemma \ref{regchaincancel2}, so (\ref{nondomcancel}) follows directly from (\ref{regchaincancelref}). Note that, should any modification procedure described in the proof of Proposition \ref{maincoef} be needed, where some $j\in Z^+$ is moved to $W^+$ and $\widetilde{\alpha_j}=\alpha_j^++\epsilon_j^+ \lambda_a^+$ is replaced by $\alpha_j^+$ in the factor $\frac{\chi_\infty(\widetilde{\alpha_j})}{\epsilon_j^+\pi i\widetilde{\alpha_j}}$, this would not affect the above equality due to the factorized structure. In addition we also have
\begin{equation}\label{additionvanish}\int_\Rb\bigg(\frac{\chi_\infty(\alpha_j^++\epsilon_j^+\lambda_a^+)}{\epsilon_j^+\pi i(\alpha_j^++\epsilon_j^+\lambda_a^+)}-\frac{\chi_\infty(\alpha_j^+)}{\epsilon_j^+\pi i\alpha_j^+}\bigg)\,\mathrm{d}\alpha_j^+=0.\end{equation} This completes the inductive proof.
\end{proof}
\subsection{Dominant couples} For a dominant couple $\Qc$, the corresponding leading term, which contains the integral of $\widetilde{\Bc}_{\Qc,Z}$, will be nonzero due to Lemma \ref{regchaincal}. Moreover, in this situation it is easy to check that any set $Z$ that appears in Proposition \ref{maincoef} must be \emph{special} as in Definition \ref{equivcpl}.
\begin{prop}\label{regcplcal} Let $\Qc$ be a dominant couple and $Z\subset\Qc^{ch}$ be special. Then the function $\Jc\widetilde{\Bc}_{\Qc,Z}(t,s)$ defined in (\ref{defintcal}) is independent of $Z$ and may be denoted $\Jc\widetilde{\Bc}_{\Qc}(t,s)$. Moreover, these functions satisfy some explicit recurrence relation, described as follows. First $\Jc\widetilde{\Bc}_\Qc(t,s)\equiv 1$ for the trivial couple.

Suppose $\Qc$ has type $1$, then it is formed from the $(1,1)$-mini couple by replacing its three leaf pairs by dominant couples $\Qc_j\,(1\leq j\leq 3)$. Then we have
\begin{equation}
\label{recurtype1}
\Jc\widetilde{\Bc}_{\Qc}(t,s)=2\int_0^{\min(t,s)}\prod_{j=1}^3\Jc\widetilde{\Bc}_{\Qc_j}(\tau,\tau)\,\mathrm{d}\tau.
\end{equation} In particular $\Jc\widetilde{\Bc}_{\Qc}=\Jc\widetilde{\Bc}_{\Qc}(\min(t,s))$ is a function of $\min(t,s)$ for type $1$ dominant couples $\Qc$.

Suppose $\Qc$ has type $2$, then $\Qc$ is formed from a regular double chain $\Qc_0$, which consists of two dominant regular chains, by replacing each leaf pair in $\Qc_0$ with a dominant couple. Using the notations in Definition \ref{defchoice}, we now have that the $j$-th pair in $\Pc^\pm$ is $\{2j-1,2j\}$, and that $\Qc_{lp}$ is trivial or has type 1. Then we have
\begin{multline}
\label{recurtype2}
\Jc\widetilde{\Bc}_{\Qc}(t,s)=\int_{t>t_1>\cdots>t_{m^+}>0}\int_{s>s_1>\cdots >s_{m^-}>0}\prod_{j=1}^{m^+}\Jc\widetilde{\Bc}_{\Qc_{j,+,1}}(t_j,t_j)\Jc\widetilde{\Bc}_{\Qc_{j,+,2}}(t_j,t_j)\\\times\prod_{j=1}^{m^-}\Jc\widetilde{\Bc}_{\Qc_{j,-,1}}(s_j,s_j)\Jc\widetilde{\Bc}_{\Qc_{j,-,2}}(s_j,s_j)\cdot\Jc\widetilde{\Bc}_{\Qc_{lp}}(\min(t_{m^+},s_{m^-}))\prod_{j=1}^{m^+}\mathrm{d}t_j\prod_{j=1}^{m^-}\mathrm{d}s_j.
\end{multline} Here we understand that $t_0=t$ and $s_0=s$.
\end{prop}
\begin{proof} We induct on the scale of $\Qc$. The base case $\Qc=\times$ is obvious. Now suppose the result is true for dominant couples of scale smaller than $n(\Qc)$, it suffices to prove for $\Qc$ and any $Z\subset \Nc^{ch}$ that $\Jc\widetilde{\Bc}_{\Qc,Z}(t,s)$ is given by (\ref{recurtype1}) if $\Qc$ has type $1$ and by (\ref{recurtype2}) if $\Qc$ has type $2$.

\emph{Case 1}. Assume $\Qc_0$ is a $(1,1)$-mini couple, and that the three leaf pairs are replaced by $\Qc_j\,(1\leq j\leq 3)$ in $\Qc$. Then by (\ref{newexp3}) we have
\begin{multline}\label{newexp5}\widetilde{\Bc}_{\Qc,Z}(t,s,\alpha[\Nc^{ch}\backslash Z])=\int_{\Rb^6}\prod_{j=1}^3\Cc_j(\lambda_{2j-1},\lambda_{2j},\alpha[\Nc_j^{ch}\backslash Z_j])\prod_{j=1}^6\mathrm{d}\lambda_j\\\times\int_0^t\int_0^s e^{\pi i\alpha_\rf(t_1-s_1)}e^{\pi i(\lambda^* t_1+\lambda^{**}s_1)}\,\mathrm{d}t_1\mathrm{d}s_1\end{multline} where $(\lambda^*,\lambda^{**})=(\lambda_1+\lambda_3+\lambda_5,\lambda_2+\lambda_4+\lambda_6)$. Here $\Cc_j$ are the functions defined in (\ref{maincoef1}) associated with the couple $\Qc_j$; note that some $\Cc_j$ may actually be the functions defined in (\ref{maincoef1}) \emph{after switching the variables $\lambda_{2j-1}$ and $\lambda_{2j}$}, but this will not affect the final result as will be clear later.

When $(\lambda_1,\cdots,\lambda_6)$ are fixed, we know that the $(t_1,s_1)$ integral in (\ref{newexp5}) gives an $L^1$ function $K(\alpha_{\rf})$, and we can calculate that
\begin{equation}\label{integralK}
\begin{aligned}\int K(\alpha_\rf)\,\mathrm{d}\alpha_\rf&=[\mathrm{G}]\int K(\alpha_\rf)\,\mathrm{d}\alpha_\rf=\lim_{\theta\to 0}\int K(\alpha_\rf)e^{-\pi\theta\alpha_\rf^2}\,\mathrm{d}\alpha_\rf\\
&=\lim_{\theta\to 0}\int_0^t\int_0^s e^{\pi i(\lambda^*t_1+\lambda^{**}s_1)}\,\mathrm{d}t_1\mathrm{d}s_1\int_\Rb e^{\pi i(t_1-s_1)\alpha_\rf-\pi\theta\alpha_\rf^2}\,\mathrm{d}\alpha_\rf\\
&=\lim_{\theta\to 0}\theta^{-1}\int_0^t\int_0^se^{\pi i(\lambda^*t_1+\lambda^{**}s_1)}e^{-\frac{\pi(t_1-s_1)^2}{4\theta^2}}\,\mathrm{d}t_1\mathrm{d}s_1.
\end{aligned}
\end{equation} Like in (\ref{approxid}), with fixed $t_1$, the $s_1$ integral tends to $0$ if $t_1>s$, and to $2e^{\pi i\lambda^{**}t_1}$ if $t_1<s$. This gives
\[\int K(\alpha_\rf)\,\mathrm{d}\alpha_\rf=2\int_0^{\min(t,s)}e^{\pi i(\lambda_1+\cdots+\lambda_6)t_1}\,\mathrm{d}t_1.\] Now we plug this into (\ref{newexp5}), and integrate in $(\lambda_1,\cdots,\lambda_6)$. Note that by definition (cf. (\ref{maincoef1}))
\begin{equation}\label{fourierexpand}\int_{\Rb^2}\Cc_j(\lambda_{2j-1},\lambda_{2j},\alpha[\Nc_j^{ch}\backslash Z_j])e^{\pi i(\lambda_{2j-1}+\lambda_{2j})t_1}\,\mathrm{d}\lambda_{2j-1}\mathrm{d}\lambda_{2j}=\widetilde{\Bc}_{\Qc_j,Z_j}(t_1,t_1,\alpha[\Nc_j^{ch}\backslash Z_j]),\end{equation}and this expression does not change even if the variables $\lambda_{2j-1}$ and $\lambda_{2j}$ are switched. Thus
\[\int\widetilde{\Bc}_{\Qc,Z}(t,s,\alpha[\Nc^{ch}\backslash Z])\,\mathrm{d}\alpha_\rf=2\int_0^{\min(t,s)}\prod_{j=1}^3\widetilde{\Bc}_{\Qc_j,Z_j}(t_1,t_1,\alpha[\Nc_j^{ch}\backslash Z_j])\,\mathrm{d}t_1,\]so after integrating in $\alpha[\Nc_j^{ch}\backslash Z_j]\,(1\leq j\leq 3)$ and applying the induction hypothesis we get (\ref{recurtype1}).

\emph{Case 2}. Now assume $\Qc_0$ is a regular double chain. We fix $\Qc$ and $Z$, and define the relevant variables and objects, such as $\alpha_j^\pm$, $Z_{j,\epsilon,\iota}$, $Z^\pm$ and others, in the same way as in Section \ref{recursive} and the proof of Proposition \ref{maincoef} in Section \ref{proofmaincoef}. As in the proof of Proposition \ref{maincancel}, we may neglect any modification procedure described in the proof of Proposition \ref{maincoef}, where some $j\in Z^+$ is moved to $W^+$ and $\widetilde{\alpha_j}=\alpha_j^++\epsilon_j^+ \lambda_a^+$ is replaced by $\alpha_j^+$ in the factor $\frac{\chi_\infty(\widetilde{\alpha_j})}{\epsilon_j^+\pi i\widetilde{\alpha_j}}$, because the difference produced will contribute $0$ to $\Jc\widetilde{\Bc}_{\Qc,Z}$ upon integrating in all the $\alpha_j$ variables, thanks to (\ref{additionvanish}). Therefore, we may omit the $\frac{\chi_\infty(\widetilde{\alpha_j})}{\epsilon_j^+\pi i\widetilde{\alpha_j}}$ factors and focus on the function $\widetilde{\Bc}_{\Qc,Z}(t,s,\alpha[\Nc^{ch}\backslash Z])$. By (\ref{newexp4}) we have \begin{equation}\label{newexp6}
\begin{aligned}
\widetilde{\Bc}_{\Qc,Z}(t,s,\alpha[\Nc^{ch}\backslash Z])&=\int\prod_{\epsilon\in\{\pm\}}\prod_{j=1}^{m^\epsilon}\prod_{\iota=1}^2\Cc_{j,\epsilon,\iota}\big(\lambda_{2j-1,\epsilon,\iota},\lambda_{2j,\epsilon,\iota},\alpha[\Nc_{j,\epsilon,\iota}^{ch}\backslash Z_{j,\epsilon,\iota}]\big)\,\mathrm{d}\lambda_{2j-1,\epsilon,\iota}\lambda_{2j,\epsilon,\iota}\\&\times\int\Cc_{lp}(\lambda_{lp,+},\lambda_{lp,-},\alpha[\Nc_{lp}^{ch}\backslash Z_{lp}])\,\mathrm{d}\lambda_{lp,+}\mathrm{d}\lambda_{lp,-}\\
&\times\widetilde{K}^+(t,\lambda_0^+,\mu_1^+,\cdots,\mu_{m^+}^+,\widetilde{\alpha}^+[W^+])\widetilde{K}^-(s,\lambda_0^-,\mu_1^-,\cdots,\mu_{m^-}^-,\widetilde{\alpha}^-[W^-]).
\end{aligned}
\end{equation} Here $Z_{j,\epsilon,\iota}$ and $Z_{lp}$ are subsets of $\Nc_{j,\epsilon,\iota}^{ch}$ and $\Nc_{lp}^{ch}$ respectively, and $\Cc_{j,\epsilon,\iota}$ and $\Cc_{lp}$ are the functions defined in (\ref{maincoef1}) associated with the couples $\Qc_{j,\epsilon,\iota}$ and $\Qc_{lp}$; again note that the order of the two $\lambda$ variables involved in each $\Cc$ function may be switched, but this will not affect the final result. Moreover, $\widetilde{K}^\pm$, which depends on $t$ (or $s$), $\lambda_0^\pm$ and $(\mu_1^\pm,\cdots,\mu_{m^\pm}^\pm)$ and $\widetilde{\alpha}^\pm[W^\pm]$, are the functions defined in Lemma \ref{regchaincancel2} and Lemma \ref{regchaincal}; here we have $\widetilde{\alpha_j}^\pm=\alpha_j^\pm+\epsilon_j^\pm\lambda_{2j-1}^\pm$ (where $\epsilon_j^\pm=\zeta_{\nf_{2j-1}^\pm}$) and $\mu_j^\pm=\lambda_{2j-1}^\pm+\lambda_{2j}^\pm$ for $1\leq j\leq m^\pm$, and $\lambda_a^\pm=\lambda_{a,\pm,1}+\lambda_{a,\pm,2}$ for $1\leq a\leq 2m^\pm$, and $\lambda_0^\pm=\lambda_{lp,\pm}$. Now, by applying Lemma \ref{regchaincal} to the functions $\widetilde{K}^\pm$, and using the equality (\ref{fourierexpand}) with $(\Qc_j,Z_j)$ replaced by $(\Qc_{j,\epsilon,\iota},Z_{j,\epsilon,\iota})$ and $(\Qc_{lp},Z_{lp})$ to integrate over the the $(\lambda_{a,\epsilon,\iota})$ and $(\lambda_{lp,\pm})$ variables, we get that
\begin{multline}
\label{recurtype2pre}
\int\widetilde{\Bc}_{\Qc,Z}(t,s,\alpha[\Nc^{ch}\backslash Z])\,\mathrm{d}\alpha^+[W^+]\mathrm{d}\alpha^-[W^-]=\int_{t>t_1>\cdots>t_{m^+}>0}\int_{s>s_1>\cdots >s_{m^-}>0}\prod_{j=1}^{m^+}\mathrm{d}t_j\prod_{j=1}^{m^-}\mathrm{d}s_j\\\times\prod_{j=1}^{m^+}\widetilde{\Bc}_{\Qc_{j,+,1}}(t_j,t_j,\alpha[\Nc_{j,+,1}^{ch}\backslash Z_{j,+,1}])\widetilde{\Bc}_{\Qc_{j,+,2}}(t_j,t_j,\alpha[\Nc_{j,+,2}^{ch}\backslash Z_{j,+,2}])\widetilde{\Bc}_{\Qc_{lp}}(t_{m^+},s_{m^-},\alpha[\Nc_{lp}^{ch}\backslash Z_{lp}])\\\times\prod_{j=1}^{m^-}\widetilde{\Bc}_{\Qc_{j,-,1}}(s_j,s_j,\alpha[\Nc_{j,-,1}^{ch}\backslash Z_{j,-,1}])\widetilde{\Bc}_{\Qc_{j,-,2}}(s_j,s_j,\alpha[\Nc_{j,-,2}^{ch}\backslash Z_{j,-,2}]).
\end{multline} Now by integrating over $\alpha[\Nc_{j,\epsilon,\iota}^{ch}\backslash Z_{j,\epsilon,\iota}]$ and $\alpha[\Nc_{lp}^{ch}\backslash Z_{lp}]$ and applying the induction hypothesis, and noticing that $\Jc\widetilde{\Bc}_{\Qc_{lp}}(t,s)=\Jc\widetilde{\Bc}_{\Qc_{lp}}(\min(t,s))$ because $\Qc_{lp}$ is trivial or has type $1$, we obtain (\ref{recurtype2}). Note that in the factor $\widetilde{\Bc}_{\Qc_{lp}}$ in (\ref{recurtype2pre}) the variables $t_{m^+}$ and $s_{m^-}$ may be switched, but this has no effect on the final result due to the symmetry of $\Jc\widetilde{\Bc}_{\Qc_{lp}}(t,s)$ in $t$ and $s$.
\end{proof}
\subsection{Combinatorics of enhanced dominant couples}\label{combine} Finally we put everything together to obtain the full asymptotics. We use $\Qs=(\Qc,Z)$ to denote enhanced dominant couples, where $Z$ is s special subset of $\Nc^{ch}$. 
\begin{prop}\label{domasymp1} We have
\begin{equation}\label{reducedom}\sum_{\substack{n(\Qc)=2n\\\Qc\,\,\mathrm{regular}}}\Kc_\Qc(t,t,k)=\sum_{\Qs}2^{-2n}\delta^n\zeta^*(\Qc)\prod_{\nf\in Z}\frac{1}{\zeta_{\nf}\pi i}\cdot\Jc\widetilde{\Bc}_\Qc(t,t)\cdot \Mc_\Qs^*(k)+\Rs,\end{equation} where $\|\Rs\|_{X_{\mathrm{loc}}^{40d}}\lesssim (C^+\delta)^nL^{-2\nu}$. Here in (\ref{reducedom}), the first summation is taken over all \emph{regular} couples $\Qc$ of scale $2n$, and the second summation is taken over all \emph{enhanced dominant} couples $\Qs=(\Qc,Z)$ of scale $2n$. The quantity $\Mc_\Qs^*(k)=\Mc_{\Qc,Z}^*(k)$ is defined as in (\ref{defintegral}).
\end{prop}
\begin{proof} This follows from combining Propositions \ref{asymptotics1}, \ref{maincancel} and \ref{regcplcal}. Note that the number of choices for $(\Qc,Z)$ is at most $C^n$, so the accumulate error term $\Rs$ still satisfies the same bound as in Proposition \ref{asymptotics1}.
\end{proof}
\begin{prop}\label{multiexp} Let $\Qs=(\Qc,Z)$ be an enhanced dominant couple. Let $\Mc_{\Qs}^*(k)=\Mc_{\Qc,Z}^*(k)$ be defined as in (\ref{defintegral}). Then, the expression $\Mc_\Qs^*(k)$ depends only on the equivalence class $\Xs$ of $\Qs$, so we can denote it by $\Mc_\Xs^*(k)$. Moreover, it satisfies the recurrence relation described as follows. If $\Xs$ is an equivalence class of type 1, then it is uniquely determined by $(\Xs_1,\Xs_2,\Xs_3)$, see Section \ref{summary}. In this case we have
\begin{equation}\label{inductm01}\Mc_\Xs^*(k)=\int_{(\Rb^d)^3}\prod_{j=1}^3\Mc_{\Xs_j}^*(k_j)\dirac(k_1-k_2+k_3-k)\dirac(|k_1|_\beta^2-|k_2|_\beta^2+|k_3|_\beta^2-|k|_\beta^2)\,\mathrm{d}k_1\mathrm{d}k_2\mathrm{d}k_3.\end{equation} Next, if $\Xs$ is an equivalence class of type 2, then it is uniquely determined by $m\geq 1$, the tuples $(\mathtt{I}_j,\mathtt{c}_j,\Xs_{j,1},\Xs_{j,2})$ for $1\leq j\leq m$, and $\Ys$ trivial or of type 1,  see Section \ref{summary}. Then we have
\begin{equation}\label{inductm02}\Mc_\Xs^*(k)=\Mc_\Ys^*(k)\cdot\prod_{j=1}^m\Mc_{(j)}^*(k),\end{equation} where for each $1\leq j\leq m$, if $(\mathtt{I}_j,\mathtt{c}_j)=(0,1)$ we have
\begin{equation}\label{inductm03}\Mc_{(j)}^*(k)=\int_{(\Rb^d)^3}\Mc_{\Xs_{j,1}}^*(k_2)\Mc_{\Xs_{j,2}}^*(k_3)\cdot\dirac(k_1-k_2+k_3-k)\dirac(|k_1|_\beta^2-|k_2|_\beta^2+|k_3|_\beta^2-|k|_\beta^2)\,\mathrm{d}k_1\mathrm{d}k_2\mathrm{d}k_3;\end{equation}if $(\mathtt{I}_j,\mathtt{c}_j)=(0,2)$ we have
\begin{equation}\label{inductm04}\Mc_{(j)}^*(k)=\int_{(\Rb^d)^3}\Mc_{\Xs_{j,1}}^*(k_1)\Mc_{\Xs_{j,2}}^*(k_3)\cdot\dirac(k_1-k_2+k_3-k)\dirac(|k_1|_\beta^2-|k_2|_\beta^2+|k_3|_\beta^2-|k|_\beta^2)\,\mathrm{d}k_1\mathrm{d}k_2\mathrm{d}k_3;\end{equation}if $(\mathtt{I}_j,\mathtt{c}_j)=(0,3)$ we have
\begin{equation}\label{inductm05}\Mc_{(j)}^*(k)=\int_{(\Rb^d)^3}\Mc_{\Xs_{j,1}}^*(k_1)\Mc_{\Xs_{j,2}}^*(k_2)\cdot\dirac(k_1-k_2+k_3-k)\dirac(|k_1|_\beta^2-|k_2|_\beta^2+|k_3|_\beta^2-|k|_\beta^2)\,\mathrm{d}k_1\mathrm{d}k_2\mathrm{d}k_3.\end{equation} If $\mathtt{I}_j=1$ then the corresponding formulas are the same as above, except that the factor \[\dirac(|k_1|_\beta^2-|k_2|_\beta^2+|k_3|_\beta^2-|k|_\beta^2)\] should be replaced by\[\frac{1}{|k_1|_\beta^2-|k_2|_\beta^2+|k_3|_\beta^2-|k|_\beta^2}.\]
\end{prop}
\begin{proof} We prove by induction. The integral (\ref{defintegral}) has two parts: the measure $\mathrm{d}\sigma$, and the integrand
\[\Is(\Qc)=\prod_{\lf\in\Lc^*}^{(+)}n_{\mathrm{in}}(k_\lf)\cdot\prod_{\nf\in\Nc^{ch}\backslash Z}\dirac(\Omega_{\nf})\prod_{\nf\in Z}\frac{1}{\Omega_{\nf}}.\] It is easy to see that if $\Qc$ is formed by the smaller couples $\Qc_j$ as in Definition \ref{defchoice}, then $\Is(\Qc)$ is equal to the product of $\Is(\Qc_j)$, multiplied by the product of the $\dirac(\Omega)$ (if the corresponding $\mathtt{I}_j=0$) or $1/\Omega$ (if $\mathtt{I}_j=1$) factors appearing in (\ref{inductm01})--(\ref{inductm05}), where $\Omega=|k_1|_\beta^2-|k_2|_\beta^2+|k_3|_\beta^2-|k|_\beta^2$. Therefore, to verify the recurrence relation we just need to consider the measure part $\mathrm{d}\sigma$.

Recall the linear submanifold $\Sigma$ and the definition of $\mathrm{d}\sigma$ in (\ref{defintegral}), which we shall denote by $\mathrm{d}\sigma_\Qc$ here. If we choose one leaf from each leaf pair to form a set $\Xc$ (the exact choice can be arbitrary and does not affect the formula), then, as described in Proposition \ref{asymptotics1}, there is a set $\Yc\subset\Xc$ of odd cardinality, such that
\[\mathrm{d}\sigma_\Qc=\dirac\bigg(\sum_{\lf\in\Yc}(\pm k_\lf)-k\bigg)\,\mathrm{d}k[\Xc].\]

Now, suppose $\Qc$ has type 1, which is composed of three dominant couples $\Qc_j\,(1\leq j\leq 3)$; let $(\Xc_j,\Yc_j)$ be associated with $\Qc_j$, then we have $\Xc=\Xc_1\cup\Xc_2\cup\Xc_3$ and $\Yc=\Yc_1\cup\Yc_2\cup\Yc_3$. Then
\[
\begin{aligned}\dirac(k_1-k_2+k_3-k)\mathrm{d}k_1\mathrm{d}k_2\mathrm{d}k_3\prod_{j=1}^3\mathrm{d}\sigma_{\Qc_j}&=\dirac(k_1-k_2+k_3-k)\prod_{j=1}^3\bigg[\dirac\bigg(\sum_{\lf\in\Yc_j}(\pm k_\lf)-k_j\bigg)\mathrm{d}k[\Xc_j]\bigg]\mathrm{d}k_1\mathrm{d}k_2\mathrm{d}k_3\\
&=\dirac\bigg(\sum_{\lf\in\Yc}(\pm k_\lf)-k\bigg)\,\mathrm{d}k[\Xc]=\mathrm{d}\sigma_\Qc.
\end{aligned}\] This can be verified, for example, by integrating any function against the measures.

Suppose $\Qc$ has type 2, we will only consider the case $m=1$, since the general case follows from iteration. Using the notations of Definitions \ref{defchoice} and \ref{equivcpl}, suppose $(m^+,m^-)=(1,0)$ and $\mathtt{c}_1=1$ (the other cases are similar), and denote $(\Qc_{lp}, \Qc_{1,+,1},\Qc_{1,+,2})=(\Qc_1,\Qc_2,\Qc_3)$, then $\Xc=\Xc_1\cup\Xc_2\cup\Xc_3$ and $\Yc=\Yc_1$, hence
\[\begin{aligned}\mathrm{d}k_2\mathrm{d}k_3\prod_{j=1}^3\mathrm{d}\sigma_{\Qc_j}&=\prod_{j=2}^3\bigg[\dirac\bigg(\sum_{\lf\in\Yc_j}(\pm k_\lf)-k_j\bigg)\mathrm{d}k[\Xc_j]\bigg]\mathrm{d}k_2\mathrm{d}k_3\cdot \dirac\bigg(\sum_{\lf\in\Yc_1}(\pm k_\lf)-k\bigg)\mathrm{d}k[\Xc_1]\\
&=\dirac\bigg(\sum_{\lf\in\Yc}(\pm k_\lf)-k\bigg)\,\mathrm{d}k[\Xc]=\mathrm{d}\sigma_\Qc.
\end{aligned}\] Therefore the measure $\mathrm{d}\sigma_\Qc$ satisfies the desired recurrence relation, so the result is proved.
\end{proof}
\begin{prop}\label{nonemptyZ} Let $\Xs$ be an equivalence class of enhanced dominant couples such that for $\Qs=(\Qc,Z)\in\Xs$ we have $Z\neq\varnothing$. Then we have
\begin{equation}
\label{nonemptyZ1}\sum_{\Qs=(\Qc,Z)\in\Xs}\bigg(\prod_{\nf\in Z}\frac{1}{\zeta_\nf\pi i}\bigg)\cdot \Jc\widetilde{\Bc}_\Qc(t,t)=0.
\end{equation}
\end{prop}
\begin{proof} First $|Z|$ is constant for all $\Qs\in\Xs$, so we may replace the product in (\ref{nonemptyZ1}) by $\prod_{\nf\in Z}\zeta_\nf$. Denote this reduced sum by $\Gc_\Xs(t)$. We prove (\ref{nonemptyZ1}) by induction. The base case is simple. Suppose (\ref{nonemptyZ1}) is true for $\Xs$ of smaller half-scale. Let $\Xs$ be composed from smaller equivalence classes $\Xs_j$ as in Section \ref{summary}, then by definition of equivalence, the summation over $\Qs=(\Qc,Z)\in\Xs$ must contain (among other things) a sub summation over $\Qs_j=(\Qc_j,Z_j)\in\Xs_j$, so in particular $\Gc_\Xs$ equals a multilinear expression of the quantities $\Gc_{\Xs_j}$ (see also (\ref{splitformula0}) below). Therefore, by induction hypothesis, we may assume that $Z_j=\varnothing$ for each $\Xs_j$, and $Z\neq\varnothing$. In particular, $\Xs$ must have type 2. By the structure of dominant double chains, it is easy to see that $\zeta_\nf=\epsilon$ if $\nf\in Z\cap \Tc^\epsilon$ with $\epsilon\in\{\pm\}$. Let $m\geq 1$, the tuples $(\mathtt{I}_j,\mathtt{c}_j,\Xs_{j,1},\Xs_{j,2})$, and $\Ys$ be fixed as in Section \ref{summary}.

If $\Qs\in\Xs$, then $\Qs_{lp}\in\Ys$, and we may decompose $m=m^++m^-$, such that the tuples $(\mathtt{I}_{j,\epsilon},\mathtt{c}_{j,\epsilon},\Xs_{j,\epsilon,1},\Xs_{j,\epsilon,2})$ where $\epsilon\in\{\pm\}$, $1\leq j\leq m^\epsilon$ and $\Xs_{j,\epsilon,\iota}$ is the equivalence class of $\Qc_{j,\epsilon,\iota}$, form a permutation of $(\mathtt{I}_j,\mathtt{c}_j,\Xs_{j,1},\Xs_{j,2})$ where $1\leq j\leq m$. Moreover, since $\mathtt{c}_{j,\epsilon}$ are just the \emph{first digits} of the codes of the mini trees appearing in the structure of $\Qc$, the corresponding \emph{second digits} can be arbitrary (and $\Bc_\Qc$ does not depend on this second digit) which results in a $2^m$ factor. Apart from this, we apply Proposition \ref{regcplcal} and sum over all possible $\Qc$'s---which means summing over all permutations of the tuples and then summing over all possible $\Qc_{j,\epsilon,\iota}$ and $\Qc_{lp}$---to get
\begin{equation} \label{splitformula0}
\begin{aligned}
\Gc_\Xs(t)&=2^m\sum_{m^++m^-=m}\sum_{(\As_1,\cdots, \As_{m^+},\Bs_1,\cdots,\Bs_{m^-})}\int_{t>t_1>\cdots >t_{m^+}>0}\int_{t>s_1>\cdots>s_{m^-}>0}\prod_{j=1}^{m^-}(-1)^{\mathtt{I}_{j,-}'}\\&\times\prod_{j=1}^{m^+}\Ms(\As_j)(t_j)\prod_{j=1}^{m^-}\Ms(\Bs_j)(s_j)\cdot \Gc_{\Ys}(\min(t_{m^+},s_{m^-}))\,\mathrm{d}t_1\cdots\mathrm{d}t_{m^+}\mathrm{d}s_1\cdots\mathrm{d}s_{m^-}.
\end{aligned}
\end{equation} Here in (\ref{splitformula0}) the summation is taken over all permutations $(\As_1,\cdots, \As_{m^+},\Bs_1,\cdots,\Bs_{m^-})$ of the tuples $(\mathtt{I}_j,\mathtt{c}_j,\Xs_{j,1},\Xs_{j,2})$. Moreover $\mathtt{I}_{j,-}'$ represents the first component of $\Bs_j$, the function $\Ms(\As_j)$ is $\Gc_{\Xs_{j,+,1}'}\cdot\Gc_{\Xs_{j,+,2}'}$ where $(\Xs_{j,+,1}',\Xs_{j,+,2}')$ represents the last two components of $\As_j$, and $\Ms(\Bs_j)$ is defined similarly.

Now fix $m^+$ and $m^-$, and consider the $m$ variables $t_1,\cdots,t_{m^+},s_1,\cdots,s_{m-}\in[0,t]$. If we fix a total ordering to these variables, then under the assumptions $t_1>\cdots>t_{m^+}$ and $s_1>\cdots >s_{m^-}$, each total ordering can be uniquely represented by a partition $(A,B)$ of $\{1,\cdots,m\}$ into an $m^+$ element subset $A$ and an $m^-$ element subset $B$. Once this total ordering is fixed, we may rearrange these variables as $t>u_1>\cdots >u_m>0$, then this term on the right hand side of (\ref{splitformula0}) becomes
\begin{equation}\label{splitformula1}\sum_{(\Cs_1,\cdots,\Cs_m)}\int_{t>u_1>\cdots>u_m>0}\prod_{j\in B}(-1)^{\mathtt{I}_j'}\prod_{j=1}^m\Ms(\Cs_j)(u_j)\cdot \Gc_\Ys(u_m)\,\mathrm{d}u_1\cdots\mathrm{d}u_m,\end{equation} where the summation is taken over all permutations $(\Cs_1,\cdots,\Cs_m)$ of the tuples $(\mathtt{I}_j,\mathtt{c}_j,\Xs_{j,1},\Xs_{j,2})$, $\mathtt{I}_j'$ represents the first component of $\Cs_j$, and the function $\Ms(\Cs_j)$ is $\Gc_{\Xs_{j,1}'}\cdot\Gc_{\Xs_{j,2}'}$ where $(\Xs_{j,1}',\Xs_{j,2}')$ represents the last two components of $\Cs_j$. After summing over $(A,B)$ and $(m^+,m^-)$, we obtain that
\begin{equation}\label{splitformula2}\Gc_\Xs(t)=2^m\sum_{(\Cs_1,\cdots,\Cs_m)}\int_{t>u_1>\cdots>u_m>0}\prod_{j=1}^m\Ms(\Cs_j)(u_j)\cdot \Gc_\Ys(u_m)\,\mathrm{d}u_1\cdots\mathrm{d}u_m\cdot\bigg[\sum_B\prod_{j\in B}(-1)^{\mathtt{I}_j'}\bigg],\end{equation} where the inner summation is taken over \emph{all} subsets $B\subset\{1,\cdots,m\}$. Since $Z\neq\varnothing$, we know that at least one $1\leq j\leq m$ is such that $\mathtt{I}_j'=1$, which implies that \[\sum_{B\subset\{1,\cdots,m\}}\prod_{j\in B}(-1)^{\mathtt{I}_j'}=\prod_{j=1}^m(1+(-1)^{\mathtt{I}_j'})=0,\] {where we understand the product is $1$ if $B=\varnothing$}. Therefore $\Gc_\Xs(t)=0$ and the proof is complete.
\end{proof}
\subsubsection{Expansions of the solution to (\ref{wke})} Now we can match the nonzero leading correlations, which come from the (enhanced) dominant couples with $Z=\varnothing$, with the terms in the Taylor expansion of the solution to (\ref{wke}).
\begin{prop}\label{wkelwp} Let $\delta$ be small enough depending on $n_{\mathrm{in}}$. Then the equation (\ref{wke}) has a unique solution $n=n(t,k)$ for $t\in[0,\delta]$. The solution has a convergent Taylor expansion
\begin{equation}\label{taylorexp}n(\delta t,k)=\sum_{n=0}^\infty \Mc_n(t,k),\quad |\Mc_n(t,k)|\lesssim(C^+\delta)^n\end{equation} for $t\in[0,1]$, where $\Mc_n(t,k)$ is defined by (\ref{iterate}). This $\Mc_n(k)$ can be expanded as 
\begin{equation}\label{expandmnk}\Mc_n(t,k)=\delta^n\sum_{n(\Tc)=n}\zeta^*(\Tc)\cdot g_\Tc(t)\cdot \widetilde{\Mc}_\Tc(k),\end{equation} where the summation is taken over all \emph{encoded} trees of scale $n$. The sign $\zeta^*(\Tc)$ is defined in (\ref{defzetat}), the function $g_\Tc(t)$ is defined inductively by
\begin{equation}\label{inductg}g_{\bullet}(t)=1,\quad g_\Tc(t)=\int_0^t g_{\Tc_1}(t')g_{\Tc_2}(t')g_{\Tc_3}(t')\,\mathrm{d}t',\end{equation} and the expression $\widetilde{\Mc}_\Tc(k)$ is defined inductively as follows. First if $\Tc=\bullet$ then define $\widetilde{\Mc}_\bullet(k)=n_{\mathrm{in}}(k)$. Now let $(\Tc_1,\Tc_2,\Tc_3)$ be the three subtrees of $\Tc$ from left to right. Then, if $\mathtt{c}_\rf=0$ where $\rf$ is the root of $\Tc$, we define
\begin{equation}\label{inductm1}\widetilde{\Mc}_\Tc(k)=\int_{(\Rb^d)^3}\widetilde{\Mc}_{\Tc_1}(k_1)\widetilde{\Mc}_{\Tc_2}(k_2)\widetilde{\Mc}_{\Tc_3}(k_3)\dirac(k_1-k_2+k_3-k)\dirac(|k_1|_\beta^2-|k_2|_\beta^2+|k_3|_\beta^2-|k|_\beta^2)\,\mathrm{d}k_1\mathrm{d}k_2\mathrm{d}k_3.\end{equation} If $\mathtt{c}_\rf=1$, we define
\begin{equation}\label{inductm2}\widetilde{\Mc}_\Tc(k)=\int_{(\Rb^d)^3}\widetilde{\Mc}_{\Tc_1}(k)\widetilde{\Mc}_{\Tc_2}(k_2)\widetilde{\Mc}_{\Tc_3}(k_3)\dirac(k_1-k_2+k_3-k)\dirac(|k_1|_\beta^2-|k_2|_\beta^2+|k_3|_\beta^2-|k|_\beta^2)\,\mathrm{d}k_1\mathrm{d}k_2\mathrm{d}k_3.\end{equation}  If $\mathtt{c}_\rf=2$ we define
\begin{equation}\label{inductm3}\widetilde{\Mc}_\Tc(k)=\int_{(\Rb^d)^3}\widetilde{\Mc}_{\Tc_1}(k_1)\widetilde{\Mc}_{\Tc_2}(k)\widetilde{\Mc}_{\Tc_3}(k_3)\dirac(k_1-k_2+k_3-k)\dirac(|k_1|_\beta^2-|k_2|_\beta^2+|k_3|_\beta^2-|k|_\beta^2)\,\mathrm{d}k_1\mathrm{d}k_2\mathrm{d}k_3.\end{equation} If $\mathtt{c}_\rf=3$ we define
\begin{equation}\label{inductm4}\widetilde{\Mc}_\Tc(k)=\int_{(\Rb^d)^3}\widetilde{\Mc}_{\Tc_1}(k_1)\widetilde{\Mc}_{\Tc_2}(k_2)\widetilde{\Mc}_{\Tc_3}(k)\dirac(k_1-k_2+k_3-k)\dirac(|k_1|_\beta^2-|k_2|_\beta^2+|k_3|_\beta^2-|k|_\beta^2)\,\mathrm{d}k_1\mathrm{d}k_2\mathrm{d}k_3.\end{equation}

The expression $\widetilde{\Mc}_\Tc(k)$ depends only on the equivalence class of $\Tc$, so we may denote it by $\widetilde{\Mc}_\Xs(k)$. For any $\Xs$, if $\Xs$ has type 1 and is determined by $(\Xs_1,\Xs_2,\Xs_3)$ as above, then we have
\begin{equation}\label{inductm5}\widetilde{\Mc}_\Xs(k)=\int_{(\Rb^d)^3}\prod_{j=1}^3\widetilde{\Mc}_{\Xs_j}(k_j)\cdot\dirac(k_1-k_2+k_3-k)\dirac(|k_1|_\beta^2-|k_2|_\beta^2+|k_3|_\beta^2-|k|_\beta^2)\,\mathrm{d}k_1\mathrm{d}k_2\mathrm{d}k_3.\end{equation} If $\Xs$ has type 2 and is determined by a positive integer $m$, triples $(\mathtt{c}_j,\Xs_{j,1},\Xs_{j,2})$ where $1\leq j\leq m$, and $\Ys$, then we have
\begin{equation}\label{inductm6}\widetilde{\Mc}_\Xs(k)=\widetilde{\Mc}_\Ys(k)\cdot\prod_{j=1}^m\widetilde{\Mc}_{(j)}(k),\end{equation} where for each $1\leq j\leq m$, if $\mathtt{c}_j=1$ we have
\begin{equation}\label{inductm7}\widetilde{\Mc}_{(j)}(k)=\int_{(\Rb^d)^3}\widetilde{\Mc}_{\Xs_{j,1}}(k_2)\widetilde{\Mc}_{\Xs_{j,2}}(k_3)\cdot\dirac(k_1-k_2+k_3-k)\dirac(|k_1|_\beta^2-|k_2|_\beta^2+|k_3|_\beta^2-|k|_\beta^2)\,\mathrm{d}k_1\mathrm{d}k_2\mathrm{d}k_3;\end{equation}if $\mathtt{c}_j=2$ we have
\begin{equation}\label{inductm8}\widetilde{\Mc}_{(j)}(k)=\int_{(\Rb^d)^3}\widetilde{\Mc}_{\Xs_{j,1}}(k_1)\widetilde{\Mc}_{\Xs_{j,2}}(k_3)\cdot\dirac(k_1-k_2+k_3-k)\dirac(|k_1|_\beta^2-|k_2|_\beta^2+|k_3|_\beta^2-|k|_\beta^2)\,\mathrm{d}k_1\mathrm{d}k_2\mathrm{d}k_3;\end{equation}if $\mathtt{c}_j=3$ we have
\begin{equation}\label{inductm9}\widetilde{\Mc}_{(j)}(k)=\int_{(\Rb^d)^3}\widetilde{\Mc}_{\Xs_{j,1}}(k_1)\widetilde{\Mc}_{\Xs_{j,2}}(k_2)\cdot\dirac(k_1-k_2+k_3-k)\dirac(|k_1|_\beta^2-|k_2|_\beta^2+|k_3|_\beta^2-|k|_\beta^2)\,\mathrm{d}k_1\mathrm{d}k_2\mathrm{d}k_3.\end{equation} Moreover, for any equivalence class $\Xs$ of dominant couples or encoded trees, we have $\widetilde{\Mc}_\Xs(k)=\Mc_\Xs^*(k)$.
\end{prop}
\begin{proof} This follows from direct calculation. First, let $\Mc_n(t,k)$ be defined by (\ref{iterate}), then the formula (\ref{expandmnk}) follows from induction. Here one notes that (i) the four cases in the recurrence relation (\ref{inductm1})--(\ref{inductm4}) defining $\widetilde{\Mc}_\Tc(k)$ exactly correspond to iterating the four different terms in the nonlinearity (\ref{wke2}), (ii) the recurrence definition (\ref{inductg}) of $g_\Tc(t)$ corresponds to applying the Duhamel formula for (\ref{wke}), and (iii) the sign $\zeta^*(\Tc)$ is uniquely determined by iterating the signs of the four terms in (\ref{wke2}).

Next, with the inductive definition (\ref{inductm1})--(\ref{inductm4}) of $\widetilde{\Mc}_\Tc(k)$, it is easy to see that (\ref{inductm5})--(\ref{inductm9}) hold. In fact, (\ref{inductm5}) is just (\ref{inductm1}), and (\ref{inductm6}) for general $m$ follows from iterating the $m=1$ case, while the three possibilities (\ref{inductm7})--(\ref{inductm9}) are just (\ref{inductm2})--(\ref{inductm4}). Since the expression (\ref{inductm6}) is invariant under permuting the different indices $1\leq j\leq m$, we can inductively prove that $\widetilde{\Mc}_\Tc(k)$ does not change if $\Tc$ is replaced by an equivalent encoded tree, so we can replace $\widetilde{\Mc}_\Tc$ by $\widetilde{\Mc}_\Xs$.

Next, let $\Xs$ be an equivalence class of dominant couples or encoded trees. For dominant couples $\Qc$ we assume $Z=\varnothing$, so in particular all the $\mathtt{I}_j$ variables (as in Section \ref{summary}) appearing in the inductive step will be $0$. As a result, the recurrence relations (\ref{inductm01})--(\ref{inductm05}) for $\Mc_{\Xs}^*(k)$ do not contain any $1/\Omega$ factor (only $\dirac(\Omega)$), and thus coincide with (\ref{inductm1})--(\ref{inductm5}). This shows $\widetilde{\Mc}_\Xs(k)=\Mc_\Xs^*(k)$. Finally, as in Remark \ref{holder} we have $|\widetilde{\Mc}_\Tc(k)|\lesssim (C^+)^n\langle k\rangle^{-40d}$ if $\Tc$ has scale $n$. Since the number of encoded trees of scale $n$ is at most $C^n$, and $g_\Tc(t)$ is homogeneous in $t$ and can easily be bounded in some smooth norm, we see that $\|\Mc_n(t,k)\|_{X_{\mathrm{loc}}^{40d}}\lesssim (C^+\delta)^n$, which proves the convergence of (\ref{taylorexp}).
\end{proof}
\begin{prop}\label{emptyZ} Let $\Xs$ be as in Proposition \ref{nonemptyZ}, but assume $Z=\varnothing$ for $\Qs=(\Qc,Z)\in\Xs$; for simplicity we write $\Qs=(\Qc,\varnothing)$ simply as $\Qc$. Then for any equivalence class $\Xs$ of half-scale $n$ we have
\begin{equation}\label{emptyZ1}\sum_{\Qc\in\Xs} \Jc\widetilde{\Bc}_\Qc(t,t)=2^{2n}\sum_{\Tc\in\Xs}g_\Tc(t).
\end{equation} 
\end{prop}
\begin{proof} Define
\[G_\Xs(t)=\sum_{\Tc\in\Xs}g_\Tc(t),\] then by definition of equivalence and the recurrence relation (\ref{inductg}) of $g_\Tc(t)$, we can show that if $\Xs$ has type 1, then
\begin{equation}\label{inductG1}G_\Xs(t)=\int_0^t \prod_{j=1}^3 G_{\Xs_j}(t')\,\mathrm{d}t'.\end{equation} If $\Xs$ has type 2, then
\begin{equation}\label{inductG2}G_\Xs(t)=\sum_{(\As_1,\cdots,\As_m)}\int_{t>t_1>\cdots>t_m>0}\prod_{j=1}^m G_{\Xs_{j,1}'}(t_j)G_{\Xs_{j,2}'}(t_j)\cdot G_\Ys(t_m)\,\mathrm{d}t_1\cdots\mathrm{d}t_m,\end{equation} where the sum is taken over all permutations $(\As_1,\cdots,\As_m)$ of the triples $(\mathtt{c}_j,\Xs_{j,1},\Xs_{j,2})_{1\leq j\leq m}$ (in particular the number of terms in this summation varies, depending on whether some of the triples coincide or not), and $(\Xs_{j,1}',\Xs_{j,2}')$ represents the last two components of $\As_j$.

In order to prove (\ref{emptyZ1}), as the base case is easily verified, it will suffice to show that the quantity
\[\Gc_\Xs(t):=\sum_{\Qc\in\Xs} \Jc\widetilde{\Bc}_\Qc(t,t)\] satisfies the same recurrence relation (\ref{inductG1})--(\ref{inductG2}), but with an extra factor of $2^2$ on the right hand side of (\ref{inductG1}), and an extra factor of $2^{2m}$ on the right hand side of (\ref{inductG2}).

The case when $\Xs$ is type $1$ is in fact quite easy, as the recurrence relation satisfied by $\Jc\widetilde{\Bc}_\Qc(t,t)$, namely (\ref{recurtype1}), has the same form as (\ref{inductG1}) assuming $t=s$. If one sums over all $\Qc\in\Xs$, which is equivalent to summing over all $\Qc_j\in\Xs_j$ for $1\leq j\leq 3$, one gets the same recurrence relation for $\Gc_\Xs(t)$ in place of $\Jc\widetilde{\Bc}_\Qc(t,t)$. The factor of $2^2$---instead of $2$ on the right hand side of (\ref{recurtype1})---comes from the two possible codes (i.e. $00$ or $01$) for the $(1,1)$-mini couple forming the structure of $\Qc$.

From now on we assume $\Xs$ has type 2. Let $m\geq 1$, the triples $(\mathtt{c}_j,\Xs_{j,1},\Xs_{j,2})$ where $1\leq j\leq m$, and $\Ys$ be fixed as in Section \ref{summary}. We can argue in essentially the same way as in the proof of Proposition \ref{nonemptyZ}, except that (i) now the $\As_j$, $\Bs_j$ and $\Cs_j$ only contain three components, for example $\Cs_j=(\mathtt{c}_j',\Xs_{j,1}',\Xs_{j,2}')$ as $\mathtt{I}_j'$ is always $0$, and (ii) we do not have the factors $(-1)^{\mathtt{I}_{j,-}'}$ in (\ref{splitformula0}) or $(-1)^{\mathtt{I}_{j}'}$ in (\ref{splitformula1}). Therefore, we do not have the cancellation as in Proposition \ref{nonemptyZ}, instead we have
\[\Gc_\Xs(t)=2^m\sum_{(\Cs_1,\cdots,\Cs_m)}\int_{t>u_1>\cdots>u_m>0}\prod_{j=1}^m\Ms(\Cs_j)(u_j)\cdot \Gc_\Ys(u_m)\,\mathrm{d}u_1\cdots\mathrm{d}u_m\cdot\bigg(\sum_B1\bigg),\] where again the inner summation is taken over all subsets $B\subset\{1,\cdots,m\}$. In this way we get $\Gc_\Xs(t)=2^{2m}\Ks$, where $\Ks$ is exactly the right hand side of (\ref{inductG2}) with the $G$ quantities replaced by $\Gc$ quantities. This verifies the recurrence relation satisfied by $\Gc$, and completes the proof.
\end{proof}
\begin{prop}\label{regcplapprox} For $n\leq N^3$, we have
\[\bigg\|\sum_{\substack{n(\Qc)=2n\\\Qc\,\,\mathrm{regular}}}\Kc_\Qc(t,t,k)-\Mc_n(t,k)\bigg\|_{X_{\mathrm{loc}}^{40d}}\lesssim (C^+\delta)^nL^{-2\nu}.\]
\end{prop}
\begin{proof} This follows from Propositions \ref{domasymp1}, \ref{multiexp}, \ref{nonemptyZ}, \ref{wkelwp} and \ref{emptyZ}.
\end{proof}
\section{Non-regular couples I: cancellation of irregular chains}\label{irchaincancel} We now turn to the study of non-regular couples, until the end of Section \ref{l1coef}. Since regular couples have been studied in Section \ref{regasymp}--\ref{domasymp}, in view of Proposition \ref{structure2}, we can reduce any non-regular couple $\Qc$ to its skeleton $\Qc_{sk}$, which is a nontrivial prime couple. Then, we will focus on the study of prime couples.
\subsection{From general to prime couples} Let $\Qc$ be a non-regular couple with skeleton $\Qc_{sk}$, then $\Qc_{sk}\neq\times$ is a prime couple. By Proposition \ref{structure2}, $\Qc$ is formed from $\Qc_{sk}$ in a unique way by replacing each branching node $\mf$ with a regular tree $\Tc^{(\mf)}$ and each leaf pair $\{\mf,\mf'\}$ with a regular couple $\Qc^{(\mf,\mf')}$. Using the results of Section \ref{regasymp}, we shall reduce $\Kc_\Qc(t,s,k)$ to an expression that has similar form with $\Kc_{\Qc_{sk}}(t,s,k)$.

In fact, by (\ref{defkq}) and (\ref{defcoefb2}) we have
\begin{equation}\label{bigformula}\Kc_\Qc(t,s,k)=\bigg(\frac{\delta}{2L^{d-1}}\bigg)^n\zeta^*(\Qc)\sum_{\Es}\int_\Ec\epsilon_\Es\prod_{\nf\in \Nc^*} e^{\zeta_\nf\pi i\cdot\delta L^2\Omega_\nf t_\nf}\,\mathrm{d}t_\nf{\prod_{\lf\in\Lc^*}^{(+)}n_{\mathrm{in}}(k_\lf)},
\end{equation}  where $n$ is the scale of $\Qc$, $\Ec$ is the domain defined in (\ref{timedom}), $\Es$ is a $k$-decoration and other objects are defined as before, all associated to the couple $\Qc$. By definition, the restriction of $\Es$ to nodes in $\Qc_{sk}$ forms a $k$-decoration of $\Qc_{sk}$, and the relevant quantities such as $\Omega_\nf$ are the same for both decorations (i.e. each $\Omega_\nf$ in the decoration of $\Qc_{sk}$ uniquely corresponds to some $\Omega_\nf$ in the corresponding decoration of $\Qc$).

Now, let $\{\mf,\mf'\}$ be a leaf pair in $\Qc_{sk}$, which becomes the roots of the regular sub-couple $\Qc^{(\mf,\mf')}$ in $\Qc$. We must have $k_\mf=k_{\mf'}$. In (\ref{bigformula}), consider the summation in the variables $k_\nf$, where $\nf$ runs over all nodes in $\Qc^{(\mf,\mf')}$ other than $\mf$ and $\mf'$ (these variables, together with $k_\mf$ and $k_{\mf'}$, form a $k_\mf$-decoration of $\Qc^{(\mf,\mf')}$), and the integration in the variables $t_\nf$, where $\nf$ runs over all branching nodes in $\Qc^{(\mf,\mf')}$, with all the other variables fixed. By definition, this summation and integration equals, up to some sign $\zeta^*(\Qc^{(\mf,\mf')})$ and some power of $\delta(2L^{d-1})^{-1}$, the \emph{exact} expression $\Kc_{\Qc^{(\mf,\mf')}}(t_{\mf^p},t_{(\mf')^p},k_\mf)$. Here we assume $\zeta_\mf=+$ and $\zeta_{\mf'}=-$, and $\mf^p$ is the parent of $\mf$ (if $\mf$ is the root then $t_{\mf^p}$ should be replaced by $t$; similarly for $(\mf')^p$).

Similarly, let $\mf$ be a branching node in $\Qc_{sk}$, which becomes the root $\pf$ and lone leaf $\qf$ of a regular tree $\Tc^{(\mf)}$ in $\Qc$. We must have $k_\pf=k_\qf$. In (\ref{bigformula}), consider the summation in the variables $k_\nf$, where $\nf$ runs over all nodes in $\Tc^{(\mf)}$ other than $\pf$ and $\qf$ (these variables, together with $k_\pf$ and $k_{\qf}$, form a $k_\mf$-decoration of $\Tc^{(\mf)}$, where $k_\mf=k_\pf=k_\qf$), and the integration in the variables $t_\nf$, where $\nf$ runs over all branching nodes in $\Tc^{(\mf)}$, with all the other variables fixed. In the same way, this summation and integration equals, up to some sign $\widetilde{\zeta}(\Tc^{(\mf)})$ and some power of $\delta(2L^{d-1})^{-1}$, the \emph{exact} expression $\Kc_{\Tc^{(\mf)}}^*(t_{\pf^p},t_{\qf},k_\pf)$. Here $\pf^p$ is the parent of $\pf$ (again, if $\pf$ is the root then $t_{\pf^p}$ should be replaced by $t$ or $s$) and the relevant notations are defined as in Proposition \ref{varregtree}.

After performing this reduction for each leaf pair and branching node of $\Qc_{sk}$, we can reduce the summation in (\ref{bigformula}) to the summation in $k_\mf$ for all leaves and branching nodes $\mf$ of $\Qc_{sk}$, i.e. a $k$-decoration of $\Qc_{sk}$. Moreover, we can reduce  the integration in (\ref{bigformula}) to the integration in $t_\mf$ for all branching nodes $\mf$ of $\Qc_{sk}$ (for a regular tree, the time variables $t_{\pf^p}$ and $t_{\qf}$ for $\Qc$ correspond to $t_{\mf^p}$ and $t_\mf$ for $\Qc_{sk}$ where $\mf^p$ is the parent of $\mf$). This implies that
\begin{multline}\label{bigformula2}\Kc_\Qc(t,s,k)=\bigg(\frac{\delta}{2L^{d-1}}\bigg)^{n_0}\zeta^*(\Qc_{sk})\sum_{\Es_{sk}}\int_{\Ec_{sk}}\epsilon_{\Es_{sk}}\prod_{\nf\in \Nc_{sk}^*} e^{\zeta_\nf\pi i\cdot\delta L^2\Omega_\nf t_\nf}\,\mathrm{d}t_\nf\\\times{\prod_{\mf\in\Lc_{sk}^*}^{(+)}\Kc_{\Qc^{(\mf,\mf')}}(t_{\mf^p},t_{(\mf')^p},k_\mf)}\prod_{\mf\in\Nc_{sk}^*}\Kc_{\Tc^{(\mf)}}^*(t_{\mf^p},t_{\mf},k_\mf),
\end{multline} where $n_0$ is the scale of $\Qc_{sk}$, $\Ec_{sk}$ is the domain defined in (\ref{timedom}), $\Es_{sk}$ is a $k$-decoration, the other objects are as before but associated to the couple $\Qc_{sk}$. Moreover in (\ref{bigformula2}), { the first product is taken over all leaves $\mf$ of sign $+$ with $\mf'$ being the leaf paired to $\mf$}, the second product is taken over all branching nodes $\mf$, and $\mf^p$ is the parent of $\mf$.

Using Propositions \ref{asymptotics1} and \ref{varregtree}, in (\ref{bigformula2}) we can decompose
\begin{equation}\label{inputdecomp}\Kc_{\Qc^{(\mf,\mf')}}=(\Kc_{\Qc^{(\mf,\mf')}})_{\mathrm{app}}+\Rs_{\Qc^{(\mf,\mf')}},\quad \Kc_{\Tc^{(\mf)}}^*=(\Kc_{\Tc^{(\mf)}}^*)_{\mathrm{app}}+\Rs_{\Tc^{(\mf)}}^*.
\end{equation} Here $(\Kc_{\Qc^{(\mf,\mf')}})_{\mathrm{app}}$ is the leading term in Proposition \ref{asymptotics1}, and is a linear combination of functions of $(t,s)$ multiplied by functions of $k$, which in turn satisfy (\ref{holderbound}) and the $X_{\mathrm{loc}}$ bound in Remark \ref{timebound}; the remainder $\Rs_{\Qc^{(\mf,\mf')}}$ is bounded in $X_{\mathrm{loc}}^{40d}$ with extra gain $L^{-\nu}$ as in Proposition \ref{asymptotics1}. The terms $(\Kc_{\Tc^{(\mf)}}^*)_{\mathrm{app}}$ and $\Rs_{\Tc^{(\mf)}}^*$ are as in Proposition \ref{varregtree}, and satisfy the bound (\ref{vardecomp}). 

We may fix a \emph{mark} in $\{\Lf,\Rf\}$ for each leaf pair and each branching node in $\Qc_{sk}$ which indicates whether we select the \emph{leading} term $(\cdots)_{\mathrm{app}}$ or the \emph{remainder} term $\Rs$ or $\Rs^*$; for a general couple $\Qc$ we can do the same but only for the nodes of its skeleton $\Qc_{sk}$. In this way we can define \emph{marked} couples, which we still denote by $\Qc$, and expressions of form (\ref{bigformula2}) but with $\Kc_{\Qc^{(\mf,\mf')}}$ and $\Kc_{\Tc^{(\mf)}}^*$ replaced by the corresponding leading or remainder terms, which we still denote by $\Kc_{\Qc}$. By definition, any sum of $\Kc_\Qc$ over unmarked couples $\Qc$ equals the corresponding sum over marked couples $\Qc$ for all possible unmarked couples and all possible markings.

Using the relevant $X_{\mathrm{loc}}^{40d}$ and $X_{\mathrm{loc}}^0$ bounds (which control weighted $L^1$ norms in time Fourier variables), we can expand the $(\cdots)_{\mathrm{app}}$ and $\Rs$ (or $\Rs^*$) factors as a Fourier integral in $(t_{\mf^p},t_{(\mf')^p})$ (or  $(t_{\mf^p},t_{\mf})$), which reduces (\ref{bigformula2}) to a formula of form similar to (\ref{defkq}) for $\Qc_{sk}$, but with the $\Omega_\nf$ variables appearing in $\Bc_{\Qc_{sk}}$ suitably shifted, $n_\mathrm{in}$ replaced by factors coming from $(\Kc_{\Qc^{(\mf,\mf')}})_{\mathrm{app}}$ or $\Rs_{\Qc^{(\mf,\mf')}}$, and with extra factors coming  from $(\Kc_{\Tc^{(\mf)}}^*)_{\mathrm{app}}$ or $\Rs_{\Tc^{(\mf)}}^*$ included. Before doing so, however, we need to exploit the cancellation between $\Kc_\Qc$ for some different couples $\Qc$ with specific symmetries. Such cancellation is linked to the notion of \emph{irregular chains}, which we now introduce.
\subsection{Irregular chains and congruence} We now introduce the main object that causes difficulty in the analysis of $\Qc_{sk}$, namely the irregular chains.
\begin{df}[Irregular chains]\label{irrechain} Given a couple $\Qc$ (or a paired tree $\Tc$), we define an \emph{irregular chain} to be a sequence of nodes $(\nf_0,\cdots,\nf_q)$, such that (i) $\nf_{j+1}$ is a child of $\nf_j$ for $0\leq j\leq q-1$, and the other two children of $\nf_j$ are leaves, and (ii) for $0\leq j\leq q-1$, there is a child $\mf_j$ of $\nf_j$, which has opposite sign with $\nf_{j+1}$, and is paired (as a leaf) to a child $\pf_{j+1}$ of $\nf_{j+1}$. We also define $\pf_0$ to be the child of $\nf_0$ other than $\nf_1$ and $\mf_0$. See Figure \ref{fig:irrechain}.
  \begin{figure}[h!]
  \includegraphics[scale=.5]{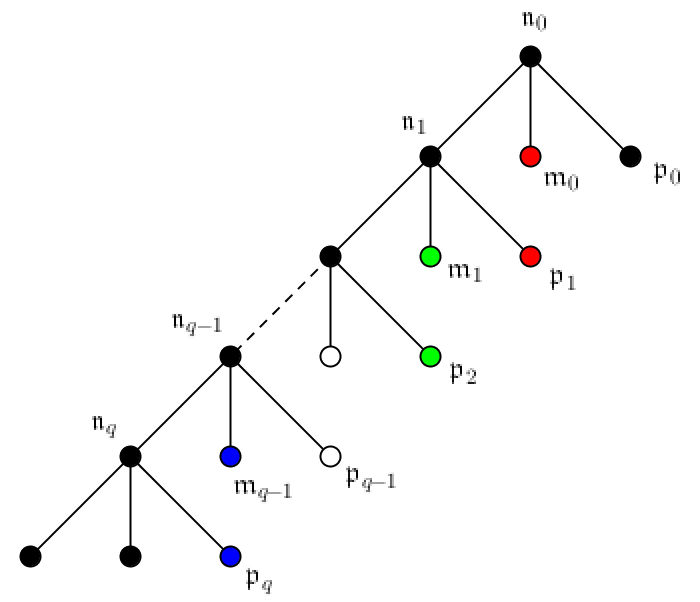}
  \caption{An irregular chain, as in Definition \ref{irrechain}. Here $\mf_j$ and $\nf_{j+1}$ are required to have opposite signs. A white leaf may be paired with a leaf in the omitted part.}
  \label{fig:irrechain}
\end{figure}
\end{df}
\begin{df}[Congruence and a relabeling]\label{equivirrechain} Consider any irregular chain $\Hc=(\nf_0,\cdots,\nf_q)$. By Definition \ref{irrechain}, we know $\pf_j$ is the child of $\nf_j$ other than $\nf_{j+1}$ and $\mf_j$ for $0\leq j\leq q-1$, thus $\pf_j$ has the same sign with $\nf_j$ (hence it is either its first or third child). Now for two irregular chains $\Hc=(\nf_0,\cdots,\nf_q)$ and $\Hc'=(\nf_0',\cdots,\nf_q')$, with $\pf_j$ and $\pf_j'$ etc. defined accordingly, we say they are \emph{congruent}, if $\zeta_{\nf_0}=\zeta_{\nf_0'}$, and for each $0\leq j\leq q-1$, either $\pf_j$ is the first child of $\nf_j$ and $\pf_j'$ is the first child of $\nf_j'$, or $\pf_j$ is the third child of $\nf_j$ and $\pf_j'$ is the third child of $\nf_j'$, counting from left to right. See Figure \ref{fig:equivirrechain}.

In particular, if $q$ and the congruence class (and hence $\zeta_{\nf_0}$) are fixed, then an irregular chain $\Hc$ is uniquely determined by the signs $\zeta_{\nf_j}$ for $1\leq j\leq q$. We relabel the nodes $\nf_j,\pf_j\,(0\leq j\leq q)$ by defining $\{\bff_j,\cf_j\}=\{\nf_j,\pf_j\}$, and that $\bff_j=\nf_j$ if and only if $\zeta_{\nf_j}=+$. Further, we label the two children of $\nf_q$ other than $\pf_q$ as $\ef$ and $\ff$, with $\zeta_\ef=+$ and $\zeta_\ff=-$.
  \begin{figure}[h!]
  \includegraphics[scale=.5]{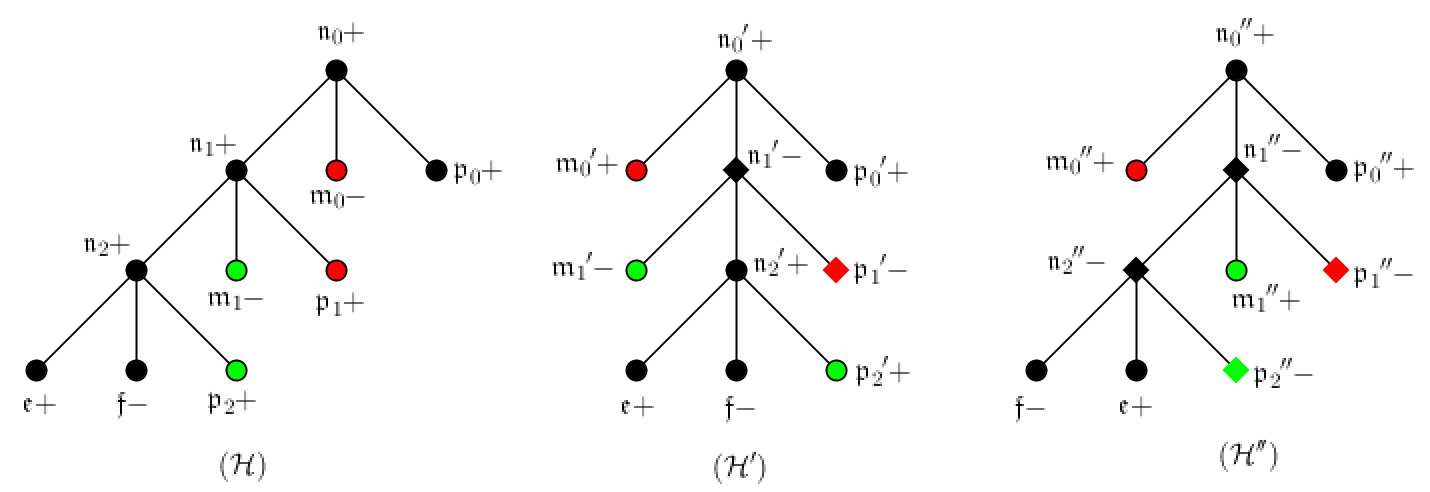}
  \caption{Three congruent irregular chains $\Hc$, $\Hc'$ and $\Hc''$, as in Definition \ref{equivirrechain}; here each $\pf_j$ (or $\pf_j'$ etc.) is the third child of $\nf_j$ (or $\nf_j'$ etc.). For convenience, we have included the sign of each node. As for the relabeling, we represent the case $(\bff_j,\cf_j)=(\nf_j,\pf_j)$ by round points, and the other case $(\bff_j,\cf_j)=(\pf_j,\nf_j)$ by diamond shaped points. Points of the same color are still paired regardless of their shapes.}
  \label{fig:equivirrechain}
\end{figure}
\end{df}
\begin{prop}\label{congdec} Let $\Hc=(\nf_0,\cdots,\nf_q)$ be an irregular chain. For any decoration $\Ds$ (or $\Es$), its restriction to $\nf_j\,(0\leq j\leq q)$ and their children is uniquely determined by $2(q+2)$ vectors $k_j,\ell_j\in\Zb_L^d\,(0\leq j\leq q+1)$, such that $k_{\bff_j}=k_j$ and $k_{\cf_j}=\ell_j$ for $0\leq j\leq q$, and $k_\ef=k_{q+1}$ and $k_\ff=\ell_{q+1}$. See Figure \ref{fig:congdec} for an example corresponding to the irregular chains in Figure \ref{fig:equivirrechain}. These vectors satisfy
\[k_0-\ell_0=k_1-\ell_1=\cdots =k_{q+1}-\ell_{q+1}:=h,\] and for each $0\leq j\leq q$ we have $\zeta_{\nf_j}\Omega_{\nf_j}=2\langle h,k_{j+1}-k_j\rangle_\beta$. Moreover $\epsilon_{k_{\nf_{j1}}k_{\nf_{j2}}k_{\nf_{j3}}}=\epsilon_{k_{j+1}\ell_{j+1}\ell_j}$, where $(\nf_{j1},\nf_{j2},\nf_{j3})$ are the children of $\nf_j$ from left to right. We say this decoration has \emph{small gap}, \emph{large gap} or \emph{zero gap} with respect to $\Hc$, if we have $0<|h|\leq \frac{1}{100\delta L}$, $|h|\geq \frac{1}{100\delta L}$ or $h=0$.
\end{prop}
  \begin{figure}[h!]
  \includegraphics[scale=.5]{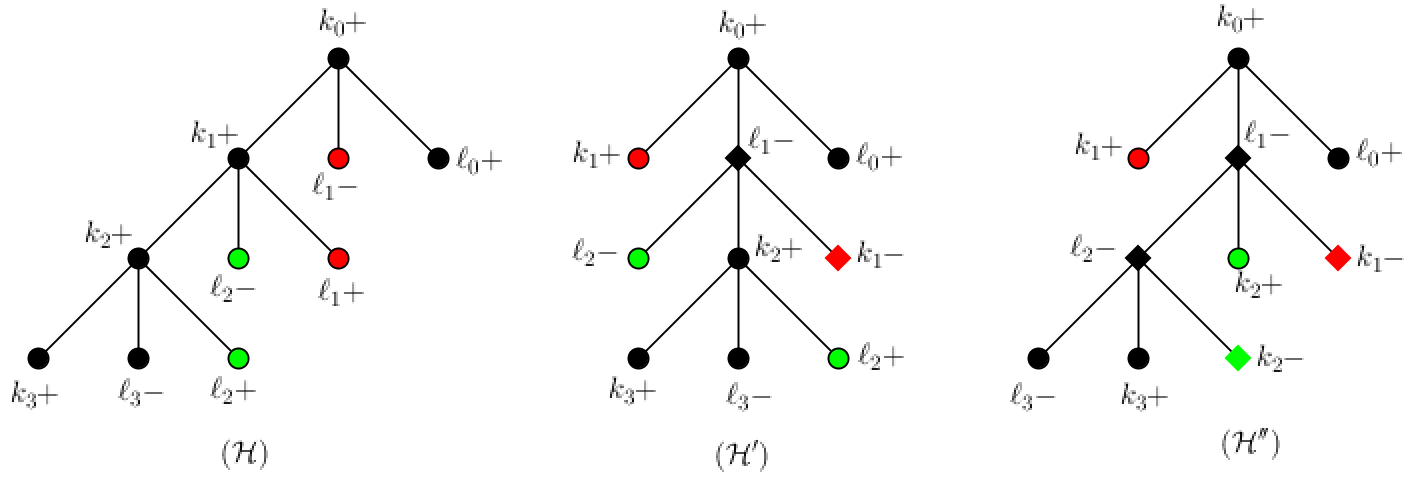}
  \caption{Decorations of irregular chains in Figure \ref{fig:equivirrechain} with sign of each node included. For the nodes $\bff_j,\cf_j$ etc. see Definition \ref{equivirrechain}.}
  \label{fig:congdec}
\end{figure}
\begin{proof} We can verify that $(k_{\nf_j},k_{\nf_{j1}},k_{\nf_{j2}},k_{\nf_{j3}})\in\{(k_j,k_{j+1},\ell_{j+1},\ell_j),(k_j,\ell_j,\ell_{j+1},k_{j+1})\}$ if $\zeta_{\nf_j}=+$, and $(k_{\nf_j},k_{\nf_{j1}},k_{\nf_{j2}},k_{\nf_{j3}})\in\{(\ell_j,\ell_{j+1},k_{j+1},k_j),(\ell_j,k_j,k_{j+1},\ell_{j+1})\}$ if $\zeta_{\nf_j}=-$. Moreover by pairing we know $k_{\mf_j}=k_{\pf_{j+1}}$ for $0\leq j\leq q-1$. The result then follows.
\end{proof}
\begin{df}\label{conggen} Let $\Hc=(\nf_0,\cdots,\nf_q)$ be an irregular chain contained in a couple $\Qc$ or a paired tree $\Tc$. If we replace $\Hc$ by a congruent irregular chain $\Hc'=(\nf_0',\cdots,\nf_q')$, then we obtain a modified couple $\Qc'$ or paired tree $\Tc'$ by (i) attaching the same subtree of $\ef$ and $\ff$ in $\Qc$ (or $\Tc$) to the bottom of $\ef'$ and $\ff'$, and (ii) assigning to $\nf_0'$ the same parent of $\nf_0$ and keeping the rest of the couple unchanged.

Given a marked prime couple $\Qc_{sk}$, we identify all the maximal irregular chains $\Hc=(\nf_0,\cdots,\nf_q)$, such that $q\geq 10^3d$, and all $\nf_j$ and their children have mark $\Lf$. For each such maximal irregular chain $\Hc$, consider $\Hc^\circ=(\nf_5,\cdots,\nf_{q-5})$ formed by omitting $5$ nodes at both ends (so that it does not affect other possible irregular chains). We define another marked prime couple $\widetilde{\Qc}_{sk}$ to be \emph{congruent} to $\Qc_{sk}$, if it can be obtained from $\Qc_{sk}$ by changing each of the irregular chains $\Hc^\circ$ to a congruent irregular chain, as described above.

Given a marked couple $\Qc$, we define $\widetilde{\Qc}$ to be congruent to $\Qc$, if it can be formed as follows. First obtain the (marked) skeleton $\Qc_{sk}$ and change it to a congruent marked prime couple $\widetilde{\Qc}_{sk}$. Then, we attach the regular couples $\Qc^{(\mf,\mf')}$ and regular trees $\Tc^{(\mf)}$ from $\Qc$ to the relevant leaf pairs and branching nodes of $\widetilde{\Qc}_{sk}$. Note that if an irregular chain $\Hc^\circ=(\nf_0,\cdots,\nf_q)$ in $\Qc_{sk}$ is replaced by $(\Hc^\circ)'=(\nf_0',\cdots,\nf_q')$ in $\widetilde{\Qc}_{sk}$, with relevant nodes $\mf_j, \pf_j$ etc. as in Definition \ref{irrechain}, then for $0\leq j\leq q-1$, the same regular couple $\Qc^{(\mf_j,\pf_{j+1})}$ is attached to the leaf pair $\{\mf_j',\pf_{j+1}'\}$ in $\widetilde{\Qc}_{sk}$. Similarly, for $1\leq j\leq q$, if $\zeta_{\nf_j'}=\zeta_{\nf_j}$ then the same regular tree $\Tc^{(\nf_j)}$ is placed at the branching node $\nf_j'$ in $\widetilde{\Qc}_{sk}$; otherwise the conjugate regular tree $\overline{\Tc^{(\nf_j)}}$ is placed at $\nf_j'$. See Figure \ref{fig:conggen} for a description of two congruent couples.
  \begin{figure}[h!]
  \includegraphics[scale=.5]{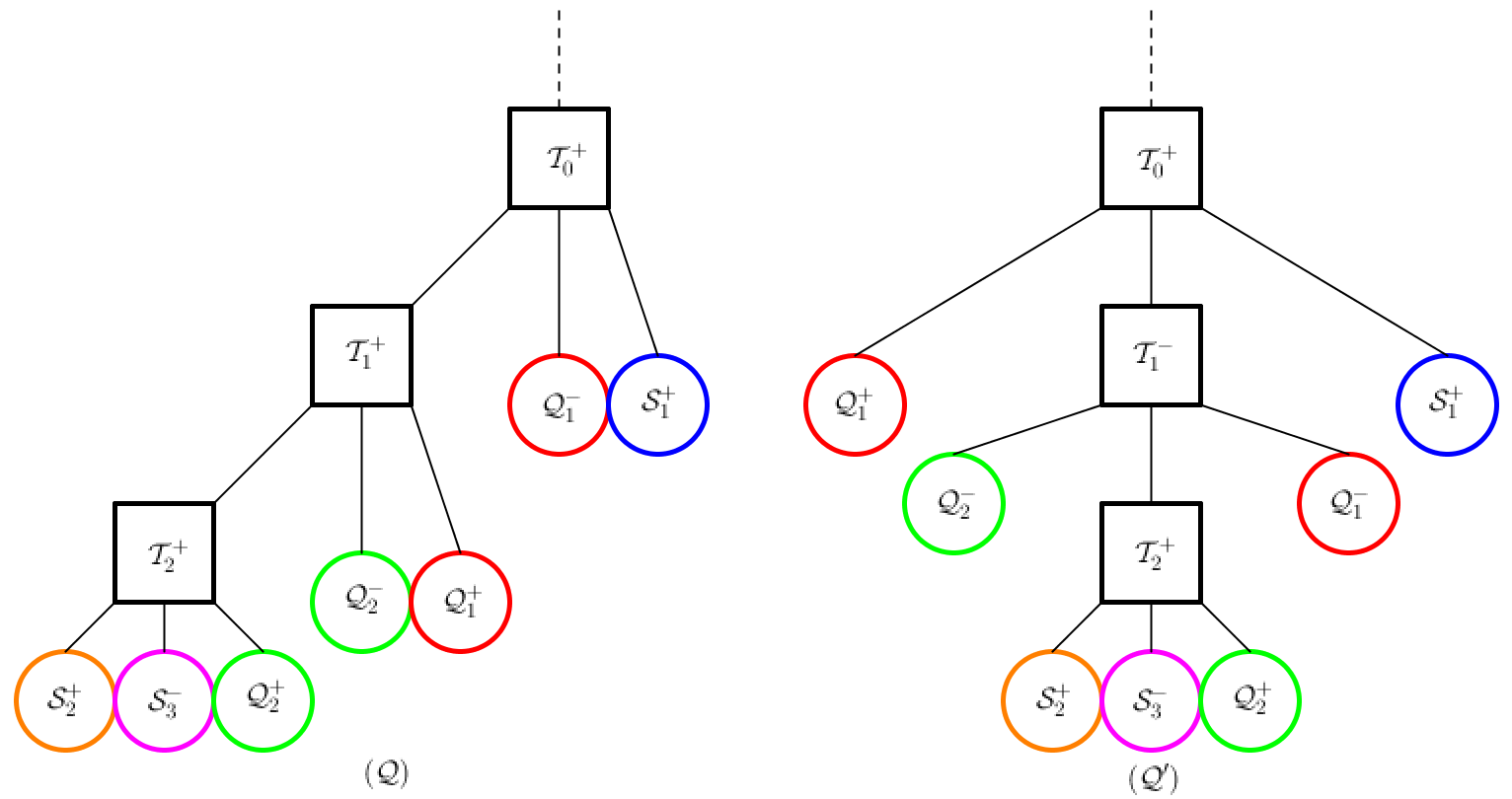}
  \caption{Two congruent couples $\Qc$ and $\Qc'$ as in Definition \ref{conggen} (formed by altering one irregular chain $\Hc^\circ$ in $\Qc_{sk}$; of course the actual $\Hc^\circ$ is much longer). Here each circle (labeled by $\Qc_j$ or $\Sc_j$) represents a paired tree, two circles of the same color (labeled by the same $\Qc_j$) in the same couple form a regular sub-couple, and each black box (labeled by $\Tc_j$) represents a regular tree. Two circles in different couples of the same color and same labeling (including signs) represent the same paired tree, two boxes of the same labeling represent the same regular tree (if they have the same sign) or two conjugate regular trees (if they have opposite signs). Finally, the signs represent the signs of the corresponding nodes in $\Qc_{sk}$ and $\Qc_{sk}'$.}
  \label{fig:conggen}
\end{figure}
\end{df}
\subsection{Expressions associated with irregular chains}\label{1chain} Given \emph{one congruence class} $\Fs$ of marked couples as in Definition \ref{conggen}, the goal of this section is to analyze the sum
\begin{equation}\label{irrechainsum}\sum_{\Qc\in\Fs}\Kc_\Qc(t,s,k),\end{equation} where the sum is taken over all \emph{marked couples} $\Qc\in\Fs$. Let the lengths of all the irregular chains $\Hc^\circ$ involved in the congruence class $\Fs$, as in Definition \ref{conggen}, be $q_1,\cdots,q_r$, then $|\Fs|=2^Q$ where $Q=q_1+\cdots+ q_r$. Since these irregular chains do not affect each other, we may focus on one individual chain, say $\Hc^\circ=(\nf_0,\cdots,\nf_q)$; that is, we only sum over $\Qc\in \Fs$ obtained by altering this irregular chain $\Hc^\circ$.

In the summation and integration in (\ref{bigformula2}), we will first fix all the variables $k_\nf$ and $t_\nf$, \emph{except} $k_\nf$ with $\nf\in\{\nf_j,\pf_j,\mf_{j-1}\}\,(1\leq j\leq q)$ and $t_\nf$ with $\nf=\nf_j\,(1\leq j\leq q-1)$, and sum and integrate over these variables. Note that we are fixing $k_{\nf_0}$ and $k_{\pf_0}$ as well as $k_\ef$ and $k_\ff$, in the notation of Definition \ref{equivirrechain}, and are thus fixing $(k_0,\ell_0,k_{q+1},\ell_{q+1})$ and $k_0-\ell_0=k_{q+1}-\ell_{q+1}=h$ as in Proposition \ref{congdec}. It is easy to see that in the summation and integration in (\ref{bigformula2}) over the \emph{fixed variables} (i.e. those $k_\nf$ and $t_\nf$ not in the above list), the summand and integrand does not depend on the way $\Hc^\circ$ is changed, because the rest of the couple is preserved under the change of $\Hc^\circ$, by Definition \ref{conggen}.

We thus only need to consider the sum and integral over the variables listed above. By Proposition \ref{congdec}, this is the same as the sum over the variables $k_j\,(1\leq j\leq q)$, with $\ell_j:=k_j-h$, and integral over the variables $t_j:=t_{\nf_j}\,(1\leq j\leq q-1)$, which satisfies $t_0>t_1>\cdots >t_{q-1}>t_q$ with $t_0:=t_{\nf_0}$ and $t_q:=t_{\nf_q}$. For any possible choice of $\Hc^\circ$ (there are $2^q$ of them), the sum and integral can be written, using (\ref{bigformula2}) and Proposition \ref{congdec}, as
\begin{multline}\label{irrechainexp}\sum_{k_1,\cdots,k_q}\int_{t_0>t_1>\cdots >t_{q-1}>t_q}\bigg(\frac{\delta}{2L^{d-1}}\bigg)^q\prod_{j=1}^q(i\zeta_{\nf_j})\prod_{j=0}^{q}\epsilon_{k_{j+1}\ell_{j+1}\ell_j}\\\times\prod_{j=0}^q e^{2\pi i\delta L^2\langle h,k_{j+1}-k_j\rangle_\beta t_j}\prod_{j=1}^{q}\Kc_{j,\Hc^\circ}\cdot \Kc_{j,\Hc^\circ}^*\,\mathrm{d}t_1\cdots\mathrm{d}t_{q-1}.\end{multline} Here in (\ref{irrechainexp}), we have
\[\Kc_{j,\Hc^\circ}=\Kc_j(t_{j},t_{j-1},k_j-h),\quad \Kc_{j,\Hc^\circ}^*=\Kc_j^*(t_{j-1},t_j,k_j)\] if $\zeta_{\nf_j}=+$, and
\[\Kc_{j,\Hc^\circ}=\Kc_j(t_{j-1},t_{j},k_j),\quad \Kc_{j,\Hc^\circ}^*=\overline{\Kc_j^*(t_{j-1},t_j,k_j-h)}\] if $\zeta_{\nf_j}=-$, where $\Kc_j=(\Kc_{\Qc^{(\pf_j,\mf_{j-1})}})_{\mathrm{app}}$ and $\Kc_j^*=(\Kc_{\Tc^{(\nf_j)}}^*)_{\mathrm{app}}$ where $\Tc^{(\nf_j)}$ is chosen to have sign $+$; note that if $\overline{\Tc}$ is the regular tree conjugate to $\Tc$ then $\Kc_{\overline{\Tc}}^*=\overline{\Kc_\Tc^*}$, and the same holds for the leading contribution $(\cdots)_{\mathrm{app}}$.

In what follows we shall study the expression (\ref{irrechainexp}), where we also sum over all possible choices of $\Hc^\circ$, i.e. all possible choices of $\zeta_{\nf_j}\,(1\leq j\leq q)$. We will view it as a function of $(k_0,\ell_0,k_{q+1},\ell_{q+1},t_0,t_q)$. Depending on the value of $h$ we have three possibilities. However, the zero gap case $h=0$ is very easy, as we have $k_j=\ell_j$, so in view of the $\epsilon_{k_{j+1}\ell_{j+1}\ell_j}$ factors we must have $k_1=\cdots=k_q=k_0$, so the expression (\ref{irrechainexp}) is bounded by $L^{-(d-1)q}/q!$, which is a large negative power of $L$ when $q$ is large (we have at least $q\geq10^3d-10$ by Definition \ref{conggen}). This term can then be easily treated, in the same way as the small gap term below.
\subsubsection{Small gap case}\label{sgcase} Assume the small gap condition $0<|h|\leq1/(100\delta L)$. Summing over all choices of $\zeta_{\nf_j}$ in (\ref{irrechainexp}), we get the expression
\begin{multline}\label{irrechainexp2}\sum_{k_1,\cdots,k_q}\int_{t_0>t_1>\cdots >t_{q-1}>t_q}\bigg(\frac{i\delta}{2L^{d-1}}\bigg)^q\prod_{j=0}^{q}\epsilon_{k_{j+1}\ell_{j+1}\ell_j}\prod_{j=0}^q e^{2\pi i\delta L^2\langle h,k_{j+1}-k_j\rangle_\beta t_j}\\\times\prod_{j=1}^{q}\big[\Kc_j(t_{j},t_{j-1},k_j-h)\Kc_j^*(t_{j-1},t_j,k_j)-\Kc_j(t_{j-1},t_{j},k_j)\overline{\Kc_j^*(t_{j-1},t_j,k_j-h)}\big]\,\mathrm{d}t_1\cdots\mathrm{d}t_{q-1}.\end{multline}
Recall that $\Kc_j$ and $\Kc_j^*$ are of form $(\cdots)_{\mathrm{app}}$, by Propositions \ref{asymptotics1} and \ref{varregtree}, they can be decomposed into terms which are products of functions of time variables $t_j$ and functions of frequency variables $k_j$. Due to bilinearity of (\ref{irrechainexp2}), we may thus assume
\begin{equation}\label{irreassump1}\Kc_j(t,s,k)=(C^+\delta)^{m_j}\Jc\Ac_j(t,s)\Mc_j(k),\quad \Kc_j^*(t,s,k)=(C^+\delta)^{m_j'}\Jc\Ac_j^*(t,s)\Mc_j^*(k),
\end{equation} where $2m_j$ and $2m_j'$ are the scales of $\Qc^{(\pf_j,\mf_{j-1})}$ and $\Tc^{(\nf_j)}$, the functions $\Jc\Ac_j,\Jc\Ac_j^*$ and $\Mc_j,\Mc_j^*$ satisfy that
\begin{equation}\label{irreassump2}\|\Jc\Ac_j\|_{X_{\mathrm{loc}}},\,\|\Jc\Ac_j^*\|_{X_{\mathrm{loc}}}\lesssim 1;\quad \sup_{|\rho|\leq 40d}\big(\langle k\rangle^{40d}|\partial^\rho\Mc_j(k)|+|\partial^\rho\Mc_j^*(k)|\big)\lesssim 1.
\end{equation} After extracting the factor $(C^+\delta)^{m_j+m_j'}$, we can write the difference factor in (\ref{irrechainexp2}) as 
\begin{multline}\label{cancellation}\Jc\Ac_j(t_j,t_{j-1})\Jc\Ac_j^*(t_{j-1},t_j)\cdot\Mc_j(k_j-h)\Mc_j^*(k_j)-\Jc\Ac_j(t_{j-1},t_j)\overline{\Jc\Ac_j^*(t_{j-1},t_j)}\cdot\Mc_j(k_j)\Mc_j^*(k_j-h)\\
=\big[\Jc\Ac_j(t_j,t_{j-1})\Jc\Ac_j^*(t_{j-1},t_j)-\Jc\Ac_j(t_{j-1},t_j)\overline{\Jc\Ac_j^*(t_{j-1},t_j)}\big]\cdot \Mc_j(k_j-h)\Mc_j^*(k_j)\\
+\Jc\Ac_j(t_{j-1},t_j)\overline{\Jc\Ac_j^*(t_{j-1},t_j)}\cdot\big[\Mc_j(k_j-h)\Mc_j^*(k_j)-\Mc_j(k_j)\Mc_j^*(k_j-h)\big],\end{multline} since $\Mc_j$ and $\Mc_j^*$ are real valued. 

For any $|h|\leq1/(100\delta L)$, by (\ref{irreassump2}) we get
\begin{equation}\label{cancel1}\sup_{|\rho|\leq 30d}\langle k_j\rangle^{30d}|\partial^\rho\big[\Mc_j(k_j-h)\Mc_j^*(k_j)-\Mc_j(k_j)\Mc_j^*(k_j-h)\big]|\lesssim \delta^{-1}L^{-1},\end{equation} which we shall use to control the second term on the right hand side of (\ref{cancellation}). To deal with the first term, we notice that both factors
\begin{equation}\label{difffactor}\Jc\Ac_j(t,s)-\Jc\Ac_j(s,t)\quad\mathrm{and}\quad\Jc\Ac_j^*(t,s)-\overline{\Jc\Ac_j^*(t,s)}\end{equation} vanish at $t=s$. In fact, if $\Tc^{(\nf_j)}$ is formed from a regular chain of scale $2m'$ (see Remark \ref{regtree}), then we may apply similar arguments as in Section \ref{domasymp}, and calculate $\Jc\Ac_j^*(t,s)$ in the same way as $\Jc\widetilde{\Bc}_\Qc(t,s)$, so that it is either $0$ or is given (up to a constant multiple) by (\ref{recurtype2}), except that the domain of integration is now $t>t_1>\cdots>t_{m'}>s$ (as the regular tree $\Tc^{(\nf_j)}$ only has one regular chain), and the irrelevant factors in the integrand are omitted. This shows that $\Jc\Ac_j^*(t,t)=0$ if $m'\geq 1$; if $m'=0$ then $\Tc^{(\nf_j)}$ is trivial so $\Jc\Ac_j^*(t,t)=1$, so in either case the desired vanishing of (\ref{difffactor}) is true. Since $\Jc\Ac_j$ and $\Jc\Ac_j^*$ are bounded in $X_{\mathrm{loc}}$ as in Remark \ref{timebound} and (\ref{vardecomp}), we can write the functions in (\ref{difffactor}) in the form
\begin{equation}\label{cancel2}(\ref{difffactor})=|t-s|^{\frac{1}{18}}\int_{\Rb^2}G(\lambda,\mu)e^{\pi i(\lambda s+\mu t)}\,\mathrm{d}\lambda\mathrm{d\mu},\quad\|(\langle \lambda\rangle+\langle \mu\rangle)^{\frac{1}{18}}G\|_{L^1}\lesssim (C^+)^{m_j},
\end{equation} for $s,t\in[0,1]$ (and also $t>s$ for the latter term in (\ref{difffactor})). Here the bound in (\ref{cancel2}) follows from the simple fact that the Fourier $L^1$ norm of $|x|^{-\gamma}\chi_0(x)(e^{\pi i\lambda x}-1)$ is bounded by $\langle\lambda\rangle^\gamma$ for $0<\gamma<1$. By (\ref{cancel1}), (\ref{cancel2}), and making further decompositions if necessary, we can rewrite (\ref{cancellation}) as a linear combination (in the form of an integral in $\lambda_j$ and $\mu_j$) of
\begin{equation}\label{cancellation2}|t_{j-1}-t_j|^{\kappa_j}e^{\pi i(\lambda_jt_j+\mu_jt_{j-1})}\cdot \Nc_j(k_j),
\end{equation} where either $\kappa_j=0$ and $\Nc_j$ satisfies (\ref{cancel1}), or $\kappa_j=1/18$ and $\Nc_j$ satisfies (\ref{cancel1}) with right hand side replaced by $1$. The coefficient of this linear combination is a function of $(\lambda_j,\mu_j)$ that satisfies the weighted bounds in (\ref{cancel2}).

By performing the above reduction for all $j$, we can rewrite (\ref{irrechainexp2}) as a linear combination (in the form of an integral in $(\lambda_j,\mu_j)$ variables) of terms, where the coefficient of this linear combination is a function of these $(\lambda_j,\mu_j)$, and is bounded in some weighted $L^1$ norm which is a tensor power of the one in (\ref{cancel2}). The term then has the following form:
\begin{multline}\label{defnewterm}\Zc:=(C^+\delta)^{m_{\mathrm{tot}}}\bigg(\frac{i\delta}{2L^{d-1}}\bigg)^q\int_{t_0>t_1>\cdots>t_{q-1}>t_q}{\,\mathrm{d}t_1\cdots\mathrm{d}t_{q-1}\cdot\prod_{j=1}^q|t_{j-1}-t_j|^{\kappa_j}e^{\pi i(\lambda_jt_j+\mu_jt_{j-1})}}\\\times e^{2\pi i\delta L^2(t_q\langle h,k_{q+1}\rangle_\beta-t_0\langle h,k_0\rangle_\beta)}\sum_{k_1,\cdots,k_q}\prod_{j=0}^q\mathbf{1}_{k_j\neq k_{j+1}}\cdot \prod_{j=1}^q e^{2\pi i\delta L^2(t_{j-1}-t_j)\langle h,k_j\rangle_\beta}\Nc_j(k_j),
\end{multline} where $(\lambda_j,\mu_j)$ are parameters as above, and $m_{\mathrm{tot}}$ is the sum of all the half-scales $m_j$ and $m_j'$.

We will first fix the $t_j$ variables and sum in $k_j$ in (\ref{defnewterm}). In this sum we may ignore the factors $\mathbf{1}_{k_j\neq k_{j+1}}$, because if any $k_j=k_{j+1}$ then in (\ref{defnewterm}) we have $e^{2\pi i\delta L^2(t_{j-1}-t_{j+1})\langle h,k_j\rangle_\beta}(\Nc_j\Nc_{j+1})(k_j)$, which can be treated in the same way with much better estimates as we are summing over fewer variables. Thus, up to lower order error terms which have the same form, we have
\begin{multline}\label{defnewterm2}\Zc:=(C^+\delta)^{m_{\mathrm{tot}}}(i\delta/2)^qe^{2\pi i\delta L^2(t_q\langle h,k_{q+1}\rangle_\beta-t_0\langle h,k_0\rangle_\beta)}\\\times\int_{t_0>t_1>\cdots>t_{q-1}>t_q}\prod_{j=1}^q|t_{j-1}-t_j|^{\kappa_j}F_j(h,{t_{j-1}-t_j})e^{\pi i(\lambda_jt_j+\mu_jt_{j-1})}\,\mathrm{d}t_1\cdots\mathrm{d}t_{q-1},
\end{multline} where $F_j$ is defined to be
\[F_j(h,t)=L^{-(d-1)}\sum_k e^{2\pi i\delta L^2t\langle h,k\rangle_\beta}\Nc_j(k).\] By Poisson summation we have (here $\widehat{\Nc_j}$ denotes the Fourier transform in $\Rb^d$)
\begin{equation}\label{poissonsum}F_j(h,t)=L\sum_{y\in\Zb^d}\widehat{\Nc_j}\big(L(y-\delta Lt (\beta^1h^1,\cdots,\beta^dh^d))\big)\end{equation} where $h^j\,(1\leq j\leq d)$ are coordinates of $h$, and we assume $\beta^j\in[1,2]$. Note that by assumption, see (\ref{cancel1}), we have $|\widehat{\Nc_j}(\xi)|\lesssim\langle \xi\rangle^{-40d}$, and also $|\delta Lt\beta^jh^j|\leq1/50$, so the sum corresponding to $y\neq 0$ in the above formula contributes at most $L^{-30d}$, hence
\begin{equation}\label{poisson}|F_j(h,t)|\lesssim L^{-30d}+L(1+\delta L^2t\cdot|h|)^{-40d}\lesssim L^{-30d}+L(1+\delta Lt)^{-40d}\end{equation} using also that $|h|\geq L^{-1}$. Moreover, for $j$ with $\kappa_j=0$, we have an extra $\delta^{-1}L^{-1}$ factor in the above bound, due to (\ref{cancel1}). In particular, for each $j$, we have
\begin{equation}\label{extradecay}\int_{|t|\leq 1}|t|^{\kappa_j}|F_j(h,t)|\,\mathrm{d}t\lesssim L^{-\frac{1}{20}}.\end{equation}

Now, let $t_0-t_q:=\sigma$, using (\ref{defnewterm2}) we can rewrite, for fixed parameters $(\lambda_j,\mu_j)$, that
\[\Zc=\Zc(k_0,\ell_0,k_{q+1},\ell_{q+1},t_0,t_q)=(C^+\delta)^{m_{\mathrm{tot}}}(i\delta/2)^qe^{-2\pi i\delta L^2\sigma\langle h,k_0\rangle_\beta}\cdot e^{\pi i\delta L^2\Omega^*t_q}e^{\pi i\lambda_q^*t_q}\cdot P(\sigma,h).\]Here $\Omega^*:=|k_{q+1}|_\beta^2-|\ell_{q+1}|_\beta^2+|\ell_0|_\beta^2-|k_0|_\beta^2$, and $\lambda_{q}^*$ is the last component of the vector $(\lambda_0^*,\cdots,\lambda_q^*)$ that satisfies
\[\sum_{j=0}^{q-1}\lambda_j^*(t_{j+1}-t_{j})+\lambda_q^*t_q\equiv \sum_{j=1}^q(\lambda_jt_j+\mu_jt_{j-1})\] (in particular each $\lambda_j^*$ is a linear combination of $\lambda_j$ and $\mu_j$). Moreover $P$ is defined by
\[P(\sigma,h)=\int_{\sigma>t_1>\cdots >t_{q-1}>0}\prod_{j=0}^{q-1}|t_{j+1}-t_j|^{\kappa_j}F_j(h,t_{j+1}-t_j)e^{\pi i\lambda_j^*(t_{j+1}-t_j)},\] where we replace $t_0$ by $\sigma$ and $t_q$ by $0$ in the above integral. By (\ref{extradecay}) and our choice of $q$ we have
\begin{equation}\label{extradecay2}\sup_{|h|\leq (100\delta L)^{-1}}\int_{|\sigma|\leq 1}|P(\sigma,h)|\,\mathrm{d}\sigma\lesssim L^{-\frac{q}{20}}\lesssim L^{-40d}.
\end{equation}

In summary, we get that
\begin{equation}\label{summary1}(\ref{irrechainexp2})=(C^+\delta)^{m_{\mathrm{tot}}}(i\delta/2)^q\int_\Rb\int_0^1G(\lambda)\Pc(\lambda,\sigma,k_0,\ell_0)\cdot\dirac(t_0-t_q-\sigma)e^{\pi i\delta L^2\Omega^*t_q}e^{\pi i\lambda t_q}\,\mathrm{d}\sigma\mathrm{d}\lambda,
\end{equation} where $\Omega^*$ is as above, and the functions $G$ and $\Pc$ satisfies
\begin{equation}\label{coefgpbd}\|\langle \lambda\rangle^{\frac{1}{18}}G\|_{L^1}\lesssim (C^+)^{m_{\mathrm{tot}}},\quad \sup_{\lambda,k_0,\ell_0}\int_0^1|\Pc(\lambda,\sigma,k_0,\ell_0)|\,\mathrm{d}\sigma\lesssim L^{-40d}.
\end{equation}
where the supremum in $(k_0, \ell_0)$ is taken here over $|k_0-\ell_0|\leq (100\delta L)^{-1}$.

Now, define the new marked couple $\Qc_{sk}^<$ by removing the irregular chain $\Hc^\circ$ from $\Qc_{sk}$; namely we set $(\pf_0,\ef,\ff)$ (see Definition \ref{equivirrechain}) to be the three children nodes of $\nf_0$, with the order determined by their signs and the relative position of $\pf_0$, and remove the other nodes (i.e. $(\nf_j,\pf_j)$ for $1\leq j\leq q$ and $\mf_j$ for $0\leq j\leq q-1$). See Figure \ref{fig:reduceirrechain}. Denote the scale of $\Qc_{sk}^<$ by $n_0^<$. Note that $\Qc_{sk}^<$ does \emph{not} depend on the choice of $\Hc^\circ$ in the fixed congruence class, and for the decoration of $\Qc_{sk}^<$ coming from the decoration of $\Qc_{sk}$, we have $\zeta_{\nf_0}\Omega_{\nf_0}=\Omega^*$ for each choice of $\Hc^\circ$. 
  \begin{figure}[h!]
  \includegraphics[scale=.5]{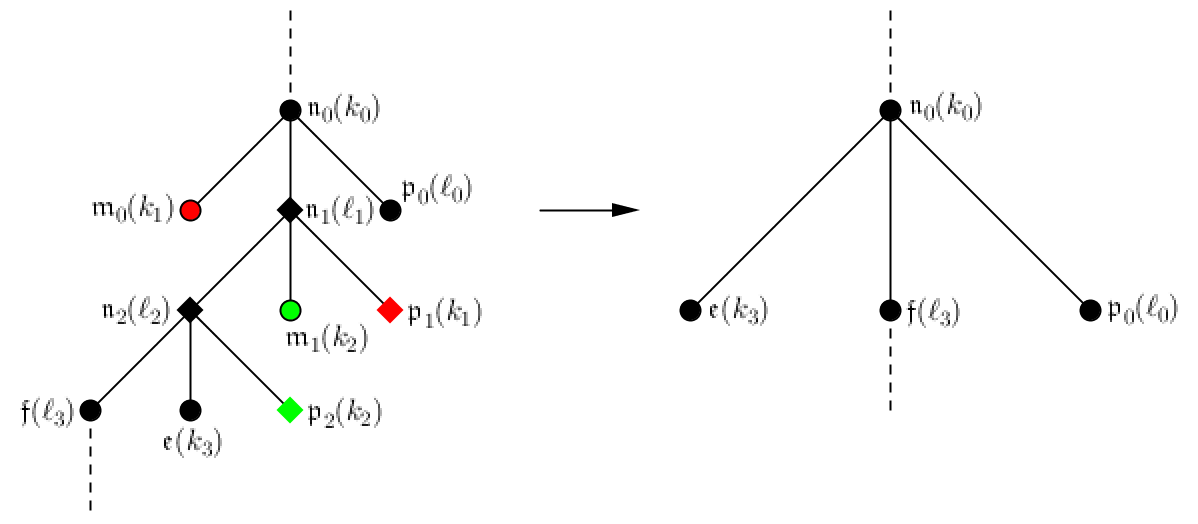}
  \caption{An example of (parts of) $\Qc_{sk}$ and $\Qc_{sk}^<$, where $q=2$. For simplicity, we have also included a decoration of $\Qc_{sk}$ and the corresponding decoration for $\Qc_{sk}^<$.}
  \label{fig:reduceirrechain}
\end{figure}

We now consider the sum
\begin{equation}\label{smallgap}\sum_{\Qc}\Kc_\Qc^{\mathrm{sg}}(t,s,k),
\end{equation} where $\Qc$ ranges over all marked couples formed by altering the irregular chain $\Hc^\circ$ in $\Qc_{sk}$, and the superscript $\mathrm{sg}$ represents the small gap case. With (\ref{summary1}), we can rewrite it as
\begin{multline}\label{smallgap2}(\ref{smallgap})=(C^+\delta)^{m_{\mathrm{tot}}}(i\delta/2)^q\cdot\bigg(\frac{\delta}{2L^{d-1}}\bigg)^{n_0^<}\zeta^*(\Qc_{sk}^<)\int_\Rb G(\lambda)\,\mathrm{d}\lambda\int_0^1\mathrm{d}\sigma\cdot\sum_{\Es_{sk}^<}\int_{\widetilde{\Ec}_{sk}^<}\widetilde{\epsilon}_{\Es_{sk}^<}\cdot\Pc(\lambda,\sigma,k_0,\ell_0)\\\times e^{\pi i\lambda t_{\nf_0}}\prod_{\nf\in (\Nc_{sk}^<)^*} e^{\zeta_\nf\pi i\cdot\delta L^2\Omega_\nf t_\nf}\,\mathrm{d}t_\nf{\prod_{\mf\in(\Lc_{sk}^<)^*}^{(+)}\Kc_{\Qc^{(\mf,\mf')}}(t_{\mf^p},t_{(\mf')^p},k_\mf)}\prod_{\mf\in(\Nc_{sk}^<)^*}\Kc_{\Tc^{(\mf)}}^*(t_{\mf^p},t_{\mf},k_\mf).
\end{multline} Here in (\ref{smallgap2}) the sum is taken over all $k$-decorations $\Es_{sk}^<$ of $\Qc_{sk}^<$, and the other notations are all associated with $\Qc_{sk}^<$, except $\widetilde{\Ec}_{sk}^<$ and $\widetilde{\epsilon}_{\Es_{sk}^<}$; instead, for $\widetilde{\Ec}_{sk}^<$ we add the one extra condition $t_{\nf_0^p}>t_{\nf_0}+\sigma$ (where $\nf_0^p$ is the parent of $\nf_0$) to the original definition (\ref{timedom}), and for $\widetilde{\epsilon}_{\Es_{sk}^<}$ we remove the one factor $\epsilon_{k_{\nf_{01}}k_{\nf_{02}}k_{\nf_{03}}}$ (where $\nf_{0j}$ are the children of $\nf_0$ from left to right) from the original definition (\ref{defcoef}). The functions $G$ and $\Pc$, and variables $(k_0,\ell_0)$ etc. are as in (\ref{summary1}),  and we may also insert the small gap restriction $0<|h|\leq 1/(100\delta L)$ in (\ref{smallgap2}). Finally, in the functions $\Kc_{\Tc^{(\nf_0)}}^*$ and $\Kc_{\Qc^{(\mf,\mf')}}$ for the leaf pair $\{\mf,\mf'\}$ containing $\pf_0$, the input variable $t_{\nf_0}$ should be replaced by $t_{\nf_0}+\sigma$.
\begin{rem}\label{newres} Due to the absence of $\epsilon_{k_{\nf_{01}}k_{\nf_{02}}k_{\nf_{03}}}$ in $\widetilde{\epsilon}_{\Es_{sk}^<}$, in the summation in (\ref{smallgap2}), the decoration $(k_\nf)$ may be resonant at the node $\nf_0$ (i.e. $(k_{\nf_{01}},k_{\nf_{02}},k_{\nf_{03}})\not\in\Sf$, see (\ref{defset})), but it must not be resonant at any other branching node. This resonance may lead to an (at most) $L^{4d}$ loss in the counting estimates in Section \ref{improvecount}, but this can always be covered by the $L^{-40d}$ gain from $\Pc$ in (\ref{coefgpbd}). See Remark \ref{countingrem} for further explanation.
\end{rem}
\subsubsection{Large gap case}\label{lgcase} Now consider the large gap case $|h|>1/(100\delta L)$. Here we will not need a big $L^{-40}$ power gain as in Section \ref{sgcase}, and it is also not necessary to exploit the cancellation. Therefore, we may fix a single choice for the irregular chain $\Hc^\circ$.

We proceed as in Section \ref{sgcase}, and in the proof below we may assume $k_0\neq k_{q+1}$ in the decoration of $\Qc_{sk}$. In fact, if $k_0=k_{q+1}$ then we must have $k_0\neq k_q$ (as $k_q\neq k_{q+1}$ in view of the factor $\epsilon_{k_{q+1}\ell_{q+1}\ell_q}$ in (\ref{irrechainexp})), so we may apply the same analysis to the shorter chain $(\nf_0,\cdots,\nf_{q-1})$, which will not make a difference in the proofs in later sections, as we leave out only one node for this chain.

We now repeat the calculations in Section \ref{sgcase} for (\ref{irrechainexp}), using again (\ref{irreassump2}). The main difference is that we do not have \eqref{cancel1}. By Poisson summation formula, we still have  (\ref{poissonsum}), but now the contribution of $y\neq 0$ is not negligible.  Still we may assume $|y|\leq C \delta L|h|$, as the remaining contribution is at most $C^+L^{-40d}$ when $|t|\leq 1$. Assume $|h^1|\geq C^{-1}|h|$, then replacing (\ref{extradecay}) we have
\begin{equation}\label{noextradecay}\int_{|t|\leq 1}|F_j(h,t)|\,\mathrm{d}t\lesssim L\sum_{|y^1|\lesssim \delta L|h|}\int_{\Rb}\frac{\mathrm{d}t}{(1+L|y^1-\delta Lt\beta^1h^1|)^{40}}\sum_{y'}\prod_{j=2}^d\frac{1}{(1+L|y^j-\delta Lt\beta^jh^j|)^{40}}\lesssim 1,
\end{equation} where $y'=(y^2,\cdots,y^d)$. This is because in (\ref{noextradecay}), the the inner sum over $y'$ is trivially bounded by $1$, so the integral over $t$ is bounded by $C^+(\delta L|h|)^{-1}L^{-1}$, and the final sum over $y^1$ is bounded by $C^+(\delta L|h|)^{-1}L^{-1}\cdot \delta L|h|=C^+L^{-1}$, noting also that $\delta L|h|\geq C^{-1}$.

With (\ref{noextradecay}) and the same arguments as before, in the end we can still write (\ref{irrechainexp}) in the form of (\ref{summary1}), together with (\ref{coefgpbd}), except that the right hand side of the second inequality of (\ref{coefgpbd}) will be $1$ instead of $L^{-40d}$. We then define the marked couple $\Qc_{sk}^<$ in the same way as in Figure \ref{fig:reduceirrechain}, which also does not depend on the choice of $\Hc^\circ$ in the fixed congruence class (except when $k_0=k_{q+1}$ and we remove the chain with one less node, where $\Qc_{sk}^<$ may depend on the last digit $\zeta_{\nf_q}$; however this is obviously acceptable and we will ignore it below), so that from (\ref{summary1}) we can again deduce (\ref{smallgap2}) for (\ref{smallgap}), except that the small gap condition in (\ref{smallgap}) should be replaced by the large gap condition. Moreover the assumption $k_0\neq k_{q+1}$ means we can recover the factor $\epsilon_{k_{\nf_{01}}k_{\nf_{02}}k_{\nf_{03}}}$ in (\ref{smallgap2}), hence instead of $\widetilde{\epsilon}_{\Es_{sk}^<}$ we have the original factor $\epsilon_{\Es_{sk}^-}$ from (\ref{defcoef}) in Definition \ref{defdec}. This means $(k_{\nf_{1}},k_{\nf_{2}},k_{\nf_{3}})\in\Sf$ for any $\nf\in(\Nc_{sk}^<)^*$ and any decoration appearing in (\ref{smallgap2}), which will allow us to apply the appropriate counting estimates in Section \ref{improvecount}.

In summary, in both small gap and large gap cases we have arrived at the formula (\ref{smallgap2}), possibly with minor differences indicated above.
\subsection{Conclusion}\label{section6summary} In Section \ref{1chain} we have fixed a single irregular chain $\Hc^\circ$ in $\Qc_{sk}$. Since different irregular chains do not affect each other, we can combine them and get an expression for the full sum (\ref{irrechainsum}). Namely, let $\Qc_{sk}^\#$ be the marked couple obtained by removing all the irregular chains $\Hc^\circ$ from $\Qc_{sk}$ as described in Figure \ref{fig:reduceirrechain} (perhaps with minor modification in the large gap case as described in Section \ref{lgcase} above, which we will ignore). This does not depend on the choice of $\Qc_{sk}$ in the fixed congruence class, nor on the choice of $\Qc\in\Fs$. We then have 
\begin{multline}\label{section6fin1} (\ref{irrechainsum})=(C^+\delta)^{n_1}\bigg(\frac{\delta}{2L^{d-1}}\bigg)^{n_0'}\zeta^*(\Qc_{sk}^\#)\int_{\Rb^\Xi} G(\vlambda)\,\mathrm{d}\vlambda\int_{[0,1]^\Xi}\mathrm{d}\vsigma\sum_{\Es_{sk}^\#}\int_{\widetilde{\Ec}_{sk}^\#}\epsilon_{\Es_{sk}^\#}\Pc(\vlambda,\vsigma,k[\Qc_{sk}^\#])\prod_{\nf\in \Xi}e^{\pi i\lambda_\nf t_\nf}\\\times\prod_{\nf\in (\Nc_{sk}^\#)^*} e^{\zeta_\nf\pi i\cdot\delta L^2\Omega_\nf t_\nf}\,\mathrm{d}t_\nf{\prod_{\mf\in(\Lc_{sk}^\#)^*}^{(+)}\Kc_{\Qc^{(\mf,\mf')}}(t_{\mf^p},t_{(\mf')^p},k_\mf)}\prod_{\mf\in(\Nc_{sk}^\#)^*}\Kc_{\Tc^{(\mf)}}^*(t_{\mf^p},t_{\mf},k_\mf).
\end{multline} Here in (\ref{section6fin1}), $n_0'$ is the scale of $\Qc_{sk}^\#$ and $n_1$ is the sum of all the $m_{\mathrm{tot}}$ and $q$ in (\ref{smallgap2}), the summation is taken over all $k$-decorations $\Es_{sk}^\#$ of $\Qc_{sk}^\#$, and the other notations are all associated with $\Qc_{sk}^\#$, except $\widetilde{\Ec}_{sk}^\#$; instead, for $\widetilde{\Ec}_{sk}^\#$ we add the extra conditions $t_{\nf^p}>t_{\nf}+\sigma_\nf$ (where $\nf^p$ is the parent of $\nf$) to the original definition (\ref{timedom}), for $\nf\in\Xi$, where $\Xi$ is a subset of the set $(\Nc_{sk}^\#)^*$ of branching nodes. The vector parameters are $\vlambda=\lambda[\Xi]\in\Rb^\Xi$ and $\vsigma=\sigma[\Xi]\in[0,1]^\Xi$ respectively, and $k[\Qc_{sk}^\#]$ is the vector of all the $k_\nf$'s. The functions $G(\vlambda)$ and $\Pc(\vlambda,\vsigma,k[\Qc_{sk}^\#])$ satisfy the bounds
\begin{equation}\label{tensorbd}\bigg\|\prod_{\nf\in\Xi}\langle \lambda_\nf\rangle^{\frac{1}{18}}G\bigg\|_{L^1}\lesssim (C^+)^n,\quad \sup_{\vlambda,k[\Qc_{sk}^\#]}\int_{[0,1]^\Xi}|\Pc(\vlambda,\vsigma,k[\Qc_{sk}^\#])|\,\mathrm{d}\vsigma\lesssim 1.
\end{equation} We may also insert various small gap or large gap conditions (including the ones coming from $k_0\neq k_{q+1}$ in Section \ref{lgcase}), and some input variables in some of the $\Kc_{\Qc^{(\mf,\mf')}}$ or $\Kc_{\Tc^{(\mf)}}^*$ functions may be translated by $\sigma_\nf$, as in (\ref{smallgap2}) in Section \ref{sgcase}. Finally, the function $\epsilon_{\Es_{sk}^\#}$ may miss a few $\epsilon_{k_{\nf}k_{\nf_1}k_{\nf_2}k_{\nf_3}}$ factors compared to the original definition (\ref{defcoef}), but for each such missing factor we can gain a power $L^{-40d}$ on the right hand side in the second inequality in (\ref{tensorbd}).

At this point, we may expand the functions $\Kc_{\Qc^{(\mf,\mf')}}$ and $\Kc_{\Tc^{(\mf)}}^*$ (or their leading or remainder contributions) using their Fourier $L^1$ (or $X_{\mathrm{loc}}^\kappa$) bounds, and combine the $\Kc$ factors and the $\Pc$ factor in (\ref{section6fin1}), to further reduce to the expression
\begin{multline}\label{section6fin2} (\ref{irrechainsum})=(C^+\delta)^{\frac{n-n_0'}{2}}\bigg(\frac{\delta}{2L^{d-1}}\bigg)^{n_0'}\zeta^*(\Qc_{sk}^\#)\int_{\Rb^{\Lambda}\times\Rb^2} G(\vlambda)\cdot e^{\pi i(\lambda t+\mu s)}\,\mathrm{d}\vlambda\int_{[0,1]^\Xi}\mathrm{d}\vsigma\sum_{\Es_{sk}^\#}\int_{\widetilde{\Ec}_{sk}^\#}\epsilon_{\Es_{sk}^\#}\\\times\prod_{\nf\in {\Lambda}} e^{\zeta_\nf\pi i\cdot\delta L^2\Omega_\nf t_\nf}\prod_{\nf\in \Lambda}e^{\pi i\lambda_\nf t_\nf}\,\mathrm{d}t_\nf\cdot \Xc_{\mathrm{tot}}(\vlambda,\vsigma,k[\Qc_{sk}^\#]),
\end{multline}
Here in (\ref{section6fin2}) the set $\Lambda=(\Nc_{sk}^\#)^*$ and $\vlambda=(\lambda[\Lambda],\lambda,\mu)\in\Rb^\Lambda\times\Rb^2$, the function $G$ is different from the one in (\ref{section6fin1}), but still satisfies the same first inequality in (\ref{tensorbd}) (with $\Xi$ replaced by $\Lambda$, and the extra factor $\langle\lambda\rangle^{\frac{1}{18}}\langle\mu\rangle^{\frac{1}{18}}$ on the left hand side). Using the second bound in (\ref{tensorbd}), the $X_{\mathrm{loc}}^\kappa$ bounds for $\Kc_{\Qc^{(\mf,\mf')}}$ and $\Kc_{\Tc^{(\mf)}}^*$ and their components, and the definition of markings $\Lf$ and $\Rf$, we deduce that the function $\Xc_{\mathrm{tot}}$ satisfies
\begin{equation}\label{xtotbound}\int_{[0,1]^\Xi}|\Xc_{\mathrm{tot}}(\vlambda,\vsigma,k[\Qc_{sk}^\#])|\,\mathrm{d}\vsigma\lesssim{\prod_{\lf\in(\Lc_{sk}^\#)^*}^{(+)}\langle k_\lf\rangle^{-40d}}\cdot L^{-2\nu r_0}
\end{equation} uniformly in $\vlambda$, where $r_0$ is the total number of branching nodes and leaf pairs that are marked $\Rf$ in the marked couple $\Qc_{sk}^\#$. In (\ref{xtotbound}) we can also gain a power $L^{-40d}$ per missing factor $\epsilon_{k_{\nf}k_{\nf_1}k_{\nf_2}k_{\nf_3}}$ in $\epsilon_{\Es_{sk}^\#}$, as described above.

Note that the couple $\Qc_{sk}^\#$ is still prime. Moreover by definition, it \emph{does not contain an irregular chain of length $>10^3d$ with all branching nodes and leaf pairs marked $\Lf$}. In particular, if $r_0$ is the number of branching nodes and leaf pairs that are marked $\Rf$, $r_{\mathrm{irr}}$ is the number of maximal irregular chains, and $Q$ is the total length of these irregular chains, then we have \begin{equation}\label{typeIcontrol}Q\leq C(r_0+r_{\mathrm{irr}}).
\end{equation}Based on this information, as well as the first inequality in (\ref{tensorbd}) and (\ref{xtotbound}), we will establish an \emph{absolute upper bound} for the expression (\ref{section6fin2}). This will be done in Sections \ref{improvecount} and \ref{l1coef}.
\section{Non-regular couples II: improved counting estimates}\label{improvecount} We shall reduce the estimate of (\ref{section6fin2}) to bounding the number of solutions to some counting problem, see (\ref{finalexp3}). In this section we first introduce and study this counting problem, and then use it to control (\ref{section6fin2}) in Section \ref{l1coef}.
\subsection{Couples and molecules} To study the counting problem, we introduce the notion of molecules, which is more flexible than couples.
\begin{df}\label{defmole0} A \emph{molecule} $\Mb$ is a directed graph, formed by {vertices (called \emph{atoms})} and edges (called \emph{bonds}), where multiple and self-connecting bonds are allowed. We will write $v\in \Mb$ and $\ell\in\Mb$ for atoms $v$ and bonds $\ell$ in $\Mb$; we also write $\ell\sim v$ if $v$ is {one of the two endpoints} of $\ell$. We further require that (i) each atom has at most 2 outgoing bonds and at most 2 incoming bonds (a self-connecting bond counts as outgoing once and incoming once), and (ii) there is no \emph{saturated} (connected) component, where a component is saturated means that it contains only degree 4 atoms. Here and below connectedness is always understood in terms of \emph{undirected graphs}.

It is clear that a subgraph of a molecule is still a molecule. We will be interested in certain special subgraphs (or types of subgraph) of molecules, which we will refer to as \emph{functional groups}. We introduce the following notation for molecules $\Mb$, which will be used throughout this section: $V$ is the number of atoms, $V_j\,(0\leq j\leq 4)$ is the number of degree $j$ atoms ($V_0$ is the number of isolated atoms), $E$ is the number of bonds, $F$ is the number of components. We also define
\begin{equation}\label{defchar}
\begin{aligned}
\chi&:=E-V+F,&\eta&:= V_3+2V_2+3V_1+4V_0-4F,\\
&&\eta_*&:=V_3+2V_2+2V_1+2V_0-2F.
\end{aligned}
\end{equation} {This $\chi$ is called the \emph{circuit rank} of $\Mb$ and represents the number of independent cycles; $\eta$ and $\eta_*$ are auxiliary quantities designed to control several types of steps in the algorithm, see Section \ref{proofofgain}.}
\end{df}
\begin{prop}\label{moleprop} In a molecule any self-connecting bond must be single, and between any two atoms there is at most a triple bond. A molecule of $n\geq 1$ atoms has at most $2n-1$ bonds; if it has exactly $2n-1$ bonds we will call it a \emph{base molecule}. Then, a base molecule must be connected. It either has two degree 3 atoms or one degree 2 atom, while all other atoms have degree 4.
\end{prop}
\begin{proof} In a molecule each atom has degree $\leq 4$, so the number of bonds is at most $2n$. Equality cannot hold when $n>0$ because otherwise each atom would have degree 4, contradicting (ii) in Definition \ref{defmole0}. For the same reason there cannot be quadruple bonds or self-connecting double bonds. For a base molecule, the degrees of atoms have to be as stated because the total degree is $4n-2$. If it is not connected, then it has at least two components, so at least one of them will contain only degree 4 atoms, contradiction.
\end{proof}
\begin{df}\label{defmole} Let $\Qc$ be a nontrivial couple, we will define a directed graph $\Mb$ associated with $\Qc$ as follows. The atoms are all $4$-node subsets of $\Qc$ that contain a branching node $\nf\in\Nc^*$ and its three children nodes. For any two atoms, we connect them by a bond if either (i) a branching node is the parent in one atom and a child in the other, or (ii) two leaves from these two atoms are paired with each other. We call this bond a PC (parent-child) bond in case (i) and a LP (leaf-pair) bond in case (ii). Note that multiple bonds are possible, and a self-connecting bond occurs when two sibling leaves are paired. {This definition applies even if one of the trees of $\Qc$ is trivial; note that in this case the root of the trivial tree is regarded as a leaf instead of a branching node.}

We fix a direction of each bond as follows. If a bond corresponds to a leaf pair, then it goes from the atom containing the leaf with $-$ sign to the atom containing the leaf with $+$ sign. If a bond corresponds to a branching node $\nf$ that is not a root, suppose $\nf$ is the parent in the atom $v_1$ and is a child in the atom $v_2$, then the bond goes from $v_1$ to $v_2$ if $\nf$ has $+$ sign, and go from $v_2$ to $v_1$ otherwise. See Figure \ref{fig:cplmole} for an example.
  \begin{figure}[h!]
  \includegraphics[scale=.5]{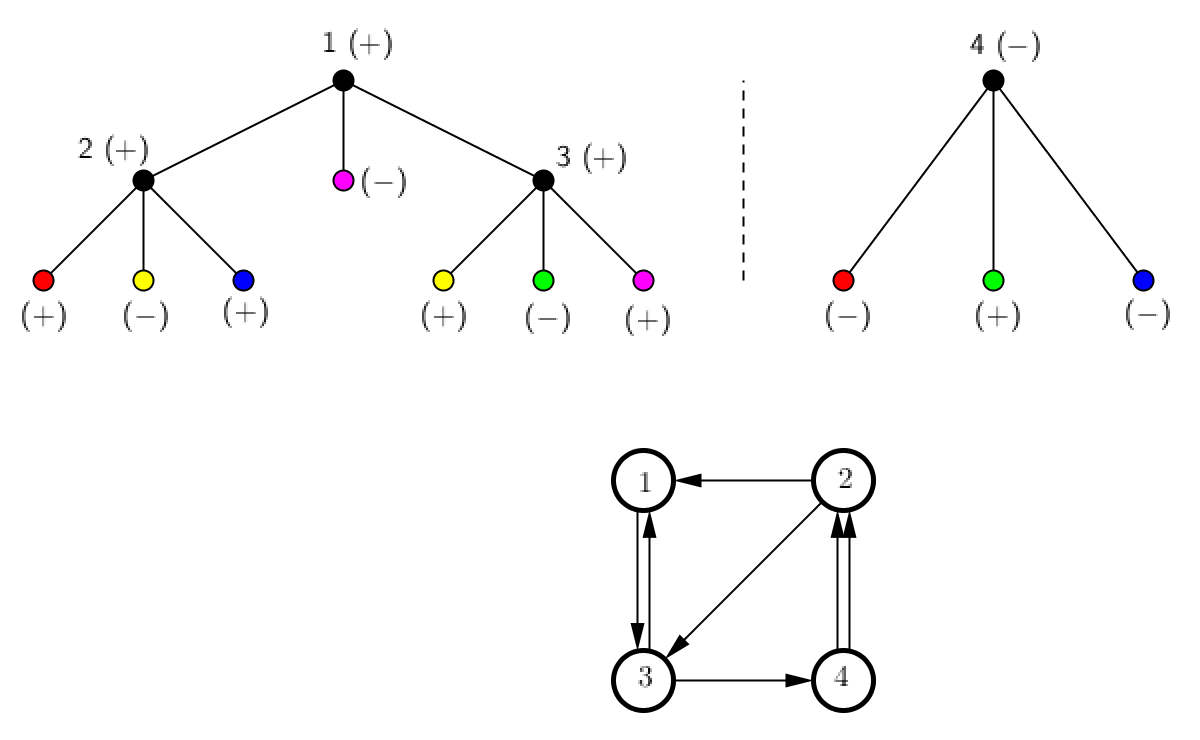}
  \caption{A base molecule (Definition \ref{defmole}), which comes from the couple in Definition \ref{defcouple}. Here each atom has the same label as its parent node in the couple.}
  \label{fig:cplmole}
\end{figure} 
\end{df}
\begin{prop} The directed graph $\Mb$ defined in Definition \ref{defmole} is a base molecule.
\end{prop}
\begin{proof} Let $n\geq 1$ be the scale of $\Qc$, then $\Mb$ has $n$ atoms and $2n-1$ bonds, because the atoms are in one-to-one correspondence with branching nodes, and the bonds are in one-to-one correspondence with non-root branching nodes and non-root leaf pairs. The statements about outgoing and incoming bonds follow directly from Definition \ref{defmole}, and it is also easy to check that $\Mb$ is connected. Therefore $\Mb$ is a base molecule.
\end{proof}
\begin{rem}In the proof below (for example in some figures), we may omit the directions of some bonds, if these directions do not play a significant role; however they still need to satisfy the conditions in Definition \ref{defmole0}. In the convention we use, arrows indicate bonds with fixed direction, segments without arrow indicate bonds with uncertain direction, and dashed segments indicate possible bond(s) connecting the given atom(s) to the the rest of the molecule. Boxes with dashed boundary indicate components after removing certain bonds or atoms.
\end{rem}
\begin{prop}\label{recover} Given a base molecule $\Mb$ with $n$ atoms as in Definition \ref{defmole0}, the number of couples $\Qc$ such that the corresponding molecule equals $\Mb$ is at most $C^n$.
\end{prop}
\begin{proof} For each atom $v\in\Mb$, each bond $\ell\sim v$ corresponds to a unique node $\nf\in v$. We may assign a code to this pair $(v,\ell)$ indicating the relative position of $\nf$ in $v$ (say code $0$ if $\nf$ is the parent in this atom, and codes $1$, $2$ or $3$ if $\nf$ is the left, mid or right child in this atom). In this way we get an encoded molecule which has a code assigned to each pair $(v,\ell)$ where $\ell\sim v$. Clearly if $\Mb$ is fixed then the corresponding encoded molecule has at most $C^n$ possibilities, so it suffices to show that $\Qc$ can be reconstructed from the encoded molecule.

In fact, if the encoded molecule is fixed, then the branching nodes of $\Qc$ uniquely correspond to the atoms of $\Mb$. Moreover, the branching node corresponding to $v_2$ is the $\alpha$-th child of the branching node corresponding to $v_1$, if and only if $v_1$ and $v_2$ are connected by a bond $\ell$ such that the codes of $(v_1,\ell)$ and $(v_2,\ell)$ are $\alpha$ and $0$ respectively. Next, we can determine the leaves of $\Qc$ by putting a leaf as the $\alpha$-th child for each branching node and each $\alpha$, as long as this position is not occupied by another branching node; moreover, the $\alpha$-th child of the branching node corresponding to $v_1$ and the $\beta$-th child of the branching node corresponding to $v_2$ are paired, if and only if $v_1$ and $v_2$ are connected by a bond $\ell$ such that the codes of $(v_1,\ell)$ and $(v_2,\ell)$ are $\alpha$ and $\beta$ respectively. Therefore $\Qc$ can be uniquely reconstructed (if one of the trees in $\Qc$ is trivial the reconstruction will be slightly different but this does affect the result).
\end{proof}
\begin{df}\label{molechain}We define two functional groups, which we call \emph{type I} and \emph{type II (molecular) chains}, as in Figure \ref{fig:molechain}. Note that type I chains are formed by double bonds, and type II chains are formed by double bonds and pairs of single bonds. For type I chains, we require that the two bonds in any double bond have opposite directions. For type II chains, we require that any pair of single bonds have opposite directions, see Figure \ref{fig:molechain}.
  \begin{figure}[h!]
  \includegraphics[scale=.5]{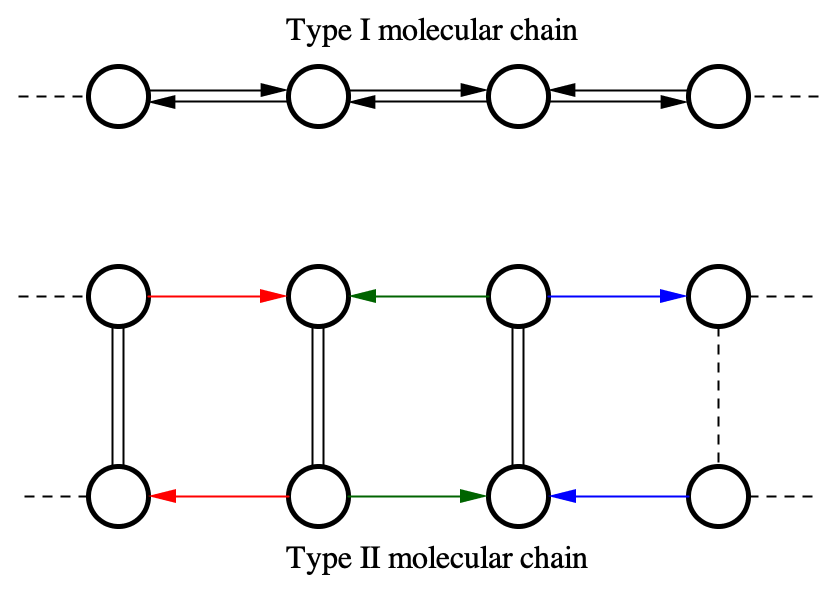}
  \caption{The two types of molecular chains. For type II, the single bonds of the same color are paired single bonds, and must have opposite directions. The directions of the double bonds are not drawn here, but they must satisfy the conditions in Proposition \ref{moleprop}.}
  \label{fig:molechain}
\end{figure} 
\end{df}
We now define the counting problem associated with a molecule (or a couple, see Remark \ref{moledec}), which is the main thing we study in the rest of this section.
\begin{df}\label{countingproblem} Given a molecule $\Mb$ and a set $S$ of atoms. Suppose we fix (i) $a_\ell\in\Zb_L^d$ for each bond $\ell\in\Mb$, (ii) $c_v\in\Zb_L^d$ for each {non-isolated} atom $v\in\Mb$, assuming $c_v=0$ if $v$ has degree 4, (iii) $\Gamma_v\in\Rb$ for each {non-isolated} atom $v$, and (iv) $f_v\in\Zb_L^d$ for each {non-isolated} $v\in S$ with $d(v)<4$. Define $\Df(\Mb)$ to be the set of vectors $k[\Mb]:=(k_\ell)_{\ell\in\Mb}$, such that each $k_\ell\in\Zb_L^d$ and $|k_\ell-a_\ell|\leq 1$, and
\begin{equation}\label{decmole}\sum_{\ell\sim v}\zeta_{v,\ell}k_\ell=c_v,\quad\bigg|\sum_{\ell\sim v}\zeta_{v,\ell}|k_\ell|_\beta^2-\Gamma_v\bigg|\leq \delta^{-1}L^{-2}\end{equation} for each {non-isolated} atom $v$. Here in (\ref{decmole}) the sum is taken over all bonds $\ell\sim v$, and $\zeta_{v,\ell}$ equals $1$ if $\ell$ is outgoing from $v$, and equals $-1$ otherwise. We also require that (a) the values of $k_\ell$ for different $\ell\sim v$ are all equal given each {non-isolated} $v\in S$, and this value equals $f_v$ if also $d(v)<4$, and (b) for any {non-isolated} $v\not\in S$ and any bonds $\ell_1,\ell_2\sim v$ of opposite directions (viewing from $v$), we have $k_{\ell_1}\neq k_{\ell_2}$. Note that this actually makes $\Df$ depending on $S$, but we will omit this dependence for simplicity. We say an atom $v$ is \emph{degenerate} if $v\in S$, and is \emph{tame} if moreover $d(v)<4$.

In addition, we may add some extra conditions to the definition of $\Df(\Mb)$. These conditions are independent of the parameters, and have the form of (combinations of) $(k_{\ell_1}-k_{\ell_2}\in E)$ for some bonds $\ell_1,\ell_2\in\Mb$ and fixed subsets $E\subset \Zb_L^d$. Let $\mathtt{Ext}$ be the set of these extra conditions, and denote the corresponding set of vectors $k[\Mb]$ be $\Df(\Mb,\mathtt{Ext})$. We are interested in the quantities $\sup\#\Df(\Mb,\mathtt{Ext})$, where the supremum is taken over all possible choices of parameters $(a_\ell,c_v,\Gamma_v,f_v)$.
\end{df}
\begin{rem}\label{moledec} The vectors $k[\Mb]$ will come from decorations of the couple $\Qc$ from which $\Mb$ is obtained. In fact, if $k[\Qc]$ is a $k$-decoration of $\Qc$, then it uniquely corresponds to a vector $k[\Mb]$. It is easy to check, using Definitions \ref{defdec} and \ref{countingproblem}, that $\sum_{\ell\sim v}\zeta_{v,\ell}k_\ell$ equals $0$ if $d(v)\in\{2,4\}$ and equals $\pm k$ if $d(v)=3$, and $\sum_{\ell\sim v}\zeta_{v,\ell}|k_\ell|_\beta^2$ equals $0$ if $d(v)=2$, equals $-\zeta_\nf\Omega_\nf$ if $d(v)=4$ (where $\nf$ is the parent node in the atom $v$), and equals $-\zeta_\nf(\Omega_\nf\pm|k|_\beta^2)$ if $d(v)=3$. Moreover, if $(k_{\nf_1},k_{\nf_2},k_{\nf_3})\in\Sf$, then either the values of $k_\ell$ for different $\ell\sim v$ are all equal (and this value equals $k$ if $d(v)<4$), or for any bonds $\ell_1,\ell_2\sim v$ of opposite directions we have $k_{\ell_1}\neq k_{\ell_2}$. Note that a degenerate atom corresponds exactly to a branching node $\nf$ for which $\epsilon_{k_{\nf_1}k_{\nf_2}k_{\nf_3}}=-1$.
\end{rem}
\begin{prop}[A rigidity theorem]\label{gain} Let $\Mb$ be a base molecule of $n$ atoms, where $1\leq n\leq (\log L)^3$, that \emph{does not contain any triple bond}. Then, $\Df(\Mb)$ is the union of at most $C^n$ subsets. Each subset has the form $\Df(\Mb,\mathtt{Ext})$, and there exists $1\leq r\leq n$, and a collection of at most $Cr$ molecular chains of either type I or type II in $\Mb$, such that (i) the number of atoms not in one of these chains is at most $Cr$, and (ii) for any type II chain in the collection and any two paired single bonds $(\ell_1,\ell_2)$ in this chain (see Figure \ref{fig:molechain}), the set $\mathtt{Ext}$ includes the condition $(k_{\ell_1}=k_{\ell_2})$. Moreover we have the estimate that  \begin{equation}\label{defect2}\sup\#\Df(\Mb,\mathtt{Ext})\leq (C^+)^n\delta^{-\frac{n+m}{2}}L^{(d-1)n-2\nu r},\end{equation} where $m$ is the number of atoms in the union of type I chains.
\end{prop}
\begin{rem}\label{countingrem} In view of Remark \ref{newres}, in Definition \ref{countingproblem} we may also fix some set $S^*$ of atoms such that neither (a) nor (b) is required for $v\in S^*$, but we are allowed to multiply the left hand side of (\ref{defect2}) by $L^{-40d\cdot|S^*|}$. In this way we can restate Proposition \ref{gain} appropriately, and the new result can be easily proved with little difference in the arguments, due to the large power gains. For simplicity we will not include this in the proof below.
\end{rem}
\subsection{The general framework}\label{moleframe} The framework of proving Proposition \ref{gain} is as follows. We will perform a sequence of operations on $\Mb$, following some specific algorithm, until reducing $\Mb$ to isolated atoms only. The operations are usually removing bonds or atoms from $\Mb$, but in some cases may also add new bonds to $\Mb$. {As is standard in graph theory, whenever we remove some atoms, we also automatically remove all bonds connected to them.}

Together with each operation we also specify an extra condition, which has the form appearing in $\mathtt{Ext}$ and will be denoted by $\Delta\mathtt{Ext}$. This is usually $\varnothing$ but in some cases may be nontrivial. The operation and the extra condition together is called a \emph{step}. A sequence of steps ending at isolated atoms is called a \emph{track}. In each track, the time immediately after a step and before the next step is called a \emph{timespot}. In our algorithm, there are timespots, which we call \emph{checkpoints}, where the next step has two choices, leading to different tracks. Any track will contain at most $Cn$ steps, and the total number of tracks is at most $C^n$.

For each step, we use the subscript $(\cdot)_{\mathrm{pre}}$ to denote any object before this step, and use $(\cdot)_{\mathrm{pos}}$ to denote the object after this step. If $X$ is a real-valued variable we define $\Delta X=X_{\mathrm{pos}}-X_{\mathrm{pre}}$. During each track we will monitor the values of various quantities associated with $\Mb$, such as $\chi$, $\eta$, etc. We will also retrospectively (i.e. in the opposite direction of the steps) define two variables $(\gamma,\kappa)$ and a set $\mathtt{Ext}$ of extra conditions. In the end state with only isolated atoms, we set $\gamma=\kappa=0$ and $\mathtt{Ext}=\varnothing$. For each step we will fix the value of $\Delta\gamma$ and $\Delta\kappa$, and will determine $\mathtt{Ext}_{\mathrm{pre}}$ from $\mathtt{Ext}_{\mathrm{pos}}$ and $\Delta\mathtt{Ext}$. Given a track and a timespot $t^*$, consider all the possible tracks that coincide with the given track up to $t^*$; these different tracks lead to different values of $\gamma$ and $\mathtt{Ext}$ calculated at $t^*$, and we define $\Upsilon$ to be the collection of all such possible $\mathtt{Ext}$'s.

We will set our steps and algorithm in such a way that, for any timespot in any track, the following conditions are always satisfied:
\begin{itemize}
\item Condition 1: $\Mb$ is always a molecule;
\item Condition 2: any vector $k[\Mb]$ must satisfy one of the conditions $\mathtt{Ext}\in\Upsilon$;
\item Condition 3: if $\Mb$ consists of components $\Mb_j$, then $\mathtt{Ext}$ is the union of $\mathtt{Ext}_j$ which only involves bonds in $\Mb_j$;
\item Condition 4: $\sup\#\Df(\Mb,\mathtt{Ext})\leq (C^+)^{n_0}\delta^{-\kappa} L^{(d-1)\gamma}$, where $n_0$ is the number of remaining steps in this track.
\end{itemize} 

The above conditions are trivially satisfied in the end state, so we only need to verify that they are preserved during the execution of the algorithm (Conditions 2--4 will be verified retrospectively). Now Condition 3 is easy to verify as the operation in each step will be restricted to one component of $\Mb$, and so will the extra condition $\Delta\mathtt{Ext}$. Condition 1 will be preserved if an operation only removes bonds or atoms; in the exceptional cases where new bonds are added, we only need to show that the (outgoing or incoming) degree of each atom does not increase, and no saturated component is created, which will be done within the definition of steps. Condition 2 depends on the algorithm, but at each non-checkpoint where the next step has only one choice, we will always set $\mathtt{Ext}_{\mathrm{pre}}=\mathtt{Ext}_{\mathrm{pos}}$, which preserves Condition 2; checkpoints will only appear in specific places where we will verify Condition 2 within the definition of the algorithm. Finally, for Condition 4, we only need to show that
\begin{equation}\label{keyineq}\sup\#\Df(\Mb_{\mathrm{pre}},\mathtt{Ext}_{\mathrm{pre}})\leq C^+\delta^{\Delta\kappa}L^{-(d-1)\Delta\gamma}\cdot\sup\#\Df(\Mb_{\mathrm{pos}},\mathtt{Ext}_{\mathrm{pos}}),
\end{equation} which will be one of the key components of the proof.

In Section \ref{oper} we define all the steps together with $(\Delta\gamma,\Delta\kappa)$ and $\Delta\mathtt{Ext}$, then prove (\ref{keyineq}), and study some properties of these steps which will be used in analyzing the algorithm. The algorithm is described in Section \ref{alg}, and we use it to prove Proposition \ref{gain} in Section \ref{proofofgain}.
\subsubsection{Some useful facts} We record some definitions and facts which will be frequently used below.
\begin{df}\label{defbridge} Given a molecule $\Mb$, we say a single bond $\ell$ is a \emph{bridge} if removing $\ell$ adds one new component, see Figure \ref{fig:bridge}. We say $\ell$ is \emph{special} if both atoms connected by $\ell$ have degree 3, and each of them has a double edge, connected to two different atoms.
  \begin{figure}[h!]
  \includegraphics[scale=.5]{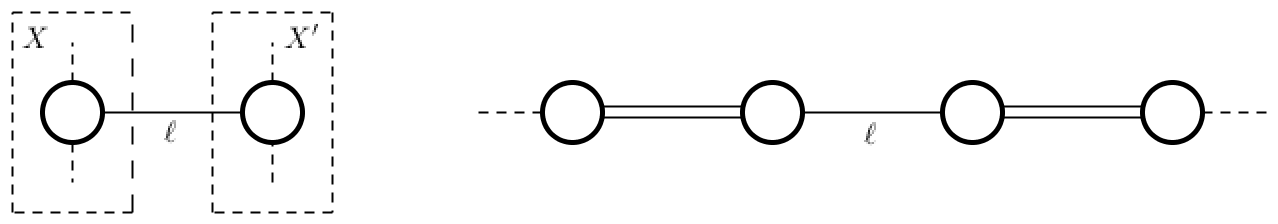}
  \caption{A bridge and a special single bond (see Definition \ref{defbridge}). The bridge can also be seen as a special case of Lemma \ref{cutlem} below with $r=1$.}
  \label{fig:bridge}
\end{figure} 
\end{df}
\begin{lem}\label{increF} Suppose $\Mb$ has no bridge. Suppose we remove a set $Y$ of atoms from $\Mb$ together with all the bonds connecting to them, and consider the possible {new} components generated by this operation. Then in $\Mb$, the total number of bonds connecting each component to $Y$ is at least 2. In particular, the number of such components is at most $h/2$, where $h$ is the total number of bonds connecting $Y$ to $Y^c$.
\end{lem}
\begin{proof} If $Z$ is one of the components and there is only one bond $\ell$ connecting $Z$ to $Y$, then since $Z$ cannot be connected to any other component, we know that $\ell$ is a bridge in the original $\Mb$. The second statement follows immediately.
\end{proof}
\begin{lem}\label{cutlem}Suppose $X$ and $X'$ form a partition of atoms in (some component of) $\Mb$, and $\ell_1,\cdots,\ell_r$ are all the bonds connecting $X$ to $X'$. then for any $k[\Mb]\in\Df(\Mb)$ we have\begin{equation}\label{cuts}\sum_{j=1}^r\zeta_jk_{\ell_j}=c_0,\quad \bigg|\sum_{j=1}^r\zeta_j|k_{\ell_j}|_\beta^2-\Gamma_0\bigg|\leq n\delta^{-1}L^{-2}\end{equation}where $n$ is the number of atoms in $\Mb$ (note that $n\leq (\log L)^3$), $\zeta_j$ equals $1$ or $-1$ depending on whether $\ell_j$ goes from $X$ to $X'$ or otherwise, $c_0$ is a constant vector depending only on the parameters $(c_v)$, and $\Gamma_0$ is a constant depending only on $(\Gamma_v)$. In particular, if $r=1$ (which means $\ell_1$ is a bridge) then $k_{\ell_1}$ is uniquely determined by $(c_v)$.
\end{lem}
\begin{proof} This follows from summing (\ref{decmole}) for all $v\in X$, and noticing that $k_\ell$, where $\ell$ is a bond connecting two atoms in $X$, appears exactly twice with opposite signs.
\end{proof}
\subsection{The steps}\label{oper} We start by defining all different types of steps. Recall the quantities defined in (\ref{defchar}). We will always have either $\Delta\gamma=\Delta\chi$ or $\Delta\gamma\geq\Delta\chi+\frac{1}{6(d-1)}$. In these two cases we call the step \emph{normal} or \emph{good}; good steps will be indicated by the letter ``G" appearing in the names.
\subsubsection{Degenerate atoms} In this step, assume $v$ is a non-isolated degenerate atom. Note that any atom with self-connecting bond must be degenerate, otherwise $\Df(\Mb)=\varnothing$ trivially.
\begin{itemize}
\item Step (DA): we remove the atom $v$, and all bonds connecting to it, and set $\Delta\mathtt{Ext}=\varnothing$.
\end{itemize}

Suppose $j\in\{0,1\}$ is the number of self-connecting bonds at $v$, and $h$ is the number of other atoms having bond(s) with $v$. Then for (DA) we have $\Delta E=-d(v)+j$, $\Delta V=-1$ and $\Delta F\leq h-1$. We define $\Delta \gamma=0$ if $d(v)\leq 3$ or if $d(v)=4$ and $\Delta F+j\geq 2$; otherwise let $\Delta\gamma=-2+\frac{1}{4}$. We also define $\Delta\kappa=0$, and $\mathtt{Ext}_{\mathrm{pre}}=\mathtt{Ext}_{\mathrm{pos}}$.
\begin{prop}\label{daprop} The step (DA) is either normal or good, and it satisfies (\ref{keyineq}). If $d(v)\geq 2$ and the step is normal we must have $\Delta\eta_*\leq -2$.
\end{prop}
\begin{proof} First, by counting the degree of $v$ we know $h+2j\leq d(v)$, so $\Delta\chi\leq -d(v)+j+h\leq 0$. If $\Delta\gamma=0$, then we have a normal or good step; if $\Delta\gamma=-2+\frac{1}{4}$, then $d(v)=4$ and $\Delta F+j\leq 1$, so $\Delta\chi=-4+j+\Delta F+1\leq -2$, and we have a good step. Now suppose $d(v)\geq 2$ and the step is normal, then $\Delta\chi=0$, hence $\Delta F=h-1$ and $d(v)=j+h$, which means that $j=0$, $d(v)=h$, and each bond connecting to $v$ is a single bond. We then have $\Delta F=d(v)-1$. As for the quantity $\rho_*:=V_3+2V_2+2V_1+2V_0$, the contribution to $\rho_*$ of each of the $d(v)$ atoms connected to $v$ changes from 0 to 1, or 1 to 2, or 2 to 2 after the removal of $v$. The contribution of $v$ itself to $\rho_*$ is $4-d(v)$ as $d(v)\geq 2$. We conclude that $\Delta\eta_*\leq d(v)-2(d(v)-1)-(4-d(v))=-2$, as desired.

Now we prove (\ref{keyineq}). Recall that $v$ is a degenerate atom, so $k_\ell$ are all equal for $\ell\sim v$, let this value be $k^*$. If $k[\Mb_{\mathrm{pre}}]\in\Df(\Mb_{\mathrm{pre}},\mathtt{Ext}_{\mathrm{pre}})$ and $k^*$ is fixed, then $k[\Mb_{\mathrm{pos}}]$, which is the restriction of $k[\Mb_{\mathrm{pre}}]$ to the bonds in $\Mb_{\mathrm{pos}}$, belongs to $\Df(\Mb_{\mathrm{pos}},\mathtt{Ext}_{\mathrm{pos}})$ with some new parameters that depend on the original parameters as well as $k^*$ (note that, if a degenerate atom $v'$ that is not tame in $\Mb_{\mathrm{pre}}$ becomes tame in $\Mb_{\mathrm{pos}}$, then $v'$ must be adjacent to $v$, so the value of $k_{\ell'}$ for $\ell'\sim v'$ must be fixed, once $k^*$ is fixed). This implies that
\[\sup\#\Df(\Mb_{\mathrm{pre}},\mathtt{Ext}_{\mathrm{pre}})\leq \Nf\cdot\sup\#\Df(\Mb_{\mathrm{pos}},\mathtt{Ext}_{\mathrm{pos}})\] where $\Nf$ is the number of choices for $k^*$. If $\Delta\gamma=-2+\frac{1}{4}$, this already implies (\ref{keyineq}), since $\Nf\lesssim L^d$ and $(d-1)(2-\frac{1}{4})>d$. If $\Delta\gamma=0$, we only need to show that $k^*$ is uniquely determined. This is true by definition if $d(v)\leq 3$; if $d(v)=4$ then $1+\Delta F> 2-j$, but the number of non-self-connecting bonds at $v$ is $2(2-j)$, so Lemma \ref{increF} implies that some bond $\ell_1$ connecting to $v$ must be a bridge. By Lemma \ref{cutlem} we know that $k_{\ell_1}$ is constant, hence $k^*$ is also constant and (\ref{keyineq}) is still true.
\end{proof}
\subsubsection{Triple bonds} From now on, in all subsequent steps, we assume $\Mb_{\mathrm{pre}}$ has no degenerate atom (and hence no self-connecting bond). In the current step, assume there is a triple bond between two atoms $v_1$ and $v_2$ in $\Mb_{\mathrm{pre}}$, such that $d(v_1)$ and $d(v_2)$ are not both $4$. In (TB-1) we assume $d(v_1)=d(v_2)=3$, so the triple bond is separated from the rest of the molecule; in (TB-2) we assume $d(v_1)=3$ and $d(v_2)=4$, so $v_2$ has an extra single bond.
\begin{itemize}
\item Steps (TB-1)--(TB2): we remove atoms $v_1$, $v_2$ and all bonds connecting to them, and set $\Delta\mathtt{Ext}=\varnothing$.
\end{itemize}

For (TB-1) we have $(\Delta V, \Delta E,\Delta F)=(-2,-3,-1)$, and for (TB-2), we have $(\Delta V, \Delta E,\Delta F)=(-2,-4,0)$. For both steps we define $\Delta\gamma=-2$, $\Delta\kappa=-1$ and $\mathtt{Ext}_{\mathrm{pre}}=\mathtt{Ext}_{\mathrm{pos}}$.
\begin{prop}\label{tbprop} The steps (TB-1) and (TB-2) are normal, and satisfy (\ref{keyineq}).
\end{prop}
\begin{proof} These steps are normal by definition, as $\Delta\chi=-2$. To prove (\ref{keyineq}), let the bonds in the triple bond be $\ell_j\,(1\leq j\leq 3)$, and let the extra single bond in the case of (TB-2) be $\ell_4$. For (TB-2) we have that $k_{\ell_4}$ is constant due to Lemma \ref{cutlem}, and in both cases $(k_{\ell_1},k_{\ell_2},k_{\ell_3})$ satisfies the system (\ref{3vcounting}) in Lemma \ref{lem:counting}, thanks to (\ref{decmole}). By Lemma \ref{lem:counting} (2) we have at most $C^+\delta^{-1}L^{2(d-1)}$ choices for these $(k_{\ell_j})$; by fixing their values and reducing to $k[\Mb_{\mathrm{pos}}]$ as in the proof of Proposition \ref{daprop} we can prove (\ref{keyineq}).
\end{proof}
\subsubsection{Bridge removal} In all subsequent steps, we assume $\Mb_{\mathrm{pre}}$ has no triple bonds. In the current step, we assume $\Mb_{\mathrm{pre}}$ contains a bridge $\ell$, which is a single bond connecting atoms $v_1$ and $v_2$.
\begin{itemize}
\item Step (BR): we remove the bond $\ell$, and set $\Delta\mathtt{Ext}=\varnothing$.
\end{itemize}
For (BR) we have $(\Delta V, \Delta E,\Delta F)=(0,-1,1)$ because removing a bridge adds one component. We also define $\Delta\gamma=\Delta\kappa=0$ and $\mathtt{Ext}_{\mathrm{pre}}=\mathtt{Ext}_{\mathrm{pos}}$.
\begin{prop}\label{brprop} The step (BR) is normal, and satisfies (\ref{keyineq}). Moreover we have $\Delta\eta=-2$ and $\Delta V_3\geq -2$, with equality holding only when $d(v_1)=d(v_2)=3$.
\end{prop}
\begin{proof} The step is normal because $\Delta\chi=0$. Let the bridge be $\ell$, then the value of $k_\ell$ must be fixed by Lemma \ref{cutlem}. Once $k_\ell$ is fixed, we can reduce to $k[\Mb_{\mathrm{pos}}]$ as before and this leads to (\ref{keyineq}).

The effect of (BR) reduces the degrees of two atoms each by 1, and adds one new component. by definition of $\eta$ we have $\Delta\eta=2-4=-2$, because the contribution to $\rho:=V_3+2V_2+3V_1+4V_0$ of each of the two atoms connected by $\ell$ changes from 0 to 1, or 1 to 2, or 2 to 3, or 3 to 4 after the removal of $\ell$. Moreover $\Delta V_3\geq -2$ is clear from definition, and equality holds only when both $v_j$ before removal of the bridge have degree 3.
\end{proof}
\subsubsection{Degree 3 atoms connected by a single bond}\label{3s3} In all subsequent steps, we assume there is no bridge in $\Mb_{\mathrm{pre}}$. In the current step, we assume that there are two degree 3 atoms $v_1$ and $v_2$, connected by a single bond $\ell_1$. Then $\Mb_{\mathrm{pre}}$ must contain one of the functional groups shown in Figures \ref{fig:3s3} and \ref{fig:3s3new}. Recall the definition of good and bad vectors in Lemma \ref{goodvec}.

In steps (3S3-1)--(3S3-4G) we assume that $v_1$ and $v_2$ each has two more single bonds $\ell_2,\ell_3$ and $\ell_4,\ell_5$, connecting to four different atoms $v_3,v_4$ and $v_5,v_6$ { labeled as in Figure \ref{fig:3s3}}. In (3S3-1)--(3S3-3G) we assume that (i) after removing $\{v_1,v_2\}$ and all bonds connecting to them, $\{v_3,v_5\}$ is in one new component, and $\{v_4,v_6\}$ is in the other new component, and that (ii) the bonds $\ell_2$ and $\ell_4$ have opposite directions (viewing from $\{v_1,v_2\}$), and the bonds $\ell_3$ and $\ell_5$ also have opposite directions. In (3S3-4G) we assume either (i) or (ii) is false. Moreover, in (3S3-1) we assume that $d(v_3)=\cdots=d(v_6)=4$, and in (3S3-3G) we assume that $d(v_3)$ and $d(v_5)$ are not both 4. Finally, in (3S3-5G) we assume the functional groups around $v_1$ and $v_2$ are like the ones shown in Figure \ref{fig:3s3new}.
  \begin{figure}[h!]
  \includegraphics[scale=.5]{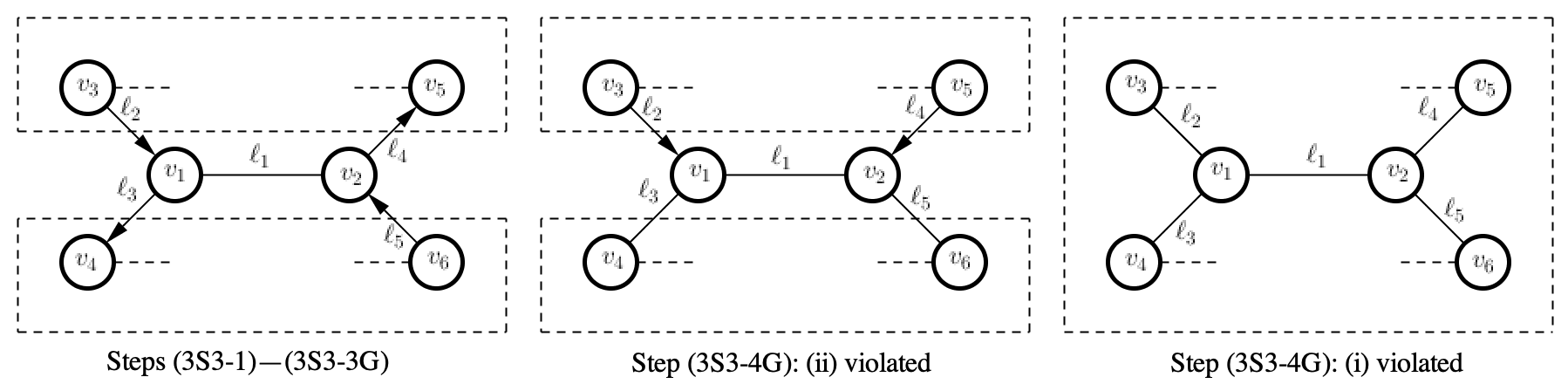}
  \caption{The functional group involved in steps (3S3-1)--(3S3-4G). In the first two pictures $\{v_3,v_5\}$ are $\{v_4,v_6\}$ are not in the same component after removing $v_1$ and $v_2$, while in the third picture they are.}
  \label{fig:3s3}
\end{figure} 
  \begin{figure}[h!]
  \includegraphics[scale=.5]{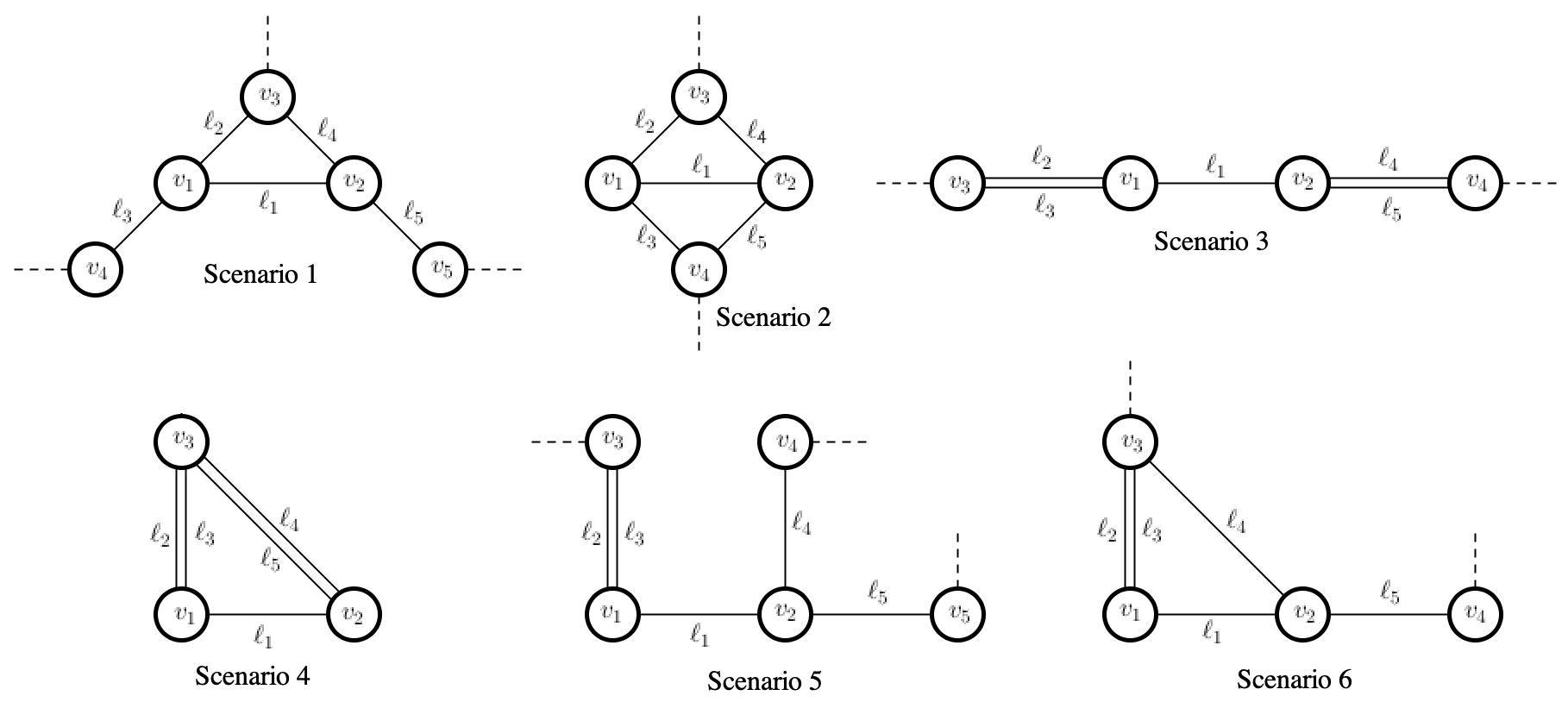}
  \caption{The functional groups involved in step (3S3-5G). In total there are 6 scenarios.}
  \label{fig:3s3new}
\end{figure} 
\begin{itemize}
\item Step (3S3-1): we remove the atoms $\{v_1,v_2\}$ and all (five) bonds connecting to them. In this step we set $\Delta\mathtt{Ext}$ to be {the condition ``$k_{\ell_2}=k_{\ell_4}$ and $k_{\ell_3}=k_{\ell_5}$ and $k_{\ell_1}-k_{\ell_3}$ is a good vector" if $\ell_1$ and $\ell_3$ have opposite directions (viewing from $v_1$), and to be the condition ``$k_{\ell_2}=k_{\ell_4}$ and $k_{\ell_3}=k_{\ell_5}$" if $\ell_1$ and $\ell_3$ have the same direction.}
\item Step (3S3-2G): we remove $\{v_1,v_2\}$ and all bonds connecting to them, but set $\Delta\mathtt{Ext}$ to be the negation {(i.e. logical $\mathtt{NOT}$)} of the condition in (3S3-1).
\item Step (3S3-3G): we remove $\{v_1,v_2\}$ and all bonds connecting to them, but add a new bond $\ell_6$ between $v_3$ and $v_5$ (not drawn in Figure \ref{fig:3s3}), which goes from $v_3$ to $v_5$ if $\ell_2$ goes from $v_3$ to $v_1$ and vice versa. We set $\Delta\mathtt{Ext}$ to be the condition in (3S3-1).
\item Step (3S3-4G)--(3S3-5G): we remove $\{v_1,v_2\}$ and all bonds connecting to them, and set $\Delta\mathtt{Ext}=\varnothing$.
\end{itemize}

{We remark that (3S3-1)--(3S3-5G) are just the possible steps one can perform; the exact choice of steps and order of performing will be fixed in the algorithm in Section \ref{alg} below.} For (3S3-1) and (3S3-2G) we have $(\Delta V,\Delta E,\Delta F)=(-2,-5,1)$, and for (3S3-3G) we have $(\Delta V,\Delta E,\Delta F)=(-2,-4,1)$. For (3S3-4G), if (i) is not violated, then we still have $(\Delta V,\Delta E,\Delta F)=(-2,-5,1)$; if (i) is violated then we must have $(\Delta V,\Delta E,\Delta F)=(-2,-5,0)$. This is because $\Delta F\leq 1$ by Lemma \ref{increF}, and if $\Delta F=1$ then we may assume $\{v_3,v_5\}$ is in one component and $\{v_4,v_6\}$ is in the other component after the removing $\{v_1,v_2\}$, since otherwise $\ell_1$ would be a bridge. As for (3S3-5G), the calculation depends on the scenario. For Scenarios 1 and 2, we have $(\Delta V,\Delta E,\Delta F)$ equals either $(-2,-5,0)$ or $(-2,-5,1)$, while for Scenarios 3--6 we must have $(\Delta V,\Delta E,\Delta F)=(-2,-5,0)$; these can be verified basically by using Lemma \ref{increF}.

We define $\Delta\gamma=-2$ for (3S3-1), and $\Delta\gamma=\Delta\chi+\frac{1}{6(d-1)}$ for all other steps. We also define $\Delta\kappa=-1$ for (3S3-1) and (3S3-3G), and $\Delta\kappa=-2$ for all other steps. For the four steps other than (3S3-3G), we define $\mathtt{Ext}_{\mathrm{pre}}=\mathtt{Ext}_{\mathrm{pos}}\cup\Delta\mathtt{Ext}$, while for (3S3-3G) we define \begin{equation}\label{renewext}\mathtt{Ext}_{\mathrm{pre}}=\mathtt{Ext}_{\mathrm{pos}}'\cup\Delta\mathtt{Ext},\end{equation} where $\mathtt{Ext}_{\mathrm{pos}}'$ is obtained by replacing each occurrence of $k_{\ell_6}$ in $\mathtt{Ext}_{\mathrm{pos}}$ with $k_{\ell_2}$.
\begin{prop}\label{3s3prop} Each of the five steps verifies Condition 1 and satisfies (\ref{keyineq}). Moreover (3S3-1) is normal and satisfies $\Delta\eta=-2$ and $\Delta V_3=2$, while the other four are good. 
\end{prop}
\begin{proof} We only need to verify Condition 1 for (3S3-3G), which adds a new bond to the molecule. This is true because the new bond is added in the component containing $v_3$ and $v_5$, and this component does not become saturated because $d(v_3)$ and $d(v_5)$ are not both 4. Moreover (3S3-1) is normal and the other four steps are good by definition, and for (3S3-1) we have $\Delta\eta=-2$ and $\Delta V_3=2$ since originally $d(v_3)=\cdots=d(v_6)=4$. Now we need to prove (\ref{keyineq}).

For (3S3-1), as part of $\mathtt{Ext}_{\mathrm{pre}}$ we have $k_{\ell_2}=k_{\ell_4}$ and $k_{\ell_3}=k_{\ell_5}$, and $(k_{\ell_1},k_{\ell_2},k_{\ell_3})$ satisfies the system (\ref{3vcounting}) in Lemma \ref{lem:counting} due to (\ref{decmole}). Therefore we have at most $C^+\delta^{-1}L^{2(d-1)}$ choices for these due to Lemma \ref{lem:counting} (2), and if $(k_{\ell_1},k_{\ell_2},k_{\ell_3})$ is fixed, we can reduce to $k[\Mb_{\mathrm{pos}}]$ and prove (\ref{keyineq}).

For (3S3-2G), (3S3-4G) and (3S3-5G) the argument is the same, except that now $(k_{\ell_1},\cdots,k_{\ell_5})$ satisfies the system (\ref{5vcounting1}) for some choice of signs $(\zeta_1,\cdots,\zeta_5)$. If $\Delta\chi=-3$ then the number of choices for $(k_{\ell_1},\cdots,k_{\ell_5})$ is at most $C^+\delta^{-2}L^{3(d-1)-\frac{1}{4}}$ by Lemma \ref{lem:counting} (5), which proves (\ref{keyineq}); so we only need to consider $\Delta\chi=-2$. In (3S3-2G) and (3S3-4G), by using Lemma \ref{cutlem} we know that in addition to (\ref{5vcounting1}) we also have (\ref{5vcounting3}); in (3S3-5G), if $\Delta\chi=-2$ then we must be in Scenarios 1 or 2, and it is easy to check that (\ref{5vcounting3}) also holds. As such, Lemma \ref{lem:counting} (7) bounds the number of choices for $(k_{\ell_1},\cdots,k_{\ell_5})$ by $C^+\delta^{-2}L^{3d-3-\frac{1}{6}}$, which proves (\ref{keyineq}), \emph{unless} $(\zeta_2,\zeta_3)=(\zeta_4,\zeta_5)$ and $(k_{\ell_2},k_{\ell_3})=(k_{\ell_4},k_{\ell_5})$. This last case cannot happen in (3S3-4G) due to the directions of $\ell_2$ and $\ell_4$, nor in (3S3-5G) as $v_3$ cannot be degenerate, so we only need to consider (3S3-2G), where $\mathtt{Ext}_{\mathrm{pre}}$ implies that $\ell_1$ and $\ell_3$ have opposite directions, and $k_{\ell_1}-k_{\ell_3}$ is a bad vector. By Lemma \ref{goodvec}, we know $k_{\ell_2}$ has at most $C^+L^{d-1-\frac{1}{4}}$ choices, and when $k_{\ell_2}=k_{\ell_4}$ is fixed, the number of choices for $(k_{\ell_1},k_{\ell_3},k_{\ell_5})$ is at most $C^+\delta^{-1}L^{d-1}$ using Lemma \ref{lem:counting} (1). Thus the number of choices for $(k_{\ell_1},\cdots,k_{\ell_5})$ is at most $C^+\delta^{-1}L^{2(d-1)-\frac{1}{4}}$, which proves (\ref{keyineq}).

Finally consider (3S3-3G). Note that $\Mb_{\mathrm{pos}}$ has two components (assuming $\Mb_{\mathrm{pre}}$ is connected; otherwise consider the current component of $\Mb_{\mathrm{pre}}$), namely $\Mb'$ containing $\{v_3,v_5\}$, and $\Mb''$ containing $\{v_4,v_6\}$. Moreover by Condition 3, $\mathtt{Ext}_{\mathrm{pos}}$ is the union of $\mathtt{Ext}'$ and $\mathtt{Ext}''$, which only involve bonds in $\Mb'$ and $\Mb''$ respectively. For any $k[\Mb_{\mathrm{pre}}]\in\Df(\Mb_{\mathrm{pre}},\mathtt{Ext}_{\mathrm{pre}})$ and assuming $k_{\ell_2}=k_{\ell_4}$, we can define \begin{equation}\label{renewk}k'[\Mb_{\mathrm{pos}}]=(k_\ell')_{\ell\in\Mb_{\mathrm{pos}}},\quad k_{\ell}'=\left\{
\begin{aligned}&k_\ell,&\ell&\neq \ell_6,\\
&k_{\ell_2},&\ell&=\ell_6.
\end{aligned}
\right.
\end{equation}Note that $k'[\Mb_{\mathrm{pos}}]$ can be divided into $k'[\Mb']$ and $k'[\Mb'']$, the latter being the restriction of $k[\Mb_{\mathrm{pre}}]$ to $\Mb''$. Moreover, we can check that $k'[\Mb']$ belongs to $\Df(\Mb',\mathtt{Ext}')$ with essentially the original parameters (where in the place of $a_{\ell_6}$ we have $a_{\ell_2}$). Once $k'[\Mb']$ is fixed, in particular $k_{\ell_2}=k_{\ell_4}=k_{\ell_6}'$ is fixed, then $k_{\ell_1}$ and $k_{\ell_3}=k_{\ell_5}$ satisfy the system (\ref{2vcounting}) in Lemma \ref{lem:counting}. If $\ell_1$ and $\ell_3$ have the same direction, then the number of choices for $(k_{\ell_1},k_{\ell_3})$ is at most $C^+\delta^{-1}L^{d-1-\frac{1}{3}}$ by Lemma \ref{lem:counting} (1); if they have opposite directions, then $k_{\ell_1}-k_{\ell_3}$ must be a good vector due to $\mathtt{Ext}_{\mathrm{pre}}$. Repeating the argument in the proof of Lemma \ref{lem:counting} (1), and using the definition of good vectors (decomposing intervals of length $\delta^{-1}L^{-2}$ into intervals of length $L^{-2}$ if necessary), we see that the number of choices for $(k_{\ell_1},k_{\ell_3})$ is at most $C^+\delta^{-1}L^{d-1-\frac{1}{4}}$. In either case, once $k_{\ell_3}=k_{\ell_5}$ is fixed, $k'[\Mb'']$ will belong to $\Df(\Mb'',\mathtt{Ext}'')$ with some new parameters that depend on the original parameters as well as $k_{\ell_3}$. This implies that
\[\sup\#\Df(\Mb_{\mathrm{pre}},\mathtt{Ext}_{\mathrm{pre}})\leq \sup\#\Df(\Mb',\mathtt{Ext}')\cdot C^+\delta^{-1}L^{d-1-\frac{1}{4}}\cdot \sup\#\Df(\Mb'',\mathtt{Ext}''),\] however since $\Mb_{\mathrm{pos}}$ is the disjoint union of $\Mb'$ and $\Mb''$ and $\mathtt{Ext}_{\mathrm{pos}}$ is the union of $\mathtt{Ext}'$ and $\mathtt{Ext}''$, it is easy to see that
\[\sup\#\Df(\Mb',\mathtt{Ext}')\cdot\sup\#\Df(\Mb'',\mathtt{Ext}'')=\sup\#\Df(\Mb_{\mathrm{pos}},\mathtt{Ext}_{\mathrm{pos}}),\] which proves (\ref{keyineq}).
\end{proof}
\subsubsection{Degree 3 atoms connected by a double bond}\label{3d3} In this step, we assume there are two degree $3$ atoms $v_1$ and $v_2$, connected by a double bond $(\ell_1,\ell_2)$, which are also connected to two other atoms $v_3$ and $v_4$ by two single bonds $\ell_3$ and $\ell_4$, see Figures \ref{fig:3d3} and \ref{fig:3d3new}. In (3D3-1)--(3D3-3G) and (3D3-6G) we assume $v_3\neq v_4$ and $\ell_3$ and $\ell_4$ are in opposite directions (viewing from $\{v_1,v_2\}$); in (3D3-1) we assume $d(v_3)=d(v_4)=4$, and in (3D3-3G) we assume that \emph{not} all atoms in the current component other than $\{v_1,v_2\}$ have degree 4. In (3D3-4G) we assume $v_3\neq v_4$ and $\ell_3$ and $\ell_4$ are in the same direction, and in (3D3-5G) we assume $v_3=v_4$. Finally, in (3D3-6G) we assume that $v_3$ is connected to $v_4$ via a single bond $\ell_5$, and $v_3$ and $v_4$ are each connected to different atoms $v_5$ and $v_6$ via double bonds $(\ell_6,\ell_7)$ and $(\ell_8,\ell_9)$, see Figure \ref{fig:3d3new}. Recall the definition of good and bad vectors in Lemma \ref{goodvec}.
  \begin{figure}[h!]
  \includegraphics[scale=.5]{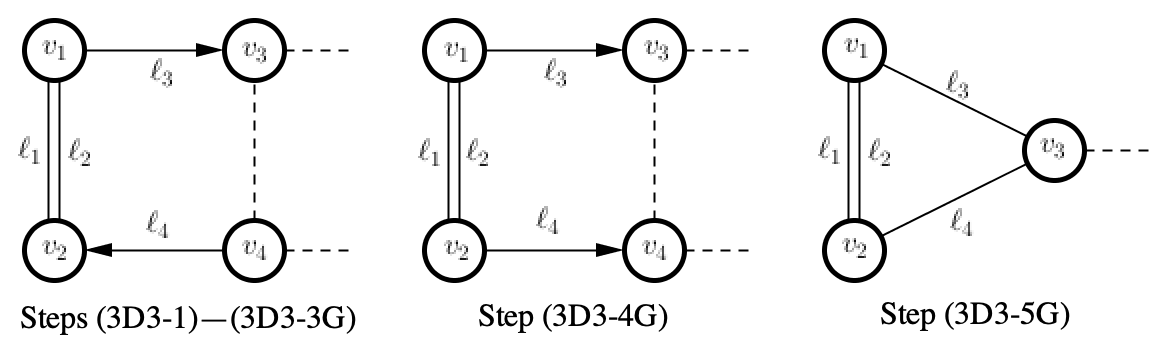}
  \caption{The functional groups involved in steps (3D3-1)--(3D3-5G).}
  \label{fig:3d3}
\end{figure} 
  \begin{figure}[h!]
  \includegraphics[scale=.5]{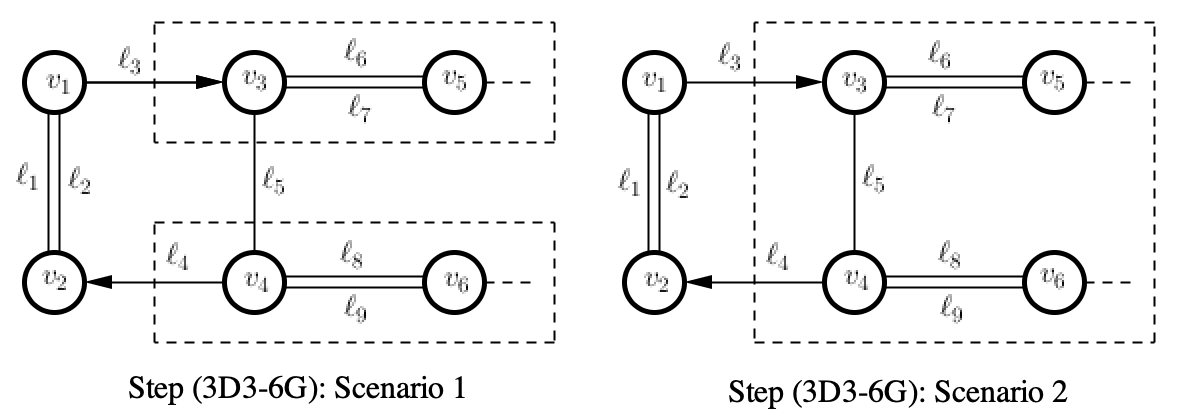}
  \caption{The functional group involved in step (3D3-6G). In the left picture $\ell_5$ becomes a bridge after removing $\{v_1,v_2\}$, while in the right picture it does not.}
  \label{fig:3d3new}
\end{figure} 
\begin{itemize}
\item Step (3D3-1): we remove the atoms $\{v_1,v_2\}$ and all (four) bonds connecting to them. In this step we set $\Delta\mathtt{Ext}$ to be {the condition that ``$k_{\ell_3}=k_{\ell_4}$ and $k_{\ell_1}-k_{\ell_2}$ is a good vector" if $\ell_1$ and $\ell_2$ have opposite directions, and to be the condition ``$k_{\ell_3}=k_{\ell_4}$" if $\ell_1$ and $\ell_2$ have the same direction.}
\item Step (3D3-2G): we remove $\{v_1,v_2\}$ and all bonds connecting to them, but set $\Delta\mathtt{Ext}$ to be the negation {(i.e. logical $\mathtt{NOT}$)} of the condition in (3D3-1).
\item Step (3D3-3G): we remove $\{v_1,v_2\}$ and all bonds connecting to them, but add a new bond $\ell_5$ between $v_3$ and $v_4$ (not drawn in Figure \ref{fig:3d3}), which goes from $v_4$ to $v_3$ if $\ell_3$ goes from $v_1$ to $v_3$ and vice versa. We set $\Delta\mathtt{Ext}$ to be the condition in (3D3-1).
\item Steps (3D3-4G)--(3D3-5G): we remove $v_1$ and $v_2$ and all bonds connecting to them, and set $\Delta\mathtt{Ext}=\varnothing$.
\item Step (3D3-6G): we remove the atoms $\{v_1,\cdots,v_4\}$ and all (nine) bonds connecting to them, and set $\Delta\mathtt{Ext}=\varnothing$.
\end{itemize}

For the four steps other than (3D3-3G) and (3D3-6G) we have $(\Delta V, \Delta E,\Delta F)=(-2,-4,0)$, where $\Delta F=0$ due to Lemma \ref{increF}, since $\Mb_{\mathrm{pre}}$ has no bridge. For (3D3-3G) we have $(\Delta V, \Delta E,\Delta F)=(-2,-3,0)$ for the same reason. Finally for (3D3-6G) we have $(\Delta V, \Delta E,\Delta F)$ equals either $(-4,-9,0)$ (if $\ell_5$ does not become a bridge after removing $\{v_1,v_2\}$) or $(-4,-9,1)$ (if it does).

Define $(\Delta\gamma,\Delta\kappa)=(-2,-1)$ for (3D3-1), $(\Delta\gamma,\Delta\kappa)=(-1+\frac{1}{4(d-1)},-1)$ for (3D3-3G), $(\Delta\gamma,\Delta\kappa)=(\Delta\chi+\frac{1}{6(d-1)},-4)$ for (3D3-6G), and $(\Delta\gamma,\Delta\kappa)=(-2+\frac{1}{4(d-1)},-2)$ for the other three steps. For the five steps other than (3D3-3G), which do not add new bonds, we define $\mathtt{Ext}_{\mathrm{pre}}=\mathtt{Ext}_{\mathrm{pos}}\cup\Delta\mathtt{Ext}$; for (3D3-3G) we define $\mathtt{Ext}_{\mathrm{pre}}$ as in (\ref{renewext}), but in $\mathtt{Ext}_{\mathrm{pos}}'$ we replace each occurrence of $k_{\ell_5}$ by $k_{\ell_3}$.
\begin{prop}\label{3d3prop} Each of the six steps verifies Condition 1 and satisfies (\ref{keyineq}). Moreover (3D3-1) is normal and and satisfies $\Delta\eta=\Delta V_3=0$, while the other five are good. 
\end{prop}
\begin{proof} We only need to verify Condition 1 for (3D3-3G). This is because the operation does not add any new component, and the existing component does not become saturated, because by assumption at least one atom in the current component other than $v_1$ and $v_2$ does not have degree 4. Moreover (3D3-1) is normal and the other four steps are good, which follows directly from definition, and in (3D3-1) we are assuming $d(v_3)=d(v_4)=4$ before the operation, so it is clear that $\Delta\eta=\Delta V_3=0$. Thus it suffices to prove (\ref{keyineq}).

For (3D3-1), as part of $\mathtt{Ext}_{\mathrm{pre}}$ we have $k_{\ell_3}=k_{\ell_4}$, and $(k_{\ell_1},k_{\ell_2},k_{\ell_3})$ satisfies the system (\ref{3vcounting}) in Lemma \ref{lem:counting} due to (\ref{decmole}). Therefore we have at most $C^+\delta^{-1}L^{2(d-1)}$ choices for these due to Lemma \ref{lem:counting}, and if $(k_{\ell_1},k_{\ell_2},k_{\ell_3})$ is fixed, we can reduce to $k[\Mb_{\mathrm{pos}}]$ and prove (\ref{keyineq}).

For (3D3-2G), (3D3-4G) and (3D3-5G) the argument is the same, except that now $(k_{\ell_1},\cdots,k_{\ell_4})$ has to satisfy the system (\ref{4vcounting2}) with some choice of signs $(\zeta_1,\cdots,\zeta_4)$. By Lemma \ref{lem:counting} (4), we get at most $C^+\delta^{-2}L^{2(d-1)-\frac{1}{4}}$ choices for $(k_{\ell_1},\cdots,k_{\ell_4})$, which proves (\ref{keyineq}), \emph{unless} $\zeta_3=\zeta_4$ and $k_{\ell_3}=k_{\ell_4}$. The latter case cannot happen in (3D3-4G) due to the directions of $\ell_3$ and $\ell_4$, nor in (3D3-5G) because $v_3$ cannot be degenerate. If it happens in (3D3-2G), then due to $\mathtt{Ext}_{\mathrm{pre}}$, we know that the directions of $\ell_1$ and $\ell_2$ must be opposite, and $k_{\ell_1}-k_{\ell_2}$ is a bad vector. Then, just like in the proof of Proposition \ref{3s3prop}, we know $k_{\ell_3}$ has at most $C^+L^{d-1-\frac{1}{4}}$ choices, and the number of choices for $(k_{\ell_1},\cdots,k_{\ell_4})$ is at most $C^+\delta^{-1}L^{2(d-1)-\frac{1}{4}}$, which proves (\ref{keyineq}).

Next consider (3D3-6G). By the same argument, we only need to bound the number of choices for $(k_{\ell_1},\cdots,k_{\ell_9})$. If $\ell_5$ does not become a bridge after removing $v_1$ and $v_2$, then $\Delta\chi=-5$. By repeating the proof above and the proof of Proposition \ref{3s3prop} (see (3S3-5G), Scenario 3), we know that (i) the number of choices for $(k_{\ell_1},\cdots,k_{\ell_4})$ is at most $C^+\delta^{-1}L^{2(d-1)}$, and (ii) once $(k_{\ell_1},\cdots,k_{\ell_4})$ is fixed, the number of choices for $(k_{\ell_5},\cdots,k_{\ell_9})$ is at most $C^+\delta^{-2}L^{3(d-1)-\frac{1}{4}}$. Therefore the number of choices for $(k_{\ell_1},\cdots,k_{\ell_9})$ is at most $C^+\delta^{-3}L^{5(d-1)-\frac{1}{4}}$, which implies (\ref{keyineq}).

Now, if $\ell_5$ does become a bridge after removing $v_1$ and $v_2$, then $\Delta\chi=-4$. By Lemma \ref{cutlem}, we know that  $(k_{\ell_3},k_{\ell_5})$ satisfies the system (\ref{2vcounting}) in Lemma \ref{lem:counting}, but with $n\delta^{-1}L^{-2}$ replacing $\delta^{-1}L^{-2}$. Since  $(k_{\ell_1},k_{\ell_2},k_{\ell_3})$ also satisfies (\ref{3vcounting}), we can apply Lemma \ref{lem:counting} (3), with a further division of intervals if necessary, to bound the number of choices for $(k_{\ell_1},k_{\ell_2},k_{\ell_3},k_{\ell_5})$ by $nC^+\delta^{-2} L^{2(d-1)-\frac{1}{4}}$. Once $(k_{\ell_1},k_{\ell_2},k_{\ell_3},k_{\ell_5})$ is fixed, then $k_{\ell_4}$ is also fixed, and number of choices for $(k_{\ell_6},\cdots,k_{\ell_9})$ is bounded by $C^+\delta^{-2}L^{2(d-1)}$ by Lemma \ref{lem:counting} (1). Therefore the number of choices for $(k_{\ell_1},\cdots,k_{\ell_9})$ is at most $C^+\delta^{-4}L^{4(d-1)-\frac{1}{6}}$ (recall $n\leq (\log L)^3$), which implies (\ref{keyineq}).

Finally consider (3D3-3G). Given any $k[\Mb_{\mathrm{pre}}]\in\Df(\Mb_{\mathrm{pre}},\mathtt{Ext}_{\mathrm{pre}})$ and assuming $k_{\ell_3}=k_{\ell_4}$, we define $k'[\Mb_{\mathrm{pos}}]$ as (\ref{renewk}), but with $k_{\ell_5'}=k_{\ell_3}$. By the same observation, we see that $k'[\Mb_{\mathrm{pos}}]$ belongs to $\Df(\Mb_{\mathrm{pos}},\mathtt{Ext}_{\mathrm{pos}})$ with essentially the original parameters (where in the place of $a_{\ell_5}$ we have $a_{\ell_3}$). Once $k'[\Mb_{\mathrm{pos}}]$ is fixed, then $(k_{\ell_1},k_{\ell_2})$ satisfies the system (\ref{2vcounting}) in Lemma \ref{lem:counting}; moreover by $\mathtt{Ext}_{\mathrm{pre}}$ we know that either $\ell_1$ and $\ell_2$ have the same direction or $k_{\ell_1}-k_{\ell_2}$ is a good vector. Just like in the proof of Proposition \ref{3s3prop}, we see that the number of choices for $(k_{\ell_1},k_{\ell_2})$ is at most $C^+\delta^{-1}L^{d-1-\frac{1}{4}}$. This implies that
\[\sup\#\Df(\Mb_{\mathrm{pre}},\mathtt{Ext}_{\mathrm{pre}})\leq \sup\#\Df(\Mb_{\mathrm{pos}},\mathtt{Ext}_{\mathrm{pos}})\cdot C^+\delta^{-1}L^{d-1-\frac{1}{4}},\] which proves (\ref{keyineq}).
\end{proof}
\subsubsection{Degree 3 and 4 atoms connected by a double bond} In this step, we assume there is an atom $v_1$ of degree 3, and another atom $v_2$ of degree 4, that are connected by a double bond $(\ell_1,\ell_2)$. Then $\Mb_{\mathrm{pre}}$ must contain one of the functional groups shown in Figure \ref{fig:3d4}.
  \begin{figure}[h!]
  \includegraphics[scale=.5]{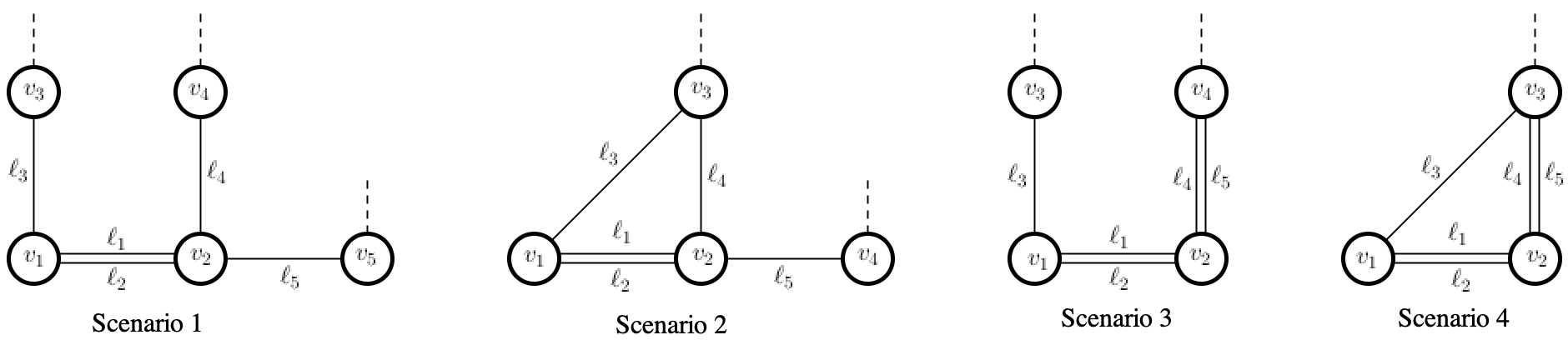}
  \caption{The functional groups involved in step (3D4G). In total there are 4 scenarios.}
  \label{fig:3d4}
\end{figure} 
\begin{itemize}
\item Step (3D4G): we remove the atoms $\{v_1,v_2\}$ and all (five) bonds connecting to them, and set $\Delta\mathtt{Ext}=\varnothing$.
\end{itemize}

For (3D4G), we can check using Lemma \ref{increF} that, in each scenario, we always have $(\Delta V,\Delta E,\Delta F)=(-2,-5,0)$. We define $\Delta\gamma=-3+\frac{1}{4(d-1)}$, $\Delta\kappa=-2$ and $\mathtt{Ext}_{\mathrm{pre}}=\mathtt{Ext}_{\mathrm{pos}}$.
\begin{prop}\label{3d4prop} The step (3D4G) is good, and satisfies (\ref{keyineq}).
\end{prop}
\begin{proof} The step is good by definition. Now by (\ref{decmole}) we know that $(k_{\ell_1},\cdots,k_{\ell_5})$ satisfies the system (\ref{5vcounting2}) in Lemma \ref{lem:counting} with some choice of signs $(\zeta_1,\cdots,\zeta_5)$. By Lemma \ref{lem:counting} (6) they have at most $C^+\delta^{-2}L^{3(d-1)-\frac{1}{4}}$ choices, and once they are fixed we can reduce to $k[\Mb_{\mathrm{pos}}]$ and prove (\ref{keyineq}).
\end{proof}
\subsubsection{Degree 3 and 2 atoms connected} In this step, we assume there is an atom $v_1$ of degree 3, and another atom $v_2$ of degree 2, that are connected. Note that they must be connected by a single bond $\ell_1$, otherwise there would be a bridge. Then, $\Mb_{\mathrm{pre}}$ must contain one of the functional groups shown in Figure \ref{fig:3s2}.
  \begin{figure}[h!]
  \includegraphics[scale=.5]{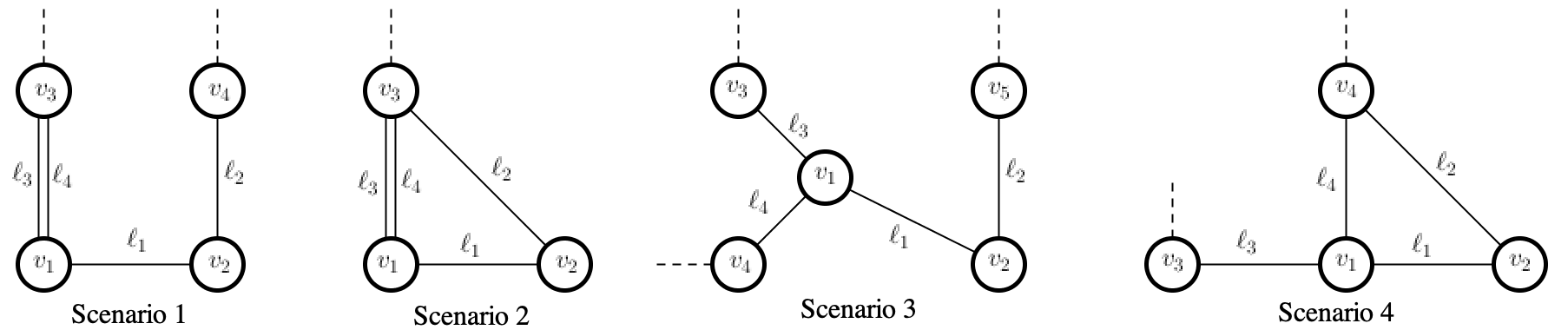}
  \caption{The functional groups involved in step (3S2G). In total there are 4 scenarios.}
  \label{fig:3s2}
\end{figure} 
\begin{itemize}
\item Step (3S2G): we remove the atoms $\{v_1,v_2\}$ and all (four) bonds connecting to them, and set $\Delta\mathtt{Ext}=\varnothing$.
\end{itemize}

For (3S2G), we can check using Lemma \ref{increF} that, in each scenario, we always have $(\Delta V,\Delta E,\Delta F)=(-2,-4,0)$. We define $\Delta\gamma=-2+\frac{1}{4(d-1)}$, $\Delta\kappa=-2$ and $\mathtt{Ext}_{\mathrm{pre}}=\mathtt{Ext}_{\mathrm{pos}}$.
\begin{prop}\label{3s2prop} The step (3S2G) is good, and satisfies (\ref{keyineq}).
\end{prop}
\begin{proof}The step is good by definition. Now by (\ref{decmole}) we know that $(k_{\ell_1},\cdots,k_{\ell_4})$ satisfies the system (\ref{4vcounting1}) in Lemma \ref{lem:counting}, with some choice of signs $(\zeta_1,\cdots,\zeta_4)$. By Lemma \ref{lem:counting} (3) they have at most $C^+\delta^{-2}L^{2(d-1)-\frac{1}{4}}$ choices, and once they are fixed we can reduce to $k[\Mb_{\mathrm{pos}}]$ and prove (\ref{keyineq}).
\end{proof}
\subsubsection{Degree 3 atom removal}\label{3r} In this step, we assume there is an atom $v$ of degree 3, which is connected to three atoms $v_j\,(1\leq j\leq 3)$ of degree 4, by three single bonds $\ell_j\,(1\leq j\leq 3)$. In step (3R-2G) we further assume that, there is a special bond $\ell_1'$ (see Definition \ref{defbridge}) in the molecule (or component) after removing the atom $v$ and the bonds $\ell_j$. In this case, suppose $\ell_1'$ connects atoms $v_1'$ and $v_2'$, $v_1'$ is connected to $v_3'$ by a double bond $(\ell_2',\ell_3')$, and $v_2'$ is connected to $v_4'$ by a double bond $(\ell_4',\ell_5')$, see Figure \ref{fig:3rproof}.
  \begin{figure}[h!]
  \includegraphics[scale=.5]{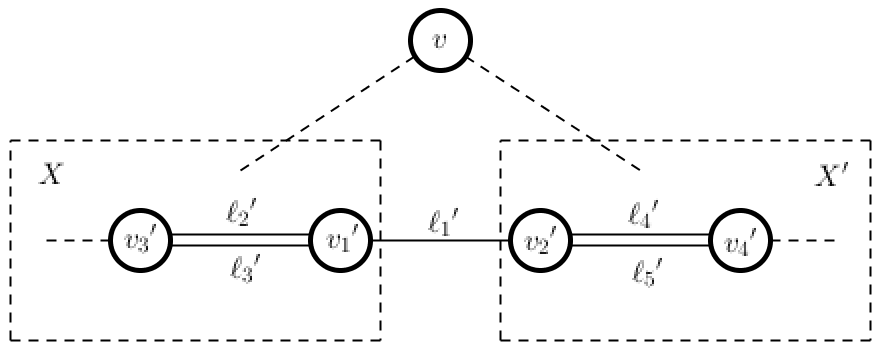}
  \caption{The functional group involved in step (3R-2G). Here $v_1,v_2,v_3$ are not drawn; some of them may coincide with some $v_j'$. Also we only draw the scenario where $\ell_1'$ becomes a bridge after removing $v$, but the other scenario is also possible}
  \label{fig:3rproof}
\end{figure} 
\begin{itemize}
\item Step (3R-1): we remove the atom $v$ and all (three) bonds connecting to it, and set $\Delta\mathtt{Ext}=\varnothing$.
\item Step (3R-2G): we we remove the atom $v$ and all (three) bonds connecting to it. Then we remove the atoms $\{v_1',v_2'\}$  and all (five) bonds connecting to them. We also set $\Delta\mathtt{Ext}=\varnothing$.
\end{itemize}

Clearly the operation of removing $v$ and $\ell_j\,(1\leq j\leq 3)$ does not increase the number of components (by Lemma \ref{increF}). Therefore for (3R-1) we have $(\Delta V,\Delta E,\Delta F)=(-1,-3,0)$. As for (3R-2), we have $(\Delta V,\Delta E,\Delta F)$ equals either $(-3,-8,0)$ or $(-3,-8,1)$, depending on whether $\ell_1'$ becomes a bridge after removing $v$. For (3R-1) we define $\Delta\gamma=-2$ and $\Delta\kappa=-1$, and for (3R-2) we define $\Delta\gamma=\Delta\chi+\frac{1}{6(d-1)}$ and $\Delta\kappa=-4$. In both cases we define $\mathtt{Ext}_{\mathrm{pre}}=\mathtt{Ext}_{\mathrm{pos}}$.
\begin{prop}\label{3rprop} The step (3R-1) is normal, and satisfies $\Delta\eta=2$ and $\Delta V_3=2$. The step (3R-2G) is good. Both satisfy (\ref{keyineq}).
\end{prop}
\begin{proof} The step (3R-1) is normal and (3R-2G) is good by definition, the equalities {for} $\Delta \eta$ and $\Delta V_3$ are also easily verified.

To prove (\ref{keyineq}), note that this is clear for (3R-1) because $(k_{\ell_1},k_{\ell_2},k_{\ell_3})$ satisfies the system (\ref{3vcounting}) in Lemma \ref{lem:counting} and thus the number of choices for these is at most $C^+\delta^{-1}L^{2(d-1)}$, and then (\ref{keyineq}) follows by reducing to $k[\Mb_{\mathrm{pos}}]$ as before. Now we only need to consider (3R-2G). If $\ell_1'$ does not become a bridge after removing $v$, then $\Delta\chi=-5$ and $\Delta\gamma=-5+\frac{1}{6(d-1)}$. In this case, by repeating the proof of Proposition \ref{3s3prop} (see (3S3-5G), Scenario 3), we know that (i) the number of choices for $(k_{\ell_1},k_{\ell_2},k_{\ell_3})$ is at most $C^+\delta^{-1}L^{2(d-1)}$, and (ii) once $(k_{\ell_1},k_{\ell_2},k_{\ell_3})$ is fixed, the number of choices for $(k_{\ell_1'},\cdots,k_{\ell_5'})$ is at most $C^+\delta^{-2}L^{3(d-1)-\frac{1}{4}}$. Therefore the number of choices for $(k_{\ell_1},k_{\ell_2},k_{\ell_3},k_{\ell_1'},\cdots,k_{\ell_5'})$ is at most $C^+\delta^{-3}L^{5(d-1)-\frac{1}{4}}$, which implies (\ref{keyineq}).

Now, we may assume $\ell_1'$ becomes a (special) bridge after removing $v$, see Figure \ref{fig:3rproof}. Since $\ell_1'$ is not a bridge in $\Mb_{\mathrm{pre}}$, we know $v$ must have at least one bond connecting to each of the two components after removing $v$ and $\ell_1'$. Without loss of generality, assume $v$ has only one bond, say $\ell_1$, connecting to an atom $v_1$ in $X$ (the component containing $\{v_1',v_3'\}$), then by Lemma \ref{cutlem} we know that  $(k_{\ell_1},k_{\ell_1'})$ satisfies the system (\ref{2vcounting}) in Lemma \ref{lem:counting}, but with $n\delta^{-1}L^{-2}$ replacing $\delta^{-1}L^{-2}$ and resonance (i.e. $k_{\ell_1}=k_{\ell_1'}$ and they have opposite signs in (\ref{2vcounting})) allowed. Since  $(k_{\ell_1},k_{\ell_2},k_{\ell_3})$ also satisfies (\ref{3vcounting}), we can apply Lemma \ref{lem:counting} (3) to bound the number of choices for $(k_{\ell_1},k_{\ell_2},k_{\ell_3},k_{\ell_1'})$ by $nC^+\delta^{-2} L^{2(d-1)-\frac{1}{4}}$, \emph{unless} $\ell_1$ and $\ell_1'$ have opposite directions (viewing from $X$) and $k_{\ell_1}=k_{\ell_1'}$. If the above improved bound holds, then the number of choices for $(k_{\ell_1},k_{\ell_2},k_{\ell_3},k_{\ell_1'},\cdots,k_{\ell_5'})$ is at most $C^+\delta^{-4}L^{4(d-1)-\frac{1}{6}}$, since once $(k_{\ell_1},k_{\ell_2},k_{\ell_3},k_{\ell_1'})$ is fixed, the number of choices for $(k_{\ell_2'},\cdots,k_{\ell_5'})$ is at most $C^+\delta^{-2}L^{2(d-1)}$ by Lemma \ref{lem:counting} (1).

Finally, suppose $\ell_1$ and $\ell_1'$ have opposite directions and $k_{\ell_1}=k_{\ell_1'}$. In particular we must have $v_1\neq v_1'$, hence $(k_{\ell_1},k_{\ell_2},k_{\ell_3},k_{\ell_2'},k_{\ell_3'})$ will satisfy the system (\ref{5vcounting1}) in Lemma \ref{lem:counting}. By Lemma \ref{lem:counting} (5), we can bound the number of choices for these by $C^+\delta^{-2}L^{3(d-1)-\frac{1}{4}}$. Once these are fixed, the number of choices for $(k_{\ell_4'},k_{\ell_5'})$ is at most $C^+\delta^{-1}L^{d-1}$ by Lemma \ref{lem:counting} (1), so the number of choices for $(k_{\ell_1},k_{\ell_2},k_{\ell_3},k_{\ell_1'},\cdots,k_{\ell_5'})$ is still at most $C^+\delta^{-4}L^{4(d-1)-\frac{1}{4}}$. This proves (\ref{keyineq}).
\end{proof}
\subsubsection{Degree 2 atom removal} In this step, we assume there is an atom $v$ of degree 2, connected to one or two atom(s) of degree 2 or 4.
\begin{itemize}
\item Step (2R-1): suppose $v$ is connected to a degree 4 atom by a double bond, where the two bonds have opposite directions. We remove the atom $v$ and the double bond.
\item Step (2R-2G): suppose $v$ is connected to a degree 4 atom by a double bond, where the two bonds have the same direction. We remove the atom $v$ and the double bond.
\item Step (2R-3): suppose $v$ is connected to a degree 4 atom by a single bond, and also connected to another atom of degree 2 or 4 by a single bond. We remove the atom $v$ and the two bonds.
\item Step (2R-4): suppose $v$ is connected to two degree 2 atoms $v_1$ and $v_2$ by two single bonds, such that neither $v_1$ nor $v_2$ is connected to a degree 3 atom. We remove the atoms $\{v, v_1,v_2\}$, and all bonds connecting to them.
\item Step (2R-5): suppose $v$ is connected to a degree 2 atom $v'$ by a double bond. We remove the atoms $v,v'$ and the double bond. In all steps we set $\Delta\mathtt{Ext}=\varnothing$.
\end{itemize}

For (2R-1)--(2R-3) we have $(\Delta V, \Delta E,\Delta F)=(-1,-2,0)$ (note that $F=0$ due to Lemma \ref{increF}). For (2R-4) we have $(\Delta V, \Delta E,\Delta F)$ can be $(-3,-4,0)$ or $(-3,-3,-1)$, and for (2R-5) we have $(\Delta V, \Delta E,\Delta F)=(-2,-2,-1)$. For (2R-2G) we define $\Delta\gamma=-1+\frac{1}{3(d-1)}$ and $\Delta\kappa=-1$, and for all others define $\Delta\gamma=-1$ and $\Delta\kappa=-1$. We also define $\mathtt{Ext}_{\mathrm{pre}}=\mathtt{Ext}_{\mathrm{pos}}$.
\begin{prop}\label{2rprop} The step (2R-2G) is good, and the other four are normal. For (2R-1) and (2R-5) we have $\Delta V_3=\Delta \eta=0$. For (2R-3) we have $\Delta\eta=0$ and $\Delta V_3\geq 1$; for (2R-4) we have $\Delta V_3\geq 0$ and $\Delta\eta\leq -2$.
\end{prop}
\begin{proof} The statements about good or normal, as well as the ones regarding $\Delta\eta$ and $\Delta V_3$, can be shown by direct verification. As for (\ref{keyineq}), if $(\ell_1,\ell_2)$ are the two bonds of $v$, then $(k_{\ell_1},k_{\ell_2})$ satisfies the system (\ref{2vcounting}) in Lemma \ref{lem:counting}, so (\ref{keyineq}) follows from Lemma \ref{lem:counting} (1) and reduction to $k[\Mb_{\mathrm{pos}}]$. For (2R-4) just notice that $v_1$ and $v_2$ become degree 1 after removing $v$, so if $\ell_3$ and $\ell_4$ are the bonds they have other than $\ell_1$ and $\ell_2$, then $k_{\ell_3}$ and $k_{\ell_4}$ must be uniquely fixed once $(k_{\ell_1},k_{\ell_2})$ is fixed, so the total number of choices for $(k_{\ell_1},\cdots,k_{\ell_4})$ is still at most $C^+\delta^{-1}L^{d-1}$.
\end{proof}
\subsection{The algorithm}\label{alg} We now describe the algorithm. It is done in two phases. In phase one we remove the degenerate atoms using steps (DA) only; moreover we remove the non-tame degenerate atoms (i.e. those with degree 4) strictly before the tame ones. Once phase one is finished we enter phase two, where there is no more degenerate atoms; note that none of our steps can create any (possibly) degenerate atom, which is easily checked by definition.

In phase two, we will describe the algorithm as a big loop. Once we enter the loop, we shall follow a set of rules so that depending on the current molecule $\Mb$, we either (i) choose the next step, or (ii) claim a checkpoint and choose the two possibilities for the next step. In some cases, we may also choose more than one steps or claim more than one checkpoints successively, again following a specific set of rules, until we are done with this execution of the loop and return to the start of the loop. The loop ends when $\Mb$ contains only isolated atoms.
\subsubsection{Phase one: degenerate atom removal} The steps in phase one are determined as follows.
\begin{itemize}
\item If there is a degenerate atom of degree 4, remove it using (DA).
\item If there is no degenerate atom of degree 4 but there is a tame atom, remove it using (DA).
\item Repeat this until there is no degenerate atom. Then enter phase two.
\end{itemize}

Note that these steps will not create new degenerate atom, or new degenerate atom of degree 4, but may transform degenerate atom of degree 4 into tame ones. At the end of phase one there will be no degenerate atom, which will be preserved for the rest of the algorithm.
\subsubsection{Phase two: description of the loop}\label{loop} We now describe the loop in phase two. For an example of this algorithm, see Appendix \ref{algexample}. {Note that there is no triple bond in the beginning.}
\begin{enumerate}
\item If $\Mb$ contains a bridge, then remove it using (BR). {Repeat until $\Mb$ contains no bridge.}
\item {Now $\Mb$ contains no bridge.} If $\Mb$ contains two degree 3 atoms $v_1$ and $v_2$ connected by a single bond $\ell_1$, then:
\begin{enumerate}
\item If $\Mb$ contains one of the functional groups in Figure \ref{fig:3s3new}, then {perform} (3S3-5G). Go to (1).
\item Otherwise, $\Mb$ contains the functional group in Figure \ref{fig:3s3}. If it satisfies (i) and (ii) in Section \ref{3s3}, and $d(v_3)=\cdots =d(v_6)=4$, then we claim a checkpoint, and choose the two possibilities for the next step to be (3S3-1) and (3S3-2G) ({the pre-assumptions for (3S3-1) and (3S3-2G) are satisfied, see Section \ref{3s3}}). Go to (1).
\item If it satisfies (i) and (ii) in Section \ref{3s3}, but (say) $d(v_3)$ and $d(v_5)$ are not both 4, then we claim a checkpoint, and choose the two possibilities for the next step to be (3S3-2G) and (3S3-3G) ({the pre-assumptions for (3S3-2G) and (3S3-3G) are satisfied, see Section \ref{3s3}}). If after (3S3-3G) a triple bond forms between $v_3$ and $v_5$, immediately remove it by (TB-1)--(TB-2). Go to (1).
\item If either (i) or (ii) in Section \ref{3s3} is violated, then we perform (3S3-4G) ({the pre-assumption for (3S3-4G) is satisfied, see Section \ref{3s3}}). Go to (1).
\end{enumerate}
\item Otherwise, if $\Mb$ contains two degree 3 atoms $v_1$ and $v_2$ connected by a double bond $(\ell_1,\ell_2)$, then:
\begin{enumerate}
\item If $\Mb$ contains the functional group in Figure \ref{fig:3d3} corresponding to (3D3-4G) or (3D3-5G), then we perform the corresponding step. Go to (1).
\item Otherwise, $\Mb$ contains the functional group in Figure \ref{fig:3d3} corresponding to (3D3-1)--(3D3-3G). This can be seen as the start of a type II chain. Now, if and while this chain \emph{continues} (i.e. $v_3$ and $v_4$ are connected by a double bond, and they are connected to two different atoms $v_5$ and $v_6$ by two single bonds of opposite directions viewing form $\{v_3,v_4\}$), we claim a checkpoint, and choose the two possibilities for the next step to be (3D3-1) and (3D3-2G) ({the pre-assumptions for (3D3-1) and (3D3-2G) are satisfied, see Section \ref{3d3}}). Proceed with (c) below.
\item Now assume the type II chain does not continue, {i.e. we have reached the end of the type II chain} {(if the type II chain does not continue in the beginning then we skip (b) and directly move to (c) here)}. Then:
\begin{enumerate}
\item If not all atoms in the current component other than $\{v_1,v_2\}$ have degree 4, then we claim a checkpoint and choose the two possibilities for the next step to be (3D3-2G) and (3D3-3G) {(the pre-assumptions for (3D3-2G) and (3D3-3G) are satisfied, see Section \ref{3d3})}. If after (3D3-3G) a triple bond forms between $v_3$ and $v_4$, immediately remove it by (TB-1)--(TB-2) {(this is always doable, see Remark 2 immediately following the description of this algorithm)}. Go to (1).
\item Otherwise, if $v_3$ and $v_4$ are like in Figure \ref{fig:3d3new}, then perform (3D3-6G) ({the pre-assumption for (3D3-6G) is satisfied, see Section \ref{3d3}}). Go to (1).
\item Otherwise, we { claim a checkpoint, and choose the two possibilities for the next step to be (3D3-1) and (3D3-2G)} ({the pre-assumptions for (3D3-1) and (3D3-2G) are satisfied, see Section \ref{3d3}}). Go to (1) but scan within this component (see explanation below).
\end{enumerate}
\end{enumerate}
\item Otherwise, if $\Mb$ contains a degree 3 atom $v_1$ connected to a degree 4 atom $v_2$ by a double bond $(\ell_1,\ell_2)$, then we have one of the functional groups in Figure \ref{fig:3d4}. We perform (3D4G). Go to (1).
\item Otherwise, if $\Mb$ contains a degree 3 atom $v_1$ connected to a degree 2 atom $v_2$, then we have one of the functional groups in Figure \ref{fig:3s2}. We perform (3S2G). Go to (1).
\item Otherwise, if $\Mb$ contains a degree 3 atom $v$, then $v$ must be connected to three degree 4 atoms $v_j\,(1\leq j\leq 3)$ by three single bonds $\ell_j\,(1\leq j\leq 3)$. Then:
\begin{enumerate}
\item If the component after removing $v$ and $\ell_j$ contains a special bond, then we perform (3R-2G) ({the pre-assumption for (3R-2G) is satisfied, see Section \ref{3r}}). Go to (1).
\item Otherwise, we perform (3R-1). Go to (1).
\end{enumerate}
\item Otherwise, $\Mb$ must only contain atoms of degree (0 and) 2 and 4. If we are in one of the cases corresponding to steps (2R-2G)--(2R-5), then perform the corresponding step. Go to (1).
\item Otherwise, there is a degree 2 atom $v$ connected to a degree 4 atom $v_1$ by a double bond of opposite directions. This can be seen as the start of a type I chain. Now, if and while this chain \emph{exists} (we do \emph{not} require this chain to continue from $v_1$, which is slightly different from (3-b)), we perform (2R-1) {until we reach the end of the type I chain}. Go to (7) but scan within this component (see explanation below).
\end{enumerate}

Before proceeding, we make a few remarks about the validity of the algorithm and Condition 2. 

1. There is no triple bond when we perform any step other than (TB-1)--(TB-2). This is because only steps (3S3-3G) and (3D3-3G) may create triple bonds, but they are immediately removed using (TB-1)--(TB-2), as in (2-c) and (3-c-i). 

2. In (3-c-i), after (3D3-3G), {suppose $v_3$ and $v_4$ are connected by a triple bond. If not both $v_3$ and $v_4$ have degree $4$, then we can perform (TB-1)--(TB-2). If $d(v_3)=d(v_4)=4$, then the two extra single bonds $\ell_1'$ and $\ell_2'$ from $v_3$ and $v_4$ must have opposite directions (by the requirement in Definition \ref{defmole0}); since the type II chain does not continue, $\ell_1'$ and $\ell_2'$ \emph{must} share a common atom, say $v_5$. The first equation in (\ref{decmole}), with $c_{v_3}=c_{v_4}=0$, and the condition $k_{\ell_3}=k_{\ell_4}$ in $\Delta\mathtt{Ext}$, then forces $k_{\ell_1'}=k_{\ell_2'}$, which is impossible as $v_5$ cannot be degenerate.}

3. When executing a ``Go to" sentence, we may proceed to scan the whole molecule for the relevant structures, except in (3-c-iii) and (8), where we only scan \emph{the current component}. Note that after performing (3D3-1) {or (3D3-2G)} in (3-c-iii), $v_3$ and $v_4$ will have degree 3, and all other atoms in the current component will have degree 4. Therefore the next step(s) we perform in this component, following our algorithm, may be (BR), (3S3-1)--(3S3-5G), (3D3-4G)--(3D3-5G), (3D4G), (3R-1)--(3R-2G), possibly accompanied by (TB-1)--(TB-2), but \emph{cannot} be (3D3-1)--(3D3-3G) because the type II chain does not continue. Similarly, after performing the last (2R-1) in (8), $v_1$ will have degree 2, and no atom in the current component may have degree 3. Therefore the next step we perform in this component may be (2R-2G)--(2R-5), but \emph{cannot} be (2R-1).

4. There is no bridge when we perform any step other than (TB-1)--(TB-2) or (BR). This is because step (BR) has the top priority {due to the ``Go to (1)" sentences in the algorithm}. Moreover, if we are in (3-b), i.e. the type II chain continues, then the steps (3D3-1) and (3D3-2G) cause the same change on $\Mb$, and this change does not create any bridge. In the same way, if we are in (8), then the step (2R-1) does not create any bridge.

5. In the whole process we never have a saturated component, thus in (7) there must be at least one degree 2 atom (unless there are only isolated atoms, in which case the loop ends; note that we are also not considering degree 1 atoms, as those imply the existence of bridges).

6. The timespots where we claim checkpoints are in (2-b), (2-c), (3-b), (3-c-i) { and (3-c-iii)}. In each case Condition 2 is preserved, because (i) by our choice, the two possible $\Delta\mathtt{Ext}$'s for the two possibilities for the next step at this checkpoint are exactly negations of each other, so any $k[\Mb_{\mathrm{pre}}]$ must satisfy one of them, and (ii) for (3S3-3G) (same for (3D3-3G)), if $k[\Mb_{\mathrm{pre}}]$ satisfies $\Delta\mathtt{Ext}$ and $k'[\Mb_{\mathrm{pos}}]$ satisfies $\mathtt{Ext}_{\mathrm{pos}}$, then $k[\Mb_{\mathrm{pre}}]$ must satisfy $\mathtt{Ext}_{\mathrm{pre}}$, which follows from (\ref{renewext}) and (\ref{renewk}).
\subsection{Proof of Proposition \ref{gain}}\label{proofofgain} The algorithm described in \ref{loop} leads to at most $C^n$ tracks. Each track contains at most $Cn$ steps as each step removes at least one bond, while there are only $2n-1$ from the beginning. We will fix a track in the discussion below. Let $r$ be the total number of \emph{good} steps in this track. Note that the change of any of the quantities we will study below, caused by any single step we defined above, is at most $C$ {(in fact, at most $100$).}
\subsubsection{Phase one} We start with phase one. {Note that $\eta_*$ must remain nonnegative due to absence of saturated components, as each component must have at least one atom of degree in $\{0,1,2\}$ or two atoms of degree $3$; moreover initially $\eta_*=0$ because there are only two atoms of degree $3$ or only one atom of degree $2$.} Let $s$ be the number of (DA) removing degree 4 degenerate atoms that are \emph{normal}, and let $s'$ be the number of (DA) removing tame atoms. After removing all the degree 4 degenerate atoms, by Proposition \ref{daprop}, we know that $0\leq \eta_*\leq -2s+Cr$, we know that $s\leq Cr$. 

At this time, the number of tame atoms is at most $2+C(s+r)\leq 2+Cr$, as originally the number of tame atoms is at most 2, and the number of newly created tame atoms is at most $C(s+r)$. Moreover, if $r=0$, then also $s=0$. If a degree 2 or degree 3 atom in the original base molecule is degenerate (hence tame), then after the first (DA) step, by Proposition \ref{daprop} we know that $\eta_*$ will become negative, which is impossible. This means that if $r=0$ then $s=s'=0$, hence in all cases $s+s'\leq Cr$.
\subsubsection{Phase two: increments of $\eta$ and $V_3$} Since the total number of steps in phase one is at most $Cr$, we know at the start of phase two, each of the quantities we will study below has changed at most $Cr$ compared to the initial state. Note that (TB-1) and (TB-2) only occur once after (3S3-3G) or (3D3-3G) which are good steps, the number of those is also at most $Cr$.

Let the number of (BR) where $d(v_1)=d(v_2)=3$ (see Proposition \ref{brprop}) be $z_1$, the number of other (BR) be $z_1'$. Let the number of (3S3-1) be $z_2$, the number of (3R-1) be $z_3$, the numbers of (2R-3)--(2R-5) be $z_4$, $z_5$ and $z_6$. By Propositions \ref{brprop}--\ref{2rprop}, we can examine the increment of $\eta$ in the whole process and get
\begin{equation}\label{increeta}-2z_1-2z_1'-2z_2+2z_3-2z_5\geq 2-Cr,
\end{equation} note that initially $\eta=-2$ and in the end $\eta=0$. In the same way, by examining the increment of $V_3$ we get
\begin{equation}\label{increv3}-2z_1-z_1'+2z_2+2z_3+z_4\leq  Cr,
\end{equation} note that initially $V_3\in\{0,2\}$ and in the end $V_3=0$. Subtracting these two inequalities yields $z_1'+z_2+z_4+z_5+2\leq Cr$. In particular we also know $r\geq 1$.
\subsubsection{Phase two: remaining steps} Next we will prove that $z_1+z_3+z_6\leq Cr$. Let $V_2^*$ be the number of degree 2 atoms with two single bonds. It is clear that $|\Delta V_2^*|\leq C$ for any step, $\Delta V_2^*=0$ for (3D3-1), (3R-1) and (2R-5), and $\Delta V_2^*\geq 0$ for (2R-1), and for (BR) assuming $d(v_1)=d(v_2)=3$. Moreover, equality holds for (BR) if and only if the bridge removed is special. Therefore, with at most $Cr$ exceptions, all the bridges appearing in (BR) are special. Then, if we consider the increment of $V_2$, we similarly see that $z_6\leq z_1+Cr$ (using also $r\geq 1$).  Combining with (\ref{increv3}) which implies $z_3\leq z_1+Cr$, we only need to prove $z_1\leq Cr$.

Consider the increment of the number of special bonds, denoted by $\xi$. Clearly $\Delta\xi=0$ for (2R-1) and (2R-5); for (BR) which removes a special bridge, we can check that this operation cannot make any existing non-special bond special, so $\Delta\xi=-1$. Moreover, by  our algorithm, whenever we perform (3R-1), it is always assumed that the component contains no special bond after this step, so $\Delta\xi\leq 0$. Similarly, whenever we perform (3D3-1) we are always in (3-b) or (3-c-iii). For (3-c-iii), $v_3$ and $v_4$ are the only two degree 3 atom in the component after performing (3D3-1) {or (3D3-2G)}, and they are not connected by a special bond (otherwise we shall perform (3D3-6G)), so this step also does not create any special bond, hence $\Delta\xi\leq 0$.

Now let us consider steps (3D3-1) occurring in (3-b). By our algorithm, if we also include the possible (3D3-2G), then such steps occur in the form of sequences which follow the type II chains in the molecule. For any step in this sequence \emph{except} the last one, we must have $\Delta\xi=0$ (because in this case, after (3D3-1), neither $v_3$ nor $v_4$ is connected to a degree 3 atom by a single bond). Moreover, if for the last one in the sequence we do have $\Delta\xi>0$, then immediately after this sequence we must have a good step (because in this case, after we finish the sequence and move to (3-c), {either $v_3$ or $v_4$} will have degree 3 instead of 4, so we must be in (3-c-i)). Since the number of good steps is at most $r$, we know that the number of steps for which $\Delta\xi>0$ is at most $Cr$. Thus, considering the increment of $\xi$, we see that $z_1\leq Cr$.
\subsubsection{Type I and type II chains} Now we see that the number of steps \emph{different from} (3D3-1) and (2R-1) is at most $Cr$. In particular steps (3D3-1) {or (3D3-2G)} occurring in (3-c-iii) is also at most $Cr$ because each of them must be followed by an operation different from (3D3-1) and (2R-1). As for the sequences of (3D3-1) or (3D3-2G) occurring in (3-b), each sequence corresponds to a type II chain, and each chain can be as long as $Cn$, but the number of chains must be at most $Cr$ for the same reason. Moreover, following each chain we have a sequence of checkpoints, and at each checkpoint we may choose (3D3-1) or (3D3-2G), but the number of (3D3-2G) chosen must be at most $Cr$. If necessary we can further divide these chains, so that (3D3-1) is chosen at each checkpoint of each type II chain.

In the same way, steps (2R-1) also occur in the form of sequences which follow the type I chains in the molecule, and at the end of each sequence we have a step different from (3D3-1) and (2R-1). Thus each sequence corresponds to a type I chain, and the number of chains is at most $Cr$. Note that some of the edges in the chains may not exist in the original base molecule, but the number of those is again at most $Cr$ because (3S3-3G) and (3D3-3G) are both good steps. Upon further dividing, we can find these (at most $Cr$) chains in the original base molecule, such that the number of atoms and bonds not belonging to one of these chains is at most $Cr$. In addition, since we are choosing (3D3-1) in type II chains, by definition, the set $\mathtt{Ext}$ obtained in the start must contain (possibly among other things) the conditions $k_{\ell_1}=k_{\ell_2}$ for any two paired single bond $(\ell_1,\ell_2)$ in any type II chain.
\subsubsection{{Conclusion}} Finally we prove (\ref{defect2}). At the initial timespot, $\Df(\Mb)$ is the union of all the possible $\Df(\Mb,\mathtt{Ext})$ for $\mathtt{Ext}\in\Upsilon$, thanks to Condition 2. The number of possible tracks is at most $C^n$, so we only need to fix one track. Now by Condition 4, we get
\[\sup\#\Df(\Mb,\mathtt{Ext})\leq (C^+)^n\delta^{-\kappa}L^{(d-1)\gamma}.\] Since initially $\chi=n$, we see that $(d-1)\gamma\leq (d-1)n-2\nu r$ for $0<\nu\leq \frac{1}{12}$ by the definition of good and normal steps. As for $\kappa$, note that $\Delta\kappa=-1$ for both (3D3-1) and (2R-1). If the total numbers of atoms in type I chains and type II chains are $m$ and $m'$, then $m'=n-m$ up to error $Cr$, and the number of steps (3D3-1) and (2R-1) are $m'/2$ and $m$ respectively (all up to error $Cr$), so initially $\kappa=m+\frac{m'}{2}=\frac{n+m}{2}$ up to error $Cr$. Clearly factors $\delta^{-Cr}$ is acceptable in view of the gain $L^{-2\nu r}$, so we have proved (\ref{defect2}).
\section{Non-regular couples III: $L^1$ bounds for coefficients}\label{l1coef} We now return to the study of the expression (\ref{section6fin2}). Let $\Qc_{sk}^\#$ and $(r_0,r_{\mathrm{irr}})$ be as in Section \ref{section6summary}. For simplicity, until the end of the proof of Proposition \ref{section8main} we will write $\Qc_{sk}^\#$ simply as $\Qc$, and the associated sets $(\Nc_{sk}^\#)^*$ as $\Nc^*$ etc. Recall, by (\ref{typeIcontrol}), that the total length of the irregular chains in $\Qc$ is at most $C(r_0+r_{\mathrm{irr}})$. Let $\Xi$ be a subset of $\Nc^*$, we may define, as in (\ref{section6fin2}), the function
\begin{equation}\label{modifiedtimeint}\Uc_{\Qc}(t,s,\vsigma,\alpha[\Nc^*])=\int_{\widetilde{\Ec}}\prod_{\nf\in\Nc^*}e^{\pi i\alpha_\nf t_\nf}\,\mathrm{d}t_\nf,
\end{equation} where $\vsigma=\sigma[\Xi]\in[0,1]^\Xi$, and the domain $\widetilde{\Ec}$ is defined as in (\ref{timedom}), but with the extra conditions $t_{\nf^p}>t_\nf+\sigma_\nf$ for $\nf\in\Xi$, where $\nf^p$ is the parent of $\nf$. Note that the definition here is slightly different from (\ref{defcoefb2}) as we include the signs $\zeta_\nf$ in the variables $\alpha_\nf$, which is more convenient for this section. Then, let $n_0'$ be the scale of $\Qc$, we can write
\begin{multline}\label{finalexp}(\ref{section6fin2})=(C^+\delta)^{\frac{n-n_0'}{2}}\bigg(\frac{\delta}{2L^{d-1}}\bigg)^{n_0'}\zeta^*(\Qc)\int_{\Rb^{\Nc^*}\times\Rb^2} G(\vlambda)\cdot e^{\pi i(\lambda t+\mu s)}\,\mathrm{d}\vlambda\int_{[0,1]^\Xi}\mathrm{d}\vsigma\\\times\sum_{\Es}\epsilon_\Es\Uc_{\Qc}\big(t,s,\vsigma, (\delta L^2\zeta_\nf\Omega_\nf+\lambda_\nf)_{\nf\in\Nc^*}\big)\cdot \Xc_{\mathrm{tot}}(\vlambda,\vsigma,k[\Qc]).
\end{multline}

Let $\Mb$ be the base molecule obtained from $\Qc$ as in Definition \ref{defmole}. It is easy to see that $\Mb$ contains no triple bond, as triple bonds in $\Mb$ can only come from $(1,1)$-mini couples and mini trees (as in Definition \ref{defmini}) in $\Qc$. By the proofs in Section \ref{improvecount}, we can introduce at most $C^{n_0'}$  sets of extra conditions $\mathtt{Ext}$, such that the summation in $\Es=k[\Qc]$ in (\ref{section6fin2}) can be decomposed into the summations with each of these sets of extra conditions imposed on $k[\Qc]$. Moreover, for each choice of $\mathtt{Ext}$ there is $1\leq r_1\leq n_0'$ such that the conclusion of Proposition \ref{gain}, including (\ref{defect2}), holds true (with $r$ replaced by $r_1$).

Notice that a type I chain in $\Mb$ can \emph{only} be obtained from either one irregular chain, or the union of two irregular chains in $\Qc$; this can be proved in the same way as in Section \ref{PClp} below (which involves the more complicated type II chains), see Remark \ref{RemarkTypeIanalysis}. Therefore, the total length $m$ of type I chains in $\Mb$ is bounded by the total length of irregular chains in $\Qc$, which is at most $C(r_0+r_{\mathrm{irr}})$. However, each irregular chain in $\Qc$ also corresponds to a type I chain in the base molecule, so $r_{\mathrm{irr}}\leq Cr_1$, hence $m\leq Cr$, where $r=r_0+r_1$. This means the number of atoms in $\Mb$ that are not in one of those (at most $Cr$) type II chains is at most $Cr$.

Now, suppose $\nf$ and $\nf'$ are two branching nodes in $\Qc$ which correspond to two atoms in $\Mb$ \emph{that are connected by a double bond in a type II chain}, then we must have $\zeta_{\nf'}\Omega_{\nf'}=-\zeta_\nf\Omega_\nf$ under the extra conditions in $\mathtt{Ext}$, see Remark \ref{moledec}. In fact we will restrict $\{\nf,\nf'\}$ to the \emph{interior} of this type II chain by omitting $5$ pairs of atoms at both ends of the chain, in the same way as in Definition \ref{conggen}. Then, we make such $\{\nf,\nf'\}$ a pair (this is related to but different from the pairing of branching nodes in Proposition \ref{branchpair}), and choose one node from each such pair to form a set $\widetilde{\Nc}^{ch}$. If it happens that one of $\{\nf, \nf'\}$ is a parent of the other, we assume the parent belongs to $\widetilde \Nc^{ch}$. Let $\Nc^{rm}$ be the set of branching nodes not in these pairs, and define $\widetilde{\Nc}=\widetilde{\Nc}^{ch}\cup\Nc^{rm}$. 

We will be interested in estimates on the function $\Uc_\Qc$ in \eqref{modifiedtimeint} where $\alpha_\nf=\delta L^2\zeta_\nf\Omega_\nf+\lambda_\nf$, which means that $\alpha_{\nf}+\alpha_{\nf'}=\mu_\nf$ for each $\nf\in\widetilde{\Nc}^{ch}$, where $\nf'$ is the node paired to $\nf$ and $\mu_\nf=\lambda_\nf+\lambda_{\nf'}$ is a parameter depending on $\vlambda$. Under this assumption on $\alpha_{\nf}$, we can write (similar to (\ref{deftildeb}))
\begin{equation}\label{eqnuv}\Uc_{\Qc}(t,s,\vsigma,\alpha[\Nc^*])=\Vc_{\Qc}(t,s,\vsigma,\alpha[\widetilde{\Nc}])
\end{equation} for some function $\Vc_{\Qc}$. This function actually depends also on the parameters $\mu_\nf$ for $\nf\in\widetilde{\Nc}^{ch}$, but we will omit this for notational convenience. The main goal of this section is to prove the following:
\begin{prop}\label{section8main} {Suppose $\Qc$ has scale $n_0'$.} For each $\nf\in\widetilde{\Nc}$, suppose $S_\nf\subset\Zb$ and $\#S_\nf\leq L^{10d}$. Then, uniformly in $(t,s)$, in the choices of $(S_\nf)_{\nf\in\widetilde{\Nc}}$, and in the parameters $(\mu_\nf)_{\nf\in\widetilde{\Nc}^{ch}}$, we have
	\begin{equation}\label{uniformL1}\delta^{n_0'/4}\cdot\sum_{(m_\nf):m_\nf\in S_\nf}\sup_{(\alpha_\nf):|\alpha_\nf-m_\nf|\leq 1}\sup_{\vsigma}\big|\Vc_{\Qc}(t,s,\vsigma,\alpha[\widetilde{\Nc}])\big|\leq (C^+)^{n_0'}L^{Cr\sqrt{\delta}}(\log L)^{Cr},
	\end{equation} where $r=r_0+r_1$.
\end{prop}	
Before proving Proposition \ref{section8main}, we first make some observations. By \eqref{modifiedtimeint}, one can see that the function $\Vc_\Qc$ is completely determined by the \emph{tree structures} of the trees of $\Qc$, as well as the pairing between branching nodes described above (i.e. it does \emph{not} depend on the pairings between leaves of $\Qc$, nor on the signs of the nodes). Thus, below we will \emph{forget} the signs of the nodes of $\Qc$ and view it as an unsigned couple, which corresponds to an undirected molecule as in Definition \ref{defmole} (we retain the pairings between branching nodes). Denote the unsigned couple by $\Qc^{ns}$ and the undirected molecule still by $\Mb$. Later in the inductive step, we may further forget the leaf pairing structure of $\Qc$, hence viewing it as a double-tree with some of the branching nodes paired, and denote it by $\Qc^{tr}$. We may then write
$$
\Vc_{\Qc}(t,s,\vsigma,\alpha[\widetilde{\Nc}])=\Vc_{\Qc^{ns}}(t,s,\vsigma,\alpha[\widetilde{\Nc}])=\Vc_{\Qc^{tr}}(t,s,\vsigma,\alpha[\widetilde{\Nc}]).
$$

Next, for any function $V=V(x)$ defined on $[0,1]^n$, by inducting on $n$ we can prove that
\begin{equation}\label{elemineq}\sup_{x\in[0,1]^n}|V(x)|\leq\sum_{\rho}\int_{[0,1]^n}|\partial_x^\rho V|\,\mathrm{d}x,\end{equation} where $\rho$ ranges over all multi-indices with each component being $0$ or $1$. This implies that 
\begin{equation}\label{suptoint}\sum_{(m_\nf):m_\nf\in S_\nf}\sup_{(\alpha_\nf):|\alpha_\nf-m_\nf|\leq 1}\sup_{\vsigma}\big|\Vc_{\Qc}(t,s,\vsigma,\alpha[\widetilde{\Nc}])\big|\leq\sum_\rho\int_{(\alpha_\nf):\alpha_\nf\in S_\nf(1)}\sup_{\vsigma}\big|\partial_\alpha^\rho\Vc_{\Qc}(t,s,\vsigma,\alpha[\widetilde{\Nc}])\big|\,\mathrm{d}\alpha[\widetilde{\Nc}],\end{equation} where $\rho$ is as above, and $S_\nf(1)$ is the $1$-neighborhood of $S_\nf$ in $\Rb$ which has measure $\leq L^{10d}$. If we fix $\rho$ in (\ref{suptoint}), which has at most $2^{n_0'}$ choices, then $\partial_\alpha^\rho\Vc_\Qc$ has a similar form as $\Vc_\Qc$ except that one has some extra $\pi it_\nf$ factors in the integral (\ref{modifiedtimeint}). From the proof below it is clear that such factors will not make a difference, so we will focus on the right hand side of (\ref{suptoint}) without the $\partial_\alpha^\rho$ derivative.

Finally, we record the following lemma, which will be useful in the proof of Proposition \ref{section8main}.
\begin{lem}\label{timeintlemma}
	Let $\Tc$ be a ternary tree, and denote by $\Nc$ the set of branching nodes. Let $\Xi\subset \Nc$, and consider 
	\begin{equation}\label{modifiedtimeint2}
	\Uc_{\Tc}(t,\vsigma,\alpha[\Nc])=\int_{\widetilde{\Dc}}\prod_{\nf\in\Nc}e^{\pi i\alpha_\nf t_\nf}\,\mathrm{d}t_\nf,
	\end{equation} where $\vsigma=\sigma[\Xi]\in[0,1]^\Xi$, and the domain $\widetilde{\Dc}$ is defined as in (\ref{timedom0}), but with the extra conditions $t_{\nf^p}>t_\nf+\sigma_\nf$ for $\nf\in\Xi$, where $\nf^p$ is the parent of $\nf$.
	
	For every choice of $d_{\nf}\in \{0, 1\}\,(\nf\in \Nc)$, we define  $q_{\nf}$ for $\nf \in \Nc$ inductively as follows: Set $q_{\nf}=0$ if $\nf$ is a leaf, and otherwise define $q_{\nf}=\alpha_\nf+d_{\nf_1}q_{\nf_1}+d_{\nf_2}q_{\nf_2}+d_{\nf_3}q_{\nf_3}$ where $\nf_1, \nf_2, \nf_3$ are the three children of $\nf$.
	
	Uniformly in $\vsigma$ and $t$, the following estimate holds:
	\begin{equation}\label{timeintlemma1}
	|\Uc_\Tc (t, \vsigma, \alpha[\Nc])| \leq (C^+)^n\sum_{d_\nf\in \{0, 1\}}\prod_{\nf \in \Nc}\frac{1}{\langle q_\nf\rangle }.
	\end{equation}
	
\end{lem}

\begin{proof}
	The proof is straightforward, see Proposition 2.3 in \cite{DH}. Note that here we have the extra parameters $\vsigma$, but they only contribute unimodular coefficients to various components of $\Uc_\Tc$ and do not affect any of the estimates.
\end{proof}

\begin{proof}[Proof of Proposition \ref{section8main}]
	The proof will proceed by induction on the size of $\widetilde\Nc^{ch}$. The base case in which $\widetilde \Nc^{ch}$ is empty is covered by Lemma \ref{timeintlemma}, since for any choice of $y_\nf \in \Rb$ and $S_\nf(1)$ of measure $\leq L^{10d}$, one has that
	$$\int_{S_\nf(1)}\frac{1}{\langle\alpha_{\nf}+y_\nf\rangle}\,\mathrm{d}\alpha_\nf\lesssim \log L
	$$
Therefore the left hand side of \eqref{uniformL1} is bounded by 
	$\delta^{n_0'/4}(C^+)^{n_0'} (\log L)^{n_0'}$ which is more than acceptable since if $\widetilde \Nc^{ch}$ is empty we must have $n_0'\leq Cr$.

	We now assume that $\widetilde \Nc^{ch}$ is nonempty and that estimate \eqref{uniformL1} hold for couples with smaller $\widetilde \Nc^{ch}$ (equivalently molecules with shorter Type II chains). To prove (\ref{uniformL1}) for $\Qc$, we first need to analyze the structure of the couple $\Qc$, which is done in the next section.
\subsection{Tree Structure near $\nf \in \widetilde \Nc^{ch}$}

	Recall the definition of PC and LP bonds in Definition \ref{defmole}. Clearly, the two edges of a double bond cannot be both PC bonds, but we can have them both being LP bonds (we call this an LP-LP double bond) or one LP and one PC bond (we call that an LP-PC double bond). Denote by $\Nf$ the set of atoms in $\Mb$ connected by double bonds in the type II chains, then each such pair of atoms corresponds to a pair of branching nodes in $\Qc$, and only one of the two nodes belongs to $\widetilde \Nc^{ch}$. There are two cases for the molecule $\Mb$:  \emph{Case 1} where there exists at least one LP-LP double bond connecting a pair of atoms in $\Nf$, or \emph{Case 2} where all double bonds connecting a pair of atoms in $\Nf$ are LP-PC double bonds. 
	
	\subsubsection{ LP-LP double bonds}\label{LPlp} Suppose that one of the double bonds appearing in $\Nf$ is an LP-LP double bond. In this case, if $(\pf,\cf_1, \cf_2, \cf_3)$ and $(\pf', \cf_1', \cf_2', \cf_3')$ denote the two $4$-node subsets corresponding to the two atoms connected by an LP-LP double bond, then $\pf$ and $\pf'$ are two branching nodes in $\Nc^*$ such that neither is a child of the other. We also have two leaf pairings between the children $\cf_{k_1}, \cf_{k_2}$ and $\cf_{j_1}', \cf_{j_2}'$ where $k_i, j_i \in \{1, 2, 3\}$, see Figure \ref{fig:Case1}. Note that $\pf$ and $\pf'$ may or may not be in the same tree. In fact one of them may be a descendant of the other, in which case Figure \ref{fig:Case1} will be depicted differently (but the proof will not be affected).
		\begin{figure}[h!]
		\includegraphics[scale=0.85]{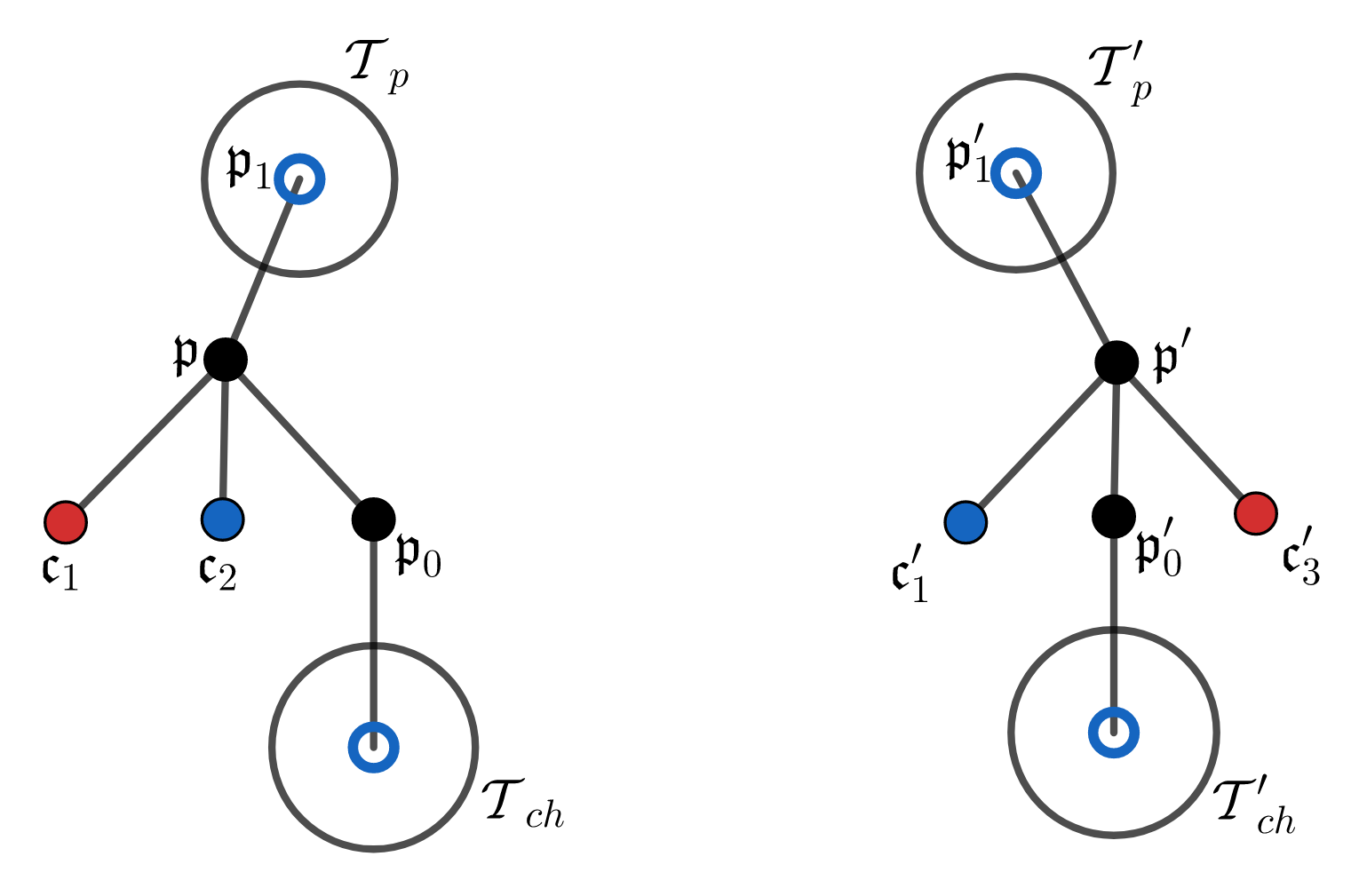}
		\caption{In \emph{Case 1}, the two paired nodes are $\pf$ and $\pf'$, neither of which is a child of the other. Each has two children leaves paired with the children of the other. Here $\Tc_p$ and $\Tc_p'$ denote the trees obtained by deleting the subtrees rooted at $\pf$ and $\pf'$ respectively (keeping the leaves $\pf$ and $\pf'$), and $\Tc_{ch}, \Tc_{ch}'$ denote the trees rooted at $\pf_0, \pf_0'$ respectively.}
		\label{fig:Case1}
	\end{figure}

	\subsubsection{LP-PC double bonds}\label{PClp} Now consider \emph{Case 2} in which all the double bonds in all type II chains connecting pairs of atoms in $\Nf$ are LP-PC double bonds. Here we can verify that, the two horizontal parallel single bonds in Figure \ref{fig:molechain} that connect two LP-PC double bonds cannot be both PC bonds (since each node in $\Nc^*$ has a single parent), which means that at least one of the two parallel single bonds is an LP bond. Since the total number of LP bonds is $\in\{n_0',n_0'+1\}$ and that of PC bonds is $\in\{n_0'-1,n_0'-2\}$, we conclude that the number of the parallel single bonds that are both LP is bounded by the number of bonds outside all the type II chains which is $Cr$.

	As a result of this, by splitting the type II chains appearing in $\Mb$ at the (at most) $Cr$ sites where the parallel single bonds are both LP bonds, we obtain that the molecule $\Mb$ has at most $Cr$ type II chains where the double bonds are all LP-PC and the parallel single bonds connecting them are such that one is LP and the other is PC. We shall abuse notation, and refer to those (possibly smaller) chains as the type II chains below and still denote by $\Nf$ the smaller set of atoms connected by such LP-PC double bonds, such that each pair of single bonds has one LP and one PC bond. See Figure \ref{fig:molechainS8}.
	
	\begin{figure}[h!]
		\includegraphics[scale=0.55]{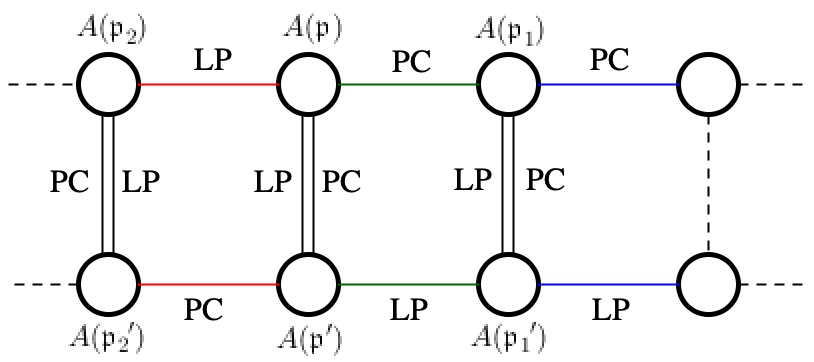}
		\caption{A type II chain in which we label the bonds as either LP (leaf pair) or PC (parent-child). Here we assume that all double bonds are LP-PC, and for any pair of parallel colored bonds, one is LP and the other is PC. Here $A(\pf)$ is the atom corresponding to the branching node $\pf$, and we assume $\pf$ is the parent of $\pf'$. Similar for $\pf_1$ and $\pf_2$.}
		\label{fig:molechainS8}
	\end{figure} 
	
	\medskip
	
	Let $(\pf,\cf_1, \cf_2, \cf_3)$ and $(\pf', \cf_1', \cf_2', \cf_3')$ denote the two $4$-node subsets corresponding to the two atoms of $\Nf$ connected by a LP-PC double bond in a type II molecular chain, and suppose that $\pf'$ is a child of $\pf$. Since there is a double bond between $A(\pf)$ and $A(\pf')$, some child $\cf_k$ of $\pf$ must be paired to a child $\cf_j'$ of $\pf'$; in particular, $\cf_k$ and $\cf_j'$ are leafs.
	
	We claim that: (1) among the one remaining child of $\pf$ and the two remaining children of $\pf'$, exactly $2$ are leaves, and the other one, denoted by $\pf_0$, is a branching node corresponding to an atom in $\Nf$; (2) the parent of $\pf$, denoted by $\pf_1$, corresponds to an atom in $\Nf$ that is connected to $A(\pf)$ by a single bond. Note that the node $\pf_0$ in (1) is either a child of $\pf$ or a child of $\pf'$; we call these \emph{Case 2A} and \emph{Case 2B}, see Figure \ref{FigCase2AB}.
\begin{figure}[h!]
\includegraphics[scale=1]{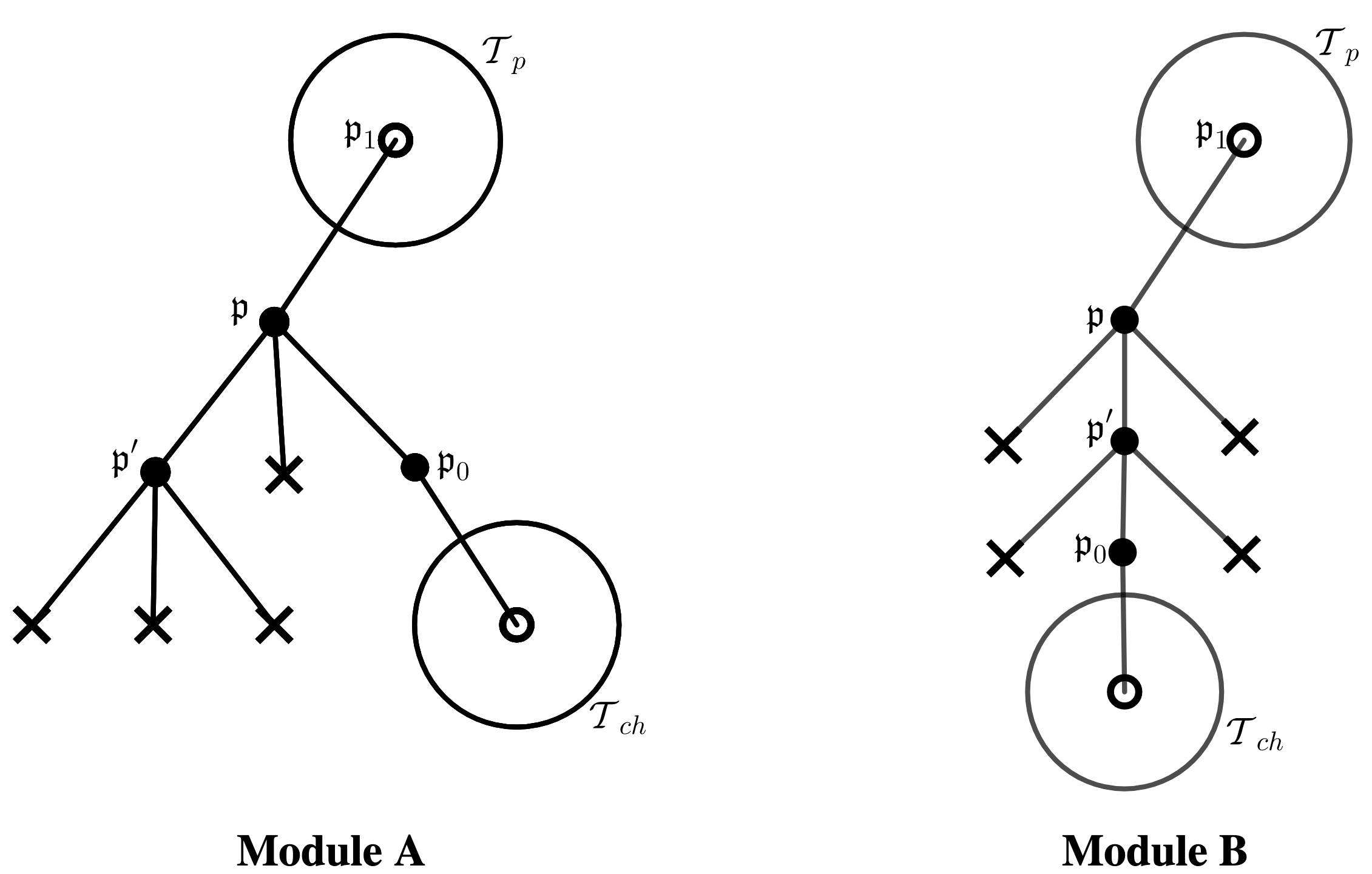}
\caption{The type II chain in \emph{Case 2} corresponds to a chain of modules of form A and B depicted here. All nodes marked by $\times$ are leafs. Each module connects to the next one either through the parent node $\pf$ or through the child node $\pf_0$. Here $\Tc_p$ denotes the tree obtained by deleting the subtree rooted at $\pf$ (keeping the leaf $\pf$), and $\Tc_{ch}$ denotes the tree rooted at $\pf_0$.}
		\label{FigCase2AB}
	\end{figure} 
	
	In fact, apart from the double bond connecting to $A(\pf')$, there is at least one more PC bond (which corresponds to the non-root branching node $\pf$) at the atom $A(\pf)$; this must be a single bond connecting to $A(\pf_1)$ where $\pf_1$ is the parent of $\pf$, so (2) is true. Now, apart from this single bond and the double bond between $A(\pf)$ and $A(\pf')$, there are three remaining bonds connecting to either $A(\pf)$ or $A(\pf')$, which correspond to the three remaining children of $\pf$ and $\pf'$ listed in (1). Among these three bonds, exactly two are LP bonds and exactly one is a PC bond (thanks to the assumption we made above), hence exactly two of the three children are leaves, and the other one, denoted by $\pf_0$, is a branching node which corresponds to an atom connected to either $A(\pf)$ or $A(\pf')$ by a single bond. This proves (1) and thus we are in either \emph{Case 2A} or \emph{Case 2B}.
	
	If we perform the above analysis for the LP-PC double bond at $\pf_0$ or $\pf_1$, and repeat this process, it is easy to see that each type II chain in $\Mb$ corresponds to a chain in $\Qc$, which is formed by repeatedly stacking one of the modules A or B depicted in Figure \ref{FigCase2AB} (with each module connecting to the next one either through the parent node $\pf$ or through the child node $\pf_0$).
\begin{rem}\label{RemarkTypeIanalysis} A similar (and simpler) argument to the above can be used to show that, each type I chain in $\Mb$ must be obtained from either one irregular chain, or the union of two irregular chains in $\Qc$. Note that unlike here, the argument for type I chains will involve signs, but this can be easily adjusted.
\end{rem}
\subsubsection{Conclusion on the tree structure of $\Qc$}\label{conclusionq} From the discussions in Sections \ref{LPlp} and \ref{PClp} we conclude the followings.
	
In \emph{Case 1}, there exist two paired branching nodes $\pf$ and $\pf'$, with $\pf \in \widetilde \Nc^{ch}$, such that neither is a child of the other. Moreover, $\pf$ and $\pf'$ each has two children leaves that form two pairs, see Figure \ref{fig:Case1}. Note that this is a property of the unsigned couple $\Qc^{ns}$.
	
In \emph{Case 2}, the couple $\Qc$ contains at most $Cr$ chains, each consisting of modules A and B as described in Figure \ref{FigCase2AB}, such that the rest of the couple has at most $Cr$ nodes. Moreover for each module A or B in this chain, the nodes $\pf$ and $\pf'$ (as in Figure \ref{FigCase2AB}) are paired with $\pf \in \widetilde \Nc^{ch}$, and $\pf'$ is a child of $\pf$. Note that this is a property for the double-tree $\Qc^{tr}$.
\subsection{Induction step} Now we can proceed with the induction step in the proof of Proposition \ref{section8main}. As stated before we will deal with the unsigned couple $\Qc^{ns}$. Recall that the couple $\Qc$ is formed by two trees $\Tc^\pm$, with corresponding sets of branching nodes $\Nc^\pm$.
	
	Suppose first that we are in \emph{Case 1}, and we fix $\pf$ and $\pf'$ as in Section \ref{conclusionq}. In this case, using the notation in Figure \ref{fig:Case1}, denote by $\Tc_{ch}$ the subtree attached to $\pf_0$, and $\Nc_{ch}$ the set of branching nodes in $\Tc_{ch}$, and let $\pf_1$ be the parent of $\pf$. Also denote by $\Tc_p$ the tree obtained by deleting the subtree rooted at $\pf$ from the tree containing $\pf$ (keeping $\pf$ as a leaf), and let $\Nc_p$ be its set of branching nodes. Without loss of generality assume $\pf\in\Tc^+$, define $\Uc_{\Tc^+}=\Uc_{\Tc^{+}}(t, \vsigma,\alpha[\Nc^{+}])$ as in \eqref{modifiedtimeint2}. Then we have that (note $\sigma_\pf$ and $\sigma_{\pf_0}$ may be replaced by $0$)
	\begin{align*}
	\Uc_{\Tc^{+}}(t, \vsigma,\alpha[\Nc^{+}])=\int_{\Fc_1}\prod_{\nf \in \Nc_{p}} \left(e^{\pi i \alpha_\nf t_{\nf}}\right) \int_{0}^{ t_{\pf_1}-\sigma_\pf} e^{\pi i \alpha_\pf t_{\pf}}\,\mathrm{d}t_\pf\int_0^{t_\pf-\sigma_{\pf_0}}e^{\pi i \alpha_{\pf_0}t_{\pf_0}}\,\mathrm{d}t_{\pf_0}\prod_{j=1}^3\Uc_{\Tc_{ch}^{(j)}}(t_{\pf_0},\vsigma, \alpha[\Nc_{ch}^{(j)}]),
	\end{align*}
	where $\Tc_{ch}^{(j)}$ are the three subtrees of $\Tc_{ch}$, $\Nc_{ch}^{(j)}$ are defined accordingly, and
	\begin{equation}\label{timedomS8}
	\Fc_1:=\big\{t[\Nc_p]:0<t_{\nf}<t_{\nf^p}<t,\text{ where }\nf^p\text{ is the parent of }\nf,\text{ and }t_\nf<t_{\nf^p}-\sigma_\nf \text{ if }\nf\in \Xi\}.
	\end{equation}
	 Interchanging the order of integration, we obtain that 
	\begin{align*}
	\Uc_{\Tc^{+}}(t, \vsigma, \alpha[\Nc^{\pm}])=&\int_{\Fc_1}\prod_{\nf \in \Nc_{p}} \left(e^{\pi i \alpha_\nf t_{\nf}}\right) \int_{0}^{ t_{\pf_1}-\sigma_{\pf}-\sigma_{\pf_0}}G(t_{\pf_1}, t_{\pf_0})e^{\pi i \alpha_{\pf_0}t_{\pf_0}} \,\mathrm{d}t_{\pf_0}\prod_{j=1}^3\Uc_{\Tc_{ch}^{(j)}}(t_{\pf_0}, \vsigma,\alpha[\Nc_{ch}^{(j)}]), \\
	G(t_{\pf_1}, t_{\pf_0})=&\frac{1}{\pi i \alpha_\pf}\left(e^{\pi i \alpha_\pf (t_{\pf_1}-\sigma_\pf)}-e^{\pi i \alpha_\pf (t_{\pf_0}+\sigma_{\pf_0})}\right)\chi_0(\frac{t_{\pf_1}}{10})\chi_0(\frac{t_{\pf_0}}{10}).
	\end{align*}
	Clearly $G$ satisfies $\|\widehat G(\eta,\theta)\|_{L_{\eta,\theta}^1L_{\sigma_{\pf},\sigma_{\pf_0}}^\infty}\leq C\langle \alpha_p\rangle^{-1}$ (where $\widehat{G}$ is the Fourier transform on $\Rb^2$), hence
	\begin{align*}
	\Uc_{\Tc^{+}}(t, \vsigma, \alpha[\Nc^{+}])=&\int_{\Rb^2}\widehat G(\eta, \theta)\int_{\Fc_1}\prod_{\nf \in \Nc^{p}} \left(e^{\pi i \alpha_\nf t_{\nf}}\right) \int_{0}^{ t_{\pf_1}-\sigma_{\pf}-\sigma_{\pf_0}}e^{2\pi i \eta t_{\pf_1}}e^{\pi i(\alpha_{\pf_0}+2 \theta)t_{\pf_0}}\mathrm{d}t_{\pf_0} \\\times&\prod_{j=1}^3\Uc_{\Tc_{ch}^{(j)}}(t_{\pf_0}, \vsigma,\alpha[\Nc_{ch}^{(j)}])
	\,\,\,=\,\,\,\int_{\Rb^2}\widehat G(\eta, \theta) \Uc_{\Tc_*^{+}}(t,\vsigma',\alpha'[{\Nc}_*^+])
	\end{align*}
	where $ \Tc_*^{+}$ is the tree obtained by replacing the subtree rooted at $\pf$ with the subtree rooted at $\pf_0$ (i.e. merging $\pf$ and $\pf_0$), which has ${\Nc}_*^+=\Nc^{\pm}\backslash\{\pf\}$ as its set of branching nodes, and $\alpha'$ is obtained from $\alpha$ by adding $2\eta$ and $2\theta$ to $\alpha_{\pf_1}$ and $\alpha_{\pf_0}$ respectively. Similarly, $\vsigma'$ is obtained from $\vsigma$ by restricting to the new set of branching nodes and replacing $\sigma_{\pf_0}$ by $\sigma_{\pf_0}+\sigma_{\pf}$.
	
	Doing the same computation for the node $\pf'$ (for which $\alpha_{\pf'}=-\alpha_{\pf}+\mu_\pf$), noticing that $\pf'\not\in\{\pf_1,\pf_0\}$. We obtain that 
	\begin{align*}
	\Vc_{\Qc^{ns}}(t, s, \vsigma,\alpha[\widetilde\Nc])&=\Uc_{\Tc^+}(t,\vsigma, \alpha[\Nc^+])\Uc_{\Tc^-}(s, \vsigma,\alpha[\Nc^-])=\int_{\Rb^4}\widehat G(\eta, \theta)\widehat G'(\eta', \theta') \Vc_{\Qc_{new}^{ns}}(t, s, \widetilde \vsigma, \widetilde \alpha[\widetilde{\Nc}_{new}])
	\end{align*}
	where $\Qc^{ns}_{new}$ is the unsigned couple obtained from $\Qc^{ns}$ by replacing the trees rooted at $\pf$ and $\pf'$ with the trees rooted at $\pf_0$ and $\pf_0'$ respectively, and has the same leaf pairing and branching node pairing structures as $\Qc^{ns}$. The set of branching nodes $\Nc^*_{new}=\Nc^*\backslash\{\pf, \pf'\}$, and $\widetilde \Nc_{new}=\widetilde \Nc\backslash\{\pf\}$ is the set obtained from $\Nc^*_{new}$ by pairing branching nodes as above. The variables $\widetilde{\alpha}[\widetilde{\Nc}_{new}]$ is the restriction of $\alpha[\widetilde{\Nc}]$ to  $\widetilde{\Nc}_{new}$, which then has at most four entries translated by some linear combinations of $(\pm2\theta, \pm2\eta, \pm2\theta', \pm2\eta')$. Similarly, $\widetilde \vsigma$ is obtained from $\vsigma$ by translations as explained above. 
	
The function $G'$ satisfies the same bound as $G$, but with the right hand side replaced by $C\langle \alpha_\pf-\mu_\pf\rangle^{-1}$. Using that $\int_\Rb\langle \alpha_\pf\rangle^{-1}\langle \alpha_\pf-\mu_\pf\rangle^{-1}\mathrm{d}\alpha_\pf\leq C$, we can directly estimate 
	\begin{equation*}
	\int_{(\alpha_\nf):\alpha_\nf\in S_\nf(1)}\sup_{\vsigma}\left| \Vc_{\Qc^{ns}}(t, s, \vsigma, \alpha[\widetilde\Nc])\right|
	\leq C\sup_{(T_\nf(1))}\int_{(\alpha_\nf):\alpha_\nf\in T_\nf(1)}\sup_{\vsigma}\left|\Vc_{ \Qc^{ns}_{new}}(t, s, \vsigma, \alpha[\widetilde \Nc_{new}])\right|,
	\end{equation*} 
where $T_\nf(1)$ ranges over all subsets of $\Rb$ with measure $\leq L^{10d}$, and we assume $\nf\in\widetilde{\Nc}$ in the first integral, and $\nf\in\widetilde{\Nc}_{new}$ in the second integral. Using the induction hypothesis on $\Qc_{new}^{ns}$, we obtain the needed estimate. 
	
	\medskip 
	
	We are thus left with \emph{Case 2} where $\Qc$ is the union of at most $p\leq Cr$ chains of modules A and B as described in Figure \ref{FigCase2AB}, plus at most $Cr$ other nodes. At this point we will forget the leaf pairing structure of $\Qc$ and view it as a double-tree $\Qc^{tr}$ with some branching nodes paired. We will prove, with $S_n(1)$ defined as above, that
	\begin{equation}\label{s9Indbd}\delta^{n_0'/4}\cdot\int_{(\alpha_\nf):\alpha_\nf\in S_\nf(1)}\sup_{\vsigma}\big|\Vc_{\Qc^{tr}}(t,s,\vsigma,\alpha[\widetilde{\Nc}])\big|\leq (C^+)^{n_0'}L^{Cp\sqrt{\delta}}(\log L)^{|\Nc^{rm}|+Cp},
	\end{equation}
	where $n_0'$ is the scale of $\Qc$. This estimate would give \eqref{uniformL1} since $|\Nc^{rm}|+p\leq Cr$.
	
	\medskip 
	
	We will prove estimate $\eqref{s9Indbd}$ by induction on $p$, with the base case $p=0$ being a consequence of Lemma \ref{timeintlemma}. Let $\pf\in \widetilde{\Nc}^{ch}$, and let us start by assuming that $\pf$ is the parent node in a Module A. Assume without loss of generality that both nodes $\pf$ and $\pf'$ belong to the tree $\Tc^+$. Let $\Tc_{ch}$ be the tree rooted at $\pf_0$ (see Figure \ref{FigCase2AB}), $\Tc_{p}$ be the tree obtained from $\Tc$ by removing the subtree rooted at $\pf$ (keeping $\pf$ as a leaf), and $\Nc_{ch}$ and $\Nc_p$ be the respective sets of branching nodes. Then if $\pf_1$ is the parent of $\pf$, we have
	\begin{align}
	\Uc_{\Tc^+}(t, \vsigma,\alpha[\Nc^{+}])&=\int_{\Fc_1}\prod_{\nf \in \Nc_{p}} \left(e^{\pi i \alpha_\nf t_{\nf}}\right) \,\mathrm{d}t_{\nf}\int_{0}^{ t_{\pf_1-\sigma_{\pf}}}\,\mathrm{d}t_\pf e^{\pi i \alpha_\pf t_{\pf}}\Uc_{\Tc_{ch}}(t_\pf,\vsigma, \alpha[\Nc_{ch}]) \int_0^{t_\pf-\sigma_{\pf'}}\,\mathrm{d}t_{\pf'}e^{-\pi i (\alpha_{\pf}-\mu_{\pf})t_{\pf'}}\nonumber \\
	\label{lowerbd}&=\int_{\Fc_1}\prod_{\nf \in \Nc_{p}} \left(e^{\pi i \alpha_\nf t_{\nf}}\right)\,\mathrm{d}t_{\nf} \int_{\sigma_{\pf'}}^{ t_{\pf_1-\sigma_{\pf}}}\,\mathrm{d}t_\pf G_\pf(t_{\pf_1}, t_{\pf}) \Uc_{\Tc_{ch}}(t_\pf,\vsigma, \alpha[\Nc_{ch}]),
	\end{align}
	where $\Fc_1:=\big\{t[\Nc_p]:0<t_{\nf}<t_{\nf^p}<t,\text{ and } t_{\nf}<t_{\nf^p}-\sigma_\nf \text{ if }\nf\in \Xi\}$. The function $G_{\pf}$ is defined by
	\begin{equation}\label{defgp}G_{\pf}(t_{\pf_1},t_\pf)=\chi_0\big(\frac{t_{\pf}}{10}\big)\chi_0\big(\frac{t_{\pf_1}}{10}\big)e^{\pi i \alpha_\pf t_{\pf}}\frac{1}{\pi i(\mu_\pf-\alpha_\pf)}\big(e^{\pi i(\mu_\pf-\alpha_\pf)(t_\pf-\sigma_{\pf'})}-1\big)\end{equation}and satisfies that
	\begin{equation}\label{gpbound1}
	\|\widehat G_{\pf}(\eta,\theta)\|_{L_{\eta,\theta}^1L_{\sigma_{\pf'}}^\infty}\leq \frac{C}{\langle \alpha_\pf-\mu_\pf\rangle}.
	\end{equation} Moreover, in view of the restriction $t_\pf>\sigma_{\pf'}$ in the last integral in (\ref{lowerbd}), we may truncate $G_\pf$ and define $G_\pf^{cut}:=G_\pf\cdot \mathbf{1}_{t_\pf-\sigma_{\pf'}\geq 0}$. This truncated function then satisfies (\ref{gpbound1}), but with the right hand side replaced by $C\langle\alpha_\pf-\mu_\pf\rangle^{-1}\log (2+|\alpha_\pf-\mu_\pf|)$, which follows from direct calculations.
	
	In case 2B, the computation is similar and one obtains that 
	\begin{align*}
	\Uc_{\Tc^{+}}(t, \alpha[\Nc^{+}])&=\int_{\Fc_1}\prod_{\nf \in \Nc_{p}} \left(e^{\pi i \alpha_\nf t_{\nf}}\right) \int_{0}^{ t_{\pf_1}-\sigma_{\pf}}\,\mathrm{d}t_\pf e^{\pi i \alpha_\pf t_{\pf}} \int_0^{t_\pf-\sigma_{\pf'}}\,\mathrm{d}t_{\pf'}e^{-\pi i (\alpha_{\pf}-\mu_{\pf})t_{\pf'}}\Uc_{\Tc_{ch}}(t_{\pf'},\vsigma, \alpha[\Nc_{ch}]) ,\\
	&=\int_{\Fc_1}\prod_{\nf \in \Nc_{p}} \left(e^{\pi i \alpha_\nf t_{\nf}}\right) \int_{0}^{ t_{\pf_1-\sigma_\pf-\sigma_{\pf'}}}\,\mathrm{d}t_{\pf'}G_{\pf}(t_{\pf_1}, t_{\pf'})\Uc_{\Tc_{ch}}(t_{\pf'}, \vsigma,\alpha[\Nc_{ch}])
	\end{align*}
	with the kernel \[G_\pf(t_{\pf_1}, t_{\pf'})=\chi_0\big(\frac{t_{\pf}}{10}\big)\chi_0\big(\frac{t_{\pf_1}}{10}\big)e^{-\pi i (\alpha_{\pf}-\mu_{\pf})t_{\pf'}}\frac{1}{\pi i\alpha_\pf}\big(e^{\pi i\alpha_\pf(t_{\pf_1}-\sigma_\pf)}-e^{\pi i\alpha_\pf(t_{\pf'}+\sigma_{\pf'})}\big)\] that satisfies the bound
	\begin{equation}\label{gpbound2}
	\|\widehat G_{\pf}(\eta,\theta)\|_{L_{\eta,\theta}^1L_{\sigma_\pf,\sigma_{\pf'}}^\infty}\leq \frac{C}{\langle \alpha_\pf\rangle}.
	\end{equation}

	Consider now one of the $p$ chains of modules A and B, suppose that it contains $\ell$ modules, which we list from top to bottom, and is contained in the tree $\Tc^+$. Define $\hf_1$ to be the $\pf$ node (see Figure \ref{FigCase2AB}) of the top module, and $\pf_1$ to be the parent of $\hf_1$; also define $\hf_{\ell+1}$ to be the $\pf_0$ node of the bottom module, and $\pf_{\ell+1}$ to be the parent of $\hf_{\ell+1}$. Define $\Tc_{above}$ to be the tree obtained by removing the subtree rooted at $\hf_1$ and keeping $\hf_1$ as a leaf, and $\Tc_{below}$ to be the tree rooted at $\hf_{\ell+1}$ (we define $\Nc_{above}$ and $\Nc_{below}$ accordingly). Note that $\Tc_{above}$ is just $\Tc_p$ for the top module and $\Tc_{below}$ is just the $\Tc_{ch}$ for the bottom module. Let $\hf_k\,(1\leq k\leq \ell)$ be the $\pf$ node of the $k$-th module from top to bottom, and write $(\alpha_k,\mu_k):=(\alpha_{\hf_k},\mu_{\hf_k})$. Then by iterating the above calculations, we have
	\begin{equation*}
	\Uc_{\Tc^{+}}(t,\vsigma, \alpha[\Nc^{+}])=\int_{\Fc_1}\prod_{\nf \in \Nc_{above}} \left(e^{\pi i \alpha_\nf t_{\nf}}\right)\mathrm{d}t_\nf \int_\Cc\prod_{k=1}^\ell                           G_k^*(t_{k},t_{k+1})\,\mathrm{d}t_{k+1}\cdot \Uc_{\Tc_{below}}(t_{\pf_{\ell+1}},\vsigma, \alpha[\Nc_{below}]).
	\end{equation*}
	Here $\Fc_1$ is the set defined before but associated with $\Nc_{above}$, and $(t_1,t_{\ell+1}):=(t_{\pf_1},t_{\pf_{\ell+1}})$. The function $G_k^*$ equals $G_{\hf_k}^{cut}$ if the $k$-th module is A and either $k\leq 4$ or $k\geq \ell-4$, and $G_k^*=G_{\hf_k}$ otherwise. The domain
\[\Cc=\{(t_2, \ldots, t_{\ell+1}): t_{k}>t_{k+1}+\widetilde \sigma_{k},\text{ for } 1\leq k \leq \ell;\,\,t_k>\sigma_k',\text{ for } 5\leq k\leq \ell-5\},\] where $\widetilde \sigma_k$ is the sum of zero, one or two $\sigma_\nf$ variables appearing in $\vsigma$, and $\sigma_k'$ equals either $0$ or some $\sigma_\nf$ that appears in $\vsigma$. The function $G_k^*$ satisfies either (\ref{gpbound1}) or (\ref{gpbound2}), and with the right hand side of (\ref{gpbound1}) multiplied by $\log(2+|\alpha_\pf-\mu_\pf|)$ \emph{only if} $k\leq 4$ or $k\geq \ell-4$. As a result, we may write, with $\widetilde \sigma_{\text{tot}}=\widetilde \sigma_1+\ldots+\widetilde \sigma_{\ell}$, that
	\begin{align*}
	\Uc_{\Tc^{+}}(t,\vsigma, \alpha[\Nc^{+}])&=\int_{(\Rb^2)^{\ell}}\prod_{k=1}^\ell \widehat {G_{k}^*}(\eta_k, \theta_k)\int_{\Fc_1}\prod_{\nf \in \Nc_{above}} \left(e^{\pi i \alpha_\nf t_{\nf}}\right)\mathrm{d}t_{\nf} \cdot e^{2\pi i \eta_1t_{\pf_1}}\\
	&\qquad \qquad \qquad  \qquad\quad \,\times \int_\Cc \prod_{k=2}^{\ell}e^{2\pi i (\eta_k+\theta_{k-1})t_{k}}\,\mathrm{d}t_k\cdot e^{2\pi i \theta_\ell t_{\pf_{\ell+1}}}\Uc_{\Tc_{below}}(t_{\pf_{\ell+1}},\vsigma, \alpha[\Nc_{below}])\,\mathrm{d}t_{\pf_{\ell+1}}\\
	&=\int_{(\Rb^2)^{\ell}}\prod_{k=1}^\ell \widehat {G_{k}^*}(\eta_k, \theta_k)\int_{\Fc_1}\prod_{\nf \in \Nc_{above}} \left(e^{\pi i \alpha_\nf t_{\nf}}\right)\mathrm{d}t_{\nf}  \\
	&\qquad \qquad \qquad  \qquad\quad \times\int_0^{t_{\pf_1}-\widetilde \sigma_{\text{tot}}} \Kc(t_{\pf_1}, t_{\pf_{\ell+1}})\cdot\Uc_{\Tc_{below}}(t_{\pf_{\ell+1}},\vsigma, \alpha[\Nc_{below}])\mathrm{d}t_{\pf_{\ell+1}},\\
	\Kc(t_{\pf_1},t_{\pf_{\ell+1}})&=\chi_0\big(\frac{t_{\pf_1}}{10}\big)\chi_0\big(\frac{t_{\pf_{\ell+1}}}{10}\big)e^{2\pi i \eta_1t_{\pf_1}}e^{2\pi i \theta_\ell t_{\pf_{\ell+1}}}\int_{\Cc_0} \prod_{k=2}^{\ell}e^{2\pi i (\eta_k+\theta_{k-1})t_{k}},
	\end{align*}
	where $\Cc_0=\{(t_2, \ldots, t_{\ell}): (t_2,\cdots,t_{\ell+1})\in\Cc\}$. 
	
	Now let $\eta_k+\theta_{k-1}=\beta_k$, then the above integral in $\Cc_0$ can be written as
	\[\int_{-\infty}^{t_{\pf_1}-\widetilde{\sigma}_1}e^{2\pi i\beta_2t_2}\mathrm{d}t_2\int_{-\infty}^{t_2-\widetilde{\sigma}_2}e^{2\pi i\beta_3t_3}\mathrm{d}t_3\int_{-\infty}^{t_3-\widetilde{\sigma}_3}e^{2\pi i\beta_4t_4}G(t_4)\mathrm{d}t_4;\qquad G(t_4):=\int_{\Cc_4}\prod_{k=5}^\ell e^{2\pi i\beta_kt_k},\] where $\Cc_4=\{(t_5,\cdots,t_\ell):(t_2,\cdots,t_{\ell+1})\in\Cc\}$. Therefore, if the derivatives do not fall on $\chi_0$ factors, we have
	\[\big|(\partial_{t_{\pf_1}}-2\pi i(\beta_3+\beta_2+\eta_1))(\partial_{t_{\pf_1}}-2\pi i(\beta_2+\eta_1))(\partial_{t_{\pf_1}}-2\pi i\eta_1)\Kc\big|\leq C\|G\|_{L^\infty}\leq\frac{C}{(\ell-5)!}\] uniformly in $(t_{\pf_1},t_{\pf_{\ell+1}})$ and $\vsigma$, noticing that $\Cc_4$ is a subset of a simplex. If any of the above derivatives falls on $\chi_0$ then we can take that derivative again and get similar estimates. Since also $|\Kc|\leq C/(\ell-1)!$, we conclude that
	\[|\widehat{\Kc}(\eta_0,\theta_0)|\leq \frac{C}{(\ell-5)!}\min(1,|\eta_0-\eta_1|^{-1}|\eta_0-(\beta_2+\eta_1)|^{-1}|\eta_0-(\beta_3+\beta_2+\eta_1)|^{-1})\] uniformly in $\vsigma$ (there may be other possibilities for denominators but the results are the same). In the same way we can get similar estimates for $\theta_0$, and combing these two yields the bound
	$$
	\sup_{(\eta_k ,\theta_k)_{k=1, \ldots \ell}}\|\widehat \Kc(\eta_0,\theta_0)\|_{L_{\eta_0,\theta_0}^1L_{\vsigma}^\infty}\leq \frac{C}{(\ell-5)!}.
	$$
As a result, we have that
	\begin{align*}
	\Uc_{\Tc^{+}}(t,\vsigma, \alpha[\Nc^{+}])&=\int_{(\Rb^2)^{\ell+1}}\widehat \Kc(\eta_0, \theta_0)\prod_{k=1}^\ell \widehat {G_{k}^*}(\eta_k, \theta_k)\int_{\Fc_1}\prod_{\nf \in \Nc_{above}} \left(e^{\pi i \alpha_\nf t_{\nf}}\right)\mathrm{d}t_{\nf} \cdot e^{2\pi i \eta_0t_{\pf_1}}\\
	&\qquad \qquad \qquad \qquad \qquad \qquad \times \int_0^{t_{\pf_1}-\widetilde \sigma_{\text{tot}}} e^{2\pi i \theta_0 t_{\pf_{\ell+1}}} \Uc_{\Tc_{below}}(t_{\pf_{\ell+1}},\vsigma, \alpha[\Nc_{below}])dt_{\pf_{\ell+1}}\\
	&=\int_{(\Rb^2)^{\ell+1}}\widehat \Kc(\eta_0, \theta_0)\prod_{k=1}^\ell \widehat {G_{k}^*}(\eta_k, \theta_k) \cdot\Uc_{\Tc^+_{new}}(t, \widetilde \vsigma, \widetilde \alpha[\Nc_{new}^+]) \, \prod_{k=0}^\ell \mathrm{d}\eta_k \mathrm{d}\theta_k,
	\end{align*}
	where $\Tc^+_{new}$ is the tree obtained from $\Tc^+$ by deleting this chain of Modules A and B as follows: Attach the tree $\Tc_{below}$ at its root as one of three children of $\hf_1$ (see Figure \ref{S9lastfig}) keeping the other two children as leaves. $\Nc_{new}^+$ is the set of branching nodes of $\Tc_{new}$, $\widetilde{\vsigma}$ is the restriction of $\vsigma$ with $\sigma_{\hf_1}=\widetilde{\sigma_{\mathrm{tot}}}$, and $\widetilde \alpha[\Nc_{new}^+]$ is obtained from $\alpha[\Nc_{new}^+]$ by translating $\alpha_{\pf_1}$ by $2\eta_0$, defining $\alpha_{\hf_1}=2\theta_0$, and keeping all remaining $\alpha_{\nf}$ for $\nf \in \Nc^+_{new}\setminus\{\pf_1, \hf_1\}$ the same.
\begin{figure}[h!]
		\includegraphics[scale=1]{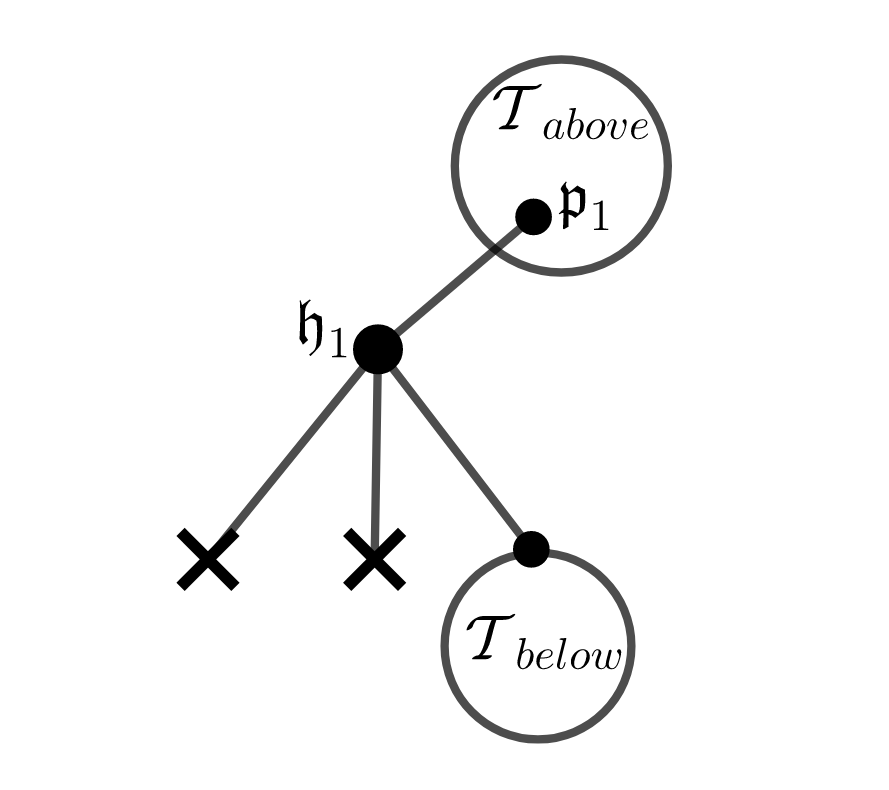}
		\caption{The tree $\Tc^+_{new}$ is obtained from $\Tc^+$ by removing the chain of A and B modules, and connecting the tree $\Tc_{below}$ at one of the children of $\hf_1$ keeping the other two children as leaves.}
		\label{S9lastfig}
	\end{figure}

	We define the double-tree $\Qc_{new}^{tr}=\Tc_{new}^+\cup \Tc^-$ (which has no leaf pairing structure), with the branching node pairing structure inherited from $\Qc^{tr}$ and not involving $\hf_1$. Also define $\widetilde{\Nc}_{new}^{tr}$ accordingly. Using the induction hypothesis, we can take supremum over $\vsigma$, then integrate in $\alpha_\nf$ for $\nf\in\widetilde{\Nc}_{new}^{tr}$, to obtain that the left hand side of \eqref{s9Indbd} is bounded by (recalling that the removed chain of A and B modules has $2\ell$ branching nodes)
	\begin{align*}
	&(C^+)^{n_0'-2\ell}L^{C(p-1)\sqrt \delta }(\log L)^{|\Nc^{rm}|+C(p-1)}\\ &\qquad \qquad \qquad \times\left(\delta^{(2\ell)/4}\int_{(\alpha_1,\cdots,\alpha_\ell):\alpha_k\in S_k(1)}\int_{(\Rb^2)^{\ell+1}}\sup_{\vsigma}|\widehat \Kc(\eta_0, \theta_0)|\prod_{k=1}^\ell |\widehat {G_{k}^*}(\eta_k, \theta_k)| \right)\\
	&\leq (C^+)^{n_0'-2\ell}C^\ell L^{C(p-1)\sqrt \delta }(\log L)^{|\Nc^{rm}|+C(p-1)} \left(\delta^{(2\ell)/4}\int_{(\alpha_1,\cdots,\alpha_\ell):\alpha_k\in S_k(1)}\frac{\prod_{k=1}^\ell \Phi_k(\alpha_{k}-\widetilde \mu_{k})}{(\ell-5)!}\right)\\
	&\leq (C^+)^{n_0'-2\ell}C^\ell L^{C(p-1)\sqrt \delta }(\log L)^{|\Nc^{rm}|+C(p-1)} \left(C^\ell\delta^{(2\ell)/4}\frac{(\log L)^{\ell+10}}{(\ell-5)!}\right)\\
	&\leq (C^+)^{n_0'-2\ell}C^{2\ell}L^{C(p-1)\sqrt \delta }(\log L)^{|\Nc^{rm}|+C(p-1)} \left((\log L)^{15}e^{C\sqrt \delta \log L}\right),
	\end{align*}
	where $S_k(1)$ is a set of measure $\leq L^{10d}$, we denoted by $\widetilde \mu_{k}$ either $0$ or $\mu_k$ (depending on whether the $k$-th module is A or B), and $\Phi_k(z)$ is either $\langle z\rangle^{-1}$ or (for at most $10$ values of $k$) $\langle z\rangle^{-1}\log (2+|z|)$. In the final step we used the bound $\frac{x^{\ell-5}}{(\ell-5)!}\leq e^x$ for any $\ell$. Note also that in applying the induction hypothesis for $\Nc_{new}^{tr}$ we have fixed the value of $\alpha_{\hf_1}=2\theta_0$ (using integrability of $\widehat{\Kc}$), but it is clear from the proof that fixing the value of any $\alpha_\nf$ will only lead to better estimates than integrating in $\alpha_\nf$. This gives the estimate \eqref{s9Indbd} and finishes the proof. 
	\end{proof}
\subsection{Proof of Propositions \ref{mainprop1} and \ref{mainprop4}} We are now ready to prove Propositions \ref{mainprop1} and \ref{mainprop4}. First we establish the absolute upper bound for (\ref{section6fin2}), which then allows us to control (\ref{irrechainsum}).
\begin{prop}\label{congsumest} Given one congruence class $\Fs$ of non-regular marked couples of scale $n$ as in Definition \ref{conggen}, the expression
\begin{equation}\label{congsum}\sum_{\Qc\in\Fs}\Kc_\Qc(t,s,k)
\end{equation} can be decomposed into at most $C^n$ terms. For each term there is an integer $1\leq r\leq n$ such that this term is bounded, uniformly in $(t,s)\in[0,1]^2$, by $(C^+\delta^{1/4})^{n}\langle k\rangle^{-20d}\cdot L^{-\nu r}$. Moreover, for each fixed $r$, the number of possibilities of $\Qc$ (or $\Fs$) that correspond to this $r$ is at most $C^n(Cr)!$.
\end{prop}
\begin{proof} As in Section \ref{irchaincancel} we can reduce to (\ref{section6fin2}), and then to (\ref{finalexp}). Note that in (\ref{finalexp}) the $\Qc$ actually means $\Qc_{sk}^\#$ by our notation. Using the decay factors in (\ref{xtotbound}) we can gain the power $\langle k\rangle^{-30d}$, and also restrict to the subset where $|k_\lf-a_\lf|\leq 1$ for some fixed parameters $(a_\lf)$ (with summability in $(a_\lf)$ guaranteed). Using the bound for $G(\vlambda)$, which is a modification of the first inequality in (\ref{tensorbd}), we may also fix the value of $\vlambda$ (and hence $\mu_\nf$).

As in Section \ref{improvecount}, by decomposing into at most $C^{n_0'}$ terms (where $n_0'$ is the scale of $\Qc_{sk}^\#$), we can add the set of extra conditions $\mathtt{Ext}$, which also defines the sets $\widetilde{\Nc}$ (as in Proposition \ref{section8main}), etc., and the value $r_1\geq 1$. Let $r=r_0+r_1$ as above, then thanks to $\mathtt{Ext}$, we can use (\ref{eqnuv}) to reduce $\Uc_{\Qc_{sk}^\#}$ to $\Vc_{\Qc_{sk}^\#}$. Moreover, for each $\nf\in\widetilde{\Nc}$, the value $\delta L^2\zeta_\nf\Omega_\nf+\lambda_\nf$ belongs to some subset of $\Rb$ of cardinality at most $L^{3d}$, as $k[\Qc_{sk}^\#]$ varies (this is because each $k_\nf$ belongs to a ball of radius at most $n\leq (\log L)^3$ under our assumptions). In particular the value $m_\nf=\lfloor \delta L^2\zeta_\nf\Omega_\nf+\lambda_\nf\rfloor$ belongs to a set $S_\nf\subset\Zb$ with cardinality at most $L^{3d}$, for all possible choices of $k[\Qc_{sk}^\#]$.

To estimate (\ref{finalexp}) with $\vlambda$ fixed, we first integrate in $\vsigma$. Using (\ref{xtotbound}), we can estimate (\ref{finalexp}) using
\begin{equation}\label{finalexp2}\sum_{\Es_{sk}^\#}|\epsilon_{\Es_{sk}^\#}|\cdot\sup_{\vsigma}\big|\Vc_{\Qc_{sk}^\#}\big(t,s,\vsigma, (\delta L^2\zeta_\nf\Omega_\nf+\lambda_\nf)_{\nf\in\widetilde{\Nc}}\big)\big|,
\end{equation} where $\Es_{sk}^\#=k[\Qc_{sk}^\#]$ is a $k$-decoration of $\Qc_{sk}^\#$ {(we also have additional factors that will be collected at the end)}. We next fix the values of $m_\nf\in S_\nf$ for each $\nf$; note that then
\[\sup_{\vsigma}\big|\Vc_{\Qc_{sk}^\#}\big(t,s,\vsigma, (\delta L^2\zeta_\nf\Omega_\nf+\lambda_\nf)_{\nf\in\widetilde{\Nc}}\big)\big|\leq\sup_{(\alpha_\nf):|\alpha_\nf-m_\nf|\leq 1}\sup_{\vsigma}\big|\Vc_{\Qc_{sk}^\#}(t,s,\vsigma,\alpha[\widetilde{\Nc}])\big|\] by definition, so if we use (\ref{uniformL1}) to sum over $(m_\nf)$ in the end, we can further estimate (\ref{finalexp2}) using
\begin{equation}\label{finalexp3}\sum_{\Es_{sk}^\#}|\epsilon_{\Es_{sk}^\#}|\cdot\prod_\lf\mathbf{1}_{|k_\lf-a_\lf|\leq 1}\prod_\nf\mathbf{1}_{|\Omega_\nf-b_\nf|\leq \delta ^{-1}L^{-2}},
\end{equation} where $a_\lf$ and $b_\nf$ are constants, and we also include the conditions in $\mathtt{Ext}$. {Now (\ref{finalexp3}) is almost exactly the counting problem $\Df(\Mb,\mathtt{Ext})$ stated in Definition \ref{countingproblem}, due to Remark \ref{moledec}, except that we only assume $|k_\lf-a_\lf|\leq 1$ for \emph{leaves} $\lf$. However, for any branching node $\nf$ there exists a child $\nf'$ of $\nf$ such that $k_\nf\pm k_{\nf'}$ belongs to a fixed ball of radius $\mu_\nf^\circ$ as in Lemma \ref{auxlem}, so by using (\ref{auxineq}), one can reduce (\ref{finalexp3}) to at most $C^{n_0'}$ counting problems, each of which having exactly the same form as $\Df(\Mb,\mathtt{Ext})$ in Definition \ref{countingproblem}.} Therefore, (\ref{finalexp3}) can be bounded using Proposition \ref{gain} {(and using Remark \ref{countingrem} if necessary)}. Collecting all the factors appearing in the above estimates, we get that
\begin{equation}\label{finalest1}{\langle k\rangle^{20d}\cdot|(\ref{congsum})|}\leq (C^+)^n\delta^{(n-n_0')/2}\delta^{3n_0'/4} L^{-(d-1)n_0'}\cdot L^{-2\nu r _0}\cdot L^{Cr\sqrt{\delta}}(\log L)^{Cr}\delta^{-(n_0'+m)/2} L^{(d-1)n_0'-2\nu r_1},
\end{equation} which is then bounded by $(C^+\delta^{1/4})^{n} L^{-3\nu r/2}\delta^{-m/2}$, where $m$ is the total length of type I chains in the molecule obtained from $\Qc_{sk}^\#$. We know $m\leq Cr$ so $\delta^{-m/2}\leq L^{\nu r/2}$, which implies the desired bound.

Finally, suppose we fix $r$, then the base molecule formed by $\Qc_{sk}^\#$ is, up to at most $Cr$ remaining atoms, a union of at most $Cr$ type II chains with total length at most $n_0'$. This clearly has at most $(Cr)!C^n$ possibilities. By Proposition \ref{recover}, the number of choices for $\Qc_{sk}^\#$ is also at most $(Cr)!C^n$. To form $\Qc_{sk}$ from $\Qc_{sk}^\#$ one needs to insert at most $Cr$ irregular chains with total length at most $n$, which also has at most $C^n$ possibilities. Finally, using Corollary \ref{corcpl}, we see that $\Qc$ has at most $(Cr)!C^n$ choices. The number of choices for markings, as well as $\mathtt{Ext}$, are also at most $C^n$ and can be accommodated.
\end{proof}
\begin{proof}[Proof of Proposition \ref{mainprop1}] By definition, we have
\[\Eb|(\Jc_n)_k(t)|^2=\sum_{\Qc}\Kc_\Qc(t,t,k),\] where the sum is taken over all couples $\Qc=(\Tc^+,\Tc^-)$ such that $n(\Tc^+)=n(\Tc^-)=n$. If $\Qc$ is regular, then the number of such $\Qc$'s is at most $C^n$ by Proposition \ref{countcouple1}, and for each $\Qc$ we have $|\Kc_\Qc(t,t,k)|\lesssim \langle k\rangle^{-20d}(C^+\delta)^n$ by Proposition \ref{asymptotics1} and Remark \ref{timebound}. Therefore, the sum over these $\Qc$'s is under control.

Now consider non-regular $\Qc$. It follows from definition that the congruence relation (as in Definition \ref{conggen}) preserves the scales of both trees of a couple. Thus, the sum over $\Qc$ can be decomposed into sums over $\Qc\in\Fs$ (i.e. sums of form (\ref{congsum})), where $\Fs$ runs over the (possible) different congruence classes. Applying Proposition \ref{congsumest}, we can regroup these terms according to the value of $1\leq r\leq 2n$ (which we call the index), such that (i) each single term with index $r$ is bounded by $(C^+\delta^{1/4})^{2n}\langle k\rangle^{-20d}\cdot L^{-\nu r}$, and (ii) the number of terms with index $r$ is at most $(Cr)!C^{2n}$. Hence
\[\Eb|(\Jc_n)_k(t)|^2\lesssim \langle k\rangle^{-20d}(C^+\delta)^n+\langle k\rangle^{-20d}(C^+\delta^{1/4})^{2n}\sum_{r=1}^{2n}L^{-\nu r} (Cr)!C^{2n}\lesssim \langle k\rangle^{-20d}(C^+\sqrt{\delta})^{n},\] noticing also that $r\leq 2n\leq 2(\log L)^3$. This completes the proof.
\end{proof}
\begin{proof}[Proof of Proposition \ref{mainprop4}] Here we are considering the sum of $\Kc_\Qc(t,t,k)$ over all couples $\Qc$ such that $n(\Qc)=m$ for some fixed value $m$. If $\Qc$ is non-regular, then using the same argument as in the above proof we can bound the corresponding contribution by $\langle k\rangle^{-20d}(C^+\delta^{1/4})^{m}L^{-\nu}$ since we also have $r\geq 1$. Therefore we only need to consider regular couples $\Qc$. If $m$ is odd then this sum is zero because the scale of regular couples must be even. If $m=2n$, we only need to show that
\[\bigg|\sum_{\substack{n(\Qc)=2n\\\Qc\,\,\mathrm{regular}}}\Kc_\Qc(t,t,k)-\Mc_n(t,k)\bigg|\lesssim \langle k\rangle^{-20d}(C^+\sqrt{\delta})^nL^{-\nu},\] but this is a consequence of Proposition \ref{regcplapprox}. This completes the proof.
\end{proof}
\section{The operator $\Ls$}\label{operatornorm} In this section we prove Proposition \ref{mainprop2}. The arguments are mostly the same as in previous sections, so we will only point out the necessary changes in the proof. First, in order to expand the kernel $(\Ls^n)_{k\ell}^\zeta(t,s)$, we need to slightly modify the definition of trees and couples.
\begin{df}\label{flowers} A \emph{flower tree} is a tree $\Tc$ with one leaf $\ff$ specified, called the \emph{flower}; different choices of $\ff$ for the same tree $\Tc$ leads to different flower trees. There is a unique path from the root $\rf$ to the flower $\ff$, which we call the \emph{stem}. A \emph{flower couple} is a couple formed by two flower trees, such the two flowers are paired (in particular they have opposite signs).

The \emph{height} of a flower tree $\Tc$ is the number of branching nodes in the stem of $\Tc$. Clearly a flower tree of height $n$ is formed by attaching two sub-trees each time, and repeating $n$ times, starting from a single node; see Figure \ref{fig:flowertree}. We say a flower tree is \emph{admissible} if all these sub-trees have scale at most $N$.
  \begin{figure}[h!]
  \includegraphics[scale=.5]{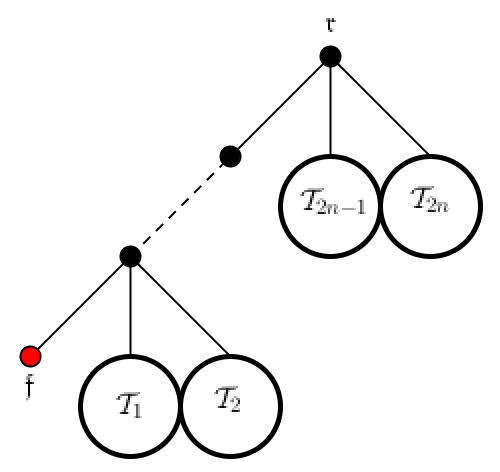}
  \caption{A flower tree, as in Definition \ref{flowers}. The red leaf $\ff$ is the flower, $\rf$ is the root, and $\Tc_j\,(1\leq j\leq 2n)$ are attached sub-trees, where $n$ is the height.}
  \label{fig:flowertree}
\end{figure} 
 \end{df}
 \begin{prop}\label{L^nexpansion} Given $\zeta\in\{\pm\}$, we can make the decomposition (\ref{decomposem}), such that for each $m$,
 \begin{equation}\label{linkerexpansion}\Eb|(\Ls^n)_{k\ell}^{m,\zeta}(t,s)|^2=\sum_{\Qc}\widetilde{\Kc}_\Qc(t,s,k,\ell),
 \end{equation} where the sum is taken over all \emph{flower} couples $\Qc=(\Tc^+,\Tc^-)$, such that both $\Tc^\pm$ are admissible, have \emph{height} $n$ and \emph{scale} $m$, and the flower of $\Tc^\pm$ has sign $\pm\zeta$. For $t>s$, the quantity $\widetilde{\Kc}_\Qc$ is defined similar to (\ref{defkq}):
 \begin{equation}\label{defaltkq}\widetilde{\Kc}_\Qc(t,s,k,\ell):=\bigg(\frac{\delta}{2L^{d-1}}\bigg)^{2m}\zeta^*(\Qc)\sum_\Es\epsilon_\Es\int_{\Ec}\prod_{\nf\in\Nc^*}e^{\zeta_\nf\pi i \delta L^2\Omega_\nf t_\nf}\mathrm{d}t_\nf\prod\dirac(t_{\ff^p}-s){\prod_{\ff\neq\lf\in\Lc^*}^{(+)}n_{\mathrm{in}}(k_\lf)}\mathbf{1}_{k_\ff=\ell},
 \end{equation} where $\Es$ is a $k$-decoration of $\Qc$, the other objects are associated with the couple $\Qc$, and the set $\Ec$ is defined as in (\ref{timedom}) but with $s$ replaced by $t$; {in the last product we assume $\lf$ has sign $+$ and is not one of the two flowers $\ff$ of the flower couple $\Qc$.}
 
 The differences between (\ref{defaltkq}) and (\ref{defkq}) are the (two) Dirac factors $\dirac(t_{\ff^p}-s)$, where $\ff^p$ is the parent of $\ff$ for both flowers $\ff$, and the (one) factor $\mathbf{1}_{k_\ff=\ell}$.
 \end{prop}
 \begin{proof} Note that
 \[\Ls b=\sum_{n(\Tc_1),n(\Tc_2)\leq N}\big\{\Ic\Cc_+(b,\Jc_{\Tc_1},\Jc_{\Tc_2})+\Ic\Cc_+(\Jc_{\Tc_1},\overline{b},\Jc_{\Tc_2})+\Ic\Cc_+(\Jc_{\Tc_1},\Jc_{\Tc_2},b)\big\},\] where the signs of the trees $\Tc_j$ are determined by the positions they appear ($+$ for the first and third inputs of $\Cc_+$ and $-$ otherwise). This corresponds to attaching two sub-trees $\Tc_{1,2}$ to a single node. Calculating $\Ls^n$ corresponds to repeating this $n$ times (obtaining an admissible flower tree $\Tc$ of height $n$), and the linear {(or conjugate linear)} part of $\Ls^n$ corresponds to the flower of $\Tc$ having the same (or opposite) sign as the root. Taking into account also the time integrations, we get
 \[(\Ls^n)_{k\ell}^\zeta(t,s)=\sum_m(\Ls^n)_{k\ell}^{m,\zeta}(t,s):=\sum_m\sum_{\Tc}\widetilde{\Jc}_\Tc(t,s,k,\ell),\] where the inner sum in the last expression is taken over all admissible flower trees $\Tc$ of height $n$ and scale $m$ such that the $\zeta_\rf=+$ and $\zeta_\ff=\zeta$, and
  \begin{equation}\label{defkerLn}\widetilde{\Jc}_\Tc(t,s,k,\ell)=\bigg(\frac{\delta}{2L^{d-1}}\bigg)^m\prod_{\nf\in\Nc}(i\zeta_\nf)\sum_\Ds\epsilon_\Ds\int_{\Dc}\prod_{\nf\in\Nc}e^{\zeta_\nf\pi i \delta L^2\Omega_\nf t_\nf}\mathrm{d}t_\nf\cdot\dirac(t_{\ff^p}-s)\prod_{\ff\neq\lf\in\Lc}\sqrt{n_{\mathrm{in}}(k_\lf)}\eta_{k_{\lf}}^{\zeta_{\lf}}(\omega)\mathbf{1}_{k_\ff=\ell},
 \end{equation} where $\Ds$ is a $k$-decoration of $\Tc$, $\Dc$ is defined as in (\ref{timedom0}), and the other objects are associated with the tree $\Tc$. Note also that if $\Tc$ is admissible and has height $n\leq N$ and scale $m$, then $n\leq m\leq (1+2N)n\leq N^3$. Then, by repeating the arguments in Section \ref{correlcal} using Lemma \ref{isserlis0}, we can deduce (\ref{defaltkq}).
 \end{proof}
 \begin{proof}[Proof of Proposition \ref{mainprop2}] We only need to control the right hand side of (\ref{linkerexpansion}). We will basically repeat the arguments in Sections \ref{regasymp}--\ref{l1coef}. The main points worth noticing are the followings.
 
 (1) If $\Qc$ is an admissible flower couple, and $\widetilde{\Qc}$ is congruent to $\Qc$ in the sense of Definition \ref{conggen}, then $\widetilde{\Qc}$ is also an admissible flower couple, if we choose its flower to be the image of the flower of $\Qc$, and has the same height and scale as $\Qc$. This will enable us to decompose the right hand side of (\ref{linkerexpansion}) into sums of form (\ref{congsum}) where $\Fs$ is a congruence class of marked flower couples (which are defined similar to Definition \ref{conggen}), which then allows for the cancellation exploited in Section \ref{irchaincancel}.
 
To prove the above claim, notice that the branching nodes in any irregular chain in $\Qc_{sk}$ are also branching nodes in $\Qc$, and this chain can be divided into two chains such that all branching nodes in the first one belong to the stem of a tree in $\Qc$, and all branching nodes in the second one contained in one of the $\Tc_j$ sub-trees that are attached in the process described in Definition \ref{flowers}.
 
 We may treat these two chains separately; at the joint of the two chains we may leave out at most $5$ nodes, but this will be acceptable similar to Section \ref{irchaincancel}. Similarly we may assume that the first chain avoids the flower, by shortening it if necessary. For the second chain all branching nodes are contained in some $\Tc_j$, so modifying it in the sense of Definition \ref{conggen} has effect only within $\Tc_j$ (and it does not affect any pairings between $\Tc_j$ and any other $\Tc_{j'}$), and does not change the scale of $\Tc_j$. For the first chain all branching nodes belong to the stem, so modifying it may result in some $\Tc_j$ being replaced by its conjugate, or being permuted with some other $\Tc_j$, see Figure \ref{fig:conggen}. Note that the nodes in the chain may not be consecutive nodes on the stem, but the part of stem between them can be obtained by including a unique path within each regular tree (represented by a black box in Figure \ref{fig:conggen}). In either case, this does not change the height or scale of $\Qc$, nor the fact that $n(\Tc_j)\leq N$ for each $j$. Therefore, the couple $\widetilde{\Qc}$ is also admissible and has the same height and scale as $\Qc$.
 
 (2) In (\ref{defaltkq}) we have the factor $\mathbf{1}_{k_\ff=\ell}$ instead of $n_{\mathrm{in}}(k_\ff)$. First notice that ${k-\zeta\ell}$ is a linear combination of the $k_\lf$ for $\ff\neq\lf\in\Lc^*$, so the decay factor $\langle {k-\zeta\ell}\rangle^{-20d}$ in (\ref{mainest2}) can be obtained from the $n_{\mathrm{in}}(k_\lf)$ factors. Moreover, since $k_\ff\in \Zb_L^d$, we can replace $\mathbf{1}_{k_\ff=\ell}$ by $\psi(L(k_\ff-\ell))$ for some suitable cutoff function $\psi$. Using this function in place of $n_{\mathrm{in}}(k_\ff)$, we can repeat all the previous arguments, with at most a $L^{3d}$ loss. For example, in Propositions \ref{asymptotics1} and \ref{varregtree} we are relying on Proposition \ref{approxnt}, which only requires the norm in (\ref{propertyw1}). The norm of $W$ is bounded by the same norm of the tensor product function $\prod_{\lf}n_{\mathrm{in}}(k_\lf)$, as $W$ is obtained from the latter by a linear change of variables; if one factor in this tensor product is replaced by $\psi(L(k_\ff-\ell))$, then its norm gets multiplied by a constant power of $L$. Therefore, all the proofs will be the same, except for a possible loss of at most $L^{3d}$.
 
 (3) In (\ref{defaltkq}) we have the Dirac factors $\dirac(t_{\ff^p}-s)$. This means that in the integral in (\ref{defaltkq}) we are omitting the integration in $t_{\ff^p}$ for both flowers $\ff$. However, this difference will cause at most another $L^{20d}$ loss. This is intuitively clear as only one node (and one time variable) is affected, and we can demonstrate it as follows.
 
 Recall the sequence of reductions in Sections \ref{regasymp}--\ref{l1coef}, where we remove from the couple $\Qc$ successively (i) the regular couples and regular trees, then (ii) the irregular chains, then (iii) the nodes corresponding to atoms in type II chains of the base molecule. In both steps (ii) and (iii) we can choose to avoid the two specific nodes $\ff^p$, so for each flower $\ff$, we only need to consider the case where (a) $\ff^p$ belongs to a regular couple or a regular chain in step (i), or (b) $\ff^p$ belongs to the rest of $\Qc$ after performing steps (i)--(iii). Let $\mf=\ff^p$.
 
 In case (a) we will further reduce the regular couple or regular tree using Proposition \ref{structure1}, and we may assume $\ff^p$ belongs to one of the regular chains in this process (we only consider \emph{Case 2} in Section \ref{recursive}; \emph{Case 1} is much easier as the expression is much simpler and we can directly calculate it). The point here is that, if we omit the integration in $t_{\mf}$, then the resulting expression, which is a function of $\alpha[\Nc^*]$ as in (\ref{defcoefb2}), satisfies the same bound as the one with $t_{\mf}$ integration, but in the \emph{weaker norm} $L_{\alpha_{\mf}}^\infty L_{\alpha[\Nc^*\backslash\{\mf\}]}^1$ instead of $L_{\alpha[\Nc^*]}^1$. To see this, consider
 \begin{equation}\label{regchainext}K(t,\alpha_1,\cdots,\alpha_m)=\int_{t>t_1>\cdots >t_{2m}>0}e^{\pi i(\beta_1t_1+\cdots +\beta_{2m}t_{2m})}\,\mathrm{d}t_1\cdots\mathrm{d}t_{2m}\end{equation} as in (\ref{regchaink}), where $\beta_a\,(1\leq a\leq 2m)$ is a permutation of $\pm\alpha_j\,(1\leq j\leq m)$ associated with a legal partition, as in Section \ref{regchainest}; for simplicity we have omitted the $\lambda_a$ variables. By the arguments in Section \ref{regchainest}, we can bound the $L_{\alpha_1,\cdots,\alpha_m}^1$ norm of $K$ (or we may extract explicit $\frac{1}{\alpha_j}$ factors from $K$ and bound the $L^1$ norm in the other $\alpha_j$ variables, see Lemma \ref{regchainlem8}; for simplicity we will omit this case). Now, suppose we insert $\dirac(t_{2m}-s)$ in (\ref{regchainext}) (note that, since a child of $\mf$ is a leaf that is paired with a leaf in the other tree, we must have $t_\mf=t_{2m}$ in the regular chain integration), then we will lose integrability in $\beta_{2m}$; however if we \emph{fix} the value of $\beta_{2m}$ then we get the expression
 \begin{equation}\label{oddexpression}e^{\pi i\beta_{2m}s}\int_{t>\cdots >t_{2m-1}>s}e^{\pi i(\beta_1t_1+\cdots +\beta_{2m-1}t_{2m-1})}\mathrm{d}t_1\cdots\mathrm{d}t_{2m-1}\end{equation} (note that there is some $a$ such that $\beta_a=-\beta_{2m}$ is also fixed). This has basically the same form as (\ref{regchainext}), except for a harmless class $J$ operator corresponding to integration in $t_a$, so we can repeat the proof in Section \ref{regchainest}, using the notions of class $J$ and $R$ operators, to obtain the same bound for this expression in the $L^1$ norm in the remaining $\alpha_j$ (i.e. excluding $\beta_a$ and $\beta_{2m}$) variables. This then implies the bound for our expression, for fixed $s$, in the $L^\infty L^1$ type weaker norm as desired.
 
 In case (b), the same argument applies, except that we replaced the $L^1$ norm by the variant in (\ref{uniformL1}). The proof is in fact easier, as the bound (\ref{uniformL1}), after removing the type II molecular chains, follows solely from the denominators $\langle q_\nf\rangle$ occurring in (\ref{timeintlemma1}) in Lemma \ref{timeintlemma}. If we omit the integration in $t_{\mf}$, then we are  at loss of only \emph{one} denominator involving $\alpha_{\mf}$, which does not affect the presence of all the other denominators. Thus, if $\alpha_{\mf}$ is fixed, the function can be bounded in the remaining variables in the norm in (\ref{uniformL1}), using the same arguments in Section \ref{l1coef}.
 
 In either case, in the end we can obtain the same bound for the modified time integral, but in weaker norms without integrability in \emph{at most two of the $\alpha_\nf$ variables}. But this bound can easily be transformed to the $L^1$ type bound involving all variables, with at most $L^{10d}$ loss, because each $\alpha_\mf$ will be replaced by $\delta L^2\Omega_\mf$ in the actual $\Kc_\Qc$ expression, which belongs to the union of at most $L^{5d}$ fixed unit intervals (at least if we restrict $|k_\lf|\leq L$ for each $\lf$; if $\max_\lf|k_\lf|:=M'\geq L$ then we may lose $(M')^{10d}$ but this will be covered by the $(M')^{100d}$ gain coming from $n_{\mathrm{in}}$). Therefore $L_{\alpha_{\mf}}^\infty$ bounds imply the corresponding $L_{\alpha_{\mf}}^1$ bounds with a loss of at most $L^{5d}$, once we insert the suitable cutoff functions adapted to these unit intervals.
 
 \smallskip
In view of the arguments (1)--(3) above, the bound for the right hand side of (\ref{linkerexpansion}) can be obtained, using the same arguments as in Sections \ref{regasymp}--\ref{l1coef}. This proves Proposition \ref{mainprop2}.
 \end{proof}
 \begin{cor}\label{extragaincor} Fix $M_0\geq L$ and $M_1\geq L^{(100d)^3}$. In (\ref{defaltkq}), suppose $|k|\geq M_0^2$, and we insert suitable cutoff functions supported in $|k_\lf|\leq M_0$ for each $\ff\neq\lf\in\Lc^*$; moreover, suppose we insert one (or more) cutoff function supported in $|\Omega_{\nf_j}|\geq M_1$ for some $1\leq j\leq n-1$, where $\nf_j$ is the $j$-th node in the stem from top to bottom (in particular $\nf_n=\ff^p$), then the resulting expression satisfies the same bound as (\ref{defaltkq}), but with an additional decay factor $M_1^{-1/9}M_0^{5d}$. The same holds for the right hand side of (\ref{linkerexpansion}).
 \end{cor}
 \begin{proof} Note that the assumption implies that for any irregular chain in $\Qc_{sk}$ with branching nodes on the stem, the gap $h$ (see Proposition \ref{congdec}) must satisfy $|h|\geq M_0^2/4$(since $|k_\nf|\geq M_0^2/2$ for any node $\nf$ on the stem, and $|k_\nf|\leq N^3M_0$ for any node $\nf$ off the stem); in particular we are in the large gap case (Section \ref{lgcase}) and thus do not need the cancellation coming from congruence couples obtained by altering this irregular chain. Thus, in carrying out the arguments in previous sections we only need to sum over $\Qc\in\Fs'$ where $\Fs'$ (unlike $\Fs$) is a subset in a fixed congruence class, formed by altering irregular chains that are completely contained in some $\Tc_j$. Since altering these chains do not affect the structure of the stem or any $\Omega_{\nf_j}$ factor, we can bound the resulting expression in the same way as (\ref{defaltkq}).
 
 To gain the extra decay in $M_1$ using the largeness of $|\Omega_{\nf_j}|$, like in the proof of Proposition \ref{L^nexpansion}, we may assume $\nf_j$ belongs to either (a) a regular couple or regular tree, or (b) the rest of the couple after removing all the special structures. In case (a), if $\nf_j$ is not paired (as a branching node) to $\nf_n=\ff^p$, we can use (\ref{maincoef2.5}) or the denominators $\alpha_\nf$ in (\ref{maincoef1}) to gain a power of $M_1$; if $\nf_j$ is paired to $\ff^p$, then we will consider \emph{Case 1} and \emph{Case 2} (in the sense of Section \ref{recursive}) separately. By direct calculation in \emph{Case 1} and examining (\ref{oddexpression}) similar to Section \ref{regchainest} in \emph{Case 2}, we can also gain a power $M_1^{-1/9}$ at a loss of at most $M_0^{5d}$, in either situation. In case (b), the decay comes from the denominators $q_\nf$ in (\ref{timeintlemma1}). If $|\Omega_{\nf_j}|\geq M_1$ then one of these denominators, say $\langle\widetilde{q}\rangle$, will be $\gtrsim M_1$; we then sum over this $\widetilde{q}$ to get
 \[\sum_{\widetilde{q}}\frac{1}{\langle \widetilde{q}\rangle}\lesssim M_1^{-1}M_0^{5d}\] since $\widetilde{q}$ belongs to a set of cardinality at most $M_0^{5d}$ as the decoration varies. This provides the needed decay, and the rest of the sum can be estimated as in Proposition \ref{section8main}.
 \end{proof}
\section{The endgame}\label{endgame} In this section we prove Theorem \ref{main}. We will do this in a few steps. Recall that $A\geq 40d$ is fixed in Section \ref{notations}, as is the even integer $p\gg_{A,\nu}1$ and $\delta\ll_{p,C^+}1$.
\begin{prop}\label{finprop1} With probability $\geq 1-L^{-A}$, we have
\begin{equation}\label{largedevest}|(\Jc_n)_k(t)|\lesssim\langle k\rangle^{-9d}(p^{2}C^+\sqrt{\delta})^{n/2}L,\quad |\Rc_k(t)|\lesssim \langle k\rangle^{-9d}(p^{2}C^+\sqrt{\delta})^{N/2}L
\end{equation} for any $k\in\Zb_L^d$, $t\in[0,1]$, and all $0\leq n\leq N^3$, where $\Rc$ is defined in (\ref{eqnbk1.5}). Note in particular that the right hand side of the second inequality in (\ref{largedevest}) is bounded, due to our choice $N=\lfloor\log L\rfloor$, by
\begin{equation}\label{negpower}(p^{2}C^+\sqrt{\delta})^{N/2}L\leq \delta ^{N/8}L\leq L^{-(100d)^3}.\end{equation}
\end{prop}
\begin{proof} First consider $\Jc_n$. For fixed $k$ and fixed $t$, $(\Jc_n)_k(t)$ is a random variable of form (\ref{multigauss}), so using Lemma \ref{largedev} and Proposition \ref{mainprop1} we get
\[\Eb|\langle k\rangle^{10d}(\Jc_n)_k(t)|^p\lesssim p^{np}(C^+\sqrt{\delta})^{np/2}.\] This being uniform in $t$, we can integrate in $t$ and sum in $k$ to obtain that
\begin{equation}\label{largedev1}\Eb\|\langle k\rangle^{9d}(\Jc_n)_k(t)\|_{L_{t,k}^p([0,1]\times\Zb_L^d)}^p\lesssim p^{np}(C^+\sqrt{\delta})^{np/2},\end{equation} where $L_k^p$ is taken with respect to $L^{-d}$ times the counting measure in $k$. Moreover we also have
\begin{equation}\label{timederiv}\Eb|\partial_t(\Jc_n)_k(t)|^2\lesssim \langle k\rangle^{-20d}(C^+\sqrt{\delta})^n L^{40d},
\end{equation} which can be proved using the arguments in Section \ref{operatornorm}, as taking $\partial_t$ derivative just corresponds to omitting the $t_\rf$ integration and producing something like (\ref{defaltkq}). This then implies that
\begin{equation}\label{largedev2}\Eb\|\langle k\rangle^{9d}\partial_t(\Jc_n)_k(t)\|_{L_{t,k}^p([0,1]\times\Zb_L^d)}^p\lesssim p^{np}(C^+\sqrt{\delta})^{np/2}L^{40dp}.\end{equation} By { using Gagliardo-Nirenberg for $t\in[0,1]$, and bounding the $L_k^\infty$ norm by the $L_k^p$ norm for $k\in \Zb_L^d$ with an extra loss $L^{d/p}$}, we conclude that
\begin{equation}\label{largedev3}\Eb\|\langle k\rangle^{9d}(\Jc_n)_k(t)\|_{L_{t,k}^\infty([0,1]\times\Zb_L^d)}^p\lesssim L^dp^{np}(C^+\sqrt{\delta})^{np/2}L^{40d},
\end{equation} thus with probability $\geq 1-L^{-p/2}$, we have
\[\sup_{t,k}|\langle k\rangle^{9d}(\Jc_n)_k(t)|\lesssim p^{n}(C^+\sqrt{\delta})^{n/2}L^{41d/p+1/2},\] which implies (\ref{largedevest}). The estimate for $\Rc$ is the same, with $n$ replaced by $N$; we just need to notice that $\Rc_k(t)$ equals the sum of $(\Jc_{\Tc^+})_k(t)$ over all trees $\Tc^+$ of scale $>N$ such that its three sub-trees all have scale $\leq N$ (in particular the scale of $\Tc^+$ is between $N$ and $3N$). This property, as well as the similar property for couples, is again invariant under congruence relations, so the same arguments in the previous sections apply.
\end{proof}
\begin{prop}\label{finprop2} With probability $\geq 1-L^{-A}$, we have
\begin{equation}\label{operatordev}\|\Ls^n\|_{Z\to Z}\lesssim (p^{2}C^+\sqrt{\delta})^{n/2}L^{60d}
\end{equation} for all $0\leq n\leq N$.
\end{prop}
\begin{proof} We only need to show, with probability $\geq 1-L^{-A}$, that
\begin{equation}\label{supink}\sup_{k,\ell}\sup_{0\leq s<t\leq 1}\langle {k-\zeta\ell}\rangle^{9d}|(\Ls^n)_{k\ell}^{m,\zeta}(t,s)|\lesssim(p^{2}C^+\sqrt{\delta})^{m/2}L^{55d}
\end{equation} for any $\zeta\in\{\pm\}$ and $n\leq m\leq N^3$. The supremum in $(t,s)$ can be treated similar to the proof of Proposition \ref{finprop1}, so the main point here is to address the supremum in $(k,\ell)$. This formally has infinitely many possibilities, but we will use Lemma \ref{finitelem}, which is a variant of Claim 3.7 in \cite{DH}, to reduce it to finitely many possibilities. {We may assume $\zeta$ equals $+$ below, as the other case is the same.}

We start by making the decomposition
\begin{equation}\label{Rdecompose}1=\sum_{R\geq L}\chi_R(k_\lf),\end{equation} where $\chi_R(z)$ is supported in $|z|\lesssim R$ if $R=L$ and in $|z|\sim R$ if $R>L$, for each $\ff\neq\lf\in\Lc^*$. Note that, as in the proof of Corollary \ref{extragaincor}, we also fix one particular stem structure of the tree $\Tc$ in (\ref{defkerLn}). Let the maximum of these $R$ for all the $k_{\lf}$ be $M_0$. Below we may assume $M_0=L$, since even if $M_0>L$, we will lose at most $M_0^{20d}$ in all subsequent arguments (see for example Corollary \ref{extragaincor}), which can be covered by the $M_0^{-200d}$ gain from the $n_{\mathrm{in}}(k_\lf)$ factor, and clearly the summation over $R$ is not a problem.

Now assume $M_0=L$. If $|k|\leq L^2$, then the number of possibilities for $(k,\ell)$ is at most $L^{8d}$. We can replace the $L_{k,\ell}^\infty$ norm in (\ref{supink}) by $L_{k,\ell}^p$ and argue as in Proposition \ref{finprop1}, applying Proposition \ref{mainprop2} and the corresponding bound for $t$ and $s$ derivatives (which can be obtained similarly as in Section \ref{operatornorm}), to get (\ref{supink}).

Suppose now $|k|\geq L^2$, then we will further make the decomposition (\ref{Rdecompose}) for the variables $\Omega_{\nf_j}\,(1\leq j\leq n-1)$ defined in Corollary \ref{extragaincor}, but with $L$ replaced by $L^{(100d)^3}$. Let the maximum of these $R$ for all the $\Omega_{\nf_j}$ be $M$. For fixed $M$, let the corresponding contribution to $(\Ls^n)_{k\ell}^{m,\zeta}$ be $(\Ls^n)_{M,k\ell}^{m,\zeta}$, then it suffices to show that
\begin{equation}\label{supink2}\mathbb{E}\big|\sup_{k,\ell}\sup_{0\leq s<t\leq 1}\langle k-\ell\rangle^{9d}|(\Ls^n)_{M,k\ell}^{m,\zeta}(t,s)|\big|^p\lesssim p^{mp}(C^+\sqrt{\delta})^{mp/2}L^{50dp}M^{-p/20}
\end{equation} for $M>L^{(100d)^3}$, and the same bound without $M^{-p/20}$ for $M=L^{(100d)^3}$.

Suppose $M$ is fixed. Clearly $|k-\ell|\leq L^2$, so we may also fix the value of $k-\ell$ at a loss of $L^{5d}$. In the formula (\ref{defkerLn}) of the terms in $\Ls^n$, we will fix the values of $k_{\ff'}$ and $k_{\ff''}$ in the decoration $\Ds$, where $\ff'$ and $\ff''$ are the two siblings of $\ff$. Clearly each of them has at most $L^{5d}$ choices, so fixing them again introduces a factor of at most $L^{15d}$.

Recall that in (\ref{defkerLn}), the whole expression depends on $\Omega_{\ff^p}$ \emph{only} through the integral
\[\int e^{\zeta_{\ff^p}\pi i\delta L^2\Omega_{\ff^p}t_{\ff^p}}\dirac(t_{\ff^p}-s)\,\mathrm{d}t_{\ff^p}=e^{\zeta_{\ff^p}\pi i\delta L^2\Omega_{\ff^p}s}.\]Moreover, once $k_{\ff'}$ and $k_{\ff''}$ are fixed, and $k_{\ff}=\ell$, then $\Omega_{\ff^p}$ is determined by $\ell$ and no longer depends on the other parts of the decoration. Therefore one can extract the factor $e^{\zeta_{\ff^p}\pi i\delta L^2\Omega_{\ff^p}s}$ and write the contribution currently under consideration in the form of
\begin{equation}\label{reduction1}e^{\zeta_{\ff^p}\pi i\delta L^2\Omega_{\ff^p}s}\cdot\bigg(\frac{\delta}{2L^{d-1}}\bigg)^m\prod_{\nf\in\Nc}(i\zeta_\nf)\sum_\Ds\epsilon_\Ds\Ac^{**}(t,s,\delta L^2\cdot\Omega[\Nc\backslash\{\ff^p\}])\prod_{\ff\neq\lf\in\Lc}\sqrt{n_{\mathrm{in}}(k_\lf)}\eta_{k_{\lf}}^{\zeta_{\lf}}(\omega)\mathbf{1}_{k_\ff=\ell},\end{equation} where $\Ac^{**}$ is a function of the \emph{remaining} variables $\Omega[\Nc\backslash\{\ff^p\}]$ and does not depend on $\Omega_{\ff^p}$. Moreover, when $\Omega_{\ff^p}$ is fixed, the function $\Ac^{**}$ (or more precisely the functions generated by $\Ac^{**}$ that occur in the proofs in the previous sections) satisfies the bounds described in the proof of Proposition \ref{mainprop2} in Section \ref{operatornorm}, hence one can get the same square moment estimate as in Proposition \ref{mainprop2}. Moreover if $M>L^{(100d)^3}$, then using the same arguments as in the proof of Corollary \ref{extragaincor}, we can gain an extra $M^{-1/9}$ compared to Proposition \ref{mainprop2}, with at most $L^{5d}$ loss.

Now, after removing the unimodular factor $e^{\zeta_{\ff^p}\pi i\delta L^2\Omega_{\ff^p}s}$ in (\ref{reduction1}), we notice that the rest of (\ref{reduction1}) depends on $k$ \emph{only} through the resonance factors $\Omega_{\nf_j}$ for $1\leq j\leq n-1$. Moreover we have assumed that $|\Omega_{\nf_j}|\leq M$ for each such $j$. If any $\nf_j$ and $\nf_{j+1}$ have opposite signs, then by definition and $|\Omega_{\nf_j}|\leq M$, we easily see that $|k|\lesssim M$, so we can replace the $L_{k,\ell}^\infty$ norm by $L_{k,\ell}^p$ and close as before, where the loss is at most $M^{C/p}$ and can either be absorbed by the gain $M^{-1/9}$ coming from Corollary \ref{extragaincor} if $M>L^{(100d)^3}$, or neglected if $M=L^{(100d)^3}$. We may thus assume all $\nf_j$ have the same sign, and in this case we have
\[\Omega_{\nf_j}=|k+a_j|_\beta^2-|k+b_j|_\beta^2+\Omega_j'=\langle k,c_j\rangle_\beta+\Omega_j''\] where $(a_j,b_j,c_j)$ are vectors, and $(\Omega_j',\Omega_j'')$ are expressions, that do not depend on $k$ (hence they are bounded by $L^3$).

We may then apply Lemma \ref{finitelem} and assume $k$ is represented by a system $(r,q,v_1,\cdots,v_q,f,y)$. Then, if $\chi_{M}(z)$ is a cutoff function supported in $|z|\leq M$, and $F$ is an arbitrary function, we have
\begin{equation}\label{finitered}\prod_{j=1}^{n-1}\chi_{M}(\Omega_{\nf_j})\cdot F(\Omega_{\nf_1},\cdots,\Omega_{\nf_{n-1}})=\prod_{j=1}^{n-1}\mathbf{1}_S(c_j)\chi_{M}(G(f,y,c_j)+\Omega_j'')\cdot F((G(f,y,c_j)+\Omega_j'')_{1\leq j\leq n-1}),\end{equation} where $S$ is the set of $z$ whose first $r$ coordinates form a linear combination of $\{v_1,\cdots,v_q\}$ and $G$ is some function. The point here is that, the right hand side of (\ref{finitered}) does \emph{not} explicitly involve $k$, but only depends on the variables $(v_1,\cdots,v_q,f,y)$, which are bounded in size by $M^C$ for a constant $C$ depending only on $d$. This is in contrast with $k$, which has no a priori upper bound and may have infinitely many choices.

Therefore, for any $k$, the function $\Ac^{**}(t,s,\delta L^2\cdot\Omega[\Nc\backslash\{\ff^p\}])$, viewed as a function of $(t,s,k[\Tc\backslash\{\rf,{\ff}\}])$, equals $\Gc^{**}(t,s,v_1,\cdots,v_q,f,y,k[\Tc\backslash\{\rf,\ff\}])$ for some function $\Gc^{**}$. This means that
\begin{equation}\label{reduct2}\sup_k\big|(\cdots[\Ac^{**}]\cdots)\big|=\sup_{(v_1,\cdots,v_q,f,y)}\big|(\cdots [\Gc^{**}]\cdots)\big|,\end{equation} where the parenthesis $(\cdots)$ in (\ref{reduct2}) represents (\ref{reduction1}) without the $e^{\zeta_{\ff^p}\pi i\delta L^2\Omega_{\ff^p}s}$ factor. Now, with the finite volume that $(v_1,\cdots,v_q,f,y)$ and $(t,s)$ occupy, we can bound
\begin{equation}\label{neweqn0}\big\|(\cdots [\Gc^{**}]\cdots)\big\|_{L_{t,s}^\infty L_{v_1,\cdots,v_q,f,y}^\infty}\lesssim\big\|(\cdots [\Gc^{**}]\cdots)\big\|_{L_{t,s}^pL_{v_1,\cdots,v_q,f,y}^p}^{1-C/p}\big\|(\partial_{t,s,y}\cdots [\Gc^{**}]\cdots)\big\|_{L_{t,s}^pL_{v_1,\cdots,v_q,f,y}^p}^{C/p}.\end{equation} Then, we take $p$-th power moments and argue as in the proof of Proposition \ref{finprop1}. Note that by Lemma \ref{finitelem}, for \emph{any} system $(r,q,v_1,\cdots,v_q,f,y)$, the $\Gc^{**}$ function is the \emph{limit} of the $\Ac^{**}$ functions for some sequence of $k$; the same limit holds for the $(\partial_t,\partial_s)$ derivative, and $\partial_y\Gc^{**}$ is the limit of $\partial_{\Omega_{\nf_j}}\Ac^{**}$ multiplied by some harmless factors (these factors are bounded by $M^C$ which is diminished by the $C/p$ power in (\ref{neweqn0}); also any derivative of $\Ac^{**}$ can be treated in the same way as $\Ac^{**}$ itself), so the $p$-th power moments for fixed $(v_1,\cdots,v_q,f,y)$ can be estimated as in Proposition \ref{mainprop2}, with the extra decay in $M$ for the term without derivatives if $M>L^{(100d)^3}$.

In summary we get\[\mathbb{E}\big|\sup_{k,\ell}\sup_{0\leq s<t\leq 1}\langle k-\ell\rangle^{9d}|(\Ls^n)_{M,k\ell}^{m,\zeta}(t,s)|\big|^p\lesssim p^{mp}(C^+\sqrt{\delta})^{mp/2}L^{45dp}M^{C-p/18},\] which clearly implies (\ref{supink2}), if $M>L^{(100d)^3}$. Here notice that we gain the power $M^{-p/18}$ from the $p$-th moment of the $(\cdots [\Gc^{**}]\cdots)$ term without derivatives, thanks to Corollary \ref{extragaincor}, and that all the losses caused by the summation or integration in $(r,q,v_1,\cdots,v_q,f,y)$, or by the $(\partial_t,\partial_s,\partial_y)$ derivatives, are at most $M^C$. If $M=L^{(100d)^3}$ then we do not have the gain $M^{-p/18}$, but the losses are still at most $M^C\leq L^p$ which is also acceptable. The proof is now complete.
\end{proof}
\begin{prop}\label{finprop3} With probability $\geq 1-L^{-A}$, the mapping defined by the right hand side of (\ref{eqnbk2}) is a contraction mapping from the set $\{b:\|b\|_Z\leq L^{-500d}\}$ to itself.
\end{prop}
\begin{proof} Suppose we exclude the exceptional set of probability $L^{-A}$ in Propositions \ref{finprop1}--\ref{finprop2}. Consider the mapping
\[b\mapsto (1-\Ls)^{-1}(\Rc+\Bs(b,b)+\Cs(b,b,b)),\] as usual, we just need to prove it maps the given set to itself, and the contraction property will follow similarly. Suppose $\|b\|_Z\leq L^{-500d}$, note that
\[(1-\Ls)^{-1}=(1-\Ls^N)^{-1}(1+\Ls+\cdots +\Ls^{N-1})\] maps $Z$ to $Z$, where $(1-\Ls^N)^{-1}$ can be constructed by Neumann series; using (\ref{operatordev}) we get that $\|(1-\Ls)^{-1}\|_{Z\to Z}\lesssim L^{62d}$. Therefore, it suffices to show that
\[\|\Rc\|_Z+\|\Bs(b,b)\|_Z+\|\Cs(b,b,b)\|_Z\lesssim L^{-600d}.\] The bound for $\Rc$ follows from (\ref{largedevest}) and (\ref{negpower}), so we only need to consider $\Bs$ and $\Cs$. But this is again easy, using the loose estimate
\[\|\Ic\Cc_+(u,v,w)\|_Z\lesssim \|\Cc_+(u,v,w)\|_Z\lesssim L^{20d}\|u\|_Z\|v\|_Z\|w\|_Z,\] together with (\ref{largedevest}) and the assumption $\|b\|_Z\leq L^{-500d}$. This completes the proof.
\end{proof}
\begin{proof}[Proof of Theorem \ref{main}] By Propositions \ref{finprop1}--\ref{finprop3}, with probability $\geq 1-L^{-A}$, the solution $a=a_k(t)$ to (\ref{akeqn})--(\ref{akeqn2}) can be written as the ansatz (\ref{ansatz}) for $t\in[0,1]$, where each $\Jc_n$ satisfies (\ref{largedevest}), and $b$ is constructed by contraction mapping and satisfies $\|b\|_Z\leq L^{-500d}$. Denote this event by $E$, so that $\mathbb{P}(E)\geq 1-L^{-A}$. Let $E_1\supset E$ be the event that (\ref{nls}) has a smooth solution on $[0,\delta\cdot T_{\mathrm{kin}}]$.

For each $\tau\in[0,\delta]$ we will calculate, with $\widehat{u}$ as in (\ref{fourier}), that
\[\Eb(|\widehat{u}(\tau\cdot T_{\mathrm{kin}},k)|^2\mathbf{1}_{E_1})=\Eb(|a_k(\delta^{-1}\tau)|^2\mathbf{1}_{E_1}).\] If we replace $\mathbf{1}_{E_1}$ by $\mathbf{1}_{E_1\backslash E}$, then the resulting contribution is bounded by $L^{-A+10d}$, since $|a_k(t)|^2$ is bounded uniformly in $k$ and $t$ by mass conservation, so we may replace $\mathbf{1}_{E_1}$ by $\mathbf{1}_E$. This then reduces to the expression
\[\sum_{0\leq n_1,n_2\leq N}\Eb((\Jc_{n_1})_k(t)\overline{(\Jc_{n_2})_k(t)}\mathbf{1}_E)+2\sum_{n=0}^N\mathrm{Re}\,\Eb((\Jc_{n})_k(t)\overline{b_k(t)}\mathbf{1}_E)+\Eb(|b_k(t)|^2\mathbf{1}_E),\] where $t=\delta^{-1}\tau$. The terms involving $b$ are obviously bounded by $L^{-100d}$ using (\ref{largedevest}) and $\|b\|_Z\leq L^{-500d}$, so we just need to consider the correlations between $\Jc_{n_1}$ and $\Jc_{n_2}$. In these correlations, if we replace $\mathbf{1}_E$ by $\mathbf{1}_{E^c}$, then the resulting contribution is
\[\big|\Eb((\Jc_{n_1})_k(t)\overline{(\Jc_{n_2})_k(t)}\mathbf{1}_{E^c})\big|\leq(\Eb|(\Jc_{n_1})_k(t)|^4)^{1/4}(\Eb|(\Jc_{n_2})_k(t)|^4)^{1/4}(\Pb(E^c))^{1/2}\lesssim L^{-A/2+10d}\] using Lemma \ref{largedev} and Proposition \ref{mainprop1}, so we may replace $\mathbf{1}_E$ by $1$, thus reducing to the expression
\[\sum _{0\leq n_1,n_2\leq N}\Eb((\Jc_{n_1})_k(t)\overline{(\Jc_{n_2})_k(t)})=\sum_\Qc\Kc_\Qc(t,t,k),\] where the last sum is taken over all couples $\Qc=(\Tc^+,\Tc^-)$ with $n(\Tc^\pm)\leq N$. We may replace this condition by the condition $n(\Qc)\leq 2N$, because each term $\Kc_\Qc$ in the difference must satisfy $N\leq n(\Qc)\leq 2N$, and the set of these $\Qc$ is invariant under congruence, so we can bound the difference by $(C^+\sqrt{\delta})^{N/2}\leq L^{-100d}$. This reduces our target, up to errors $O( L^{-10d})$, to
\[\sum_{n=0}^{2N}\sum_{n(\Qc)=n}\Kc_\Qc(t,t,k)=\sum_{n=0}^N\Mc_n(t,k)+O(L^{-\nu})=n(\delta t,k)+O(L^{-\nu})=n(\tau,k)+O(L^{-\nu}),\] where the last steps are due to Propositions \ref{mainprop4} and \ref{wkelwp}. In the end we get that
\[\Eb(|\widehat{u}(\tau\cdot T_{\mathrm{kin}},k)|^2\mathbf{1}_{E_1})=n(\tau,k)+O(L^{-\nu}),\] uniformly in $\tau\in[0,\delta]$ and $k\in \Zb_L^d$. This proves Theorem \ref{main}.
\end{proof}
\appendix
\section{Preliminary lemmas}
\subsection{The exceptional set $\Zf$} We will define in Lemma \ref{genericity} the Lebesgue null set $\Zf$ used in Theorem \ref{main}. Once Lemma \ref{genericity} is proved, for the rest of the paper we will fix one $\beta\in(\Rb^+)^d\backslash \Zf$.
\begin{lem}[The genericity condition]\label{genericity} There exists a Lebesgue null set $\Zf\subset(\Rb^+)^d$ such that the followings hold for any $\beta=(\beta^1,\cdots,\beta^d)\in(\Rb^+)^d\backslash \Zf$.

(1) For any integers $(K^1,K^2)\neq (0,0)$, we have
\begin{equation}\label{generic1}|\beta^1K^1+\beta^2K^2|\gtrsim(1+|K^1|+|K^2|)^{-1}\log^{-4}(2+|K^1|+|K^2|);\end{equation}

(2) The numbers $\beta^1,\cdots,\beta^d$ are algebraically independent over $\Qb$, and for any $R$ we have
\begin{equation}\label{generic1.5}\#\left\{(X,Y,Z)\in(\Zb^d)^3:|X|,|Y|,|Z|\leq R,\,X\neq 0,\,\max(|\langle X,Y\rangle_\beta|,|\langle X,Z\rangle_\beta|)\leq 1\right\}\lesssim R^{3d-4+\frac{1}{6}}.
\end{equation}
\end{lem}
\begin{proof} (1) This is standard in Diophantine approximation, {which can be proved by summing over all $(K^1,K^2)$ the measure of the set of $(\beta^1,\beta^2)$ not satisfying (\ref{generic1}) and applying Borel-Centelli.}

(2) Without loss of generality we may assume $\beta\in[1,2]^d$. If $X^jY^j=X^jZ^j=0$ for all $j$, since $X\neq 0$, the number of choices for $(X,Y,Z)$ is clearly at most $R^{2d-1}$ which satisfies (\ref{generic1.5}) since $d\geq 3$.  If $X^iX^j(Y^iZ^j-Y^jZ^i)=0$ for all $(i,j)$, but not all $X^jY^j$ and $X^jZ^j$ are zero, say $X^1Y^1\neq 0$, then for fixed $(X,Y,Z)$, the Lebesgue measure of the set
\[E:=\bigg\{\beta\in[1,2]^d:\bigg|\sum_{\ell=1}^d\beta^\ell X^\ell Y^\ell\bigg|\leq 1,\,\bigg|\sum_{\ell=1}^d\beta^\ell X^\ell Z^\ell\bigg|\leq 1\bigg\}\] is bounded by $C|X^1Y^1|^{-1}$. Moreover, once $X^1$ and $Y^1$ are fixed, the number of choices for $(X^j,Y^j,Z^j)$ for each $j\geq 2$ is at most $R^2$. This implies that
\[\int_{[1,2]^d}(\mathrm{left\ hand\ side\ of\ (\ref{generic1.5})})\,\mathrm{d}\beta\leq C\sum_{{|X|,|Y|,|Z|\leq R}}|X^1Y^1|^{-1}\leq CR^{2d-1+\frac{1}{8}},\] {where the sum in $(X,Y,Z)$ is taken under the assumption $X^iX^j(Y^iZ^j-Y^jZ^i)=0$ and $X^1Y^1\neq 0$}. By Borel-Cantelli lemma, and using that $2d-1\leq 3d-4$, we get (\ref{generic1.5}) for any $R$ and almost all $\beta$.

Now suppose there is $1\leq i<j\leq d$ such that $X^iX^j(Y^iZ^j-Y^jZ^i)\neq 0$, say $(i,j)=(1,2)$. Then for fixed $(X,Y,Z)$, the Lebesgue measure of $E$ is bounded by $C|X^1X^2|^{-1}\cdot|Y^1Z^2-Y^2Z^1|^{-1}$, therefore
\[\int_{[1,2]^d}(\mathrm{left\ hand\ side\ of\ (\ref{generic1.5})})\,\mathrm{d}\beta\leq C\sum_{{|X|,|Y|,|Z|\leq R}}|X^1X^2|^{-1}\cdot|Y^1Z^2-Y^2Z^1|^{-1},\] {where the sum in $(X,Y,Z)$ is taken under the assumption $X^1X^2(Y^1Z^2-Y^2Z^1)\neq0$.} The last sum is bounded by $R^{3d-4+\frac{1}{8}}$, by fixing the values of $X^1$, $X^2$ and $Y^1Z^2-Y^2Z^1$ and using the divisor estimate {(i.e. the number of divisors of any nonzero integer $x$ is $O_{\epsilon}(|x|^\epsilon)$ for any $\epsilon>0$)}. Again by Borel-Cantelli, we obtain (\ref{generic1.5}) for almost all $\beta$.
\end{proof}
\subsection{Miscellaneous results} We collect some auxiliary results needed in the main proof.
\begin{lem}[Complex Isserlis' theorem]\label{isserlis0} Given $k_j\in\Zb_L^d$ (not necessarily distinct) and $\zeta_j\in\{\pm\}$ for $1\leq j\leq n$, then
\begin{equation}\label{isserlis}\Eb\bigg[\prod_{j=1}^n\eta_{k_j}^{\zeta_j}(\omega)\bigg]=\sum_\Ps\prod_{\{j,j'\}\in\Ps}\mathbf{1}_{k_j=k_{j'}},\end{equation}where the summation is taken over all partitions $\Ps$ of $\{1,2,\cdots,n\}$ into two-element subsets $\{j,j'\}$ such that $\zeta_{j'}=-\zeta_j$.
\end{lem}
\begin{proof} Let all the different vectors in $\{k_j:1\leq j\leq n\}$ be $k^{(1)},\cdots, k^{(r)}$. Assume for each $1\leq i\leq r$ and $\zeta\in\{\pm\}$ that the number of $j$'s such that $k_j=k^{(i)}$ and $\zeta_j=\zeta$ is $a_i^\zeta$, then we have
\[\Eb\bigg[\prod_{j=1}^n\eta_{k_j}^{\zeta_j}(\omega)\bigg]=\Eb\bigg[\prod_{i=1}^r(g_{k^{(i)}})^{a_i^+}(\overline{g_{k^{(i)}}})^{a_i^-}\bigg]=
\left\{
\begin{split}(a_1^+)!\cdots (a_r^+)!,&&\mathrm{if\ }a_i^+=a_i^-\mathrm{\ for\ all\ }i,\\
0,&&\mathrm{otherwise}.
\end{split}
\right.\] On the other hand, for fixed $\Ps$ the product on the right hand side of (\ref{isserlis}) is either $0$ or $1$, and equals $1$ if and only if each pair $\{j,j'\}\in\Ps$ is such that $k_j=k_{j'}=k^{(i)}$ for some $i$, and $\zeta_{j'}=-\zeta_j$. If $a_i^+=a_i^-$ for all $i$, then the number of choices for $\Ps$ clearly equals $(a_1^+)!\cdots (a_r^+)!$; otherwise no such $\Ps$ exists. This proves (\ref{isserlis}).
\end{proof}
\begin{lem}[Gaussian hypercontractivity]\label{largedev} Given $n$ and $\zeta_j\in\{\pm\}$ for $1\leq j\leq n$, suppose the random variable $X$ has the form
\begin{equation}\label{multigauss}X=\sum_{k_1,\cdots,k_n}a_{k_1\cdots k_n}\prod_{j=1}^n\eta_{k_j}^{\zeta_j}(\omega),\end{equation} where $a_{k_1\cdots k_n}$ are constants, then for any $q\geq 2$ we have
\begin{equation}\label{hyper}\Eb|X|^q\leq(q-1)^{\frac{nq}{2}}\cdot(\Eb|X|^2)^{\frac{q}{2}}.\end{equation}
\end{lem}
\begin{proof} This is the standard hypercontractivity estimate for Gaussians, see \cite{OT}, {Lemma 2.6}.
\end{proof}
\begin{lem}[A combinatoric inequality]\label{combineq} Given a multi-index $\rho$, we have
\begin{equation}\label{combineq2}
\sum_{\rho^1+\cdots +\rho^9=\rho}\frac{\rho!}{(\rho^1)!\cdots(\rho^9)!}\cdot(2|\rho^1|)!\cdots (2|\rho^9|)!\leq C (2|\rho|)!.
\end{equation}
\end{lem}
\begin{proof} We first fix $\rho_*^2:=\rho^2+\cdots +\rho^9$ and sum over $(\rho^2,\cdots,\rho^9)$, then sum over $(\rho^1,\rho_*^2)$. The sum over $(\rho^2,\cdots,\rho^9)$ can be bounded by inductively repeating this process, provided one can bound the sum over $(\rho^1,\rho_*^2)$. To bound this latter sum, if $\rho=(a_1,\cdots, a_n)$ then it can be written as
\[\sum_{0\leq b_j\leq a_j}\prod_{j=1}^n\binom{a_j}{b_j}(2A-2B)!(2B)!\] where $A=a_1+\cdots+a_n$ and $B=b_1+\cdots+b_n$. If $B$ is fixed, then the sum over $(b_1,\cdots,b_n)$ equals $\binom{A}{B}$ by a simple application of the binomial theorem (or the Vandermonde identity), so (\ref{combineq2}) would follow from the inequality
\begin{equation}\label{expdecay}\sum_{0\leq B\leq A}\binom{A}{B}\binom{2A}{2B}^{-1}\leq C.\end{equation} By symmetry, in (\ref{expdecay}) we may assume $B\leq A/2$, so $\binom{A}{B}/\binom{2A}{2B}\leq 1/\binom{A}{B}\leq1/ \binom{2B}{B}\leq 2^{-B}$, which proves (\ref{expdecay}) and hence (\ref{combineq2}) by induction.
\end{proof}
\begin{lem}[A sharp Hua's lemma]\label{hua} For $s ,r \in \Rb$ and $h\in\Zb$, define the Gauss sums 
\begin{equation}\label{huagausssum}G_h(s, r,n)=\sum_{p=h}^{h+n}e(s p^2+rp), n \in \Nb; \qquad \mathrm{and} \qquad G_h(s, r,x)=G_h(s, r, \lfloor x\rfloor), x\in \Rb_+,\end{equation} where $\lfloor x\rfloor$ is the floor function, and $e(z)=e^{2\pi iz}$. Then we have
\begin{equation}\label{sharphua}\|G_h(\cdot,r,n)\|_{L^4([0,1])}^4\lesssim n^2\log(2+n),\quad \|G_h(\cdot,r,n)\|_{L^6([0,1])}^6\lesssim n^4
\end{equation}uniformly in $(r,h)$. The constants involved in $\lesssim$ here are absolute constants.
\end{lem}
\begin{proof} We only need to bound the cardinalities of the sets
\begin{multline*}A_4=\{(a,b,c,d)\in[h,h+n]^4:a^2-b^2+c^2-d^2=0\},\\\mathrm{and}\,\,A_6=\{(a,\cdots,f)\in[h,h+n]^6:a^2-b^2+c^2-d^2+e^2-f^2=0\}.\end{multline*} By changing variables $(a,b)\mapsto (a+b,a-b)$ etc., we can reduce to the sets
\[B_4=\{(a,b,c,d):ab+cd=0\}\quad\mathrm{and}\quad B_6=\{(a,\cdots,f):ab+cd+ef=0\},\] where $|a|,|c|,|e|\leq n$ and $b,d,f\in[2h,2h+2n]$. To count $\#B_6$, we may assume $a\geq |c|\geq |e|$ (and $a>0$). Note that $f$ belongs to a fixed residue class modulo $\gcd(a,c)/\gcd(a,c,e)$; once $f$ is fixed, then $d$ belongs to a fixed residue class modulo $a/\gcd(a,c)$. When $f$ and $d$ are fixed then $b$ is unique. This implies that
\[\#B_6\lesssim\sum_{a\geq|c|\geq |e|}\frac{n}{\gcd(a,c)/\gcd(a,c,e)}\cdot\frac{n}{a/\gcd(a,c)}=n^2\sum_{a\geq |c|\geq |e|}\frac{\gcd(a,c,e)}{a}.\] For the last sum, let $\gcd(a,c,e)=\Delta$ and $a=a'\Delta $ etc., then
\[\sum_{a\geq |c|\geq |e|}\frac{\gcd(a,c,e)}{a}\lesssim\sum_{0<\Delta\leq n}\sum_{0<a'\leq n/\Delta}\sum_{|c'|,|e'|\leq a'}\frac{1}{a'}\lesssim\sum_{0<\Delta\leq n}(n/\Delta)^2\lesssim n^2,\] hence $\#B_6\lesssim n^4$. In the same way we can bound $\# B_4\lesssim n^2\log (2+n)$.
\end{proof}
\begin{lem}\label{finitelem} Fix $M\geq L^{(100d)^3}$. Consider $k=(k^1,\cdots,k^d)\in \Zb_L^d$ and a system $(r,q,v_1,\cdots,v_q,f,y)$ where $0\leq q\leq r\leq d$, $v_j\in \Zb_L^r\,(1\leq j\leq q)$ are nonzero orthogonal vectors and $f=(f^{r+1},\cdots,f^d)\in\Zb_L^{d-r}$ and $y=(y^1,\cdots,y^q)\in\Rb^q$, such that $|f|,|v_j|,|y|\leq M^{C_0(d)}$ where $C_0(d)$ is a fixed large constant depending on $d$, and that the linear span of $\{v_1,\cdots,v_q\}$ does not contain any coordinate vector in $\Rb^r$. We say that the system $(r,q,v_1,\cdots,v_q,f,y)$  \emph{represents} $k$, if the followings hold:
\begin{enumerate}
\item If $z=(z^1,\cdots,z^d)\in \Zb_L^d$, $|z|\leq M$ and $|\langle k,z\rangle_\beta|\leq M$, then the vector $(z^1,\cdots,z^r)$ is a linear combination of $\{v_1,\cdots,v_q\}$;
\item If $(z^1,\cdots,z^r)=\gamma^1v_1+\cdots +\gamma^qv_q$, then we have $\langle k,z\rangle_\beta=y^1\gamma^1+\cdots +y^q\gamma^q+\beta^{r+1}f^{r+1}z^{r+1}+\cdots +\beta^df^dz^d$.
\end{enumerate}

Then, each $k\in \Zb^d$, after possibly permuting the coordinates, is represented by some system $(r,q,v_1,\cdots,v_q,f,y)$. Conversely, for each system $(r,q,v_1,\cdots,v_q,f,y)$ and $\theta>0$, there exists $k\in\Zb^d$ represented by a system $(r,q,v_1,\cdots,v_q,f,y')$, such that $|y'-y|<\theta$.
\end{lem}
\begin{proof}[Proof of Lemma \ref{finitelem}] As $M\geq L^{(100d)^3}$, upon multiplying everything by $L$, we may replace $\Zb_L$ by $\Zb$. As our convention, in the proof $C$ will denote any large constant depending only on $d$. In the first part, given $k$ we will construct the system $(r,q,v_1,\cdots,v_q,f,y)$, which is done by induction in $d$. The case $d=1$ is obvious, now suppose the result is true for $d-1$, with constant $C_0=C_0(d-1)$. We also denote $k_\beta=(k^1\beta^1,\cdots,k^d\beta^d)$ for $k=(k^1,\cdots,k^d)$ and $\beta=(\beta^1,\cdots,\beta^d)$, so that $\langle k,z\rangle_\beta=\langle k_\beta,z\rangle$.

Fix $k\in\Zb^d$, consider the set $H$ of $z\in\Zb^d$ such that $|z|\leq M$ and $|\langle k,z\rangle_\beta|\leq M$ (clearly $0\in H$). Let $q$ be the maximal number of linearly independent vectors in $H$, we may fix a maximum independent set $\{w_1,\cdots, w_q\}\subset H$, and apply Gram-Schmidt process to get orthogonal vectors $(v_1,\cdots,v_q)$. Since each $w_j\in\Zb^d$ and $|w_j|\leq M$, we can easily make $v_j\in\Zb^d$ and $|v_j|\leq M^{C}$. If the linear span of $\{v_1,\cdots,v_q\}$ does not contain any coordinate vector in $\Rb^d$, then we shall prove the result with $r=d$. In fact, (1) is already satisfied by definition; since $|\langle k,w_j\rangle_\beta|\leq M$ for $1\leq j\leq q$, we also know that $|\langle k,v_j\rangle_\beta|\leq M^{C}$ for $1\leq j\leq q$. Let $y^j=\langle k,v_j\rangle_\beta$, then $|y|\leq M^C$ and (2) is also satisfied.

If, instead, the linear span of $\{v_1,\cdots,v_q\}$ contains a coordinate vector in $\Rb^d$, say $e_d=(0,\cdots,0,1)$, we shall apply the induction hypothesis. In this case we have $|\langle k,v_j\rangle_\beta|\leq M^{C}$ for $1\leq j\leq q$ and hence $|k^d|\leq M^C$. By induction hypothesis (with $M$ replaced by $M^C$), the vector $(k^1,\cdots,k^{d-1})$ is represented by some system $(r,q,v_1,\cdots,v_q,f,y)$ where $0\leq q\leq r\leq d-1$ and $|f|,|v_j|,|y|\leq M^{CC_0}$. Now we claim that $k$ is represented by $(r,q,v_1,\cdots,v_q,f',y)$ where $f'=(f,k^d)$; in fact (2) is satisfied by definition, as for (1), if $|z|\leq M$ and $|\langle k,z\rangle_\beta|\leq M$, then $|\beta^1k^1z^1+\cdots+\beta^{d-1}k^{d-1}z^{d-1}|\leq M^C$ as $|k^d|\leq M^C$, so we may apply the induction hypothesis (with $M$ replaced by $M^C$) to show that $(z^1,\cdots,z^r)$ is a linear combination of $\{v_1,\cdots,v_q\}$. In either case we have constructed the desired system, with $C_0(d)=C\cdot C_0(d-1)$.

\smallskip
Now, suppose a system $(r,q,v_1,\cdots,v_q,f,y)$, and $\theta>0$, is fixed. We may choose $k^j=f^j$ for $r+1\leq j\leq d$, and again notice that $|z|\leq M$ and $|\langle k,z\rangle_\beta|\leq M$ implies that $|\beta^1k^1z^1+\cdots+\beta^{r}k^{r}z^{r}|\leq M^C$. Therefore we only need to consider $r=d$. Select vectors $u_1,\cdots,u_{d-q}\in\Zb^d$ such that they form an orthogonal basis with $\{v_1,\cdots,v_q\}$, and $|u_j|\leq M^{C}$ for $1\leq j\leq d-q$. Now choose $k^*=\rho_1u_1+\cdots +\rho_{d-q}u_{d-q}$, where $\rho_j$ are large integers, and assume $k$ is chosen such that $|k_\beta-k^*|\leq M^{C}$. We will assume $0\leq \rho_j\leq B$ and $B\gg_{M,\theta}1$. Clearly, if $|z|\leq M$ and $|\langle k,z\rangle_\beta|\leq M$, then $|\langle k^*,z\rangle|\leq M^{C}$.

Suppose $|z|\leq M$ and $|\langle k^*,z\rangle|\leq M^{C}$. If we decompose $z=x_1u_1+\cdots +x_{d-q}u_{d-q}+z'$ where $z'$ is a linear combination of $(v_1,\cdots,v_q)$, then
\[\langle k^*,z\rangle=\sum_{j=1}^{d-q}\rho_j|u_j|^2x_j.\] Each $x_j$ is a rational number and $\max_j|x_j|\geq M^{-C}$ unless all $x_j=0$, and the number of choices for $(x_1,\cdots,x_{d-q})$ is at most $M^C$ when $z$ varies. For each fixed nonzero $(x_1,\cdots, x_{d-q})$ the number of choices for $(\rho_1,\cdots,\rho_{d-q})$ satisfying $|\langle k^*,z\rangle|\leq M^{C}$ is at most $M^{C}B^{d-q-1}$, so for at least $B^{d-q}-M^{C}B^{d-q-1}$ choices of $(\rho_1,\cdots,\rho_{d-q})$ (and for any choice of $k$ satisfying $|k_\beta-k^*|\leq M^{C}$ given $k^*$), we have that $|z|\leq M$ and $|\langle k,z\rangle_\beta|\leq M$ implies that $z$ is a linear combination of $v_j\,(1\leq j\leq q)$.

For such choices we already have (1). Clearly $z$ is then represented by $(q,v_1,\cdots,v_q,y')$ (with $r=d$ and $f$ being void) where $(y')^j=\langle k_\beta,v_j\rangle=-\langle k^*-k_\beta,v_j\rangle$. Given $(v_j)$ and $y$, we may fix a vector $h\in\Rb^d$, where $|h|\leq M^{C}$, such that $\langle h,v_j\rangle=-y^j$ for $1\leq j\leq q$; it then suffices to choose $k^*$ and $k$ such that $|k^*-k_\beta-h|\leq M^{-C_1}\theta$ with $C_1$ larger than all the $C$ appearing above. (note that this also implies $|k_\beta-k^*|\leq M^{C}$). Since $k$ can be arbitrarily chosen, it suffices to have
\[\bigg\{\frac{(k^*)^j-h^j}{\beta^j}\bigg\}< M^{-2C_1}\theta\quad\mathrm{for}\quad1\leq j\leq d,\] where $\{\cdot\}$ means the distance to the nearest integer. Clearly
\[\frac{(k^*)^j-h^j}{\beta^j}=\sum_{i=1}^{d-q}\frac{(u_i)^j}{\beta^j}\rho_i-\frac{h^j}{\beta^j}.\]Since the linear span of $\{v_j\}$ does not contain any coordinate vector in $\Rb^d$, we know that, for each $1\leq j\leq d$, there exists $1\leq i=i(j)\leq d-q$ such that $(u_i)^j\neq 0$.

We may choose $\rho_i$ such that
\[\bigg\{\frac{(u_i)^j}{\beta^j}\rho_i\bigg\}<M^{-3C_1}\theta\,\,(i\neq i(j));\quad\bigg\{\frac{(u_i)^j}{\beta^j}\rho_i-\frac{h^j}{\beta^j}\bigg\}<M^{-3C_1}\theta\,\,(i=i(j)).\] Given $i$, since all the nonzero numbers in the set $\{1,(u_i)^j(\beta^j)^{-1}:1\leq j\leq d\}$ are $\mathbb{Q}$-linearly independent, by Weyl's equidistribution theorem, we see that the number of $(\rho_1,\cdots,\rho_{d-q})$ satisfying all the above conditions is $\gtrsim_{M,\theta}B^{d-q}$. Therefore, we may choose $(\rho_1,\cdots,\rho_{d-q})$, and hence $k^*$ and $k$, such that both (1) and (2) are satisfied.
\end{proof}
\subsection{Lattice point counting bounds} We list the various lattice point counting bounds, which are the main technical tools used in Section \ref{oper}.
\begin{lem}[Sphere counting]\label{jensen} Uniformly in $a\in\Zb_L^d$ and $\gamma\in\Rb$, we have the bound
\[\#\big\{x\in\Zb_L^d:|x-a|\leq 1,\,\,||x|_\beta^2-\gamma|\leq \delta^{-1} L^{-2}\big\}\lesssim \delta^{-1}L^{d-\frac{4}{3}}.\]
\end{lem}
\begin{proof} By dividing an interval of length $\delta^{-1}L^{-2}$ into $O(\delta^{-1})$ intervals of length $L^{-2}$ we may assume $\delta=1$. Let $y=(y^1,\cdots,y^d)=(\sqrt{\beta^1}x^1,\cdots,\sqrt{\beta^d}x^d)$, then $|y|^2=\gamma+O(L^{-2})$, where $|y|$ is the usual norm in $\Rb^d$. If we fix the coordinates $y^j\,(3\leq j\leq d)$, noticing that each $y^j\,(3\leq j\leq d)$ has $\lesssim L$ choices, it then suffices to prove that
\begin{equation}\label{jensen2}\#\big\{(u,v)\in(\sqrt{\beta^1}\Zb_L)\times(\sqrt{\beta^2}\Zb_L):u^2+v^2=\gamma+O(L^{-2}),\,\,|u-u_0|+|v-v_0|\lesssim 1\big\}\lesssim L^{\frac{2}{3}}\end{equation} uniformly in $(u_0,v_0,\gamma)\in\Rb^3$.

Let $|\gamma|\sim R^2$ (we may assume $R\gg L^{-1}$), then $(u,v)$ belongs to the $O(\varepsilon)$ neighborhood of a circle centered at the origin of radius $\sim R$, where $\varepsilon=L^{-2}R^{-1}$. Since $(u,v)$ also belongs to a disc of radius $O(1)$, we know that $(u,v)$ actually belongs to the $O(\varepsilon)$ neighborhood of an arc of length $O(\min(R,1))$ on the circle. Let $\Gamma:=(\sqrt{\beta^1}\Zb_L)\times(\sqrt{\beta^2}\Zb_L)$ be a fixed lattice, it will suffice to prove that the number of points in $\Gamma$ that belong to this neighborhood is $\lesssim L^{\frac{2}{3}}$.

Now, we may decompose the above arc of length $O(\min(R,1))$ into at most $O(L^{\frac{2}{3}})$ sub-arcs, each with length $\ll\rho$, where $\rho=L^{-\frac{2}{3}}R^{\frac{1}{3}}$, note that $\varepsilon\ll\rho\ll R$. Thus is suffices to prove that the $O(\varepsilon)$ neighborhood of each sub-arc contains $O(1)$ points in $\Gamma$. Let this neighborhood be $M$, from elementary geometry we can calculate that the area of the convex hull of $M$ is
\[A\ll \left(\frac{\rho}{R}\right)R\varepsilon+R^2\left(\frac{\rho}{R}\right)^3=\rho\varepsilon+\frac{\rho^3}{R}\lesssim L^{-2}.\] But any nondegenerate triangle with vertices in $\Gamma$ have area $\gtrsim L^{-2}$, so the points in $\Gamma\cap M$ must be collinear; however $M$ is contained in an annulus of width $2\varepsilon$, and any straight line contains at most two segments in this annulus, each having length at most $O(\sqrt{\varepsilon R})=O(L^{-1})$, so in any case the number of points in $\Gamma\cap M$ is at most $O(1)$.
\end{proof}
\begin{lem}[Good and bad vectors]\label{goodvec} We say a vector $0\neq x\in\Zb_L^d$ is a \emph{bad} vector, if
\begin{equation}\label{badvector}
\#\big\{y\in\Zb_L^d:|y-b|\leq 1,\,|\langle x,y\rangle_\beta-\Gamma|\leq L^{-2}\big\}\geq L^{d-1-\frac{1}{4}}
\end{equation} for \emph{some} $b\in\Zb_L^d$ and $\Gamma\in\Rb$; otherwise we say $x$ is a \emph{good} vector. Then, {when $L$ is large enough,} for any $a\in\Zb_L^d$, the number of bad vectors $x$ satisfying $|x-a|\leq 1$ is at most $L^{d-1-\frac{1}{4}}$.
\end{lem}
\begin{proof} If (\ref{badvector}) is true for some $(b,\Gamma)$, then it is actually true for $b=\Gamma=0$ up to some constant, by fixing $x$ taking the difference between any two possibilities of $y$. We will show that
\begin{equation}\label{l2sum}\sum_{x\neq 0,|x-a|\leq 1}\sum_{|y|,|z|\leq 1}\mathbf{1}_{|\langle x,y\rangle_\beta|\leq L^{-2}}\cdot \mathbf{1}_{|\langle x,z\rangle_\beta|\leq L^{-2}}\lesssim L^{3d-4+\frac{1}{6}},\end{equation} which implies the desired result, as the left hand side of (\ref{l2sum}) is just the sum of the square of the left hand side of (\ref{badvector}) over $|x-a|\leq 1$. However the left hand side of (\ref{l2sum}) is bounded by the same expression with $a=0$, again by fixing $(y,z)$ and taking the difference between any two possibilities of $x$. Moreover, if we set $(x,y,z)=L^{-1}(X,Y,Z)$ with $(X,Y,Z)\in(\Zb^d)^3$, then (\ref{l2sum}) with $a=0$ is just (\ref{generic1.5}) which follows from the definition of $\Zf$. This completes the proof.
\end{proof}
\begin{lem}[Atom counting bounds]\label{lem:counting} For $1\leq j\leq 5$, let $x_j\in\Zb_L^d$ be variables, $a_j\in\Zb_L^d$ and $\zeta_j\in\{\pm\}$ be fixed parameters. Fix also the parameters $k_i\in\Zb_L^d$ and $\gamma_i\in\Rb$ for $i\in\{1,2\}$. We require that $|x_j-a_j|\leq 1$ for each $j$. If any of the statements below involves an equation of form $\sum_{j\in A}\zeta_jx_j=k_i$ with some set $A$, then we also require that (i) no three of $\zeta_j\,(j\in A)$ are the same, and (ii) if $j,j'\in A$ and $\zeta_{j'}+\zeta_j=0$ then $x_j\neq x_{j'}$.

We have the following estimates, where the implicit constants only depend on $d$ and $\beta$, and do not depend on $(\delta,L)$ or any of the parameters $(a_j,\gamma_i,k_i)$:

(1) (Two-vector counting) If we require
		\begin{equation}\label{2vcounting}\zeta_1x_1+\zeta_2x_2=k_1,\quad \big|\zeta_1|x_1|_\beta^2+\zeta_2|x_2|_\beta^2-\gamma_1\big|\leq \delta^{-1}L^{-2},\end{equation} then the number of choices for $(x_1,x_2)$ is $\lesssim \delta^{-1}L^{d-1}$, and is $\lesssim \delta^{-1}L^{d-1-\frac{1}{3}}$ if $\zeta_1=\zeta_2$.
		
(2) (Three-vector counting) If we require
				\begin{equation}\label{3vcounting}\zeta_1x_1+\zeta_2x_2+\zeta_3x_3=k_1,\quad \big|\zeta_1|x_1|_\beta^2+\zeta_2|x_2|_\beta^2+\zeta_3|x_3|_\beta^2-\gamma_1\big|\leq \delta^{-1}L^{-2},\end{equation}
		then the number of choices for $(x_1,x_2,x_3)$ is $\lesssim \delta^{-1}L^{2d-2}$.
		
(3) (Four-vector counting 1) If we require
		\begin{equation}\label{4vcounting1}
		\left\{
		\begin{aligned}
		&\zeta_1x_1+\zeta_2x_2=k_1,\quad \big|\zeta_1|x_1|_\beta^2+\zeta_2|x_2|_\beta^2-\gamma_1\big|\leq \delta^{-1}L^{-2},\\
		&\zeta_1x_1+\zeta_3x_3+\zeta_4x_4=k_2,\quad\big|\zeta_1|x_1|_\beta^2+\zeta_3|x_3|_\beta^2+\zeta_4|x_4|_\beta^2-\gamma_2\big|\leq \delta^{-1}L^{-2},
		\end{aligned}
		\right.\end{equation}
	then the number of choices for $(x_1,\cdots,x_4)$ is $\lesssim \delta^{-2}L^{2d-2-\frac{1}{4}}$.

(4) (Four-vector counting 2) If we require
		\begin{equation}\label{4vcounting2}
		\left\{
		\begin{aligned}
		&\zeta_1x_1+\zeta_2x_2+\zeta_3x_3=k_1,\quad \big|\zeta_1|x_1|_\beta^2+\zeta_2|x_2|_\beta^2+\zeta_3|x_3|_\beta^2-\gamma_1\big|\leq \delta^{-1}L^{-2},\\
		&\zeta_1x_1+\zeta_2x_2+\zeta_4x_4=k_2,\quad\big|\zeta_1|x_1|_\beta^2+\zeta_2|x_2|_\beta^2+\zeta_4|x_4|_\beta^2-\gamma_2\big|\leq \delta^{-1}L^{-2},
		\end{aligned}
		\right.\end{equation}
	\emph{and assume that $(\zeta_3,x_3)\neq (\zeta_4,x_4)$}, then the number of choices for $(x_1,\cdots,x_4)$ is $\lesssim \delta^{-2}L^{2d-2-\frac{1}{4}}$.
	
(5) (Five-vector counting 1) If we require
		\begin{equation}\label{5vcounting1}
		\left\{
		\begin{aligned}
		&\zeta_1x_1+\zeta_2x_2+\zeta_3x_3=k_1,\quad\big|\zeta_1|x_1|_\beta^2+\zeta_2|x_2|_\beta^2+\zeta_3|x_3|_\beta^2-\gamma_1\big|\leq \delta^{-1}L^{-2},\\
		&\zeta_1x_1+\zeta_4x_4+\zeta_5x_5=k_2,\quad\big|\zeta_1|x_1|_\beta^2+\zeta_4|x_4|_\beta^2+\zeta_5|x_5|_\beta^2-\gamma_2\big|\leq \delta^{-1}L^{-2},
		\end{aligned}
		\right.\end{equation}
	then the number of choices for $(x_1,\cdots,x_5)$ is $\lesssim \delta^{-2}L^{3d-3-\frac{1}{4}}$.
	
(6) (Five-vector counting 2) If we require
		\begin{equation}\label{5vcounting2}
		\left\{
		\begin{aligned}
		&\zeta_1x_1+\zeta_2x_2+\zeta_3x_3=k_1,\quad\big|\zeta_1|x_1|_\beta^2+\zeta_2|x_2|_\beta^2+\zeta_3|x_3|_\beta^2-\gamma_1\big|\leq \delta^{-1}L^{-2},\\
		&\zeta_1x_1+\zeta_2x_2+\zeta_4x_4+\zeta_5x_5=k_2,\quad\big|\zeta_1|x_1|_\beta^2+\zeta_2|x_2|_\beta^2+\zeta_4|x_4|_\beta^2+\zeta_5|x_5|_\beta^2-\gamma_2\big|\leq \delta^{-1}L^{-2},
		\end{aligned}
		\right.\end{equation}
	then the number of choices for $(x_1,\cdots,x_5)$ is $\lesssim \delta^{-2}L^{3d-3-\frac{1}{4}}$.
	
(7) (Five-vector counting 3) If we require (\ref{5vcounting1}), and that 
\begin{equation}\label{5vcounting3}
		\left\{
		\begin{aligned}
		&\zeta_2x_2-\zeta_4x_4=k_1^*,\quad\big|\zeta_2|x_2|_\beta^2-\zeta_4|x_4|_\beta^2-\gamma_1^*\big|\leq n\delta^{-1}L^{-2},\\
		&\zeta_3x_3-\zeta_5x_5=k_2^*,\quad\big|\zeta_3|x_3|_\beta^2-\zeta_5|x_5|_\beta^2-\gamma_2^*\big|\leq n\delta^{-1}L^{-2}
		\end{aligned}
		\right.\end{equation}for some constants $(k_1^*,k_2^*,\gamma_1^*,\gamma_2^*)$ and $n\geq 1$, \emph{and assume that $(\zeta_2,\zeta_3,x_2,x_3)\neq (\zeta_4,\zeta_5,x_4,x_5)$}, then the number of choices for $(x_1,\cdots,x_5)$ is $\lesssim n\delta^{-2}L^{2d-2-\frac{1}{4}}$.

\smallskip
Note that, if $n\leq (\log L)^3$ then the bound in (7) can be replaced by $\delta^{-2}L^{2d-2-\frac{1}{6}}$.
\end{lem}
\begin{proof} In all the proofs, we may assume $\delta=1$ as above, by dividing an interval of length $\delta^{-1}L^{-2}$ into $O(\delta^{-1})$ intervals of length $L^{-2}$.

(1) If $\zeta_1=\zeta_2$, then $x_1+x_2=\pm k_1$ is fixed. For $y=x_1-x_2$, we have that \[|y|_\beta^2=\pm 2\gamma_1-|k_1|_\beta^2+O(L^{-2});\] moreover as $|x_1-a_1|\leq 1$ we have that $|y-(2a_1-k_1)|\leq 2$. Lemma \ref{jensen} then implies that the number of choices for $y$, and hence for $(x_1,x_2)$, is $\lesssim L^{d-\frac{4}{3}}$.

Otherwise, we may assume $\zeta_1=+$ and $\zeta_2=-$, then $x_1-x_2=k_1\neq 0$. Let $y=x_1+x_2$, then we have that
\[\langle k_1,y\rangle_\beta=\gamma_1+O(L^{-2});\] moreover as $|x_1-a_1|\leq 1$ we have that $|y-(2a_1-k_1)|\leq 2$. We may assume that the first coordinate $k_1^1$ of $k_1$ is nonzero, then $|k_1^1|\geq L^{-1}$. Thus, when the coordinates $y^j\,(2\leq j\leq d)$ are fixed, $y^1$ will belong to an interval of length $\lesssim L^{-1}$ and will have $\lesssim 1$ choices. As each $y^j\,(2\leq j\leq d)$ has $\lesssim L$ choices, we conclude that the number of choices for $y$, and hence for $(x_1,x_2)$, is $\lesssim L^{d-1}$.

Note that, if in addition we assume $x_1-x_2$ is a good vector, then by definition we can bound the number of choices of $(x_1,x_2)$ by $L^{d-1-\frac{1}{4}}$.

(2) By assumption the $\zeta_j\,(1\leq j\leq 3)$ are not all equal, so we may assume $\zeta_1=\zeta_3=+$ and $\zeta_2=-$. Let $y=k_1-x_1$ and $z=k_1-x_3$, then we have that
\[\langle y,z\rangle_\beta=\frac{|k_1|_\beta^2-\gamma_1}{2}+O(L^{-2});\] moreover as $|x_1-a_1|\leq 1$ and $|x_3-a_3|\leq 1$ we have that $|y-(k_1-a_1)|\leq 1$ and $|z-(k_1-a_3)|\leq 1$. By applying Proposition \ref{approxnt} with $n=1$, where $W$ and $\Psi$ are two fixed (translates of) nonnegative compactly supported smooth cutoff functions, we know that the number of choices for $(y,z)$, and hence for $(x_1,x_2,x_3)$, is $\lesssim L^{2d-2}$. Note that the second inequality in (\ref{propertypsi2}) is not needed if one only needs the upper bound instead of asymptotics.

(3) If $\zeta_3=\zeta_4$, then by (1) we know that $(x_1,x_2)$ have at most $L^{d-1}$ choices, while for $x_1$ fixed, $(x_3,x_4)$ has at most $L^{d-1-\frac{1}{3}}$ choices, so the total number of choices for $(x_1,\cdots,x_4)$ is at most $L^{2d-2-\frac{1}{3}}$. The same is true (with $\frac{1}{4}$ instead of $\frac{1}{3}$) if $\zeta_3+\zeta_4=0$ and $x_3-x_4$ is a good vector. But if $x_3-x_4$ is a bad vector, then $x_1$ is a fixed translate of a bad vector which belongs to a fixed ball of radius $1$, so by Lemma \ref{goodvec}, the number of choices for $x_1$, and hence $(x_1,x_2)$, is at most $L^{d-1-\frac{1}{4}}$, so we get the same result.

(4) This follows from (3) by taking the difference of the two equations, and noticing that if $\zeta_3=\zeta_4$, we must have $x_3\neq x_4$.

(5) If $\zeta_4=\zeta_5$, then by (2) we know that $(x_1,x_2,x_3)$ have at most $L^{2d-2}$ choices, while for $x_1$ fixed, $(x_4,x_5)$ has at most $L^{d-1-\frac{1}{3}}$ choices, so the total number of choices for $(x_1,\cdots,x_5)$ is at most $L^{3d-3-\frac{1}{3}}$; note that this estimate is valid even if we allow $\zeta_3=\zeta_4=\zeta_5$. The same is true (with $\frac{1}{4}$ instead of $\frac{1}{3}$) if $\zeta_4+\zeta_5=0$ and $x_4-x_5$ is a good vector. If $x_4-x_5$ is a bad vector, then $x_1$ has at most $L^{d-1-\frac{1}{4}}$ choices by Lemma \ref{goodvec}. For $x_1$ fixed, the number of $(x_2,x_3)$ and $(x_4,x_5)$ can be bounded by $L^{d-1}$ by (1), so we get the same result.

(6) This follows from the first part of (5) by taking the difference of the two equations, provided $(\zeta_3,x_3)\neq (\zeta_j,x_j)$ for $j\in\{4,5\}$; now suppose, say, $\zeta_3=\zeta_4$ and $x_3=x_4$, then the value of $x_5$ is fixed and the number of choices for $(x_1,x_2,x_3)$ is at most $L^{2d-2}$ by (2), so the result is still true.

(7) By subdividing one interval of length $nL^{-2}$ we may assume $n=1$. The result then follows from (3), since we may assume (for example) either $\zeta_2\neq \zeta_4$ or $x_2\neq x_4$, and simply exploit the first equation in (\ref{5vcounting1}) and the first equation in (\ref{5vcounting3}).
\end{proof}
\section{An example of molecule reduction}\label{algexample} Here we provide an example of the molecule reduction algorithm described in Section \ref{loop}. For simplicity we only consider phase two.
  \begin{figure}[h!]
  \includegraphics[scale=.45]{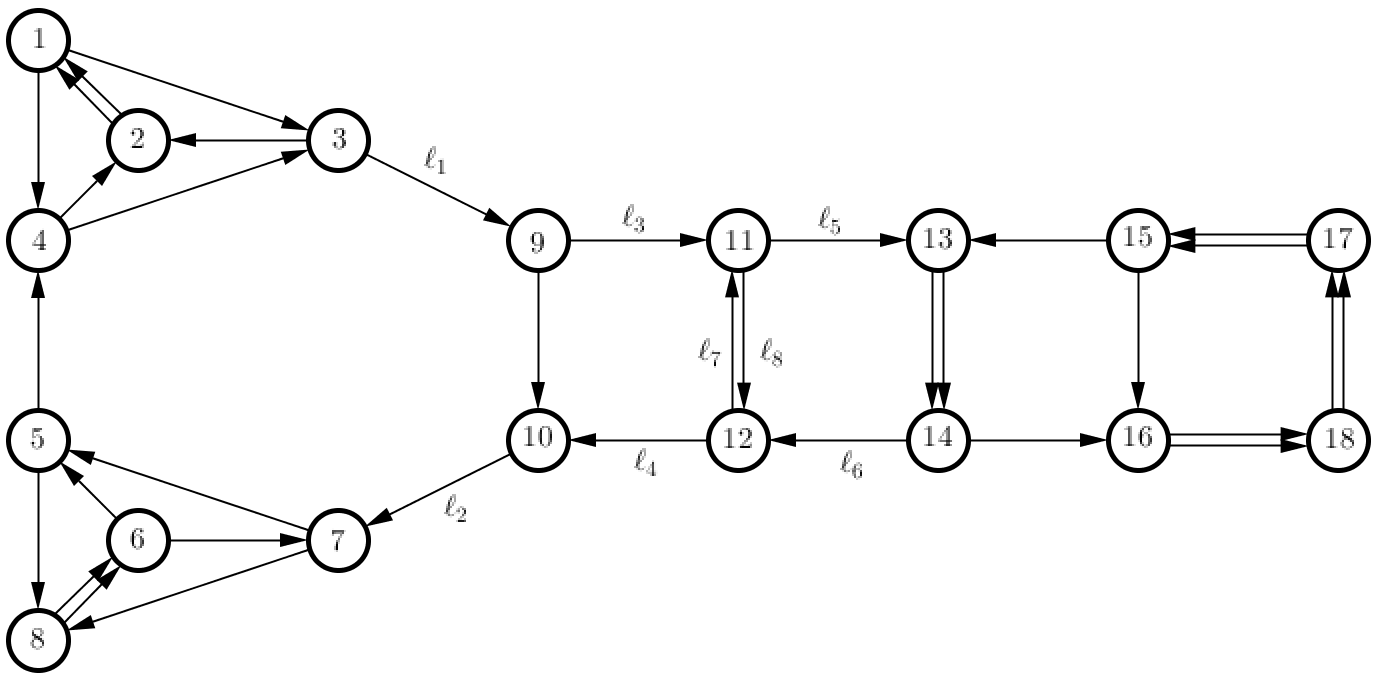}
  \caption{The original base molecule ``flashlight". The bonds $\ell_1$ to $\ell_8$ will appear in the $\mathtt{Ext}$ condition obtained by the algorithm.}
  \label{fig:graphexample1}
\end{figure}

Suppose the original molecule is a base molecule as in Figure \ref{fig:graphexample1}. Then, according to the algorithm, we first treat the two degree 3 atoms (labeled 9 and 10) connected by a single bond. As in (2-b) we claim a checkpoint and perform either (3S3-1) or (3S3-2G). In either case $\Mb$ is reduced to the one in Figure \ref{fig:graphexample2}.
  \begin{figure}[h!]
  \includegraphics[scale=.45]{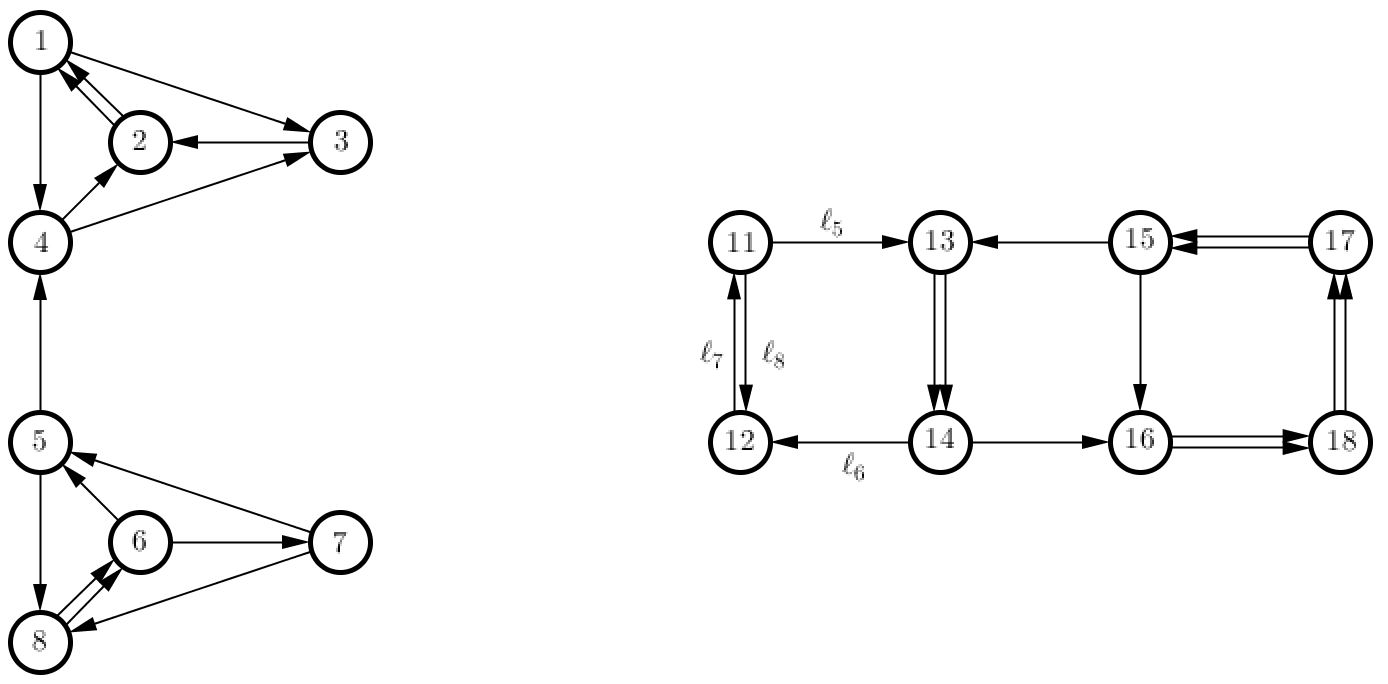}
  \caption{The molecule obtained after performing (3S3-1) or (3S3-2G). Note that this is a checkpoint and corresponds to two possible steps (though the operation on $\Mb$ is the same and the only difference is $\Delta\mathtt{Ext}$).}
  \label{fig:graphexample2}
\end{figure}

Next, as in (1), we perform (BR) and remove the bridge connecting atoms labeled 4 and 5. Then $\Mb$ is reduced to the one in Figure \ref{fig:graphexample3}.
  \begin{figure}[h!]
  \includegraphics[scale=.45]{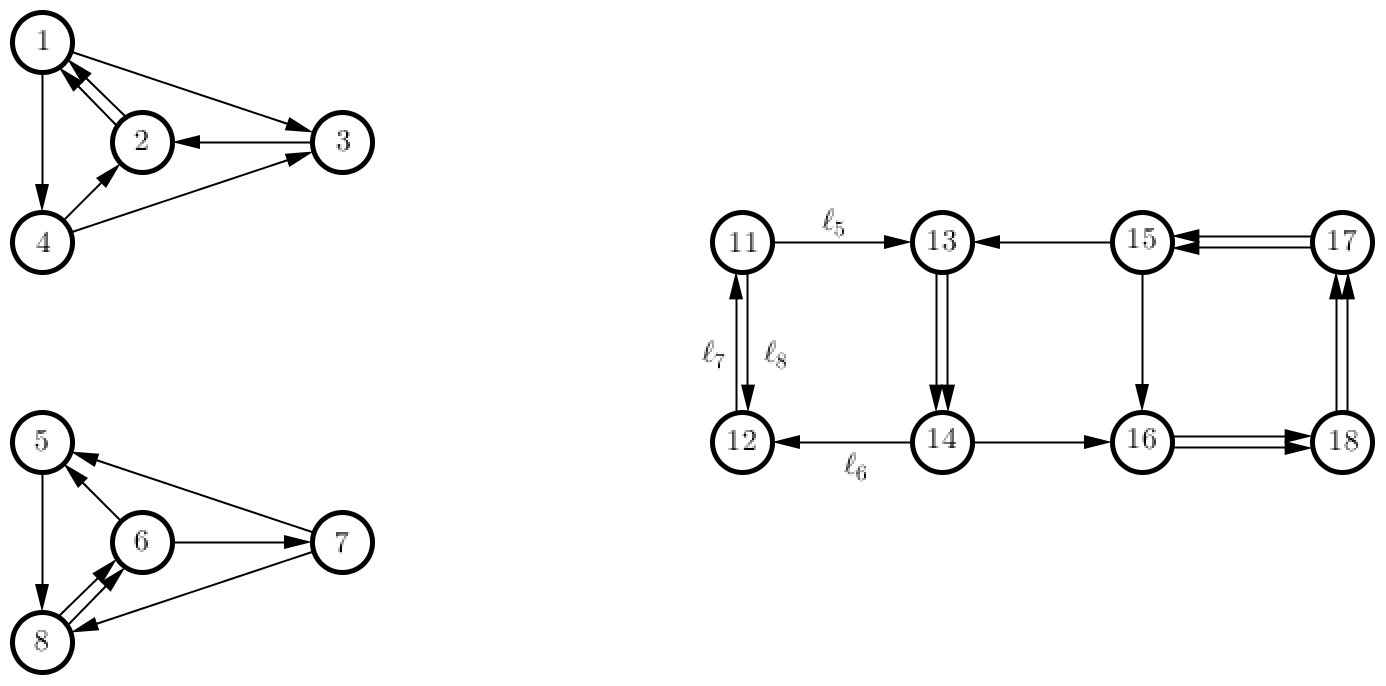}
  \caption{The molecule obtained after performing (BR).}
  \label{fig:graphexample3}
\end{figure}

Next, we treat the two pairs of degree 3 atoms (labeled $(3,4)$ and $(5,7)$) connected by two single bonds. As in (2-a) we perform (3S3-5G) (Scenario 2) twice, and reduce $\Mb$ to the one in Figure \ref{fig:graphexample4}.
  \begin{figure}[h!]
  \includegraphics[scale=.23]{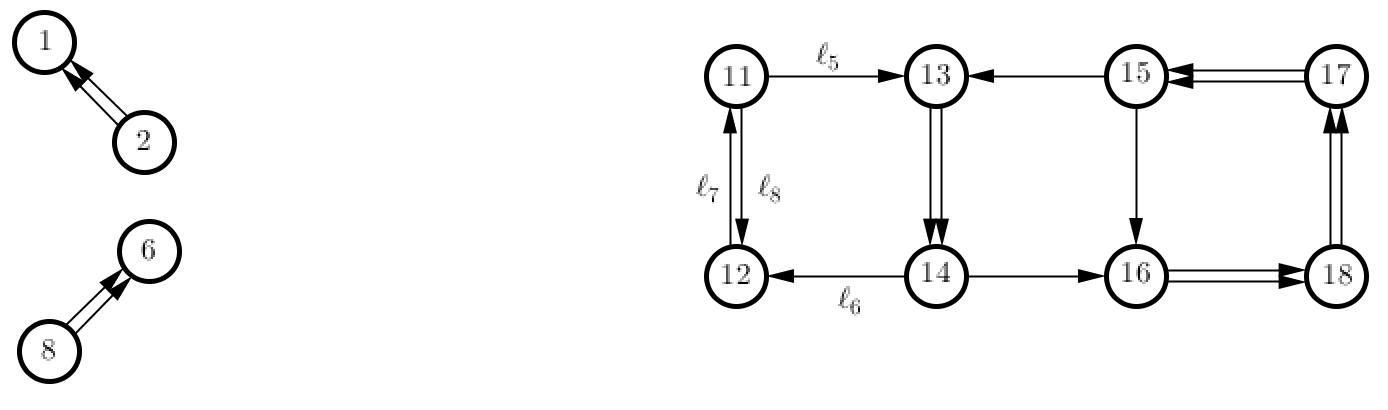}
  \caption{The molecule obtained after performing (3S3-5G) twice.}
  \label{fig:graphexample4}
\end{figure}

Next, we treat the two degree 3 atoms (labeled 11 and 12) connected by a double bond. Since the type II chain continues, as in (3-b) we claim a checkpoint and perform either (3D3-1) or (3D3-2G). In either case $\Mb$ is reduced to the one in Figure \ref{fig:graphexample5}.
  \begin{figure}[h!]
  \includegraphics[scale=.23]{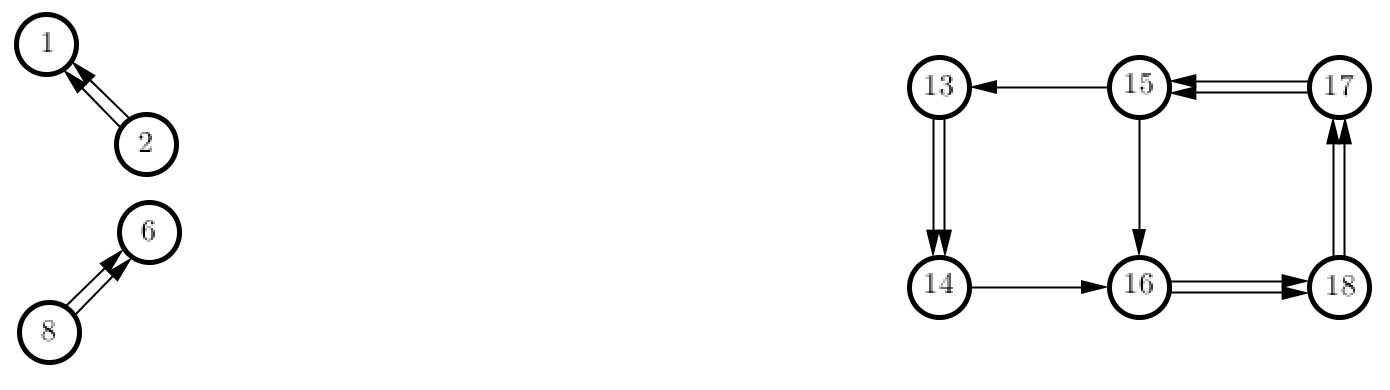}
  \caption{The molecule obtained after performing (3D3-1) or (3D3-2G). Again this is a checkpoint and the only difference between two possible steps is $\Delta\mathtt{Ext}$.}
  \label{fig:graphexample5}
\end{figure}

Next, we treat the two degree 3 atoms (labeled 13 and 14) connected by a double bond. The type II chain does not continue, so as in (3-c-ii) we perform (3D3-6G) and reduce $\Mb$ to the one in Figure \ref{fig:graphexample6}.
  \begin{figure}[h!]
  \includegraphics[scale=.23]{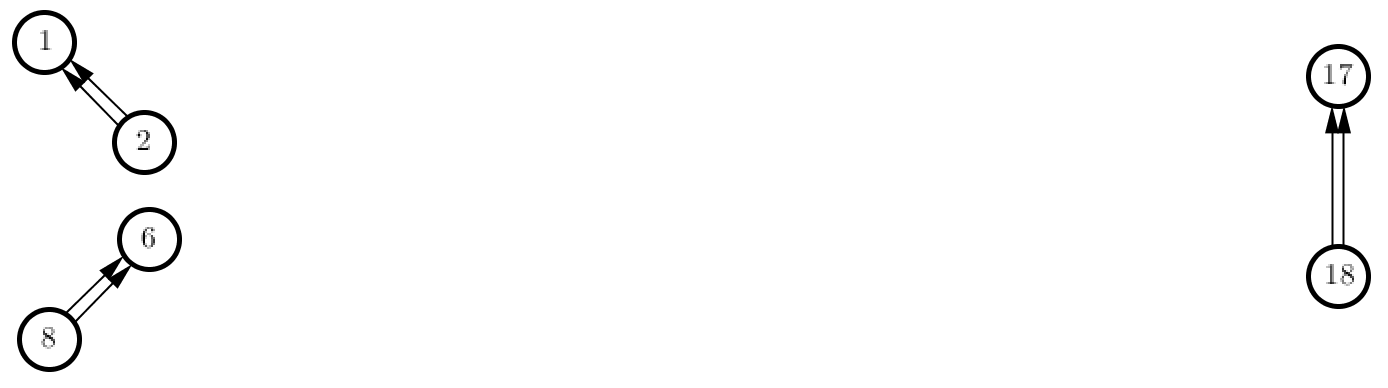}
  \caption{The molecule obtained after performing (3D3-6G).}
  \label{fig:graphexample6}
 \end{figure}
 
 Finally, we treat the remaining three pairs of degree 2 atoms connected by three double bonds. As in (7) we perform (2R-5) three times and reduce $\Mb$ to the empty graph.
 
Following the algorithm we have performed at least three good steps ($r\geq 3$). The two checkpoints provide four possible tracks, which correspond to different possibilities of $\mathtt{Ext}$ in the beginning; for example if we choose (3S3-1) and (3D3-1) then the $\mathtt{Ext}$ we obtain in the beginning is
\[\{k_{\ell_1}=k_{\ell_2},\,\,k_{\ell_3}=k_{\ell_4},\,\,k_{\ell_5}=k_{\ell_6},\,\,k_{\ell_7}-k_{\ell_8}\mathrm{\ is\ a\ good\ vector}\}.\] In this track (other tracks will have better exponents) we can calculate $\gamma=18-\frac{1}{2(d-1)}$ at the beginning, so we have, omitting powers of $\delta$, that
\[\sup\#\Df(\Mb,\mathtt{Ext})\lesssim L^{18(d-1)-\frac{1}{2}}.\]
\section{Table of notations} Here we list some important notations used in this paper. These are mainly concerned about trees, couples, molecules and their structures. Table \ref{table1} contains the basic notations and the corresponding symbols. Table \ref{table2} contains further notations, including different types of couples. Table \ref{table3} contains notations related to molecules.

\smallskip
\begin{tabular}{l c p{.45\textwidth}}
\toprule
Concept           &  Symbol                & Where defined \\
\midrule
Tree              &$\Tc$           & Definition \ref{deftree}\\
Root, node, leaf  &$\rf,\nf,\lf$ &Definition \ref{deftree}\\
Leaf set (tree) & $\Lc$ & Definition \ref{deftree}\\
Branching node set (tree) &$\Nc$ &Definition \ref{deftree}\\
Sign & $\zeta,\zeta_{\nf}$ & Definition \ref{deftree}\\
Scale (tree) & $n(\Tc)$ & Definition \ref{deftree}\\
Couple & $\Qc$ & Definition \ref{defcouple}\\
Leaf set (couple) & $\Lc^*$ & Definition \ref{defcouple}\\
Branching node set (couple) & $\Nc^*$ & Definition \ref{defcouple}\\
Scale (couple) & $n(\Qc)$ & Definition \ref{defcouple}\\
Paired tree, saturated paired tree & --- & Definition \ref{defcouple}\\
Lone leaf & --- & Definition \ref{defcouple}\\
Decoration & $\Ds,\Es$ & Definition \ref{defdec}\\
\bottomrule
\end{tabular}
\captionof{table}{Basic notations about trees and couples.}\label{table1}

\begin{tabular}{l c p{.4\textwidth}}
\toprule
Concept           &  Symbol                & Where defined \\
\midrule
$(1,1)$-mini couple, mini tree              &---           & Definition \ref{defmini}\\
Code (mini couple, mini tree) &---&Definition \ref{defmini}\\
Regular couple & --- & Definition \ref{defreg}\\
Legal partition, dominant partition & $\Pc$ & Definition \ref{deflegal}\\
Regular chain, regular double chain &---&Definition \ref{defregchain}\\
Type (regular couple) &---& Proposition \ref{structure1.5}\\
Prime couple &---&Definition \ref{defsub}\\
Skeleton & $\Qc_{sk}$ & Proposition \ref{reduceprocess}\\
Regular tree &---&Remark \ref{regtree}\\
Dominant couple &---&Definition \ref{defstd}\\
Special set &$Z$& Definition \ref{equivcpl}\\
Equivalence (dominant couple)&---& Definition \ref{equivcpl}\\
Encoded tree, equivalence (encoded tree) &---& Section \ref{encodedtree}, Definition \ref{equcodetree}\\
Associated encoded tree &---&Definition \ref{cpl-tree}\\
Irregular chain &$\Hc,\Hc^\circ$&Definition \ref{irrechain}\\
Congruence &---&Definition \ref{equivirrechain}, \ref{conggen}\\
\bottomrule
\end{tabular}
\captionof{table}{Further notations about trees and couples.}\label{table2}
\begin{tabular}{l c p{.5\textwidth}}
\toprule
Concept           &  Symbol                & Where defined \\
\midrule
Molecule, atom, bond             &$\Mb,v,\ell$           & Definition \ref{defmole0}\\
Saturated component &---& Definition \ref{defmole0}\\
Base molecule &---&Proposition \ref{moleprop}\\
Molecule associated to a couple &---& Definition \ref{defmole}\\
Type I and II chains &---& Definition \ref{molechain}\\
Degenerate atom, tame atom &---& Definition \ref{countingproblem}\\
Extra conditions &$\mathtt{Ext}$& Definition \ref{countingproblem}\\
Step, track, checkpoint &---& Section \ref{moleframe}\\
Bridge, special bond &---& Definition \ref{defbridge}\\
Good step, normal step &---& Beginning of Section \ref{oper}\\
\bottomrule
\end{tabular}
\captionof{table}{Notations about molecules.}\label{table3}

\end{document}